\documentclass{amsart}

\usepackage{amsmath,amssymb,amsthm, mathrsfs, mathtools, bm, amsfonts}
\usepackage{mathabx}
\usepackage{amscd,mathtools}
\usepackage[utf8]{inputenc}
\usepackage[english]{babel}
\usepackage{cite}
\usepackage{color}
\usepackage[pagebackref,colorlinks,citecolor=blue,linkcolor=red]{hyperref}

\renewcommand*{\backrefalt}[4]{\ifcase #1 (Not cited).\or (Cited p.~#2).\else (Cited pp.~#2).\fi} 
\usepackage{float}
\usepackage{tikz}
\usepackage{lmodern}
\usetikzlibrary{decorations.pathmorphing}

\usepackage{multicol}
\tikzset{snake it/.style={decorate, decoration=snake}}
\usetikzlibrary{automata}
\usetikzlibrary{cd}
\usepackage[margin=3cm]{geometry}

\usepackage{comment}
\usepackage{mathtools}

\usepackage{csquotes}

\usepackage{amsthm}
\usepackage{amssymb}

\usepackage{tikz}
\usetikzlibrary{calc}
\usetikzlibrary{arrows}
\usetikzlibrary{decorations.pathreplacing}
\usetikzlibrary{intersections}

\usepackage{hyperref}
\usepackage{tikz}
\usepackage{caption, subcaption}
\usepackage{tikz-cd}
\usetikzlibrary{calc,decorations.pathreplacing}
\usepackage{tikz}
\usetikzlibrary{calc}
\usetikzlibrary{arrows}
\usetikzlibrary{decorations.pathreplacing}
\usetikzlibrary{intersections}

 \newtheorem{theorem}{Theorem}[section]
  \newtheorem{proposition}[theorem]{Proposition}

         \usepackage{thmtools}
\usepackage{thm-restate}

  \newtheorem{corollary}[theorem]{Corollary}
  \newtheorem{lemma}[theorem]{Lemma}

  \newtheorem{question}[theorem]{Question}

  \newtheorem{introthm}{Theorem}
  \newtheorem{introcor}[introthm]{Corollary}

  \theoremstyle{definition}
  \newtheorem{definition}[theorem]{Definition}
  \newtheorem{claim}[theorem]{Claim}
  
  \newtheorem*{claim*}{Claim}

    \newtheorem{notation}[theorem]{Notation}

  \newtheorem*{question*}{Question}
  \newtheorem*{answer*}{Answer}
  \newtheorem*{application*}{Application}

  \theoremstyle{remark}
  \newtheorem{remark}[theorem]{Remark}
  \newtheorem*{remark*}{Remark}

\newcommand{\transverse}{\pitchfork}
\newcommand{\orth}{\perp}
\newcommand{\nest}{\sqsubset}

\newcommand{\hpi}{\widehat{\pi}}
\newcommand{\hrho}{\widehat{\rho}}

\newcommand{\symdiff}{\bigtriangleup}

\newcommand{\MCG}{\mathrm{MCG}}

\newcommand{\ep}{\epsilon}
\newcommand{\ltree}{\lambda}

\newcommand{\hT}{\widehat{T}}
\newcommand{\hd}{\hat{\delta}}
\newcommand{\hull}{\mathrm{hull}}

 \newcommand{\calA}{\mathcal{A}}
  \newcommand{\calB}{\mathcal{B}}
  \newcommand{\calC}{\mathcal{C}}
  \newcommand{\calD}{\mathcal{D}}
  \newcommand{\calE}{\mathcal{E}}
  \newcommand{\calF}{\mathcal{F}}
  \newcommand{\calG}{\mathcal{G}}
  \newcommand{\calH}{\mathcal{H}}
  \newcommand{\calI}{\mathcal{I}}

  \newcommand{\calN}{\mathcal{N}}

  \newcommand{\calQ}{\mathcal{Q}}
  \newcommand{\calR}{\mathcal{R}}
  
\newcommand{\calU}{\mathcal{U}}
  
    \newcommand{\calV}{\mathcal{V}}
  \newcommand{\calW}{\mathcal{W}}
  \newcommand{\calX}{\mathcal{X}}
    \newcommand{\calY}{\mathcal{Y}}
  \newcommand{\calZ}{\mathcal{Z}}
  \newcommand{\fg}{\mathfrak{g}}
  \newcommand{\ff}{\mathfrak{f}}

  \newcommand{\diff}{\mathrm{diff}}

  \newcommand{\bT}{\mathbb{T}}

  \newcommand{\wa}{\widehat{\alpha}}

      \newcommand{\diam}{\mathrm{diam}}

\newcommand{\Teich}{\mathrm{Teich}}

  \newcommand{\hb}{\hat{b}}
\newcommand{\ha}{\hat{a}}

    \newcommand{\hy}{\hat{y}}
\newcommand{\hx}{\hat{x}}

    \newcommand{\hO}{\widehat{\Omega}}
    
    \newcommand{\hpsi}{\hat{\psi}}
    \newcommand{\hPsi}{\widehat{\Psi}}
        
    \newcommand{\hf}{\hat{f}}

    \newcommand{\go}{\mathfrak{o}}
    \newcommand{\dc}{\mathbf{dc}}

    \newcommand{\spor}{\mathrm{spor}}

 \newcommand{\hh}{{\sf h}} 
    \newcommand{\PP}{\mathbf{P}}

    \newcommand{\fake}{\mathrm{fake}}

    \newcommand{\Rel}{\mathrm{Rel}}

\newcommand{\cuco}[1]{{\mathcal #1}}

    \newcommand{\ES}{E_{\mathfrak S}}

    \newcommand{\hg}{\hat{g}}

 \usepackage{mathabx}
\newcommand{\propnest}{\sqsubsetneq}

\usepackage[shortlabels]{enumitem}

\newtheorem{rem}{Remark}

\catcode`\@=11

\long\def\@savemarbox#1#2{\global\setbox#1\vtop{\hsize\marginparwidth 
  \@parboxrestore\tiny\raggedright #2}}
\catcode`\@=12

\begin{document}

\title[Asymp CAT(0) spaces, $\calZ$-structures, and the FJC]{Asymptotically CAT(0) spaces, $\calZ$-structures, and the Farrell--Jones Conjecture
}

 \author   {Matthew Gentry Durham}
 \address{Department of Mathematics, University of California, Riverside, CA }
 \email{mdurham@ucr.edu}
 
  \author   {Yair Minsky}
 \address{Department of Mathematics, Yale University, New Haven, CT }
 \email{yair.minsky@yale.edu}

\author{Alessandro Sisto}
    \address{Department of Mathematics, Heriot-Watt University and Maxwell Institute for Mathematical Sciences, Edinburgh, UK}
    \email{a.sisto@hw.ac.uk}

\begin{abstract}
We show that colorable hierarchically hyperbolic groups (HHGs) admit asymptotically CAT(0) metrics, that is, roughly, metrics where the CAT(0) inequality holds up to sublinear error in the size of the triangle.

We use the asymptotically CAT(0) metrics to construct contractible simplicial complexes and compactifications that provide $\calZ$-structures in the sense of Bestvina and Dranishnikov. It was previously unknown that mapping class groups are asymptotically CAT(0) and admit $\calZ$-structures. As an application, we prove that many HHGs satisfy the Farrell--Jones Conjecture, including extra large-type Artin groups.

To construct asymptotically CAT(0) metrics, we show that hulls of finitely many points in a colorable HHGs can be approximated by CAT(0) cube complexes in a way that adding a point to the finite set corresponds, up to finitely many hyperplanes deletions, to a convex embedding. 
\end{abstract}

\maketitle
\setcounter{tocdepth}{1}
\tableofcontents

\section{Introduction}

Hierarchically hyperbolic groups (HHGs) form a very large class of groups which includes hyperbolic groups, mapping class groups (of finite-type surfaces), and fundamental groups of compact special cube complexes, see e.g. \cite{HHS_II,HS:cubical, BR:combination,BR:graph_prods, BHMS:kill_twists,DDLS,MMS:Artin} for more examples. Many HHGs, including mapping class groups \cite[Theorem 4.2]{KL:actions}\cite[Theorem II.7.26]{bridson-haefliger}, are not CAT(0) groups. Nonetheless, in this paper we show that all natural examples of HHGs admit a coarse version of a CAT(0) metric, which following Kar \cite{Kar_asymp} we call an \emph{asymptotically CAT(0)} metric. Roughly, this means that triangles of rough geodesics satisfy the CAT(0) inequality up to a sublinear error in the size of the triangle, see Definition \ref{defn:asymp CAT0}. We note that our version of the definition does not require the space to be geodesic, but rather roughly geodesic, and in this regard it is more general than Kar's.

\begin{introthm}\label{thm:asymp CAT(0) intro}
    Every colorable HHG $G$ admits a $G$-invariant asymptotically CAT(0) metric equivariantly quasi-isometric to word metrics.
\end{introthm}

The colorability assumption is satisfied by all naturally occurring examples, with the only non-colorable known HHGs having been constructed specifically to fail this property \cite{Hagen:non-col}, see Definition \ref{defn:colorable}. Mapping class groups were not known to be asymptotically CAT(0), and prior to this work there were very few known examples of asymptotically CAT(0) spaces which were neither CAT(0) nor hyperbolic. The $G$-invariance is crucial for all our applications, but a non-$G$-invariant version of the theorem is already known since colorable HHGs are quasi-isometric to CAT(0) cube complexes \cite{petyt2021mapping}.

Theorem \ref{thm:asymp CAT(0) intro} can be compared to Haettel--Hoda--Petyt's result that HHGs act properly and coboundedly on injective spaces \cite{HHP} (see also \cite{petyt2024constructing}), another kind of non-positive curvature. However, while their result has several strong applications, we were not able to obtain any of the applications below using these metrics, which was our initial plan.

We will give more details on this later, but the construction of asymptotically CAT(0) metrics relies on a result on stability of cubulations of hulls of finitely many points (Theorem \ref{thm:LQC}) which improves on our previous work \cite{DMS_bary, Dur_infcube}.

We then import cubical metrics from these cubical approximations (Theorem \ref{thm:cubical metrics intro}), and importing the CAT(0) metric induces the required asymptotically CAT(0) metric.

Asymptotically CAT(0) spaces are amenable to standard techniques from CAT(0) and hyperbolic geometry, as they are a common generalization of both.  One such feature is coarse contractibility.  CAT(0) spaces are contractible, while Rips proved that any sufficiently deep Vietoris--Rips \cite{vietoris1927hoheren} complex over a hyperbolic group is contractible.  In \cite{Zaremsky_contract}, Zaremsky proved that asymptotically CAT(0) spaces are coarsely contractible in this sense (see also Theorem \ref{prop:Rips contract} where we deal with possibly non-geodesic spaces), hence we obtain the same for colorable HHGs with the metric from Theorem \ref{thm:asymp CAT(0) intro}.

We next prove that every asymptotically CAT(0) space admits a natural visual compactification (Theorem \ref{thm:compact_boundary}). In fact, for asymptotically CAT(0) groups of finite Assouad-Nagata dimension, we show that compactifying Vietoris--Rips complexes with this boundary yields $\calZ$-structures in the sense of Bestvina \cite{bestvina1996local, BestMess} and Dranishnikov \cite{Dra:BM_formula} (who extended Bestvina's notion to allow for groups with torsion). Roughly, this means that the compactified space is a Euclidean retract, that is, it can be embedded into some $\mathbb R^n$ as a retract, and the boundary can be ``locally homotoped'' inside the space. We recall the full definition in Definition \ref{defn:Z-space}. As a consequence of our results we obtain:

\begin{introthm}\label{thm:Z-boundary intro}
Colorable HHGs admit $\calZ$-structures.
\end{introthm}

In fact, the $\calZ$-structures we construct are $EZ$-structures as defined in \cite{FL:EZ} in the case of torsion-free groups (groups with torsion don't have a free action on the ``interior'').  In \cite[Subsection 3.1]{bestvina1996local}, Bestvina asks whether every group $G$ of type $F$, i.e. with a finite $K(G,1)$, admits a $\calZ$-structure, and given the more general definition of $\calZ$-structure \cite{Dra:BM_formula}, the same question can be asked for groups with a finite dimensional classifying space for proper actions.  Hence Theorem \ref{thm:Z-boundary intro} answers this decades old question for mapping class groups (and their finite-index subgroups), and at the same time for all colorable hierarchically hyperbolic groups. Beyond CAT(0) and hyperbolic groups, $\calZ$-structures were previously known to exist for systolic groups \cite{osajda2009boundaries}, Baumslag--Solitar groups \cite{guilbault2019boundaries}, torsion-free groups hyperbolic relative to a group admitting a finite classifying space \cite{dahmani2003classifying}, certain nonpositively curved complexes of groups \cite{martin2014non}, Helly groups \cite{CCGHO}, and groups acting geometrically on finite-dimensional spaces with suitable geodesic bicombings \cite{danielski2025boundaries}.

The existence of a $\calZ$-structure for a group is a powerful tool.  Our main application relates to the Farrell--Jones Conjecture \cite{farrell1993isomorphism} on $K$- and $L$-groups of group rings. This conjecture for specific groups has many applications towards the Borel Conjecture, classification of $h$-cobordisms, Kaplanski's conjecture, and several more.  We refer the reader to both Luck's and Bartels' ICM proceedings for more discussion and applications \cite{Lueck:ICM,Bartels:ICM}.

Using the axiomatic setup in \cite{BB:FJ_MCG} (which is the latest improvement on previous work on criteria to prove the Farrell--Jones Conjecture, see e.g. \cite{BL:hyp_CAT0,BLR:hyperbolic, BFL:lattices}), we show that, for a colorable HHG $G$, to prove the Farrell--Jones Conjecture for $G$ it suffices to proves it for the analogues of curve stabilizers in mapping class groups.

\begin{introthm}\label{thm:FJ intro}
    Let $(G,\mathfrak S)$ be a colorable HHG. Suppose that for all $U\propnest S$, where $S$ is the $\nest$-maximal element of the HHS structure, we have that $Stab(U)$ satisfies the Farrell--Jones Conjecture. Then $G$ satisfies the Farrell--Jones Conjecture.
\end{introthm}

The theorem can be in fact used inductively in many cases. We can define an HHG to be \emph{decomposable} by declaring that hyperbolic groups are decomposable, HHGs that are virtual direct products and central extensions of decomposable HHGs are decomposable, and that an HHG is decomposable if it has a colorable HHG structure such that for all $U\propnest S$ we have that $Stab(U)$ is a decomposable HHG. Then, using well-known properties of the Farrell--Jones Conjecture from \cite{BLR:hyperbolic,BFL:lattices} similarly to \cite[Corollary 4.10]{BB:FJ_MCG} (e.g. that products of groups satisfying the Farrell--Jones Conjecture also satisfy the Farrell--Jones Conjecture), we get:

\begin{introcor}
\label{introcor:decomp}
    Decomposable HHGs satisfy the Farrell--Jones Conjecture.
\end{introcor}

This recovers the Farrell--Jones Conjecture for mapping class groups, proven in \cite{BB:FJ_MCG}, but also implies the Farrell--Jones conjecture in many cases that were not previously known. For instance, in the case of extra-large type Artin groups (and more generally Artin groups of large and hyperbolic type) the relevant stabilizers are simply central extensions of virtually free groups, see the description of the HHS structure given in \cite{MMS:Artin}, and therefore:

\begin{introcor}
    Extra-large type Artin groups satisfy the Farrell--Jones Conjecture.
\end{introcor}

Other new examples that are covered by Corollary \ref{introcor:decomp} include quotients of mapping class groups by large powers of Dehn twists (proven to be HHGs in \cite{BHMS:kill_twists}), as well as random quotients of decomposable HHGs in the sense of \cite{Random_quot}.

There are HHGs that are (colorable but) not decomposable, for instance Burger-Mozes groups \cite{burger2000lattices} (stabilizers of proper domains in their standard HHG structures coincide with the whole group, so Theorem \ref{thm:FJ intro} holds vacuously in that case).  These however do satisfy the Farrell--Jones Conjecture because they are CAT(0) \cite{BL:hyp_CAT0, Wegner:FJ}. In fact, we are not aware of a single concrete HHG for which the Farrell--Jones Conjecture is now not known to hold.

We note that, while we apply the same criterion for the Farrell--Jones Conjecture as in \cite{BB:FJ_MCG}, we apply it to a different space in the case of mapping class groups. Indeed, in \cite{BB:FJ_MCG} the authors use the compactification of Teichm\"uller space, while we apply the criterion to the compactification of a simplicial complex on which the mapping class group acts cocompactly.  Because of this, a simpler version of the criterion from \cite{BB:FJ_MCG} suffices for us.

We also note that, despite \cite{HHP}, the results on the Farrell--Jones Conjecture from \cite{Helly} do not apply to hierarchically hyperbolic groups, or even mapping class groups. In fact, the injective spaces from \cite{HHP} are not graphs and moreover the corresponding groups act coboundedly but not necessarily cocompactly. It is possible that one cannot improve their result, as a sufficiently strong version of the Flat Torus Theorem for injective spaces, yielding that centralizers virtually split, would prevent this for mapping class groups. Note that a version of the Flat Torus Theorem does hold for injective spaces \cite{DL:flat_torus}.

\subsection{Local quasi-cubicality and cubical metrics}\label{subsec:LQC intro}

Behrstock--Hagen--Sisto proved in \cite{BHS:quasi} that HHSs are locally approximated by CAT(0) cube complexes, a higher-rank generalization of the fact that Gromov hyperbolic spaces are locally approximated by simplicial trees.  This result has been the foundational for the resolution of a number of long-standing conjectures about mapping class groups, including Farb's Rank Conjecture \cite{BHS:quasi}, semi-hyperbolicity of mapping class groups and  bicombability of Teichm\"uller spaces \cite{DMS_bary, HHP, PZ_walls}, as well as uniqueness (up to bi-Lipschitz equivalence) of asymptotic cones for the former \cite{CRHK}.

Roughly speaking, \cite[Theorem F]{BHS:quasi} proved that the coarse convex hull of any finite set of points $F$ in an HHS $\calX$ is quasi-median, quasi-isometric to a CAT(0) cube complex $\calQ_F$, which we call a \emph{cubical model} for $F$.  As in the hyperbolic case, the quality of the quasi-isometry depends only on the number of points $\#F$ and the ambient HHS $\calX$ (e.g., the topology of $S$ when $\calX = \MCG(S)$).  This was later generalized to HHS-like coarse median spaces by Bowditch \cite{Bow:cubulation} and then completely rebuilt and extended to allow for modeling hulls of finitely-many interior points and hierarchy rays in \cite{Dur_infcube}.

In this paper, we prove a strong stabilization result for cubical approximations of any colorable HHS (Definition \ref{defn:colorable}), building on our previous work \cite{DMS_bary} and utilizing the refined construction from \cite{Dur_infcube}.  For motivation, we first turn to a discussion of how one might want to use cubical models to define metrics on HHSs.

Given a pair of points $a,b \in \calX$, their cubical model $\calQ_{a,b}$ is a CAT(0) cube complex with a uniform quasi-isometry $\hO_{a,b}:\calQ_{a,b} \to \hull_{\calX}(a,b)$ to their hierarchical hull.  The cube complex $\calQ_{a,b}$ has vertices $\widehat{a}, \widehat{b} \in \calQ^0_{a,b}$ so that $\hO_{a,b}(\widehat{a}) = a$ and $\hO_{a,b}(\widehat{b}) = b$.  In fact, $\calQ_{a,b}$ is the cubical convex hull of $\ha,\hb$.

Every CAT(0) cube complex admits a variety of $\ell^p$-metrics $d^p$ by extending the $\ell^p$-norm on each cube (for $1\leq p \leq \infty$); e.g., $d^2$ is the CAT(0) metric and $d^1$ is the combinatorial metric.  As such, one might hope that the function
\begin{equation}\label{eq:distance}
\hat{d}^p_{\calX}(a,b) = 
d^p_{\calQ_{a,b}}(\ha, \hb)
\end{equation}

defines a metric on $\calX$ for each choice of $1 \leq p \leq \infty$.

The main obstacle here is the triangle inequality.  When confirming the triangle inequality for a triple $a,b,c \in \calX$, one might hope to compare distances in the $2$-point models $\calQ_{a,b}, \calQ_{a,c}, \calQ_{b,c}$ with their respective distances in the $3$-point model $\calQ_{a,b,c}$.  However, while the constructions in \cite{BHS:quasi, Bow:cubulation, Dur_infcube} provide a quasi-isometric embedding of each of $\calQ_{a,b}, \calQ_{a,c}, \calQ_{b,c}$ into $\calQ_{a,b,c}$, the multiplicative constant creates an unbounded error in any triangle inequality calculation.

Our main technical result gives a construction of cubical models which is stable under addition of points, allowing us to remove the multiplicative error.  The following is somewhat informal, see Theorem \ref{thm:stabler cubulations} for a precise version:

\begin{introthm}[Local quasi-cubicality]\label{thm:LQC}
Given a colorable HHS $\calX$, one can $Aut(\calX)$-equivariantly assign to each finite subset $F\subseteq \calX$ a cubical model $\hO_F:\calQ_F \to \hull_{\calX}(F)$ in such a way that the following holds. Whenever we have finite sets $F\subseteq F'\subseteq \calX$, there exists an $L$-cubical convex embedding $\Phi:\calQ_F \to \calQ_{F'}$ so that $\hO_F$ and $\hO_{F'} \circ \Phi$ agree up to error $L = L(|F'|, \calX)>0$.
\end{introthm}

By an $L$-\emph{cubical convex embedding}, we mean that the map $\Phi:\calQ_F \to \calQ_{F'}$ becomes a cubical convex embedding when we delete at most $L$-many hyperplanes from $\calQ_F$ and $\calQ_{F'}$.  In particular, $\Phi$ is a $(1,L)$-quasi-isometric embedding.

Perhaps the main upshot here is that the family of functions $\hat{d}^p_{\calX}$ in Equation \eqref{eq:distance} satisfy the triangle inequality up to a bounded additive error.  Hence equivariantly adding this error gives a genuine $\mathrm{Isom}(\calX)$-invariant metric $d^p_{\calX}$ on $\calX$ for each $p$; see Subsection \ref{subsec:new metrics} for details.

\begin{introthm}[Cubical metrics]\label{thm:cubical metrics intro}
    For each $1 \leq p \leq \infty$ and any colorable HHS $\calX$, there exists $C = C(\calX)>0$ and an $\mathrm{Isom}(\calX)$-invariant metric $d^p_{\calX}$ which satisfies:
    \begin{enumerate}
        \item The identity map $id_{\calX}:(\calX, d_{\calX}) \to (\calX, d^p_{\calX})$ is a $(C,C)$-quasi-isometry.
        \item For any finite subset $F \subset \calX$, the cubical model map $\hO_F:(\calQ_F, d^p_{\calQ_F}) \to (\calX, d^p_{\calX})$ is a $(1,C)$-quasi-isometric embedding.
        \item $(\calX, d^p_{\calX})$ is roughly geodesic: Every pair of points in $\calX$ is connected by a $(1,C)$-quasi-geodesic in $d^p_{\calX}$, which is furthermore a hierarchy path of uniform quality.
    \end{enumerate}
\end{introthm}

In particular, our cubical metrics are all roughly geodesic.  As we discuss next in Subsection \ref{subsec:asymp CAT(0) intro}, the metric $d^2_{\calX}$ is asymptotically CAT(0).

The power of Theorems \ref{thm:LQC} and \ref{thm:cubical metrics intro} is that they allow us to make many arguments in a colorable HHS $\calX$, such as $\MCG(S)$ or $\Teich(S)$, by passing to appropriate cubical models and only paying the cost of an \emph{additive} error.  In this sense, any colorable HHS is \emph{locally quasi-cubical}, as one can interpolate between local cubical models for overlapping finite subsets of $\calX$ like local charts on a manifold which agree on the overlaps up to cubical almost-isomorphisms. 

Our proof of Theorem \ref{thm:LQC} depends on a substantial strengthening of our stabilization techniques from \cite{DMS_bary}, and crucially on the refined cubical model construction from \cite{Dur_infcube}. 

Finally, we note that while a large portion of this paper develops an array of powerful techniques within the framework of hierarchically hyperbolic spaces in order to prove Theorems \ref{thm:LQC} and \ref{thm:cubical metrics intro}, the initial part of the paper analyzes these new metrics, with the hierarchical techniques essentially in the background.

\subsection{Asymptotically CAT(0) metrics}\label{subsec:asymp CAT(0) intro}

Theorem \ref{thm:cubical metrics intro} allows us to construct a family of $\ell^p$-like metrics on colorable HHSs like $\MCG(S)$ and $\Teich(S)$ for $1 \leq p \leq \infty$.  Recent papers of Haettel--Hoda--Petyt \cite{HHP} and Petyt--Zalloum \cite{PZ_walls} also give constructions of $\ell^{\infty}$-like metrics on any HHS.  These metrics are \emph{coarsely injective}, a powerful non-positive curvature property.  While we expect that our metric $d^{\infty}_{\calX}$ is also coarsely injective, our main focus is on the intermediate values of $p$, of which \cite{HHP, PZ_walls} do not construct analogues.

The case of $p = 2$ is of the most immediate interest.  The metric $d^2$ on CAT(0) cube complex is CAT(0).  Hence one would hope that our new metric $d^2_{\calX}$ satisfies a weakened form of the CAT(0) property.  This is indeed the case.

By Theorem \ref{thm:cubical metrics intro}, any pair of points $a,b \in \calX$ can be connected by a $(1,C)$-quasi-geodesic in the metric $d^2_{\calX}$, where $C = C(\calX)>0$.  It is not hard to show (see Lemma \ref{lem:CAT0_1C}) that any triangle $\Delta$ of $(1,C)$-quasi-geodesics in a CAT(0) space satisfy the CAT(0) inequality up to an error roughly on the order of $C\sqrt{\diam(\Delta)} + C$, in particular sublinear in the size of $\Delta$.  Hence combining this fact with Theorem \ref{thm:LQC}, which says that the metric $d^2_{\calX}$ is roughly a metric on a CAT(0) space up to bounded additive error, we obtain:

\begin{introthm} \label{thm:asymp CAT(0) intro 2}
    For any colorable HHS $\calX$, $(\calX, d^2_{\calX})$ is asymptotically CAT(0).

\end{introthm}

In \cite{Kar_asymp}, Kar introduced the notion of an asymptotically CAT(0) space as a simultaneous generalization of CAT(0) and Gromov hyperbolic spaces.  Importantly, she proved that cocompact lattices in $\widetilde{SL_2(\mathbb{R})}$ (one of Thurston's eight geometries) are asymptotically CAT(0).  We note that beyond these foundational examples, we are unaware of any other significant classes of asymptotically CAT(0) groups or spaces.  Hence Theorem \ref{thm:asymp CAT(0) intro 2} adds a wide variety of new examples---namely every colorable HHS, including mapping class groups and Teichm\"uller spaces. Incidentally, we note that $\widetilde{SL_2(\mathbb{R})}$ is also a colorable HHS.

In \cite{Kar_asymp}, Kar proved that asymptotically CAT(0) groups (namely those acting geometrically on asymptotically CAT(0) spaces) are of type $\mathrm{FP}_{\infty}$ and have finitely-many conjugacy classes of finite subgroups.  Hence Theorem \ref{thm:asymp CAT(0) intro 2} recovers \cite[Theorem G]{HHP} for colorable HHGs.  They are also strongly shortcut, in the sense of Hoda \cite[Theorem D]{hoda2024strongly}, recovering \cite[Corollary E]{HHP} for colorable HHSs. As a final remark, Kar's definition is equivalent, for geodesic spaces, to all asymptotic cones being CAT(0), but it is not hard to prove that being asymptotically CAT(0) for a rough geodesic space still implies that all asymptotic cones are CAT(0). In particular, mapping class groups and colorable HHGs admit metrics equivariantly quasi-isometric to word metrics and with CAT(0) asymptotic cones. This improves on a result of Bowditch (building on work of Behrstock--Drutu--Sapir \cite{behrstock2011median}), who showed that asymptotic cones of mapping class groups admit CAT(0) metrics bilipschitz equivalent to metrics coming from word metrics \cite{Bow:CAT0_cone}. Also, Theorem \ref{thm:asymp CAT(0) intro 2} answers \cite[Question 37.6]{CRHK}.

\begin{remark*}
    Chatterji and Petyt pointed out to us a promising potential construction of asymptotically CAT(0) metrics arising from coarsely median-preserving quasi-isometric embeddings into products of hyperbolic spaces (such embeddings are known to exist for colorable HHGs \cite{HagenPetyt}, and have been used in a similar spirit in \cite[Proposition 35.2]{CRHK}). We believe that this is worth further investigation, but we were not able to construct asymptotically CAT(0) metrics from quasi-isometric embeddings that, while coarsely preserving medians, might have image which is not median quasi-convex (as in the case of colorable HHGs).
\end{remark*}

\subsection{Outline of the paper} \label{subsec:proof sketches}

In Section \ref{sec:prelim}, after a brief discussion of generalities on HHSs, we state the precise version of Theorem \ref{thm:LQC}, which is Theorem \ref{thm:stabler cubulations}. The proof of this theorem takes up most of this paper, but we will first use the theorem in Sections \ref{sec:asymp CAT(0)}--\ref{sec:FJ} as a black-box.

In Section \ref{sec:asymp CAT(0)} we construct asymptotically CAT(0) metrics, as well as other cubical metrics, proving Theorems \ref{thm:asymp CAT(0) intro} and \ref{thm:cubical metrics intro}.

The main result in Section \ref{sec:contract} is Theorem \ref{prop:Rips contract} on contractibility of Vietoris--Rips complexes of asymptotically CAT(0) spaces. This was essentially proven by Zaremsky \cite{Zaremsky_contract}, whose results are however stated for geodesic spaces. We give a complete argument for another reason as well, which is that in later sections we need a technical improvement on contractibility, given by Proposition \ref{prop:Rips_combing}.

In Section \ref{sec:boundary} we construct boundaries for asymptotically CAT(0) spaces. The main result here is Theorem \ref{thm:compact_boundary}, which summarizes the properties of our compactifications. We note that to construct the topology we in fact construct a ``weak metric'', that is, a function on pairs of points which satisfies a weakening of the triangle inequality.

In Section \ref{sec:finite_dim} we show that our boundaries have finite covering dimension provided that the asymptotically CAT(0) space has finite Assouad-Nagata dimension, a controlled version of asymptotic dimension. We show this in Theorem \ref{thm:finite_dim}.

In Section \ref{sec:ER} we complete the construction of $\calZ$-structures from compactifications of Vietoris--Rips complexes, in Theorem \ref{thm:Z}. This applies to asymptotically CAT(0) spaces of finite Assouad-Nagata dimension and where balls are uniformly locally finite. 

In Section \ref{sec:FJ}, we obtain our applications to the Farrell--Jones Conjecture, see Theorem \ref{thm:rel_FJ}, by checking that the axiomatic setup of \cite{BB:FJ_MCG} applies to our compactifications of Vietoris--Rips complexes for colorable HHGs.

At this point of the paper, we start the proof of Theorem \ref{thm:stabler cubulations}. In Section \ref{sec:controlling domains} we consider finite subsets $F\subseteq F'$ of an HHS and study how the collection of hyperbolic spaces where $F$ and $F'$ have large diameter projections can differ. This is where we use colorability, as it allows us to perturb the HHS projections to minimize the difference between the two.

In Section \ref{sec:stable trees} we consider abstract setups in a hyperbolic space modeling the data coming from a finite set in an HHS via projections. Roughly, we explain what happens when changing additional data associated to a finite set in Theorem \ref{thm:stable tree}. We need to keep track of extensive amounts of data, see Definition \ref{defn:stable decomp}, and indeed the argument is rather involved, we refer the reader to the discussion in the section for more details and heuristics. 

In Section \ref{sec:stabler trees} we analyze what happens in a single hyperbolic space from the HHS structure when passing from a finite set $F$ to a larger finite set $F'$, see Theorem \ref{thm:stable tree}.

Finally, in Section \ref{sec:stabler cubulations}, we put the information we obtained in the various hyperbolic spaces together to prove Theorem \ref{thm:stabler cubulations}.

\subsection*{Acknowledgments}

We would like to thank Jason Behrstock, Mladen Bestvina, Indira Chatterji, Daniel Groves, Thomas Haettel, Mark Hagen, Nima Hoda, Marissa Loving, Harsh Patil, Harry Petyt, Sam Taylor, Brandis Whitfield, Wenyuan Yang, and Abdul Zalloum for interesting conversations and useful comments.  Durham was partially supported by NSF grant DMS-1906487.  Minsky was partially supported by NSF grant DMS-2005328.

\section{Preliminaries and statement of the stabler hull cubulation theorem}
\label{sec:prelim}

In this section, state the main technical result of the paper, the Stabler Cubulations Theorem \ref{thm:stabler cubulations}.  In order to do so, we give some preliminaries on HHSs.

\subsection{Generalities on HHSs}

We refer the reader to \cite{HHS_survey} for generalities on HHSs, here we simply recall the main features of an HHS structure on a metric space $\cuco X$. The key data required by an HHS structure is a family $\{\calC(U)\}_{U\in\mathfrak S}$ of uniformly hyperbolic spaces and uniformly Lipschitz maps $\pi_U:\calX\to \calC(U)$. The elements of the index set $\mathfrak S$ are called domains, and there are three relations on $\mathfrak S$, namely orthogonality, nesting, and transversality, denoted $\orth,\nest,\transverse$ respectively. An automorphism of an HHS is, roughly, a map of the HHS to itself that comes with a permutation of $\mathfrak S$ and is compatible with the maps $\pi_U$ as above, and preserves the three relations on $\mathfrak S$.

For the purposes of the statement of Theorem \ref{thm:stabler cubulations} we recall two further facts. Firstly, any HHS $(\calX,\mathfrak S)$ has a coarse median structure, which in particular means that there exists a particular coarsely Lipschitz map $\mu:\calX^3\to\calX$. The map $\mu$ is coarsely determined by requiring that for all $x,y,z\in\calX$ and $U\in\mathfrak S$ we have that $\pi_U(\mu(x,y,z))$ is a coarse center for the (thin) triangle with vertices $\pi_U(x),\pi_U(y),\pi_U(z)$. Secondly, in HHSs there is a notion of hull of subsets. In a hyperbolic space, the hull of a set is simply the union of all geodesics connecting points of the set, while for $A\subseteq \calX$ the hull $\hull_{\calX}(A)$ is the set of all $x\in\calX$ that project $\theta$-close to the hull of $\pi_U(A)$ for all $U\in\mathfrak S$. Here, $\theta$ is a sufficiently large constant, and for any HHS we automatically fix one such constant, see \cite{HHS_II} for more details.

We note that when working with an HHS $(\calX, \mathfrak S)$ in this paper, essentially all of the constants that appear in the various definitions and statements depend in part on the ambient HHS.  We will simply refer to a constant depending on $\mathfrak S$ when this is the case.

\subsection{Cubulation of hulls}

We can now state the full version of our main theorem on cubulations of hulls in HHSs. Recall from the introduction that it roughly says that hulls of finitely many points can be approximated by CAT(0) cube complexes in a way that inclusions of finite sets correspond to convex embeddings up to finitely many hyperplane deletions. In Sections \ref{sec:asymp CAT(0)}--\ref{sec:FJ} we will use the theorem, while the proof is given in the later sections, completed in Section \ref{sec:stabler cubulations}.

\begin{restatable}[]{thm}{stabler}
\label{thm:stabler cubulations}
  Let $(\calX,\mathfrak S)$ be a $G$-colorable HHS for $G < \mathrm{Aut}(\mathfrak S)$. Then for each $k$ there exist $K, N$ depending on $k, \mathfrak S$ with the following properties.  To each subset $F\subseteq \calX$ of cardinality at most $k$ one can assign a triple $(\calQ_F,\Phi_F,\psi_F)$ satisfying:
 \begin{enumerate}
  \item $\calQ_F$ is a CAT(0) cube complex of dimension at most the maximal number of pairwise orthogonal domains of $(\calX,\mathfrak S)$,
  \item $\Phi_F:\calQ_F \to \hull_{\calX}(F)$ is a $K$--median
    $(K,K)$--quasi-isometry,
  \item $\psi_F:F\to\mathcal (\calQ_F)^{(0)}$ satisfies $d_{\cuco X}(\Phi_F \circ \psi_F(f), f) \leq  K$ for each $f\in F$.
 \end{enumerate}
Moreover, suppose that $F'\subseteq \calX$ is another subset of cardinality at most $k$, and $gF\subseteq N_1(F')$ for some $g\in G$. Choose any map $\iota_{F}:F\to F'$ such that  $d_\calX(\iota_{F}(f),gf)\leq 1$ for all $f\in F$. Then the following holds. There are CAT(0) cube complexes $\mathcal R_F, \mathcal R_{F'}$, which fit into a diagram 
  \begin{equation}\label{eq:Phi diagram}
   \begin{tikzcd}[ampersand replacement=\&]
      F\arrow[r,"\psi_{F}"]\arrow[dddd,"\iota_F" left] \&\calQ_{F} \arrow[ddr,"g\circ\Phi_{F}"] \arrow[d,"\eta" left] \&  \\
         \& \mathcal R_F  \arrow[dr,"\Phi_0" below]\arrow[dd,"\theta" left] \& \\
      \& \& \calX \\    
      \& \mathcal R_{F'} \arrow[ur,"\Phi'_0"] \& \\
      F'\arrow[r,"\psi_{F'}"]\&\calQ_{F'}\arrow[uur,"\ \ \ \Phi_{F'}" below] \arrow[u,"\eta'"] \& \\
   \end{tikzcd}
  \end{equation}
which commutes up to error at most $K$, where  $\theta$ is a convex embedding, 
$\Phi_0$ and $\Phi'_0$ are  $K$--median $(K,K)$--quasi-isometric embeddings,
and $\eta$ and $\eta'$ are hyperplane deletion maps that delete at most $N$ hyperplanes. The left side commutes exactly, that is, we have $\theta\circ\eta\circ\psi_F=\eta'\circ\psi_{F'}\circ\iota_{F}$. Finally, $\theta$ is an isomorphism if $d^{Haus}_{\calX}(gF,F')\leq 1$.
\end{restatable}

Please see Definition \ref{defn:colorable} for the definition of a colorable HHS, and Definition \ref{defn:hyperplane deletion} for the definition of a hyperplane deletion map.

\section{Asymptotically CAT(0) metrics from stable cubulations}\label{sec:asymp CAT(0)}

The main result of this section is Theorem \ref{thm:asymp CAT0}, which shows the existence of asymptotically CAT(0) metrics (Definition \ref{defn:asymp CAT0}) on suitable HHSs.

\subsection{Preliminary lemmas on CAT(0) cube complexes and spaces}

We first collect some basic facts about CAT(0) cube complexes and spaces. We will consider $\ell^p$ metrics on CAT(0) cube complexes, with particular emphasis on $p=2$, the CAT(0) metric. A large part of the arguments work for any $p$, though, and in particular for $p=\infty$, the injective metric.  See \cite{schwer2324cat} for generalities of cube complexes.

\begin{lemma}
\label{lem:convex_isometric}
 A convex embedding (in the combinatorial sense) between CAT(0) cube complexes is an isometric embedding with respect to the $\ell^p$ metrics for any $p\in [1,\infty]$.
\end{lemma}

\begin{proof}
 This follows from the fact that the (combinatorial) retraction onto the image of the embedding is easily seen to be 1-Lipschitz.
\end{proof}

The following is of interest to us because of the hyperplane deletion maps appearing in the stable cubulation theorem.

\begin{lemma}
\label{lem:hyp_collapse_qi}
 A hyperplane deletion map is a $(1,1)$-quasi-isometry with respect to any $\ell^p$ metric for $p\geq 1$.
\end{lemma}

\begin{proof}
A hyperplane deletion map $q:X\to Y$ is clearly 1-Lipschitz. Let $H'$ be the image in $Y$ of the hyperplane $H$ that has been collapsed. If some geodesic $\gamma$ connecting points $q(x)$, $q(y)$ does not intersect $H'$, then it clearly comes from a geodesic of $X$. If not, we can consider the first and last points $\xi,\eta$ of $\gamma\cap H'$, and construct a path $\alpha$ from $x$ to $y$ by concatenating ``lifts'' of the initial and terminal segments of $\gamma$, and a geodesic in the carrier of $H$. Lemma \ref{lem:convex_isometric} guarantees that the distance between $\xi$ and $\eta$ is realized by a path entirely contained in $H'$, and using this it is readily seen that the length of $\alpha$ is at most the length of $\gamma$ plus 1, easily implying the required conclusion.
\end{proof}

Finally, the following lemma is the key fact about CAT(0) geometry that allows us to construct asymptotically CAT(0) metrics.

\begin{lemma}
\label{lem:CAT0_1C} 
 Let $X$ be a CAT(0) space, and fix $C\geq 0$. If $z$ lies on a $(1,C)$-quasi-geodesic from $x$ to $y$, then
 $$d(z,[x,y])\leq 3\sqrt{C \min\{d(x,z),d(y,z)\}+C^2}.$$
\end{lemma}

\begin{proof}
It suffices to prove the statement for $X$ the Euclidean plane. In turn, in order to do so it suffices to consider points $x,y,z$ where $x$ is the origin, $y$ is of the form $(d,0)$ for some $d>0$, and $z$ is of the form $(a,\ell)$ for some $a$ and $\ell>0$, and we have to show that if
$$d(x,z)+d(z,y)\leq d(x,y)+3C\ \ \ (*)$$
then the inequality in the statement of the lemma holds. Here, the ``$3C$'' comes from the fact that points along a $(1,C)$-quasi-geodesic satisfy the triangle inequality up to an error of at most $3C$.

 First we treat the case where either $a<0$ or $a>d$. In fact, up to applying a reflection across the bisector of $x$ and $y$, we can reduce to the case $a<0$. In this case we have $d(z,[x,y])=d(x,z)$ and $d(y,z)\geq d(x,y)$. The latter and $(*)$ give $d(x,z)\leq 3C$, so that $d(z,[x,y])\leq 3C\leq 3\sqrt{C^2}$, and we are done.

Suppose now $0\leq a\leq d$, and set $b=d-a$. In this case $d(z,[x,y])=\ell$. Up to applying a reflection across the bisector of $x$ and $y$, we can assume $a\leq b$, and in particular $d(x,z)\leq d(y,z)$. Note that $(*)$ becomes $\sqrt{a^2+\ell^2}+\sqrt{b^2+\ell^2}\leq d+3C$, and in particular we have $\sqrt{a^2+\ell^2}+b\leq d+3C$, or $\sqrt{a^2+\ell^2}\leq a+3C$, since $d-b=a$. By taking squares and simplifying the ``$a^2$'' we get
$\ell^2\leq 6Ca+9C^2,$
leading to $\ell\leq 3\sqrt{Ca+C^2}$. We are done since $a\leq d(x,z)$.
\end{proof}

\subsection{Key lemma on hull inclusions}

From now and until the end of the section we fix a $G$-colorable HHS $(\cuco X,\mathfrak S)$, for $G<\mathrm{Aut}(\mathfrak S)$. We will often use the setup of Theorem \ref{thm:stabler cubulations}, in particular the CAT(0) cube complexes $\calQ_F$ and related objects.

The following lemma compares distances measured in the approximating CAT(0) cube complexes for two sets $F\subseteq F'$, and most of our uses of Theorem \ref{thm:stabler cubulations} factor through it.

\begin{lemma}
\label{lem:change_cube_cplx}
 For every $k\in\mathbb N$ there exists $C = C(\mathfrak S, k)>0$ with the following properties. Let $g\in G$ and $F,F'$ be subsets of $\cuco X$ with $gF\subseteq F'\subseteq \cuco X$ and $|F'|\leq k$. Let $x,y\in F$. Then, endowing both $\calQ_F$ and $\calQ_{F'}$ with the $\ell^p$ metric for some $p\in [1,\infty]$ (same $p$ for both), we have
$$|d_{\calQ_F}(\psi_F(x),\psi_F(y))- d_{\calQ_{F'}}(\psi_{F'}(gx),\psi_{F'}(gy))|\leq C.$$
\end{lemma}

\begin{proof}
Fix the setup of Theorem \ref{thm:stabler cubulations}, where we can take $\iota_F$ to be multiplication by $g$.

  By Lemma \ref{lem:convex_isometric} and Lemma \ref{lem:hyp_collapse_qi}, the maps $\eta,\theta,\eta'$ are $(1,C)$-quasi-isometries (for a uniform $C$) when all cube complexes in Diagram \ref{eq:Phi diagram} are endowed with their $\ell^p$ metrics. Hence, it suffices to observe that $\theta(\eta(\psi_F(x)))=\eta'(\psi_{F'}(x))$ and similarly for $y$ (since then, roughly, we can map $x$ and $y$ to the same CAT(0) cube complex $\mathcal R_{F'}$ using $(1,C)$-quasi-isometries). This holds since, the right-hand side can be written as $\eta'\circ\psi_{F'}\circ\iota_{F}(x)$, and we have  $\theta\circ\eta\circ\psi_F\circ\iota_F=\eta'\circ\psi_{F'}\circ\iota_{F}$ by Theorem \ref{thm:stabler cubulations}.
\end{proof}

\begin{remark}
     From now and until the end of the section, we fix $k=5$ for all applications of Theorem \ref{thm:stabler cubulations}, simply because this is the maximal number of points in the configurations that we will consider in this section. When we write, for instance, $\calQ_F$, $\psi_F$, etc., this should be interpreted as the output of the theorem for $k=5$.
\end{remark}

\subsection{Approximate comparison triangles and asymptotically CAT(0) spaces}

We now introduce approximate comparison triangles and asymptotically CAT(0) spaces.

Given 3 points $x,y,z$ in a metric space, we denote $\bar\Delta(x,y,z)$ a comparison triangle in $\mathbb E^2$, with vertices $\bar x,\bar y,\bar z$. Given a point $p$ on a $(1,C)$-quasi-geodesic $\gamma$ joining $x,y$, for a fixed $C$, we call a point $\bar p$ on the geodesic $[\bar x,\bar y]$ a \emph{comparison point} if $
|d(p,x)-d(\bar p, \bar x)|\leq C$.
 Also, we will denote $\delta_{x,y}(p)=\min\{d(x,p),d(p,y)\}-C$. When no confusion can arise, we simply use the notation $\delta(p)$.

 \begin{remark}\label{rmk:-C}
     The ``$-C$'' is for convenience only, and in particular it guarantees that if $\gamma:[0,a]\to \calX$ is a $(1,C)$-quasi-geodesic then $\delta(\gamma(t))=\delta_{\gamma(0),\gamma(a)}(\gamma(t))\leq t$.
 \end{remark}

\begin{definition}[Sublinear CAT(0)]\label{defn:sub CAT0}
    Given a sublinear and non-decreasing function $\kappa$, we say that a triangle $\Delta$ of $(1,C)$-quasi-geodesics satisfies the \emph{CAT(0) condition up to $\kappa$} if the following holds. Let $x,y,z$ be the vertices of the triangle, and let $p$ and $q$ be points on the triangle. Fixing a comparison triangle and comparison points, we have
$$d(p,q)\leq d(\bar p,\bar q)+\kappa(\delta(p))+\kappa(\delta(q)).$$
\end{definition}

For clarity, in the above if $p$ lies, for instance, on the side connecting $x$ to $y$ then by $\delta(p)$ we mean $\delta_{x,y}(p)$.

The following definition is a variation on Kar \cite[Definition 6]{Kar_asymp}:

\begin{definition}\label{defn:asymp CAT0}
    We say a metric space $X$ is \emph{asymptotically CAT(0)} if there exists $C>0$ and sublinear function $\kappa$ so that the following hold:
    \begin{enumerate}
        \item Every pair of points of $X$ is connected by a $(1,C)$-quasi-geodesic.
        \item Every triangle $\Delta$ of $(1,C)$-quasi-geodesics satisfies the CAT(0) condition up to $\kappa$.
    \end{enumerate}
\end{definition}

\subsection{Asymptotically CAT(0) metric for HHSs} \label{subsec:new metrics}

Given $x,y\in \cuco X$, and with the notation of Theorem \ref{thm:stabler cubulations} with $F=\{x,y\}$, denote
$$\hat{d}_2(x,y)=d_{\calQ_F}(\psi_F(x),\psi_F(y)),$$
where $d_{\calQ_F}$ is the $CAT(0)$ metric on $\calQ_F$.

The following lemma shows that $\hat{d}_2$ satisfies the triangle inequality up to an additive error, and is coarsely $G$-equivariant. The triangle inequality roughly follows from the fact that the approximating CAT(0) cube complexes for pairs of points almost embed in those for three points.

\begin{lemma}
\label{lem:triangle_ineq}
There exists $C = C(\cuco X)>0$ so that for all $x,y,z\in \cuco X$, we have $\hat{d}_2(x,y)\leq \hat{d}_2(x,z)+\hat{d}_2(z,y)+C$. Moreover, for all $x,y\in\cuco X$ and $g\in G$ we have $|\hat{d}_2(x,y)-\hat{d}_2(gx,gy)|\leq C$.
\end{lemma}

\begin{proof}
 
 Let $F = \{x,y\}$ and $F' = \{x,y,z\}$.  By Lemma \ref{lem:change_cube_cplx}, $d_{\calQ_F}(\psi_F(x),\psi_F(y))$ and $d_{\calQ_{F'}}(\psi_{F'}(x),\psi_{F'}(y))$ differ by a bounded amount, and similarly for the other pairs from $\{x,y,z\}$.  Thus the claim follows from the triangle inequality in $\calQ_{F'}$.

Regarding the ``moreover'' part, we let $F=\{x,y\}$ and $F'=\{gx,gy\}$ and the conclusion follows immediately from Lemma \ref{lem:change_cube_cplx}.
\end{proof}

By adding to the metric the error in the approximate triangle inequality we can obtain an actual metric on $\cuco X$. Since we want a $G$-equivariant metric, we also take a supremum over orbits (which affects the metric only a bounded amount by the ``moreover'' part of Lemma \ref{lem:triangle_ineq}).

\begin{definition}\label{defn:ell 2}
For any $x,y \in \cuco X$, set

$$d_2(x,y) = \begin{cases}
0 & \textrm{if } x = y\\
\sup_{g\in G}\hat{d}_2(gx,gy) + C & \textrm{if } x \neq y.
\end{cases}
$$
\end{definition}

The metric $d_2$ satisfies the following useful properties:

\begin{proposition}{}
\label{prop:asymp_CAT0}
 There exist $C,D>0 $ depending only on $\cuco X$ with the following properties.
 \begin{enumerate}
  \item\label{item:quasi_old} The identity map $\mathrm{id}_{\cuco X}: (\cuco X, d_2) \to (\cuco X, d)$ is a $(C,C)$-quasiisometry.
  \item\label{item:rough_embed} Let $F = \{x,y\}$.  The map $\Phi_{F}: (\calQ_F, d_{\calQ_F}) \to (\cuco X,d_2)$ from Theorem \ref{thm:stabler cubulations} is a $(1,C$)-quasiisometric embedding.
  \item\label{item:rough_geod} Any two points of $\cuco X$ are joined by a $(1,C)$-quasi-geodesic of $d_2$.
  \item\label{item:approx_CAT} Given a triangle $\Delta$ of $(1,C)$-quasi-geodesics, $\Delta$ satisfies the CAT(0) condition up to $D\sqrt{t}+D$.
 \end{enumerate}
\end{proposition}

\begin{proof}
\eqref{item:quasi_old}: This is an immediate consequence of Theorem \ref{thm:stabler cubulations}.

\smallskip

\eqref{item:rough_embed}: Let $F = \{x,y\}$ and $a, b \in \calQ_F$. Set $a' = \Phi_F(a)$, $b' = \Phi_F(b)$, and $F' = \{x,y,a',b'\}$. We will consider the diagram from Theorem \ref{thm:stabler cubulations} for the inclusion $F\subseteq F'$ (with $g$ the identity).

For $F_0=\{a,b\}$, we have that $d_2(a',b')$ coarsely coincides with $d_{\calQ_{F_0}}(\psi_{F_0}(a'),\psi_{F_0}(b'))$. Applying Lemma \ref{lem:change_cube_cplx} to the inclusion $F_0\subseteq F'$ we get that $d_2(a',b')$ coarsely coincides with $d_{\calQ_{F'}}(\psi_{F'}(a'),\psi_{F'}(b'))$. On the other hand, $d_{\calQ_F}(a,b)$ coarsely coincides with $d_{\calR_{F'}}(\theta\circ\eta(a),\theta\circ\eta(b))$ since Lemma \ref{lem:hyp_collapse_qi} guarantees that $\eta$ is a $(1,C)$-quasi-isometric embedding for some uniform $C$ and Lemma \ref{lem:convex_isometric} says that $\theta$ is an isometric embedding. Since $\eta'$ is also a $(1,C)$-quasi-isometric embedding for some uniform $C$ (again by Lemma \ref{lem:hyp_collapse_qi}), to conclude it suffices to argue that $\bar a'=\eta'\circ\psi_{F'}(a')$ lies within bounded distance of $\bar a=\theta\circ\eta(a)$ in $\calR_{F'}$, and similarly for $b'$ and $b$. To do so, we argue that $\bar a$ and $\bar a'$ map uniformly close in $\calX$ under the quasi-isometric embedding $\Phi'_0$, as they both map uniformly close to $a'$. Indeed, by coarse commutativity of the diagram from Theorem \ref{thm:stabler cubulations}, $\Phi'_0(\bar a')=\Phi'_0\circ\eta'\circ\psi_{F'}(a')$ maps uniformly close to $\Phi_{F'}\circ\psi_{F'}(a')$, which coarsely coincides with $a'$ by Theorem \ref{thm:stabler cubulations}-(3). On the other hand, again by coarse commutativity we have that $\Phi'_0(\bar a)=\Phi'_0\circ\theta\circ\eta(a)$ coarsely coincides with $\Phi_F(a)=a'$, as required.

\smallskip

\eqref{item:rough_geod}: This is an immediate consequence of \eqref{item:rough_embed}, since for any $x, y\in \cuco X$, the CAT(0) geodesic in $\calQ_{\{x,y\}}$ between $x,y$ is sent via $\Phi_{\{x,y\}}$ to a $(1,C)$-quasigeodesic in $\cuco X$ between $x,y$.

\smallskip
\eqref{item:approx_CAT}: Let the vertices of the triangle as in the statement be $x,y,z$ and let $p,q$ be points on two of the sides, say on the quasi-geodesic between $x$ and $y$ and between $x$ and $z$ respectively. Denote $F'=\{x,y,z,p,q\}$ and let $\hat{x}=\psi_{F'}(x)$ and similarly for the others. By Lemma \ref{lem:change_cube_cplx} there exists a constant $C'$ depending on $\cuco X$ and $C$ such that all pairwise distances in $\calQ_{F'}$ between points in $\psi_{F'}(\{x,y,z,p,q\})$ coincide up to error at most $C'$ with the corresponding $d_2$-distance, that is, $|d_{\calQ_{F'}}(\hat x,\hat y)-d_2(x,y)|\leq C',|d_{\calQ_{F'}}(\hat x,\hat z)-d_2(x,z)|\leq C',$ etc. In particular, up to uniformly increasing $C'$ we have that $\hat{p}$ lies on a $(1,C')$-quasi-geodesic from $\hat{x}$ to $\hat{y}$, and similarly for $\hat{q}$. Applying Lemma \ref{lem:CAT0_1C} and the CAT(0) inequality in $\calQ_{F'}$ yields the desired conclusion.

\end{proof}

We can now state and complete the proof of the main theorem of this section.

\begin{theorem}\label{thm:asymp CAT0}
Let $(\cuco X,\mathfrak S)$ be a $G$-colorable hierarchically hyperbolic space for some $G<\mathrm{Aut}(\mathfrak S)$. Then $\cuco X$ admits a $G$-invariant asymptotically CAT(0) metric which is quasi-isometric to the original metric. Moreover, there exists $D$ such that any pair of points is joined by a $(1,C)$-quasi-geodesic which is a $D$-hierarchy path.
\end{theorem}

\begin{proof}
    We proved most of the claims in Proposition \ref{prop:asymp_CAT0}, and in particular we are only left to show the moreover part. This follows easily from item \eqref{item:rough_embed} of Proposition \ref{prop:asymp_CAT0}, as CAT(0) geodesics in CAT(0) cube complexes are median paths. Indeed, Theorem \ref{thm:stabler cubulations} guarantees that the images of CAT(0) geodesics in any $\calQ_F$ are quasi-median paths in $\cuco X$, and quasi-median paths are hierarchy paths \cite{BHS:quasi, russell2023convexity}.
\end{proof}

\section{Contractibility of the Vietoris--Rips complex over $\calX$}\label{sec:contract}

In this section, we prove that any Vietoris--Rips complex (with sufficiently high threshold) over an asymptotically CAT(0) space is contractible.  This is a key piece of building a $\calZ$-structure for such a space $\calX$, which we will do later.

\subsection{Standing assumption}
In what follows, we will assume that $\calX$ is asymptotically CAT(0) as in Definition \ref{defn:asymp CAT0}, and we fix the corresponding parameters $C$ and $\kappa$.

\subsection{Diacenters}

The first step is to define a notion of barycenter, which we call diacenter. Roughly, the diacenter of a finite set is a coarse midpoint for an arbitrary choice of furthest pair of points in the set.

\begin{definition}[Diacenter]\label{defn:diacenter}

For any finite subset $A \subset \cuco X$, a point $c\in\calX$ is called a \emph{diacenter} of $A$ if there exist two points $a,b\in A$ at maximal distance (among points in $A$) such that 
\begin{itemize}
    \item $d(a,b)\geq d(a,c)+d(c,b)-3C$, and
    \item $|d(a,c)-d(a,b)/2|\leq C$.
\end{itemize}
we fix once and for all a choice of diacenter $\dc(A)$ for each finite $A\subset\calX$, which we denote ``the'' diacenter. We set $\dc(\{x\})=x$ for all $x$.
\end{definition}
\begin{remark}
    Notice that a diacenter exists because of the existence of $(1,C)$-quasi-geodesics connecting any pair of points.
\end{remark}

The following proposition says that diacenters have a kind of contraction property under taking subsets, up to an additive error. It can be regarded as a consequence of the \emph{strict} convexity of CAT(0) metrics.

\begin{proposition}
\label{prop:unif_conv}
Let $\calX$ be asymptotically CAT(0), with parameters $C,\kappa$.
   There exists $\epsilon = \epsilon(\calX )\geq 0$ with $\epsilon < 1$ and $C' = C'(\calX)\geq 0$ such that the following holds.
   \begin{itemize}
       \item If $A\subseteq B$ are finite, then
   $d(\dc(A),\dc(B))\leq (1-\epsilon)\diam(B)+C'.$
   \end{itemize}
\end{proposition}

\begin{proof}
    Considering furthest pairs in $A$ and $B$, the proposition reduces to the following statement involving 4 (possibly non-distinct) points:

    \par\medskip

    $(*)$ There exist $\epsilon>0$ and $C'\geq 0$ such that the following holds. Let $p,q,r,s\in \cuco X$ and let $R=d(p,q)$. Suppose that all distances between the 4 points are at most $R$. Then $d(\ \dc(\{p,q\})\ ,\ \dc(\{r,s\})\ )\leq (1-\epsilon)R+C$.

    \par\medskip

    Let $a=\dc(\{p,q\})$ and $b=\dc(\{r,s\})$. Note that a comparison triangle for $p,r,s$ has diameter at most $R$, by convexity of the Euclidean metric. Hence, by the approximate CAT(0) inequality
    we get $d(p,b)\leq R+\kappa(R)$ (where we used that $\kappa$ is non-decreasing and $d(b)\leq R$). Similarly, considering the triangle with vertices $q,r,s$, we get $d(q,b)\leq R+\kappa(R)$.

    We now consider a triangle with vertices $p,q,b$, with the point $a$ on the side connecting $p,q$, as well as a comparison triangle and comparison point $\bar a$ for $a$. We can in fact choose $\bar a$ to be the midpoint of the side of the comparison triangle joining $\bar p,\bar q$. 

    \par\medskip
    
    {\bf Claim:} For any $\epsilon\in (0,1-\sqrt{3}/2)$ there exists $R_0=R_0(\kappa)$ such that if $R\geq R_0$ then we have
$$d(\bar a,\bar b)\leq (1-\epsilon)R.$$

    \begin{proof}
We are interested in a Euclidean triangle where all sides have length at most $R+\kappa(R)$, and more specifically specifically in the distance between a vertex $\bar b$ and the midpoint of the opposite edge, connecting $\bar p$ to $\bar q$, which has length exactly $R$.

We can assume that $\bar b$ is the origin, so that $d(\bar a,\bar b)=|(\bar p+\bar q)/2|$. We have the following identity:
$$\left|\frac{\bar p+\bar q}{2}\right|^2=|\bar p|^2/2+|\bar q|^2/2-|\bar p-\bar q|^2/4.$$
The first two terms on the right-hand side are bounded above, by $(R+\kappa(R))^2/2$, while the third one is equal to $R^2/4$. Hence, 
$$d(\bar a,\bar b)^2=\left|\frac{\bar p+\bar q}{2}\right|^2\leq  \frac{4(1+\kappa(R)/R)^4-1}{4}\cdot  R^2,$$
which yields 
$$d(\bar a,\bar b)\leq \frac{\sqrt{4(1+\kappa(R)/R)^4-1}}{2}\cdot R.$$
For $R$ tending to infinity, the coefficient that multiplies $R$ tends to $\sqrt{3}/2<1$, so for all sufficiently large $R$ the right hand-side is bounded above by $(1-\epsilon)R$.
    \end{proof}
        
Fixing any $\epsilon$ as in the Claim, by the asymptotically CAT(0) inequality for the triangle with vertices $p,q,b$ as above, we have $d(a,b)\leq (1-\epsilon)R+\kappa(R)$, if $R$ is sufficiently large.
       
Up to decreasing $\epsilon$ an arbitrarily small amount and for a suitably large $C'$ (that takes care of small values of $R$) this quantity can be bounded above by $(1-\epsilon)R+C'$, as required.
\end{proof}

\subsection{Contractibility of the Vietoris--Rips complex}

Recall that $\cuco X$ is a fixed asymptotically CAT(0) space.  For a constant $T$ and $B\subseteq \cuco X$ we denote by $R_T(B)$ the corresponding Vietoris--Rips complex, that is, the simplicial complex where vertices are points of $B$ and a finite subset of $B$ forms a simplex whenever the pairwise distances are at most $T$.

Vietoris--Rips complexes were introduced  by Vietoris \cite{vietoris1927hoheren}, and Rips famously proved any sufficiently deep Vietoris--Rips complex over a hyperbolic group is contractible \cite{Gromov87}.  

\begin{remark}
    Establishing contractibility of Vietoris--Rips complexes for groups in general is a tricky problem, and this does not hold in general \cite{bux2013higher}.  Virk \cite{virk2024contractibility} only recently proved that sufficiently deep Vietoris--Rips complexes over $\mathbb Z^n$ for $n=1,2,3$ with its standard word metrics are contractible (see also \cite{zaremsky2024contractible}).  More generally (but not covering the standard word metric case for $\mathbb Z^n$), Vietoris--Rips complexes over locally finite Helly graphs are contractible \cite[Lemma 5.28]{CCGHO}. 
\end{remark}

We will deal with simplicial homotopies below, by which we mean simplicial maps from triangulations of the product of a simplicial complex with an interval. 

\begin{theorem}\label{prop:Rips contract}
    For $\cuco X\neq \emptyset$ asymptotically CAT(0) with parameters $C,\kappa$ and all sufficiently large $T$, $R_T(\cuco X)$ is contractible. Moreover, again for all sufficiently large $T$, for any ball $B$ in $\cuco X$ and simplicial complex $\PP$, any simplicial map $\theta:\PP\to R_T(B)\subseteq R_T(\cuco X)$ can be homotoped to a constant map inside $R_T(N_{T+C}(B))$.
\end{theorem}

\begin{remark}
    Theorem \ref{prop:Rips contract} mildly generalizes work of Zaremsky \cite[Theorem 6.2]{Zaremsky_contract}, who developed a general criterion for geodesic spaces using Bestvina--Brady Morse theory \cite{bestvina1997morse}.  We expect an appropriate modification of his techniques  would recover Theorem \ref{prop:Rips contract}. 
\end{remark}

Note that it suffices to show the second half of the theorem, as it implies that all homotopy groups vanish.

For $A\leq B$ we will always regard $R_A(\cuco X)$ as a subcomplex of $R_B(\cuco X)$.

We first want to show that any simplicial map $\theta:\PP\to R_T(\cuco X)$ for $T$ sufficiently large can homotoped into $R_{T'}(\cuco X)$, for some $T'$ smaller than $T$, and moreover the homotopy only involves vertices of the Vietoris--Rips complex corresponding to points of $\cuco X$ in a controlled neighborhood of $\theta(\PP^{(0)})$.

From now on we identify the (geometric realization of the) barycentric subdivision of a simplicial complex with the complex itself as a topological space, so that in particular a map from a simplicial complex can be homotopic to a map from its barycentric subdivision.

\begin{lemma}\label{lem:subdivide_and_conquer}
Let $\epsilon,C'$ be as in Proposition \ref{prop:unif_conv}. For all sufficiently large $T$ the following holds.
 Let $\theta:\PP\to R_T(\cuco X)$, where $\PP$ is a finite simplicial complex, be a simplicial map. Then, for $T'=(1-\epsilon)T+C'$, $\theta$ is homotopic within $R_T(N_T(\theta(\PP^{(0)})))$ to a simplicial map $\hat\theta:\hat\PP\to R_{T'}(\cuco X)$, where $\hat\PP$ is the first barycentric subdivision of $\PP$.
\end{lemma}

\begin{proof}
    The map $\hat\theta$ maps each vertex of $\hat\PP$, which by definition is a collection of vertices $\{v_i\}$ of $\PP$ spanning a simplex, to $\dc(\theta(\{v_i\}))$ (seen as a vertex of $R_T(\cuco X)$). Then $\hat\theta$ extends to simplices provided that it maps vertex sets of simplices of $\hat\PP$ to vertex sets of simplices of $R_{T'}(\cuco X)$. In fact, since both simplicial complexes are flag complexes we only need to show this for edges. Now, the vertices $\{v_i\}$ and $\{w_j\}$ of $\hat\PP$ span an edge if only if one of the sets is contained in the other. But then the distance between their images via $\hat\theta$ is bounded above by $T'$ by Proposition \ref{prop:unif_conv}.

    We now have to show that $\hat\theta$ is homotopic to $\theta$, and in fact all simplices involved in the homotopy have vertex set contained in the union of the images of $\theta$ and $\hat\theta$, justifying the claim about the support of the homotopy (note that given a collection of vertices $\{v_i\}$ of $\PP$ spanning a simplex we have that $\dc(\theta(\{v_i\}))$ lies within $T$ of $\{v_i\}$ if $T$ is large). To do so, we define a simplicial map on a subdivision of $\PP\times [0,1]$ (the same one used in proofs of the excision axiom). The simplicial complex $\calQ$ homeomorphic to this subdivision is determined by the following: it has vertex set $\PP^{(0)}\sqcup \hat\PP^{(0)}$, it is flag, it contains the edges of $\PP$ and $\hat\PP$, and the vertex $\{v_i\}$ of $\hat\PP$ is connected by an edge to each $v_i$. Note that $\theta$ and $\hat\theta$ determine a map on the vertex set of $\calQ$, so we are left to show that the endpoints of each edge are mapped to vertices of $R_{T}(\cuco X)$ connected by an edge, when $T$ is large enough. This holds by assumption for the edges of $\PP$, while for the edges of $\hat\PP$ this follows from the facts that $\hat \theta$ maps $\hat\PP$ into $R_{T'}(\cuco X)$, and that $T'\leq T$ when $T$ is large enough. For the remaining type of edges, $\theta(v)=\hat\theta(\{v\})$, so again we can use the fact that $\hat \theta$ maps $\hat\PP$ into $R_{T'}(\cuco X)$.
\end{proof}

\begin{proof}[Proof of Theorem \ref{prop:Rips contract}]
We start by fixing constants.
\begin{itemize}
\item Let $\epsilon,C'$ be as in Proposition \ref{prop:unif_conv}.

\item Fix $T$ large enough that Lemma \ref{lem:subdivide_and_conquer} applies and, for $T'$ as in Lemma \ref{lem:subdivide_and_conquer}, we have $$T'+2\kappa(T+2C+2)+C\leq T.$$
\end{itemize}

Consider a simplicial map $\theta:\PP\to R_T(B)$ as in the statement. Applying Lemma \ref{lem:subdivide_and_conquer} we can homotope $\theta$ into $R_{T'}(N_{T}(B))$; we call the homotoped map $\theta'$.

We will now ``drag'' $\theta'$ along a choice of $(1,C)$-quasigeodesics towards the center $p$ of the ball $B$; call these segments for short. More formally, we will construct a homotopy between $\theta'$ and another map whose image is supported on a smaller ball.

For $v$ a vertex of $\PP$, let $\gamma_v$ be the corresponding segment, oriented so that $\gamma_v(0)=\theta'(v)$. Let $\ell=\lfloor T+C+2\rfloor$ (and if the domain of $\gamma_v$ is smaller than $[0,\ell]$, extend $\gamma_v$ by setting $\gamma_v(t)=p$ for larger values of $t$). We now define maps $\theta_i$, $i=0,\dots,\ell$, with $\theta_0=\theta'$. On a vertex $v$ of $\PP$ we define $\theta_i(v)=\gamma_v(i)$. We now claim that each $\theta_i$ extends to a simplicial map $\PP\to R_T(\cuco X)$ and that moreover $\theta_i$ is homotopic to $\theta_{i+1}$. Both facts hold provided that whenever vertices $v,w$ of $\PP$ are adjacent or coincide, and $|t-t'|\leq 1$ then $d(\gamma_v(t),\gamma_w(t'))\leq T$ (for the homotopy, note that this implies that each simplex of $\PP$ is mapped by both $\theta_i$ and $\theta_{i+1}$ inside a common simplex of $R_T(\cuco X)$). Using the asymptotically CAT(0) condition and convexity of the Euclidean metric, we indeed have $d(\gamma_v(t),\gamma_w(t'))\leq T'+C+\kappa(t)+\kappa(t')\leq T$, by the choice of $T$, as required.

Now, the map $\theta'$ is homotopic to $\theta_\ell$, and the image of the homotopy is contained in $R_T(\bigcup \gamma_v)$. We have that $\bigcup \gamma_v$ is contained in $N_C(N_T(B))$ (where recall that the image of $\theta'$ is contained in $R_T(N_T(B))$). Hence, for $R$ the radius of $B$, the image of $\theta_\ell$ is contained in the ball $B'$ with the same center as $B$ and radius $\max\{R+T-\ell-2C,0\}$. Since $\ell>T+2C+1$, we can homotope $\theta$ either to a point, or to the Vietoris--Rips complex of a ball with radius at least 1 smaller than that of $B$, and so a straightforward induction concludes the proof.
\end{proof}

We will also need another contractibility property of Vietoris--Rips complexes. Given a subset $\calY$ of $\calX$, a point $p$ of $\calX$, and a constant $M$, we consider the following subset of $\calX$. For each $x\in \calY$ choose a $(1,C)$-quasi-geodesic $\gamma_x$ from $p$ to $x$, and recall the quantity $\delta(q)=\delta_{p,x}(q)=\min\{d(p,q),d(x,q)\}-C$ associated to any point $q$ on such $\gamma_x$. We set
$$hull_M(\calY;p)=\{z\in\calX: \exists x\in \calY,q\in\gamma_x\ \ d(z,q)\leq M\kappa(\delta(q))\}.$$

\begin{proposition}\label{prop:Rips_combing}
    Let $\cuco X$ asymptotically CAT(0) with parameters $C,\kappa$ and suppose $\kappa\geq 1$. For all sufficiently large $T$ the following holds. For all $n$ there exists $M$ such that any point $p$ of $\calX$ and any simplicial map $\theta:\PP\to\calX$, where $\PP$ is an $n$-dimensional simplicial complex, we have that $\theta$ can be homotoped to a constant map in $R_T(hull_M(\theta(\PP^{(0)});p))$.
\end{proposition}

\begin{proof}
    We proceed inductively on skeleta of $\PP$. For the purposes of the induction, however, we need to retain more information. In this proof we will use subdivisions of $\PP$ and its skeleta, and use the same identification of the geometric realization of a subdivision with the original complex that we used for barycentric subdivisions.
    
    We will show that for all $n$ there exists $M$ such that any point $p$ of $\calX$ and any simplicial map $\theta:\PP\to\calX$, where $\PP$ is an $n$-dimensional simplicial complex, we have  a continuous map $H:\PP\times [0,\ell]\to R_T(\cuco X)$, for some integer $\ell$, such that the following hold.
    \begin{enumerate}
        \item $H(x,0)=\theta(x)$ and $H(x,1)=p$ for all $x\in \PP$.
        \item For all vertices $v$ of $\PP$ we have $H(v,t)=\gamma_v(t)$, where $\gamma_v$ is oriented so that $\gamma_v(0)=v$, and we set $\gamma_v(t)=p$ if $t$ is not in the domain of $\gamma_v$.
        \item If $x\in\PP$ lies in a simplex containing the vertex $v$ then $H(\{x\}\times [t,t+1] )$ is contained in $B(\gamma_v(t),M\kappa(\delta(\gamma_v(t))))$.
\item There is a subdivision of $\PP\times [0,\ell]$ such that $H$ is a simplicial map, each $\PP\times t$, for $t$ an integer, is a subcomplex, as is each $\Delta\times[t,t+1]$, for $\Delta$ a simplex and $t$ an integer.
    \end{enumerate}

    Notice that property (3) implies that the image of $H$ lies in $R_T(hull_M(\theta(\PP^{(0)});p))$.

  The statement easily holds for $n=0$, as we can use the $\gamma_x$ to construct the homotopy. Suppose that the statement holds for a given $n$, and let us show it for $n+1$.  We apply the statement to $\PP^{(n)}$ to obtain a map $H$, which we will extend to a map $\hat{H}:\PP\times[0,\ell]\to R_T(\cuco X)$. We first extend $H$ to all $\PP\times t$, for $t$ an integer, which requires extending $H$ from the boundary of each simplex times $t$ to the whole simplex times $t$. Fix a simplex $\Delta$ of $\PP$, and an arbitrary vertex $v$ of $\Delta$. First of all, we argue that $H(\partial\Delta\times t)$ is contained in the Vietoris--Rips complex of a controlled ball around $H(v,t)=\gamma_v(t)$. For ease of notation, in the rest of this paragraph we conflate balls with their Vietoris--Rips complexes, and drop various ``$\times t$''. By the inductive hypothesis, any simplex of $\partial \Delta$ containing $v$ gets mapped by $H$ inside $B(\gamma_v(t),M\kappa(\delta(\gamma_v(t))))$, and the remaining simplex of $\partial \Delta$ has image contained in $B(\gamma_w(t),M\kappa(\delta(\gamma_w(t))))$ for some other vertex $w$ of $\Delta$. We claim that we have
  $$\delta(\gamma_w(t))\leq M\kappa(\delta(\gamma_v(t)))+T.$$
  Indeed, first of all we have $d(\gamma_w(t),\gamma_v(t))\leq M\kappa(\delta(\gamma_v(t)))$ by (3). Also, $\gamma_v$ and $\gamma_w$ share an endpoint while their other endpoints lie within $T$ of each other. As a consequence, the distances between $\gamma_w(t)$ and the endpoints of $\gamma_w$ can differ by at most $M\kappa(\delta(\gamma_v(t)))+T$ from the corresponding distances for $\gamma_v(t)$, yielding the claim.  Since $M\kappa(\delta(\gamma_v(t)))+T\leq (M+T)\kappa(\delta(\gamma_v(t)))$ as $\kappa\geq 1$, we have that $H(\partial \Delta \times t)\subseteq B(\gamma_v(t),(2M+T)\kappa(\delta(\gamma_v(t))))$. We can now apply Theorem \ref{prop:Rips contract} to extend $H$ across $\Delta$ in a way that the extended map is contained in a slightly larger ball, namely the radius is bounded by $(2M+2T+C)\kappa(\delta(\gamma_v(t)))$.

  \begin{figure}[h]
    \begin{center}
    \includegraphics[scale=2]{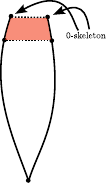}
\end{center}
\caption{A schematic of the two types of fillings we perform at the beginning of the inductive procedure on the skeleta. One is indicated by the dotted line and the other one by the shaded area.}
\end{figure}
  
  Similarly, we can then extend $H$ to each $\Delta\times[t,t+1]$, again using Theorem \ref{prop:Rips contract} as $H$ is now defined on the boundary of (the prism) $\Delta\times[t,t+1]$, and we can ensure that the image is contained in a ball of radius controlled linearly in $\kappa(\delta(\gamma_v(t)))$. This ensure that property (3) is retained up to a controlled increase of $M$, while notice that property (4) holds by construction. Property (1) is retained if the filling at times $0$ and $\ell$ are chosen to be constant and the one given by $\theta$ respectively, and property (2) is retained as we did not modify the homotopy on the 0-skeleton. This concludes the inductive argument.
  
\end{proof}

\section{Boundaries of asymptotically CAT(0) spaces}
\label{sec:boundary}

In this section, we introduce a notion of boundary for a discrete asymptotically CAT(0) space $\calX$. The main result of the Section is Theorem \ref{thm:compact_boundary} summarizing the crucial properties of our boundary. The idea is to take asymptotic classes of $(1,C)$-quasi-geodesic rays, where two rays are equivalent if they stay sublinearly close. In many ways, this boundary behaves like the visual boundary of a CAT(0) space, except we are forced to work with sublinear errors (via Proposition \ref{prop:asymp_CAT0}).  See Moran \cite{Moran_metric} for a similar metric construction in the genuine CAT(0) setting.

\begin{remark}
    We note that on page 77 of her thesis \cite{Kar_thesis}, Kar suggests a potential way to define a boundary for an asymptotically CAT(0) space, and even asks whether it would provide a $\calZ$-structure.  It would be interesting to flesh out her suggested construction and see how it compares with ours.
\end{remark}

\subsection{Standing assumptions}
    In this section, we will assume the following about $\calX$:
\begin{enumerate}
    \item $\calX$ is asymptotically CAT(0) as in Definition \ref{defn:asymp CAT0}.  We fix a constant $C$ such that any two points of $\cuco X$ are joined by a discrete $(1,C)$-quasi-geodesic. Moreover, we denote by $\kappa$ the function in the asymptotically CAT(0) condition.
\item Balls of $\calX$ are finite.
\end{enumerate}

Dealing with discrete geodesics and rays will be more convenient for us. In what follows we refer to discrete $(1,C)$-quasi-geodesic rays simply as \emph{rays}, and to discrete $(1,C)$-quasigeodesics as \emph{segments}.

\subsection{Sublinearly close rays}

We consider the following equivalence relation on the set of all rays in $\calX$. Given two rays $\gamma_1,\gamma_2$, we write $\gamma_1\sim_{sl}\gamma(t)$ if $t\mapsto d(\gamma_1(t),\gamma_2(t))$ is a sublinear function.  We remark that we are not requiring this sublinear function to be related to $\kappa$, but Corollary \ref{cor:unif_close} below says we can take it to be $4\kappa$.

We can now define our boundary $\mathfrak B \calX$ as a set:

\begin{definition}
    We define $\mathfrak B \calX$ to be the set of $\sim_{sl}$-equivalence classes of (discrete $(1,C)$-quasi-geodesic) rays in $\calX$. Similarly, for a fixed basepoint $\go\in\calX$ we define $\mathfrak B^\go \calX$ similarly but only considering rays starting at $\go$.
\end{definition}

As is the case for various notions of boundaries, it is important that boundary points have representatives at any given basepoint:

\begin{lemma}\label{lem:change_basepoint}
    For all $[\gamma]\in \mathfrak B \calX$ and all $\go\in\calX$ there exists a ray $\gamma'$ starting at $\go$ and such that $[\gamma]=[\gamma']$. Moreover, for any such $\gamma$ and $\gamma'$, and if $\gamma$ starts at $p$, then we have $d(\gamma(t),\gamma'(t))\leq 2\kappa(t)+d(\go,p)$.
\end{lemma}

\begin{proof}
Fix $\gamma,p,$ and $\go$ as in the statement. We can consider segments $\gamma'_n$ joining $\go$ to $\gamma(n)$, which have a subsequence which pointwise converges to a ray $\gamma'$ by Arzel\`a-Ascoli. We are left to show that, for a fixed $t$, we have $d(\gamma(t),\gamma_n(t))\leq 2\kappa(t)+d(\go,p)$ for all sufficiently large $n$, as this implies the required bound for $\gamma'$. This is a straightforward use of a comparison triangle.
\end{proof}

It will also be important that two equivalent rays starting at the same basepoint stay within distance bounded in terms of a given function, which follows directly from the case $p=\go$ of the previous lemma.  This is a sublinear version of the fact that points in the visual boundary of a CAT(0) space have unique geodesic representatives.

\begin{corollary}\label{cor:unif_close}
    If $\gamma,\gamma'$ are rays starting at the same point $\go\in\calX$ and $\gamma\sim_{sl}\gamma'$ then $d(\gamma(t),\gamma'(t))\leq 2\kappa(t)$ for all $t$.
\end{corollary}

\subsection{Weak metrics}

The next goal is to define a topology on $\mathfrak B \calX$. We will do so by endowing it with a structure similar enough to a metric:

\begin{definition}
    A function $d:X \times X \to \mathbb{R}_{\geq 0}$ on a space $X$ is a \emph{weak metric} if
    \begin{enumerate}
        \item $d(x,y) = d(y,x)$ for all $x, y \in X$,
       \item $d(x,x) = 0$,
    
        \item There exists a function $\ff:\mathbb{R}_{\geq 0} \to \mathbb{R}_{\geq 0}$ such that
        \begin{enumerate}
            \item $\ff$ is non-decreasing and $\lim_{t \to 0} \ff(t) = 0$, and
            \item For all $x,y,z \in \calX$, we have
             $$d(x,z) \leq \ff\left(\max \{d(x,y),d(y,z)\}\right).$$
        \end{enumerate}
    \end{enumerate}
\end{definition}

\begin{remark}
    Various weak versions of metrics appear frequently in the study of boundaries of groups.  See for instance the notion of a quasi-metric in the context of hyperbolic groups \cite{buyalo2007elements}. 
\end{remark}

Given a weak metric $d$ on a space $X$, we can consider (weak) metric balls $B(p,R)$ for $p \in X$ and $R>0$. We can also define a topology, which we refer to as the \emph{metric topology}, in the same way that one defines the topology coming from a metric: A set $U$ is open if for all $x \in U$ there is a ball $B(x,R)$, for some $R>0$, contained in $U$.

In the case of weak metrics it is not clear (and probably false in general) that balls are open, but we still have:

\begin{lemma} \label{lemma:weak metric balls}
    Let $d$ be a weak metric on a space $X$, and consider the metric topology on $X$.  Then for all $p \in X$ and $r>0$, we have that $p$ is contained in the interior of the (weak) metric ball $B(p,r)$. Moreover, let $r_0$ be such that $\ff(r_0)<r$. Then $B(p,r_0)$ is contained in the interior of of $B(p,r)$.
\end{lemma}

\begin{proof}
    Let $r_0$ be such that $\ff(r_0)<r$, and inductively let $r_i$ be such that $\ff(r_i)<r_{i-1}$ and $r_i<r_{i-1}$. Let $A_0=B(p,r_0)$ and inductively let $A_i=\bigcup_{x\in A_{i-1}} B(x,r_i)$. It is clear that $\bigcup A_i$ is open, and we will check that $A$ is contained in $B(p,r)$.

    Let $x\in A_i$, so that in particular $x\in B(x_{i-1},r_i)$ for some $x_{i-1}\in A_{i-1}$. We in fact have points $x_j$ for $j=-1,\dots i$ with $x=x_i, p=x_{-1}$ and $d(x_{j-1},x_{j})<r_j$.

For $j=i$ we obviously have $d(x,x_j)<r_j$. We now show inductively (starting at $j=i$ and ending with $j=-1$) that the same holds for all $j$. Indeed, $d(x,x_{j-1})\leq \ff (\max\{d(x_{j-1},x_j),d(x_j,x)\})$. The first term in the max is at most $r_j$ by construction, while we can assume that the second one is at most $r_j$ by the inductive hypothesis. Hence, $d(x,x_{j-1})\leq \ff (r_j)<r_{j-1}$, as required.
\end{proof}

\subsection{Weak metric on the boundary}
\label{subsec:weak_metric}

In this subsection, we implement a sublinear version of Moran's construction of a metric on the visual boundary of a CAT(0) space from \cite[Section 3]{Moran_metric}.  In Moran's construction, the idea is to measure how long it takes two CAT(0) geodesic rays to get further apart than some constant.  For our purposes, we want to see how long two asymptotic classes take to separate further apart than a fixed sublinear function.

Recall that we fixed a constant $C$ and a sublinear function $\kappa$, and that we call discrete $(1,C)$-quasi-geodesic rays and segments simply rays and segments. Fix a basepoint $\go\in\calX$. To uniformize notation for $\calX\cup \mathfrak B^\go\calX$, we identify a point $p$ of $\calX$ with the set of all segments $\gamma:\{0,\dots,a\}\to\calX$ from $\go$ to $p$.

Given $A,B\in \calX\cup \mathfrak B^\go\calX$ we define $d_\kappa(A,B)=0$ if $A=B$ and otherwise
$$d_{\kappa}(A,B) = \inf_{\alpha \in A, \beta \in B} \ \frac{1}{\sup \{t|d(\alpha(t), \beta(t)) \leq 3 \kappa(t)\}}.$$

For clarity, the supremum is taken over all $t$ such that $\alpha(t)$ and $\beta(t)$ are both defined. Note that the supremum is in fact realized as $\alpha,\beta$ are discrete.

For later purposes we also make the following remark:

\begin{remark}\label{rem:discrete}
    For all $x\in \calX$ there exists a radius $r$ such that the $d_\kappa$-ball around $x$ of radius $r$ only contains $x$. 
\end{remark}

While $d_\kappa$ does not obviously satisfy the triangle inequality, we will show that it is a weak metric.

\begin{proposition}\label{prop:weak metric}
The function $d_{\kappa}$ is a weak metric on $\bar\calX=\calX\cup\mathfrak B^\go\calX$.
\end{proposition}

We will prove the proposition below, after some preliminaries.

\subsection{The crucial divergence lemma}

The following lemma is crucial not only to prove that $d_\kappa$ is a weak metric, but also for our further study of the boundary. Roughly speaking, it says that if two rays are closer than some fixed (large) sublinear function $\eta$ at some time $t$, they are in fact at most $3\kappa$ apart at a previous time depending on $\eta$ and $t$ only. It can be seen as the counterpart of the fact that in a hyperbolic space if two rays are closer than a certain constant at some time, then they are closer than $\delta$ at a previous, controlled time.

\begin{lemma}\label{lem:weak divergence}(Uniform divergence lemma)
For all sublinear functions $\eta$ there exists a diverging function $\fg:\mathbb{R}_{\geq 0} \to \mathbb{R}_{\geq 0}$ such that the following holds. For any pair $\gamma_1, \gamma_2$ where each is a segment or a ray, if
$$d(\gamma_1(t), \gamma_2(t)) \leq \eta(t)$$
then
$$d(\gamma_1(t'), \gamma_2(t')) \leq 3 \kappa(t')$$
for all $t'\leq \fg(t)$.
\end{lemma}

\begin{proof}
We can set $\fg(t)=0$ for all $t\leq C$ and all values of $t$ such that $\eta(t)/(t-C)\geq \kappa(1)$. For other (larger) values of $t$, we let $\fg(t)$ be largest such that for all $t'\leq \fg(t)$ we have
$$\eta(t) \frac{t'}{t-C}\leq \kappa(t').$$

Note that $\fg$ is a diverging function since $\eta$ is sublinear.

    We consider a comparison triangle for $\go$, $\gamma_1(t), \gamma_2(t)$, where the sides from $\go$ to the $\gamma_i(t)$ are subpaths of the $\gamma_i$.  The comparison triangle has two sides of length at least $t-C$ and one side of length at most $\eta(t)$. Given $t'$, we have comparison points $\bar p, \bar q$ for $\gamma_1(t'), \gamma_2(t')$ on the former sides, each at distance at most $t'$ from the comparison point for $\go$ (unless the length of the side is less than $t'$, in which case we take the endnpoint of the side as comparison point). By basic Euclidean geometry we have
    $$d(\bar p,\bar q)\leq \eta(t) \frac{t'}{t-C}.$$
    Therefore, we in fact have $d(\bar p,\bar q)\leq \kappa(t')$. 
    By the CAT(0) condition up to $\kappa$, and $d(\gamma(t')),\ d((\gamma'(t'))\leq t'$ (see Remark \ref{rmk:-C}), we have
    $$d(\gamma_1(t'), \gamma_2(t')) \leq  2 \kappa(t')+d(\bar p,\bar q)\leq 3\kappa(t'),$$
    as required.   
\end{proof}

\subsection{Key properties of the boundary}

The following is not needed to show that $d_\kappa$ is a weak metric, but rather that $\bar X$ is Hausdorff; we include it here since it serves to illustrate the usefulness of the Uniform divergence lemma.

\begin{lemma}\label{lem:nonzero}
    If $A,B\in \calX\cup \mathfrak B^\go\calX$ are distinct, then $d_\kappa(A,B)>0$.
\end{lemma}

\begin{proof}
    It suffices to consider $A,B\in\mathfrak B^\go\calX$. Consider any representative rays $\alpha_0,\beta_0$. Since $A$ and $B$ are distinct, we have $d(\alpha_0(t_0),\beta_0(t_0))>3\kappa(t_0)$ for some $t_0$ (for any fixed sublinear function we can find some $t_0$ such that the distance at $t$ is larger than that function). Say that for some other representatives we have $d(\alpha(t),\beta(t))\leq 3\kappa(t)$ for some $t$; we have to give an upper bound on $t$. We also have $d(\alpha_0(t),\beta_0(t))\leq 7\kappa(t)$ by Corollary \ref{cor:unif_close}. By Lemma \ref{lem:weak divergence}, we then have $d(\alpha_0(t'),\beta_0(t'))\leq 3\kappa(t')$ for any $t'\leq \mathfrak g(t)$ for the diverging function as in the lemma, with $\eta=7\kappa$. But then we must have $\mathfrak g(t)\leq t_0$, which gives an upper bound for $t$, as required.
\end{proof}

In the proof that $d_\kappa$ is a metric we will need a coarse form of convexity for asymptotically CAT(0) metrics. The proof is a straightforward use of a comparison triangle, together with convexity of the metric of the Euclidean plane.

\begin{lemma}\label{lem:convex}
    Let $\calX$ be as in the standing assumptions, and let $\gamma,\gamma'$ be rays or segments originating from the same point $\go$ of $\calX$. Then for all $C< t\leq t'$ (such that the following formula is defined) we have
$$d(\gamma(t),\gamma'(t))\leq \frac{t}{t'-C}d(\gamma(t'),\gamma'(t'))+2\kappa(t).$$
\end{lemma}

\begin{proof}
Consider a comparison triangle for $\go$, $\gamma(t'),\gamma'(t')$. One side of it has length $d(\gamma(t'),\gamma'(t'))$ and the other two have lengths $\ell_1,\ell_2$ in $[t'-C,t'+C]$. The points at distance $\max\{t,\ell_i\}$ on those sides are comparison points for $\gamma(t),\gamma'(t)$. Basic Euclidean geometry implies that they lie within distance $\frac{t}{t'-C}d(\gamma(t'),\gamma'(t'))$ of each other, and the required inequality follows since $\delta(\gamma(t)),\delta(\gamma'(t))\leq t$ (see Remark \ref{rmk:-C}).
\end{proof}

We are now ready to prove that $d_\kappa$ is a weak metric:

\begin{proof}[Proof of Proposition \ref{prop:weak metric}]
Let $A,B,D$ lie in $\calX\cup \mathfrak B^\go\calX$, and we can assume that they are pairwise distinct so that by Lemma \ref{lem:nonzero} the (weak) distances between them are positive. Let $1/t=d_{\kappa}(A,B)$, $1/t'=d_\kappa(B,D)$, say with $t\geq t'$. We have to show that $d_\kappa(A,D)$ is bounded above by a function of $t$ which goes to $0$ as $t$ goes to $0$.

The assumption means that there exist $\alpha\in A$, $\beta,\beta'\in B, \gamma\in D$ such that $d(\alpha(t),\beta(t))\leq 3\kappa(t)$ and $d(\beta'(t'),\gamma(t'))\leq 3\kappa(t')$.

By Lemma \ref{lem:change_basepoint}, we have $d(\beta(t'),\gamma(t'))\leq 5\kappa(t')$. By Lemma \ref{lem:convex} we then see that $d(\beta(t),\gamma(t))\leq 2\kappa(t)+5\kappa(t')t/(t'-C)$. Therefore,
$$d(\alpha(t),\gamma(t))\leq 5\kappa(t)+5\kappa(t')t/(t'-C).$$
We can define a function $\eta(t)$ by taking the supremum of over all $t'\geq t$ of the expression on the right hand side, and it is readily seen that $\eta$ is sublinear. Therefore, we can apply Lemma \ref{lem:weak divergence}, which yields a diverging function $\fg$ and the inequality $d(\alpha(\fg(t)),\gamma(\fg(t)))\leq 3\kappa(\fg(t))$. This implies $d_\kappa(A,D)\leq 1/\fg(t)$, which is the required bound.
\end{proof}

We can identify $\mathfrak B\calX$ and $\mathfrak B^\go\calX$ as sets, in view of Lemma \ref{lem:change_basepoint}, and hence regard the weak metric $d_\kappa$ as defined on $\calX\cup\mathfrak B\calX$. This topology has many desirable properties, summarized in the main result of this section, which we state below.

\begin{theorem}
\label{thm:compact_boundary}
    Let $\calX$ be asymptotically CAT(0) with finite balls. Consider the topology on $\bar\calX=\calX\cup\mathfrak B\calX$ generated by metric balls of the weak metric $d_\kappa$. Then $\bar\calX$ is compact, Hausdorff, metrizable, and the $G$-action on $\calX$ extends continuously on $\bar\calX$, where $G=Isom(X)$. Moreover, for any radius $R>0$ and open cover $\calU$ of $\bar\calX$ we have that all but finitely many balls in $\calX$ of radius $R$ are $\calU$-small.
\end{theorem}

\begin{proof}
    \emph{Hausdorff}. Let $x,y \in \bar\calX$ be distinct, so that $d_{\kappa}(x,y)>0$ by Lemma \ref{lem:nonzero}.

    Let $B_r(x)$ and $B_r(y)$ be disjoint $d_{\kappa}$-balls around $x,y$, respectively, which exist since $d_\kappa(x,y)>0$ and since $d_\kappa$ is a weak metric.  Then Lemma \ref{lemma:weak metric balls} implies that the interiors of $B_r(x)$ and $B_s(y)$ are disjoint open sets containing $x,y$, so we are done.

    \par\medskip

    \emph{Sequentially compact}. We now show that $\bar\calX$ is sequentially compact. Any sequence that has infinitely many elements in some ball of $\calX$ subconverges, since balls in $\calX$ are finite. For a sequence $(x_n)$ without this property, we choose rays, resp. $(1,C)$-quasi-geodesics, $\gamma_n$ starting at $\go$ and representing, resp, ending at $x_n$. A point-wise limit of a subsequence of the $\gamma_n$ is a ray $\gamma$ which shares larger and larger initial segments with the $\gamma_n$ of the subsequence. It is readily seen that $(x_n)$ subconverges to $[\gamma]$.

    \par\medskip

    \emph{Compact metrizable}. Since $\bar\calX$ is sequentially compact and Hausdorff, by Urysohn's metrization theorem we only need to show that it is second-countable to show that it is compact and metrizable. We claim in fact that interiors of $d_{\kappa}$-balls of rational radius around points of $\calX$ generate the topology. To do so, it suffices to find for each $p\in \mathfrak B\calX$ and $r>0$ a $d_\kappa$-ball with rational radius $r'$, centered at a point $x$ of $\calX$, containing $p$ in its interior, and contained in the $d_\kappa$-ball $B$ of radius $r$ around $p$. Choose $r''$ small enough that $\mathfrak{f}(r'')<r$, and $r'<r''$ rational and small enough that $\mathfrak{f}(r')<r''$, for $\mathfrak f$ as in the definition of weak metric. If we choose $x$ to be a point sufficiently far along a ray towards $p$, then we have $d_\kappa(p,x)<r'$. By Lemma \ref{lemma:weak metric balls} (the ``moreover'' part), $p$ is contained in the interior of $d_\kappa$-ball $B_{r''}(x)$. Moreover, since $d_\kappa$ is a weak metric, any point in this ball lies at $d_\kappa$-distance at most $\mathfrak{f}(\max\{r'', d_\kappa(x,p)\})=\mathfrak{f}(r'')<r$ from $p$, that is, the ball is contained in $B$, as required.

    \par\medskip

    \emph{$G$-action}. To show that the $G$-action on $\calX$ extends continuously on $\bar\calX$ it suffices to show that the metric $d_\kappa$ changes in a controlled way when changing the basepoint. More precisely, let $p$ be a basepoint, and let $d_{\kappa}^p$ be the weak metric defined exactly as $d_\kappa$ except that we consider representative rays and segments starting at $p$. We have to show that for all $x\in\bar X$ and $r>0$ there exists a $d_\kappa^p$-ball centered at $x$ and contained in the $d_\kappa$-ball centered at $x$ of radius $r$. Consider $y$ in the latter ball, which means that there exist rays/segments $\alpha$ and $\beta$ (depending on whether $x/y$ lie in the interior or in the boundary) representing $x$ and $y$ and starting at $\go$ such that $d(\alpha(t'),\beta(t'))\leq 3\kappa(t')$ for some $t'\in[1/(2r),1/r]$. By Lemma \ref{lem:change_basepoint} (or a simpler argument in the case of segments) there are rays/segments $\alpha'$ and $\beta'$ representing $x$ and $y$ such that $d(\alpha(t'),\beta(t'))\leq 3\kappa(t')+2d(\go,p)$. By Lemma \ref{lem:weak divergence} for the function $t\mapsto 3\kappa(t)+2d(\go,p)$, we have a diverging function $\fg$ (which we can assume to be non-decreasing) such that $d(\alpha(\fg(t')),\beta(\fg(t')))\leq 3\kappa(\fg(t'))$, showing that $d_\kappa^p(x,y)\leq 1/\fg(t')\leq 1/\fg(1/(2r))$, as required.    

    \par\medskip

    \emph{Null balls}. Finally, the requirement about balls being small with respect to covers follows from combining two facts. First, there is a straightforward generalization of the existence of Lebesgue numbers for covers of compact metric spaces to the case of weak metrics; we observe this in Lemma \ref{lem:Lebesgue_number} below. Second, it is easy to see that given any $R,\epsilon>0$ there are only finitely many $R$-balls in $\calX$ with $d_\kappa$-diameter larger than $\epsilon$.
\end{proof}

\begin{lemma}\label{lem:Lebesgue_number}
    Let $d$ be a weak metric on the set $Z$, and suppose that $Z$ endowed with the metric topology is compact. Then for any open cover $\mathcal U$ of $Z$ there exists $\epsilon>0$ such that for all $z\in Z$ we have that $B(z,\epsilon)$ is contained in some $U\in\mathcal U$.
\end{lemma}

\begin{proof}
    Let $\ff$ be the function as in the definition of weak metric. For each $z\in Z$ let $\epsilon_z>0$ be such that $B(z,\ff(\epsilon_z))$ is contained in some $U_z\in\mathcal U$. Since balls are neighborhoods of their center by Lemma \ref{lemma:weak metric balls}, and since $Z$ is compact, there exist finitely many $z_i$ such that the corresponding balls cover $Z$. Let $\epsilon=\min_i \epsilon_{z_i}$. Then for any $z\in Z$ we have $d(z,z_i)<\epsilon$ for some $i$. If $z'$ is such that $d(z,z')<\epsilon$ then
    $$d(z_i,z')\leq \ff(\max\{d(z,z_i),(z,z')\})\leq \ff(\epsilon)\leq \ff(\epsilon_i),$$
    since $\ff$ is non-decreasing. This shows that $B(z,\epsilon)$ is contained in $U_{z_i}\in\mathcal U$, as required.
\end{proof}

\section{Finite dimension}
\label{sec:finite_dim} 

In this section, we prove that the combing boundary of a locally finite asymptotically CAT(0) space has finite covering dimension, provided that the interior has finite Assouad-Nagata dimension. We recall the definition of this notion:

\begin{definition}[Assouad-Nagata dimension]
    Let $(X,d)$ be a metric space.
    \begin{itemize}
        \item Given $R>0$, we say that a covering $\calU$ of $X$ is \emph{$R$-bounded} if $\diam(U)<R$ for all $U \in \calU$.
        \item Given $r>0$, we say that a covering $\calU$ has \emph{$r$-multiplicity} at most $k$ if for every $A \subset X$ with $\diam(A)<r$, we have $\#\{U \in \calU| U \cap A \neq \emptyset\}\leq k$.
        \item The \emph{Assouad-Nagata dimension} of $(X,d)$ is the smallest integer $n$ for which there exists a constant $M > 0$ such that for all $r > 0$, the space $X$ has a $Mr$-bounded covering with $r$-multiplicity at most $n + 1$.
        \begin{itemize}
            \item If no such integer exists, then we say $(X,d)$ is infinite dimensional.
        \end{itemize}
    \end{itemize}
\end{definition}

For a uniformly discrete space, that is, if there exists $\epsilon>0$ such that distinct points lie at distance at least $\epsilon$ from each other, for any small enough $r$ there always exist $r$-bounded coverings with $r$-multiplicity at most $1$. Therefore, in this context the Assouad-Nagata dimension is a linearly-controlled version of Gromov's asymptotic dimension.

We want to apply the results in this section to HHSs, which we can do in view of the following proposition that combines results in the literature.

\begin{proposition} \label{prop:fin AN dim HHS}
    If a colorable hierarchically hyperbolic space is proper and uniformly discrete then it has finite Assouad-Nagata dimension.
\end{proposition}

\begin{proof}
    Let $\calX$ be an HHS as in the statement.  As observed in \cite{HagenPetyt}, $\calX$ admits a quasi-isometric embedding into a product of projection complexes in the sense of Bestvina-Bromberg-Fujiwara \cite{BBF}. It is proved in \cite[Corollary 3.3]{HHS:asdim} that $\calC(U)$ has asymptotic dimension bounded in terms of $\calX$ (generalizing \cite{BellFuj_asdim}). It follows then, as in \cite{BBF}, that each of the projection complexes, which are build from these hyperbolic spaces, has finite asymptotic dimension.

    Here are two ways to complete the proof: We can invoke Lang-Schlichenmaier \cite[Proposition 3.5]{LangSchlich_Nagata}, which says that each of these projection complexes has finite Assouad-Nagata dimension (since they are hyperbolic with finite asymptotic dimension).  Alternatively, we can invoke Hume \cite[Proposition 5.2]{Hume_andim}, which says that each projection complex can be quasi-isometrically embedded in a finite product of simplicial trees, and make the same conclusion.  Either way, $\calX$ quasi-isometrically embeds into a finite product of spaces with finite Assouad-Nagata dimension, and we are done.
\end{proof}

It seems possible that colorability is an unnecessary assumption, and in particular that the arguments of \cite{HHS:asdim} can be adapted to show finite Assouad-Nagata dimension, so we ask:

\begin{question}
    Does every hierarchically hyperbolic space which is proper and uniformly discrete discrete have finite Assouad-Nagata dimension?
\end{question}

\subsection{Standing assumptions}

For the rest of this section, we will make the following assumptions on $(\calX,d)$:
\begin{enumerate}
    \item  $\calX$ is asymptotically CAT(0) as in Definition \ref{defn:asymp CAT0}.  We fix a constant $C$ such that any two points of $\cuco X$ are joined by a discrete $(1,C)$-quasi-geodesic. Moreover, we denote by $\kappa$ the function in the asymptotically CAT(0) condition. For convenience, we assume $\kappa(t)\geq 10C$ for all $t$.

    \item Balls in $\calX$ are finite.
\end{enumerate}

We use the same convention as in Section \ref{sec:boundary} and refer to discrete $(1,C)$-quasi-geodesic rays simply as rays.

\begin{theorem}
\label{thm:finite_dim}
    Under the assumptions above, $\mathfrak B \calX$ has covering dimension bounded by the Assouad-Nagata dimension of $\calX$.
\end{theorem}

From now on we fix the setup of Theorem \ref{thm:finite_dim}, and we fix a basepoint $\go\in\calX$.  For $T>0$ we consider the closed annulus $A_T$ in $\calX$ of inner radius $T-C$ and outer radius $T+C$, centered at $\go$. We have maps $\sigma_T:\mathfrak B\calX\to A_T$ given by letting $\sigma_T(x)$ be $\gamma(T)$ for an arbitrarily chosen ray starting at $\go$ representing $x$. First of all, we show that that the images of small balls in $\mathfrak B\calX$ (with respect to the weak metric $d_\kappa$ as in Subsection \ref{subsec:weak_metric}) have controlled image in large annuli:

\begin{lemma}\label{lem:small_image}
    For all $T>0$ there exists $r>0$ such that the following holds. For all $x,y\in\mathfrak B\calX$ we have that if $d_\kappa(x,y)<r$ then $d(\sigma_T(x),\sigma_T(y))\leq 7\kappa(T)$.
\end{lemma}

\begin{proof}
For all sufficiently small $r$, in view of Lemma \ref{lem:weak divergence} we have that there exist rays $\alpha',\beta'$ representing $x,y$ such that $d(\alpha'(T),\beta'(T))\leq 3\kappa(T)$.
Let $\alpha,\beta$ be the rays used to construct $\sigma_T(x),\sigma_T(y)$. Since any two rays starting at the same point and representing the same boundary point stay within $2\kappa$ of each other (Corollary \ref{cor:unif_close}), we have $d(\alpha(T),\beta(T))\leq 7\kappa(T)$, as required.
\end{proof}

Conversely, pre-images under $\sigma_T$ of sets of controlled size are small:

\begin{lemma}\label{lem:small_preimage}
    For all $\epsilon_0>0$ and $K>0$ there exists $T>0$ such that if $U\subseteq A_T$ has diameter at most $K\kappa(T)$ then $\sigma_T^{-1}(U)$ has diameter at most $\epsilon_0$. 
\end{lemma}

\begin{proof}
   This follows from Lemma \ref{lem:weak divergence}, which implies that given two rays $\alpha,\beta$ with $d(\alpha(T),\beta(T))\leq K\kappa(T)$ we have $d(\alpha(\fg(T)),\beta(\fg(T)))\leq 3 \kappa(\fg(T))$, for some fixed diverging function $\fg$. We can choose $T$ large enough that $1/\fg(T)<\epsilon_0$, and the conclusion follows directly from the definition of $d_\kappa$.
\end{proof}

We are now ready to prove the theorem.

\begin{proof}[Proof of Theorem \ref{thm:finite_dim}]
    Let $n$ be the Assouad-Nagata dimension of $\calX$ and let $M$ be as in the definition.
    
    Fix an open cover $\mathcal U$ of $\mathfrak B\calX$, which we have to refine to a cover of multiplicity at most $n+1$.
    
    We now fix various constants. Let $\epsilon>0$ be (a Lebesgue number) as in Lemma \ref{lem:Lebesgue_number}, and let $\epsilon_0>0$ be such that $\ff(\max\{\epsilon_0,\ff(\epsilon_0)\})<\epsilon$, where $d_\kappa$ is a weak metric with parameter $\ff$ (see Proposition \ref{prop:weak metric}). Let $T$ be as in Lemma \ref{lem:small_preimage} for this $\epsilon_0$ and $K= 10M$. Let $r$ be as in Lemma \ref{lem:small_image} for the given $T$.
    
    Consider now any $14M\kappa(T)$-bounded cover $\mathcal A$ of $A_T$ of $14\kappa(T)$-multiplicity at most $n+1$, which exists by the choice of $n$ and $M$. For each $A\in \mathcal A$ we consider the open subset of $\mathfrak B\calX$ defined by $U_A=\bigcup_{x\in\sigma_T^{-1}(A)} B(x,\min\{\epsilon_0,r\})$. The $U_A$ clearly cover $\mathfrak B\calX$, so we are only left to show that each is contained in some $U\in\mathcal U$, and that the multiplicity of the cover is bounded by $n+1$.
    
    \emph{$\mathcal U$-smallness}. We prove the first property by showing that each $U_A$ has diameter less than $\epsilon$. Given any $x',y'\in U_A$, respectively lying in $B(x,\epsilon_0)$ and $B(y,\epsilon_0)$ for some $x,y\in \sigma_T^{-1}(A)$ we can estimate:
    $$d_\kappa(x,y')\leq \ff(\max\{d_\kappa(x,y),d_\kappa(y,y')\})\leq \ff(\epsilon_0),$$
    using the conclusion of Lemma \ref{lem:small_preimage}, and in turn
$$d_\kappa(x',y')\leq \ff(\max\{d_\kappa(x',x), d_\kappa(x,y')\})\leq \ff(\max\{\epsilon_0,\ff(\epsilon_0)\})<\epsilon,$$
as required.

\emph{Multiplicity.} To bound the multiplicity, we show that if $x'\in U_{A_1}\cap U_{A_2}$ for some $A_i\in\mathcal A$ then $B(\sigma_T(x'),7\kappa(T))$ intersects both $A$ and $A'$. We have that $d_\kappa(x',x_i)<r$ for some $x_i$ with $\pi_T(x_i)=a_i\in A_i$. In view of the conclusion of Lemma \ref{lem:small_image}, we have $d(\pi_T(x'),a_i)<7\kappa(T)$, as required.
\end{proof}

\section{A $\calZ$-structure}
\label{sec:ER}

The goal of this section is Theorem \ref{thm:Z} which states that, under certain assumptions on an asymptotically CAT(0) space $\calX$, our boundary yields a compactification of (Vietoris--Rips complexes of) $\calX$ that is ``topologically controlled''. In particular, said compactification will be contractible, and more specifically it will be a Euclidean retract, a notion we recall below, alongside various related notions.

\subsection{Euclidean retracts, $\calZ$-structures and related definitions}

In this subsection, we recall the definitions of Euclidean retract, absolute neighborhood retract, and $\calZ$-set.  We refer the reader to Guilbault-Moran \cite{GuilMor} for a discussion of $\calZ$-structures, and Borsuk-Dydak \cite{BorDyd} for more on shape theory. We will rarely use the actual definitions of the notions recalled below, rather using various criteria to check them, and using known results when we apply them.

The most important notion for us is that of Euclidean retract, as it is required in the criterion for the Farrell--Jones conjecture that we will check for suitable hierarchically hyperbolic groups.

\begin{definition}[Euclidean retract]
A compact space $X$ is a Euclidean retract (or ER) if it can be embedded
in some $\mathbb R^n$ as a retract.
\end{definition}

The way that we will check that compactifications of (Vietoris--Rips complexes of) certain asymptotically CAT(0) spaces are ER uses absolute neighborhood retracts and $\calZ$-sets. We now recall the definition of the former notion even though we will never use it explicitly:

\begin{definition}[ANR]\label{defn:ANR}
    A locally compact space $X$ is an \emph{absolute neighborhood retract (ANR)} if whenever $X$ is embedded as a closed subset of a space $Y$, some neighborhood of $X$ in $Y$ retracts onto $X$.
\end{definition}

Locally finite simpicial complexes are known to be ANRs \cite[Corollary 3.5]{polyhedron_anr}:

\begin{lemma}
 \label{lem:interior ANR}
 Any locally finite simplicial complex is an ANR.
\end{lemma}

We can now define $\calZ$-sets:

\begin{definition}[$\calZ$-set]\label{defn:Z-structure}
    A closed subset $Z$ in a compact ANR $X$ is a $\calZ$-set if for every open set $U$ of $X$ the inclusion $U - Z\hookrightarrow U$ is a homotopy equivalence.
\end{definition}

ERs and ANRs are related in the following well-known way (a converse for compact metrizable spaces also holds, but we do not need it):

\begin{theorem}
\label{thm:ER}
Any compact, metrizable, contractible space of finite covering dimension space is a ER.
\end{theorem}

\begin{proof}
    Any contractible ANR is an absolute retract, meaning that it is a retract of any space where it can be embedded as a closed subset, see \cite[Proposition II.7.2]{Hu:retracts}. Also, any compact, metrizable space of finite covering dimension can be embedded into some Euclidean space, see e.g. \cite[Theorem 1.11.4]{dimension_theory}, concluding the proof.
\end{proof}

\subsection{$\calZ$-compactifications for asymptotically CAT(0) spaces}

The main goal of this section is to prove the following theorem:

\begin{theorem}\label{thm:Z}
    Let $\calX$ be asymptotically CAT(0) and assume that balls in $\calX$ are uniformly finite and that $\calX$ has finite Assouad-Nagata dimension. Then for all sufficiently large $T$ we have that $\bar X_T=\bar \calX=R_T(\calX)\cup \mathfrak B \calX$ has a topology which
    \begin{enumerate}
        \item restricts to the respective topologies on $R_T(\calX)$ and $\mathfrak B\calX$,
        \item all conclusions of Theorem \ref{thm:compact_boundary} hold, where the balls in the ``moreover'' part are now replaced by balls in the any simplicial metric on $R_T(\calX)$,
        \item  $\bar \calX$ is a ER and $\mathfrak B\calX$ is a $\calZ$-set.
    \end{enumerate}
\end{theorem}

Before proving the theorem, we observe that it implies a slightly more general version of Theorem \ref{thm:Z-boundary intro}. For convenience we recall the definition of a $\calZ$-structure, and then we state the desired corollary.

\begin{definition}\cite{bestvina1996local,Dra:BM_formula}
\label{defn:Z-space}
    Let $G$ be a discrete group. A $\calZ$-structure on $G$ is a pair $(X, Z)$ of spaces satisfying the following:

    \begin{itemize}
        \item $X$ is a Euclidean retract,
        \item $Z$ is a $\calZ$-set in X,
        \item $X-Z$ admits a proper cocompact $G$-action,
        \item for every open cover $\calU$ of $X$ and compact set $K$ of $X-Z$, all but finitely many $G$-translates of $K$ are $\calU$-small.
    \end{itemize}
\end{definition}

The space $Z$ as above is called a $\calZ$-boundary. We note that any two $\calZ$-boundaries of a group are shape equivalent \cite[Corollary 1.6]{GuilMor}, though they need not be homeomorphic, even in the CAT(0) setting \cite{CrokeKleiner}.

\begin{corollary}
\label{cor:Z}
    Let $G$ act geometrically on an asymptotically CAT(0) space with finite balls and finite Assouad-Nagata dimension. Then $G$ admits a $\calZ$-structure.
\end{corollary}

\begin{proof}
    The required $Z$-structure is provided by the $G$-action on $\bar X_T$.The fourth condition follows from the fact that all but finitely many balls in $R_T(\calX)$ of a given radius are small with respect to any given open cover of $\bar X_T$.
\end{proof}

For the rest of this section, we will make the following assumptions on $(\calX,d)$:
\begin{enumerate}
    \item $\calX$ is asymptotically CAT(0) as in Definition \ref{defn:asymp CAT0}, with parameters $C$, $\kappa$. Up to increasing $C$, any two points of $\cuco X$ are joined by a discrete $(1,C)$-quasi-geodesic, which, as in Section \ref{sec:boundary}, we refer to simply as segments. For convenience, we also assume $\kappa(t)\geq 10C$ for all $t$.

    \item Balls in $\calX$ are uniformly finite.
    \item $\calX$ has finite Assouad-Nagata dimension.
\end{enumerate}

The following proposition, which is due to Bestvina-Mess \cite[Proposition 2.1]{BestMess}, gives us a criterion for a compactification to be an ANR (here, $Z$ should be thought of as the boundary and $X$ as the whole compactification).

\begin{proposition}\label{prop:BestMess}
    Suppose $X$ is a compactum and $Z \subset X$ is a closed subset such that
    \begin{enumerate}
        \item $\mathrm{int} Z = \emptyset$;
        \item $\dim X = n < \infty$;
        \item For every $k \in \{0, \dots, n\}$, every point $z \in Z$, and every neighborhood $U$ of $z$, there is a neighborhood $V$ of $z$ so that every map $\alpha:S^k \to V - Z$ extends to a map $\tilde{\alpha}:B^{k+1} \to U - Z$;
        \item $X - Z$ is an ANR. \label{item:ANR}
    \end{enumerate}
    Then $X$ is an ANR and $Z \subset X$ is a $\calZ$-set.
\end{proposition}

The main goal will be to prove that $X=\bar \calX$ and $Z=\mathfrak B\calX$ satisfy the conditions of Proposition \ref{prop:BestMess} (when $\bar\calX$ is endowed with a natural topology), but in fact we verified several of the properties already.

\begin{proof}[Proof of Theorem \ref{thm:Z}.]
We define a topology on $\bar\calX_T$ as follows. Note that the set  $\bar \calX$ from Theorem \ref{thm:compact_boundary} is $\bar\calX_0$ in the current notation, and it is contained in each $\bar \calX_T$. We declare a subset $U$ of $\bar\calX_T$ to be open if
\begin{itemize}
    \item $U\cap R_T(\calX)$ is open, and
    \item there is an open subset $B_U$ of $\bar\calX_0$ such that $U\cap \mathfrak B\calX=B_U\cap \mathfrak B\calX$ and the full subcomplex of $R_T(\calX)$ spanned by $B_U\cap \calX$ is contained in $U$.
\end{itemize}

Note that this is indeed a topology, and we now check that this topology still has all the properties listed in Theorem \ref{thm:compact_boundary}. As a general preliminary observation, here is a way to make an open set of $\bar\calX_T$. First, we denote by $\rho$ the path metric on $R_T(\calX)$ where all simplices are isometric to regular euclidean simplices with side length 1. Start with an open set $V$ of $\bar\calX_0$. Then, denoting by $\Delta$ the subcomplex of $R_T(\calX)$ spanned by $V\cap \calX$, we have that $V\cup N^\rho_\epsilon(\Delta)$ is open for all $\epsilon<1$. Indeed, $N^\rho_\epsilon(\Delta)$ and $\Delta$ contain the same vertices of $R_T(\calX)$.

We are ready to check the properties.

\begin{itemize}
    \item Hausdorff. All pairs of points can be separated by pairs of open sets as described above. For example, if $p\in\mathfrak B\calX$ and $x\in R_T(\calX)$, then we can consider disjoint open sets $V_1,V_2$ of $\bar\calX_0$ with $p\in V_1$ and the vertices of a simplex containing $x$ all contained in $V_2$. Performing the construction above with any $\epsilon<1/2$ for both $V_1$ and $V_2$ yields the required open sets of $\bar\calX_T$.

    \item Compactness. Any open cover $\mathcal U$ of $\bar\calX_T$ restricts to an open cover $\mathcal V$ of $\bar\calX_0$, which then has a finite subcover $\mathcal V'$. Considering the corresponding subcover $\mathcal U'$ of $\mathcal U$, we claim that at most finitely many simplices have vertex set not entirely contained in some element of $\mathcal U'$, which easily yields the existence of the required finite subcover. If not, we could consider a sequence of distinct simplices $\sigma_n$ with vertex set not entirely contained in some element of $\mathcal U'$. Up to passing to a subsequence there exists a sequence of vertices $v_n$ of $\sigma_n$ which converges to some point $p$ of $\bar\calX_0$, necessarily in the boundary as the simplices are distinct. But then it is easily seen that all sequences of vertices of the $\sigma_n$ also converges to $p$ (in view of Lemma \ref{lem:weak divergence} and the fact that $d_\kappa$ is a weak metric). This in turn implies that, for all sufficiently large $n$, all vertices of $\sigma_n$ are contained in some $V\in\mathcal V'$, which in turn implies that $\sigma_n$ is contained in some $U\in\mathcal U'$, a contradiction.

    \item Metrizability. For metrizability, by Urysohn's lemma we are only left to show second-countability. 
    
    We claim that a countable basis for the topology is given by a countable basis for the topology of $R_T(\calX)$, and by any open set obtained from the procedure described above, for $\epsilon=1/2$, starting with the interior of $d_\kappa$-balls of rational radius around points of $G$. (Recall that the center of a $d_\kappa$-ball lies in its interior by Lemma \ref{lemma:weak metric balls}.)

Given an open set $U$ of $\bar\calX_T$, it is clear that $U\cap R_T(\calX)$ is a union of basis elements, so that we only need to find for each $p\in U\cap \mathfrak B\calX$ an open set of the second type containing $p$ and contained in $U$. There is a $d_\kappa$-ball $B_0$ around $p$ such that the subcomplex spanned by $B_0\cap R_T(\calX)$ is contained in $U$. In fact, we can shrink $B_0$ to make sure that the subcomplex spanned by all vertices of $R_T(\calX)$ adjacent to some vertex of $B_0\cap R_T(\calX)$ is contained in $U$ (since such vertices lie within a small $d_\kappa$-distance of $B_0\cap R_T(\calX)$, so that we can use the weak triangle inequality to conclude that they are close to $p$ if $B_0$ has sufficiently small radius). We can now choose a point far enough along a ray to $p$ and a ball $B$ of small rational radius such that $B\cap R_T(\calX)$ is contained in $B_0$, and it is readily seen that the open set constructed from the interior of said ball is contained in $U$, as required.

\item $G$-action. The fact that $G$ still acts on $\bar\calX_T$ is easily seen from the definition of the topology.

\item Null balls. Finally, the property on smallness of balls easily follows from the analogous property for $\bar\calX$.
\end{itemize}

    We will verify the various properties of Proposition \ref{prop:BestMess} below, but first we argue that this suffices to prove the theorem. Indeed, the proposition yields that $\bar \calX$ is an ANR and that $\mathfrak B\calX$ is a $\calZ$-set. In view of Theorem \ref{thm:ER}, we are left to argue that $\mathfrak B\calX$ is contractible (since finite-dimensionality is one of the properties we will check). But contractibility follows from the fact that $\bar\calX$ is homotopy equivalent to $R_T(\calX)$ by definition of $\calZ$-set (applied with the open set being the whole space), and the latter is contractible by Theorem \ref{prop:Rips contract}.

    Towards checking the conditions of Proposition \ref{prop:BestMess}, first of all we have that $\bar \calX$ is a compactum by the discussion above.

    For the numbered conditions:

    (1) This follows immediately from the definition of the topology, as points along a ray are arbitrarily close to the boundary point corresponding to the ray.

    (2) By Theorem \ref{thm:finite_dim} we have that $\mathfrak B\calX$ is finite-dimensional, and so is $R_T\calX$ since it is a finite-dimensional simplicial complex. Hence their union $\bar\calX$ has finite dimension by classical results in dimension theory, see e.g. \cite[Corollary 1.5.5, Theorem 4.1.5]{dimension_theory}.

    (4) This follows from Theorem \ref{prop:Rips contract} and Lemma \ref{lem:interior ANR}.

    \par\medskip

    We are only left to check (3). By Proposition \ref{prop:Rips_combing} it suffices to show that for all $M\geq 1$ and all $d_\kappa$-balls $U$ in $\bar\calX_0$ around some $x\in\mathfrak B\calX$ there exists another $d_\kappa$-ball $V$ around $x$ and some $p\in\calX$ such that $hull_M(V\cap \calX);p)$ is contained in $U$.

    Suppose $U=B(x_0,1/t_0)$, and that $\gamma$ is a ray from a fixed basepoint $\go$ to $x_0$. We will choose $p$ and $V$ as follows, for some $t_1$ large enough to be determined. Fix $t=\mathfrak g(t_1)$, where $\mathfrak g$ as in Lemma \ref{lem:weak divergence} for $\eta=7\kappa$. Then we will choose $p=\gamma(t)$ and $V=B(x_0,1/t_1)$.

    As a first observation about the choice of $t$, notice the following. Consider a point $q\in V\cap \calX$. By definition of the weak metric and Corollary \ref{cor:unif_close} there exists $\hat t\geq t_1$ and a segment $\alpha$ from $\go$ to $p$ such that $d(\gamma(\hat t),\alpha(\hat t))\leq 7\kappa(\hat t)$, so that we have
    $$d(\gamma(t),\alpha(t))\leq 3\kappa(t)$$
    by the conclusion of Lemma \ref{lem:weak divergence}, and for later purposes we note that the same also holds replacing all ``$t$'' with ``$t/2$''.

    We now expand the definition of $hull_M(V\cap \calX);p)$.  Consider a segment $\beta$ from $p$ to some $x\in V\cap \calX$, a point $q=\beta(s)$ on it, and a point $z$ such that $d(q,z)\leq M\kappa(s)$. Then any point $z$ in $hull_M(V\cap \calX);p)$ is of this form (see Remark \ref{rmk:-C}).

\begin{figure}[h]
\begin{center}
    \includegraphics[scale=1.2]{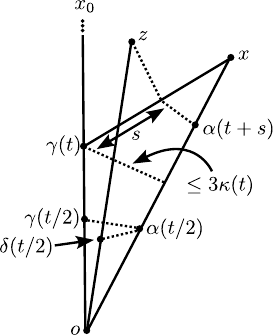}
\end{center}
\caption{}
\end{figure}

    Via a comparison triangle (with vertices $p=\gamma(t),\gamma'(t)$, and $x$), we see that $q$ lies within distance $2\kappa(s)+3\kappa(t)$ from $\alpha(t+s)$ (or $x$ if $\alpha(t+s)$ is not defined), and hence $z$ lies within $(M+2)\kappa(s)+3\kappa(t)$ from $\alpha(t+s)$.

    Let $\delta$ be a segment from $\go$ to $z$, with $z=\delta(t')$. Note that
    $t'\geq t+s-(M+2)\kappa(s)-3\kappa(t)-2C.$
    We assume that $t_1$ (which we still have to determine, and whose value determines that of $t$ via the diverging function $\mathfrak g$) is large enough that $t'\geq t/2$, which is possible since the function $s\mapsto s-(M+2)\kappa(s)-3\kappa(t)-2C$ is bounded from below (for a fixed $M$).

    Consider now the triangle with vertices $\go$, $z$, and $\alpha(t+s)$ or $q$ if $\alpha(t+s)$ is not defined (and containing $\delta$ and a subpath of $\alpha$). Using convexity of the metric in a comparison triangle we see that
    $$d(\delta(t/2), \alpha(t/2))\leq$$
    $$((M+2)\kappa(s)+3\kappa(t)) \frac{t}{2(t+s)}+ ((M+2)\kappa(s)-3\kappa(t)+2C)\frac{t}{2(t+s)} + 2\kappa(t/2),$$
    where the second term accounts for the difference between $t/2$ and the length of the comparison side for $\delta$ multiplied by $t/(2(t+s))$.

    We can then define a a sublinear function $\eta'$ by setting $\eta'(t/2)$ to be the supremum of the left-hand side over all $s\geq 0$. We then have $d(\gamma(t/2),\delta(t/2))\leq \eta'(t/2)+3\kappa(t/2)$, combining the displayed inequality and $d(\gamma(t/2),\alpha(t/2))\leq 3\kappa(t)$. In view of Lemma \ref{lem:weak divergence} (applied with $\eta=\eta'+3\kappa$), for any sufficiently large ($t_1$, whence) $t$ we have $d(\gamma(t_0),\delta(t_0))\leq 3\kappa(t_0)$, implying $z\in U$ by definition of the weak metric, as required.
\end{proof}

\section{The Farrell--Jones conjecture}
\label{sec:FJ}

In this section we prove the Farrell--Jones conjecture for suitable HHGs, using a criterion formulated by Bartels-Bestvina \cite[Section 1]{BB:FJ_MCG}. We will prove the following:

 \begin{theorem}
 \label{thm:rel_FJ}
     Let $(G,\mathfrak S)$ be a colorable hierarchically hyperbolic group. Then $G$ admits a finitely $\mathcal F$-amenable action on a compact ER, where $\mathcal F$ is the set of all subgroups of $G$ that are virtually cyclic or the stabilizer of some $Y\in\mathfrak S$ which is not $\nest$-maximal.
 \end{theorem}

 We do not need the definition of finitely $\mathcal F$-amenable, only known facts about this notion that lead to the following corollary, which proves Theorem \ref{thm:FJ intro} from the introduction:

 \begin{corollary}
  Let $(G,\mathfrak S)$ be a colorable hierarchically hyperbolic group. Suppose that the stabilizers of all $Y\in\mathfrak S$ which are not $\nest$-maximal satisfy the Farrell--Jones conjecture. Then $G$ satisfies the Farrell--Jones conjecture.
 \end{corollary}

\begin{proof}
 As noted in \cite[Theorem 4.6 and 4.8]{BB:FJ_MCG}, by results in \cite{BLR:hyperbolic,BL:hyp_CAT0,BFL:lattices} any group with a finitely $\mathcal F$-amenable action on a compact ER where the subgroups in $\mathcal F$ satisfy the Farrell--Jones conjecture also satisfies the Farrell--Jones conjecture.
\end{proof}

\subsection{Setup}

The rest of this section is devoted to the proof of Theorem \ref{thm:rel_FJ}.

Fix a colorable HHG $(G,\mathfrak S)$, endow it with the asymptotically CAT(0) metric from Theorem \ref{thm:asymp CAT0}, with parameters $C,\kappa$, which we denote by $d$. Recall that $d$ is uniformly quasi-isometric to any word metric. We can furthermore assume that there exists $D$ such that the moreover part of Theorem \ref{thm:asymp CAT0} on hierarchy paths holds with these constants. In fact, any two points of $G$ are joined by a discrete $(1,C)$-quasi-geodesic, which is furthermore a $D$-hierarchy path. As in Section \ref{sec:boundary}, we simply refer to discrete $(1,C)$-quasi-geodesic as segments, and similarly for rays.

Let $\bar G$ be the compact spaces provided by Theorem \ref{thm:Z} (with $\calX=G$); as noted in Proposition \ref{prop:fin AN dim HHS} the theorem indeed applies. We check that $\bar G$ satisfies \cite[Theorem 1.11]{BB:FJ_MCG}. Since $\bar G$ is a compact ER by Theorem \ref{thm:Z}, we are only left to show that \cite[Axioms 1.4-1.6]{BB:FJ_MCG} are satisfied; we recall the content of the various axioms along the way.

\subsection{Boundary subsurface projections}

Most of the conditions of \cite[Axiom 1.4]{BB:FJ_MCG} follow directly from the HHS axioms, the main missing piece is the fact, that we show below, that the projections $\pi_Y:G\to \mathcal C(Y)$ extend to an open set $\Delta(Y)\subseteq \bar G$, with the projection now allowed to take values in $\partial C(Y)$ as well. 

We would like to define $\pi_Y(p)$, for $p\in\mathfrak B G$ in terms of $\pi_Y(\gamma(t))$ for $\gamma$ a ``nice'' ray representing $p$. The following says that such a nice ray exists:

  \begin{lemma}\label{lem:hierarchy_ray}
        There exists $D_0$ such that for any $p\in\mathfrak B G$ and any basepoint $\go$ there exists a ray $\gamma=\gamma_p$ which is also a $D_0$-hierarchy path and $[\gamma]=p$.
    \end{lemma}

    \begin{proof}
        By assumption, any pair of interior points is connected by a segment which is also a $D$-hierarchy path. The same limiting argument as in Lemma \ref{lem:change_basepoint}, using these hierarchy paths allows us to construct the required hierarchy ray.
    \end{proof}

    When a basepoint is fixed, for $p\in\bar G$, we will denote by $\gamma_p$ either the ray from Lemma \ref{lem:hierarchy_ray} or a segment which is also a $D$-hierarchy path starting at the basepoint, depending on whether $p$ lies in the boundary or in the interior.

In order to show that the $\gamma_p$ can be used to extend $\pi_Y$ to a ``coarsely continuous'' map, we need the following. Recall that in an HHS there are, for each $Y\in\mathfrak S$, standard product regions $P_Y=F_Y\times E_Y$, where moving in the $F_Y$ factor (resp. $E_Y$ factor) corresponds to moving only in hyperbolic spaces nested into (resp. orthogonal to) $Y$.   Technically, $P_Y$, $E_Y$, and $F_Y$ are not subsets of the HHS, they come with quasi-isometric embeddings into it, and we will conflate them with their images under said embeddings. These are well-defined only up to finite Hausdorff distance. See \cite[Subsection 5B]{HHS_II} for more details.

\begin{proposition}\label{prop:coarse_cont}
    There exists $\theta\geq 0$ with the following property. Fix a basepoint $\go\in G$. For $Y\in\mathfrak S$, let $\Delta(Y)$ be the complement of the limit set of $E_Y$ in $\mathfrak B G\subseteq \bar G$. Then for any point $p\in \Delta(Y)$ and $t_0\geq 0$ there exists a neighborhood $U$ such that for any $q\in U\cap G$ we have $d_Y(\gamma_p(t_0),\gamma_q)\leq \theta$.
 \end{proposition}

\begin{proof}
    The proposition holds for $p$ an interior point since the $\pi_Y$ are coarsely Lipschitz and the $\mathcal C Y$ hyperbolic, so we only need to consider the case that $p$ is a boundary point.

We consider the weak metric $d_\kappa$ based at the fixed basepoint $\go$ that defines the topology of $\mathfrak B G$.  All constants $D_i$ appearing in the lemmas in this proof can be chosen independently of $Y$, $p$, $\go$.

Set $\gamma=\gamma_p$. We now show that $\gamma$ diverges linearly from $E_Y$.

\begin{lemma}\label{lem:hierarchy_diverges}
    There exists $\epsilon>0$ (allowed to depend on $p$ and $\gamma$) such that for all large enough $t$ we have $d(\gamma(t),E_Y)>\epsilon t$. Moreover, if $p$ does not lie in the limit set of $P_Y$, then the same holds replacing $E_Y$ with $P_Y$.
\end{lemma}

\begin{proof}
The proof for the $P_Y$ version is identical, so we only spell out the $E_Y$ version.
    We will use the gate map $\mathfrak p_{E_Y}:G\to E_Y$, which is a map with the property that $\pi_Z(\mathfrak p_{E_Y}(x))$ is approximately $\pi_Z(x)$ for all $Z$ orthogonal to $Y$, see \cite[Section 5]{HHS_II} for a general discussion. A consequence of the distance formula is that $\mathfrak p_{E_Y}(x)$ is a coarse closest point to $x$ in $E_Y$, meaning that $d(x,\mathfrak p_{E_Y}(x))$ is bounded above up to uniform multiplicative and additive errors by $d(x,E_Y)$. Consider the gate points $p_n=\mathfrak p_{E_Y}(\gamma(n))$. Suppose by contradiction that there exists a subsequence of the $p_n$ along which $d(\gamma(n),p_n)<\eta(n)$ for some sublinear function $\eta$. Up to passing to a further subsequence, we can assume that the $p_n$ converge to some $p_\infty$, which by definition lies in the limit set of $E_Y$, so that $p\neq p_{\infty}$. However, by Lemma \ref{lem:weak divergence} we have $d_\kappa(p_n,\gamma(n))<1/\mathfrak g(n)$ for a diverging function $\mathfrak g$, so that $d_\kappa(p_\infty,p)=0$, contradicting Lemma \ref{lem:nonzero}.
\end{proof}

Now the proof splits into two cases.

\par\medskip

{\bf Case 1:} $p$ does not lie in the limit set of $P_Y$.

\par\medskip

To certify that certain pairs of points have nearly the same projection to $\mathcal C(Y)$ we will use the following lemma, saying that it suffices to check the existence of a hierarchy path that stays far from the product region $P_Y$.

\begin{lemma}\label{lem:BGI_hierarchy}
There exists $D_1$ such that if two points $x,y\in G$ are connected by a $D_0$-hierarchy path that stays $D_1$-far from $P_Y$, then $d_Y(x,y)\leq D_1$.
\end{lemma}

\begin{proof}
    This is \cite[Proposition 5.17]{HHS_II}, a corrected version of which appears in \cite[Proposition 18.1]{CRHK}.
\end{proof}

\begin{lemma}\label{lem:two_hierarchies}
    Suppose that $p$ does not lie in the limit set of $P_Y$. For all $t_0>0$ there exists $\epsilon'>0$ (allowed to depend on $p$ and $\gamma$) such that the following holds. Let $q\in \bar G$ be such that $d_\kappa(q,p)<\epsilon'$. Then there is a $D_0$-hierarchy path $\alpha$ connecting $\gamma_p(t)$ to $\gamma_q(t)$ which stays outside the $D_1$-neighborhood of $P_Y$ for some $t>t_0$.
\end{lemma}

\begin{proof}
    Suppose that $d_\kappa(q,p)$ is sufficiently small. By definition of $d_\kappa$ and the fact that all $(1,C)$-quasi-geodesics with the same endpoints (possibly at infinity) stay quantifiably close (Corollary \ref{cor:unif_close}) there exists some point $\gamma_q(t)$, with $t>t_0$, which lies sublinearly close to $\gamma_p(t)$, meaning that $d(\gamma_p(t),\gamma_q(t))$ is bounded above by a fixed sublinear function. Since $\gamma(t)$ is linearly far from $P_Y$ (Lemma \ref{lem:hierarchy_diverges}), when $d_\kappa(q,p)$ is sufficiently small we have that any segment from $\gamma_q(t)$ to $\gamma_p(t)$ stays $D_1$-far from $P_Y$. In particular, there is a $D_0$-hierarchy path connecting those two points which stays $D_1$-far from $P_Y$. This concludes the proof.
\end{proof}

To conclude the proof of the proposition in Case 1, it suffices to notice that, for $U$ a sufficiently small neighborhood of $p$, the conclusion of Lemma \ref{lem:two_hierarchies} applies to any $q\in U$, so that in particular $\gamma_p(t)$ and $\gamma_q(t)$ project close in $\mathcal C(Y)$ by Lemma \ref{lem:BGI_hierarchy}. But then $\pi_Y(\gamma_q)$ needs to pass uniformly close to $\pi_Y(\gamma(s))$ for all $s\leq t$ (since they are both quasi-geodesics in the hyperbolic space $\mathcal C(Y)$), and hence in particular for $s=t_0$, as required.

\par\medskip

{\bf Case 2:} $p$ lies in the limit set of $P_Y$.

\par\medskip

Note that we are of course still assuming that $p$ does not lie in the limit set of $E_Y$. It is easy to see that $d(\gamma(t),P_Y)$ is sublinear in $t$, at least along a subsequence of $t$s. Combining this with Lemma \ref{lem:hierarchy_diverges} and the fact that the distance from $E_Y$ of a point $x$ in $P_Y$ is comparable to $\Sigma(\go,x):=\sum_{Z\nest Y} [d_Z(x,\go)]_L$, for any sufficiently large $L$, we see that $\Sigma(\go,\gamma(t))$ grows linearly in $t$, at least along the same subsequence. But then so does $\Sigma(\gamma(t_0),\gamma(t))$ (with bounds depending on $t_0$).

Now consider $q$ with the property that $\pi_Y(\gamma_q)$ is far (sufficiently so to run the argument below) from $\pi_Y(\gamma(t_0))$. We then argue that any term in $\Sigma(\gamma(t_0),\gamma(t))$ also makes a contribution to $\Sigma(\gamma_q(t),\gamma(t))$. Indeed, for the $Y$ term of the sum, hyperbolicity of $\mathcal C(Y)$ forces any geodesic from $\pi_Y(\gamma(t))$ to $\pi_Y(\gamma_q(t))$ to pass uniformly close to $\pi_Y(\gamma(t_0))$. For the terms corresponding to some $Z$ properly nested into $Y$, in view of bounded geodesic image we have that $\pi_Z(\gamma_q(t))$ coarsely coincides with $\pi_Z(\go)$, and $\pi_Z(\gamma(t_0))$ is uniformly close to a geodesic from $\pi_Z(\go)$ to $\pi_Z(\gamma(t))$ since $\gamma$ is a hierarchy path, see Figure \ref{fig:term_Z_nest_Y}.

\begin{figure}[h]

\includegraphics{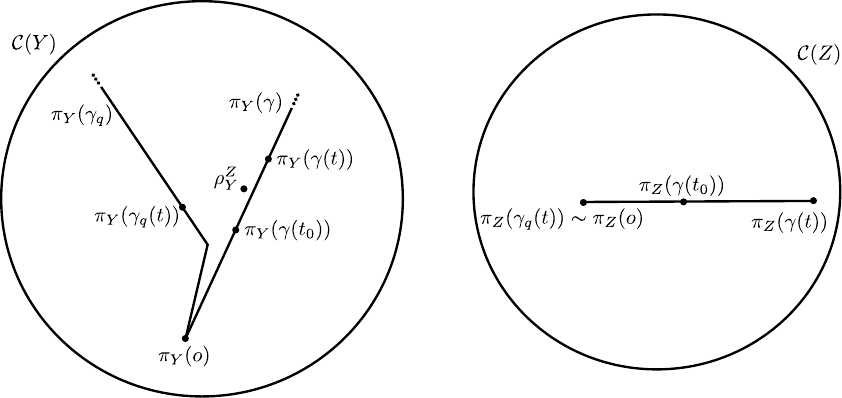}
\caption{The position of $\rho^Z_Y$ is  dictated by the fact that $d_Z(\gamma(t),\gamma(t_0))$ is large.}\label{fig:term_Z_nest_Y}.
\end{figure}

To summarize the above, if $\pi_Y(\gamma_q)$ is far from $\pi_Y(\gamma(t_0))$ then $\Sigma(\gamma_q(t),\gamma(t))$ is bounded from below by a linear function of $t$.

Now, $\Sigma(\gamma_q(t),\gamma(t))$ is bounded sublinearly in $t$ for arbitrarily large $t$ whenever $d_\kappa(p,q)$ is sufficiently small. Therefore, if $d_\kappa(p,q)$ is sufficiently small, then $\pi_Y(\gamma_q)$ needs to pass within uniformly bounded distance of $\pi_Y(\gamma(t_0))$, as required.

\end{proof}

Given a hierarchy ray $\gamma$ and $Y\in\mathfrak S$, $\pi_Y\circ\gamma$ has a coarsely-well defined endpoint in $\mathcal C(Y)\cup \partial C(Y)$. For $p\in\mathfrak B G$ we can then define $\pi_Y(p)$ as the coarse endpoint of $\pi_Y\circ\gamma_p$. For consistency of notation, for $Y,Z\in\mathfrak S$ with $\rho^Y_Z$ well-defined, write $\pi_Y(Z)=\rho^Y_Z$. We can then set $d^\pi_Y(x,z)=d_{\mathcal C(Y)\cup\partial \mathcal C(Y)}(\pi_Y(x),\pi_Y(z))\in[0,\infty]$ for $y,z\in \mathfrak S\sqcup \Delta(Y)$ for which the quantity is well-defined. Here, a point of $\partial C(Y)$ is infinitely far away from all other points.

Note that each $\pi_Y$ also extends to the Vietoris--Rips complex $R_T(G)$ contained in $\bar G$ (and that in fact that $R_T(G)$ in an HHS with these projection maps). We still denote $\pi_Y$ the extended maps, and with this we can define $d^\pi_Y(x,z)$ also when $x$ and/or $z$ are in the interior of a simple of $R_T(G)$. The axioms we have to check are coarse in nature, so essentially they hold for $G\cup \mathfrak BG$ if and only if they hold for $R_T(G)\cup\mathfrak BG$.

\begin{proposition}
    Each color of $\mathcal S$ satisfies \cite[Axiom 1.4]{BB:FJ_MCG}.
\end{proposition}

\begin{proof}
    Property (P1) (symmetry) holds since $d_{\mathcal C(Y)\cup\partial \mathcal C(Y)}$ is symmetric, and similarly (P2) (triangle inequality) holds since it is a metric (allowed to take value infinity).

    Property (P3) (inequality on triples) follows from the Behrstock inequality for the HHS $\cuco X$, since within each color all distinct pairs are transverse.

    Property (P4) (finiteness) only involves elements of the color, and is well-known and observed in e.g. \cite{HagenPetyt} to show that the axioms of \cite{BBF} apply.

    Finally, property (P5) (coarse semi-continuity) is an immediate consequence of Proposition \ref{prop:coarse_cont}.
\end{proof}

\subsection{Flow axioms}

In \cite{BB:FJ_MCG}, the authors consider the Thurston compactification $\bar T$ of Teichm\"uller space $T$, and the authors axiomatize properties of thick Teichm\"uller rays in Axioms 1.5 and 1.6. The setup is that for each compact subspace $K$ of $T$ there is a specified collection $\mathcal G_K$ of rays contained in the union of all translates of $K$. Then Axiom 1.5 says that for all $\Theta>0$ there is a compact set $K$ such that that given any point $x$ in the orbit of some fixed $x_0$ in $T$ and a boundary point $p$, either
\begin{itemize}
\item there is $c\in \mathcal G_K$ connecting $x$ to $p$, or
\item there is $Y$ with $p\in \Delta(Y)$ and $d_Y(x,p)>\Theta$.
\end{itemize}
(In \cite{BB:FJ_MCG} the second bullet is slightly different, as $x_0$ is replaced by a fixed ``base index'', but for our purposes there is no difference.)

In our case, we are considering an action which is cocompact, so we can do the following to tautologically satisfy Axiom 1.5. We assign arbitrarily to compact subsets of $R_T(G)$ whose $G$-translates cover the whole space some constant $\Theta=\Theta(K)$, in a way that the constants assigned to compact sets are arbitrarily large. Then, we declare $\mathcal G_K$ to be a set of rays containing one ray $\gamma_{x,p}$ connecting a vertex in a simplex containing $x\in R_T(G)$ to $p\in \mathfrak B G$ (which we can still regard as a quasi-geodesic with uniform constants starting at $p$ up to moving the starting point) whenever $x,p$ are such that $d_Y(x,p)\leq \Theta(K)$ for all $Y$ with $p\in \Delta(Y)$. As above, the axioms we will have to check are coarse, so it suffices to consider vertices of $R_T(G)$, which we implicitly do below.

We are left to check that the $\mathcal G_K$ as above satisfy Axiom 1.6, which is in three parts. It will be very important that all rays in $\mathcal G_K$ are Morse:

\begin{proposition}
     For all compact subspaces $K$ there exists a Morse gauge $M$ such that any ray in $\mathcal G_K$ is $M$-Morse.
\end{proposition}

\begin{proof}
    This is a direct consequence of \cite[Theorem D]{ABD}.
\end{proof}

We now check $(F_1)-(F_3)$ from \cite[Axiom 1.6]{BB:FJ_MCG} in the three following lemmas. Fix a compact subset $K$.

\begin{lemma}[Axiom (F1), small at infinity.] Let $c_n\in \mathcal G_K$ and $x_n\in R_T(G)$. Suppose that $d(c_n,x_0)$ and $d(c_n(0), x_n)$ are bounded. Then $x_n\to p\in \mathfrak B G$ if and only if $c_n(0)\to p$.
\end{lemma}

\begin{proof}
Provided that either of $d(c_n(0),x_0)$ or $d(x_n,x_0)$ diverge, we have $d_\kappa(c_n(0),x_n)\to 0$ by Lemma \ref{lem:weak divergence} (applied to a constant function $\eta$).

\end{proof}

\begin{lemma}[Axiom (F2), fellow-travelling.] For any $\rho>0$ there exists $R>0$ such that the following holds. For all $x\in R_T(G)$, $p_+\in \mathfrak B G$, and $t\in[0,\infty)$ there exists an open neighborhood $U_+$ of $p_+$ with the following property. Let $c,c'\in\mathcal G_K$ both start in the $\rho$-neighborhood of $x$ and both end in $U_+$. Then $d(c(t),c'(t))<R$.
\end{lemma}

\begin{proof}
 For a small enough neighborhood $U_+$, we have that $d(c(T),c'(T))$ is smaller than some fixed sublinear function for some $T\gg t$ (in view of the definition of $d_{\kappa}$ and the fact that all representative rays stay close in a controlled way by Corollary \ref{cor:unif_close}). Triangles with endpoints $c(0),c(T),c'(T)$ are thin with thinness constant depending only on $K$ and $\rho$ by Morse-ness, see \cite[Lemma 2.2]{Cordes:boundary}, and it is readily seen that $d(c(t),c'(t))$ is controlled by this thinness constant, the Morse-ness of $c'$, and $\rho$, see Figure \ref{fig:fellow_travel}.
\end{proof}

\begin{figure}[h]
 \includegraphics{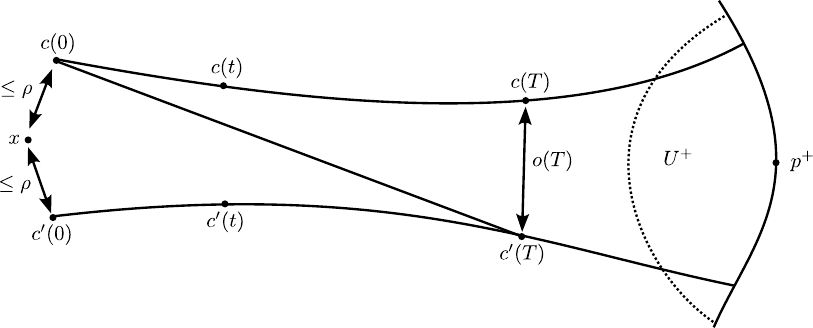}
 \caption{}\label{fig:fellow_travel}
\end{figure}

We denote by $c(\infty)$ the boundary point corresponding to a ray $c\in\mathcal G_K$.

\begin{lemma}[Axiom (F3), infinite quasi-geodesics] For any $\rho>0$ there exists $R>0$ such that the following holds. For $p_-,p_+\in\mathfrak B G$ let $T_{K,\rho}(p_-,p_+)\subseteq R_T(G)$ be the set of all $x$ such that there are $c_n\in\mathcal G_K$ with $c_n(0)\to p_-$ and $c_n(\infty)\to p_+$ and $d(c_n,x)\leq \rho$ for all $n$. If this set is non-empty, then it is contained in the $R$-neighborhood of a $(1,C)$-quasi-geodesic line $c$.
\end{lemma}

 \begin{proof}
     Assuming that there exists a sequence $c_n$ with endpoints converging to $p_-,p_+$ as in the statement and all $c_n$ intersect a fixed ball, up to replacing $c_n$ with a subsequence we can consider the limit $c$ (in the sense of Arzel\`a-Ascoli). It is not hard to see that $c$ consists of two rays, one converging to $p_-$ and one converging to $p_+$.

     Consider now $c'_n$ another sequence with endpoints converging to $p_-,p_+$, and $y$ some point $\rho$-close to all $c'_n$. We want to show that $x$ lies within bounded distance of $c$.

     It will be convenient to consider the weak metric $d_\kappa$ with basepoint $x$. For $n$ large, $d_\kappa(c'_n(0),p_-)$ becomes arbitrarily small, so that we see that there is a diverging sequence $t_n^->0$ such that $d(c'_n(0),c(t_n^-))$ is bounded by a fixed sublinear function of $t_n^-$. Similarly, there are diverging sequences $s_n,t_n^+$ such that $d(c'_n(s_n),c(t_n^+))$ is bounded by a fixed sublinear function of $t_n^+$. Now, by e.g. \cite[Lemma 2.5]{QR_CAT0}, given a point $z$ and a quasigeodesic $\alpha$ in a metric space, concatenating a $(1,C)$-quasi-geodesic  from $z$ to a closest point on $\alpha$ and a subpath of $\alpha$ yields a quasi-geodesic with constants depending only on $C$ and the constants of $\alpha$. Applying this (on two sides that are very far apart) we see that there is a subpath of $c$ that is contained in a quasi-geodesic with uniform constants with endpoints $c'_n(0)$ and $c'_n(s_n)$, see Figure \ref{fig:double_concat}. Note that for sufficiently large $n$, the point $x$ is close to $c'_n(s'_n)$ with $0\ll s'_n\ll s_n$. Since $c'_n$ is Morse, we see that $c$ needs to pass close to $x$, as required.
     
 \end{proof}

 \begin{figure}[h]
 \includegraphics{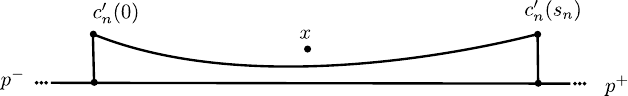}
 \caption{We obtain a quasi-geodesic by traveling from $c'_n(0)$ to a closest point in $c$, then along $c$ to a closest point to $c'_n(s_n)$, and then to $c'_n(s_n)$.}\label{fig:double_concat}
\end{figure}

\section{Controlling domains} \label{sec:controlling domains}

With this section, we begin laying the foundation for the proof of our Stabler Cubulations Theorem \ref{thm:stabler cubulations}, the main technical result on which the previous sections depended. 

Before outlining the section, we begin by setting the following standard notation, which we will use throughout the rest of the paper:

\begin{notation}
\label{not:rel}
    Given a subset $A$ of an HHS $(\calX,\mathfrak S)$ and a constant $K$, we denote $\Rel_K(F)$ the collection of all $Y\in\mathfrak S$ such that the projection $\pi_Y(A)$ of $A$ to the hyperbolic space $\calC(Y)$ has diameter at least $K$.
\end{notation}

Let $F \subset F' \subset \calX$ be finite sets in an HHS $(\calX, \mathfrak S)$.  Fixing a largeness threshold $K = K(\mathfrak S)>0$, let $\calU = \Rel_K(F)$ and $\calU' = \Rel_K(F')$.  Observe that $\calU \subset \calU'$ by definition.

\medskip

In this section, we will prove a number of statements which control the number of domains in $\calU' - \calU$ which appear in various relevant locations.  These results are in support of the Stabler Tree \ref{thm:stabler tree} construction in Section \ref{sec:stable trees} and how those trees interact with the cubical model construction in Section \ref{sec:stabler cubulations}.  Roughly speaking, the raw materials for the stabler tree construction are the projections of $F,F'$ to the domains in $\calU, \calU'$, as well as the relative projections coming from the HHS relations between the domains in $\calU,\calU'$.  The results in this section allow us to control this raw material.

In many ways, the results in this section reflect the discussion in \cite[Section 3]{DMS_bary}, where we considered the case where $d^{Haus}_{\calX}(F, F') \leq 1$, and used our notion of stable projections (see Subsection \ref{subsec:stable projections} below) to show that $|\calU \symdiff \calU'|$ was uniformly bounded in this case.

However, from this perspective, our current situation is quite different: Since $x' \in F' - F$ can be arbitrarily far from $F$, it is possible for $|\calU \symdiff \calU'| = |\calU' - \calU|$ (since $\calU \subset \calU'$) to be arbitrarily large.  Nonetheless, we will show that all but a bounded number of these extra domains are irrelevant to the way in which we need the combinatorial setup for $F$ to communicate with that for $F'$.

There are three main results here.  The first two, Propositions \ref{prop:distinguished domains} and Proposition \ref{prop:involved domains}, will bound the number of domains in $\calU$ for which the stabler tree constructions in Section \ref{sec:stable trees} for $F,F'$ are not identical.  The last, Proposition \ref{prop:sporadic domains}, gives us the control of what happens in the boundedly-many domains where the tree constructions are not identical, in particular motivating the statement of the Stabler Tree Theorem \ref{thm:stabler tree} itself.  See Lemma \ref{lem:controlling unstable parts} for the most consequential application of the statements in this section.

We first begin with some background on key properties of (colorable) HHSs.

\subsection{Colorable HHSs and stable projections}\label{subsec:stable projections}

The following definition of colorability is \cite[Definition 2.8]{DMS_bary}.  It was inspired by work of Bestvina-Bromberg-Fujiwara \cite{BBF}, who proved that the curve graph is finitely-colorable, a fact which implies that the HHS structure for the mapping class group is colorable in the following sense:

\begin{definition}\label{defn:colorable}
Let $(\calX,\mathfrak S)$ be an HHS and let $G < \mathrm{Aut}(\mathfrak S)$.  We say that
$(\calX, \mathfrak S)$ is $G$-{\em colorable} if there exists a decomposition of
$\mathfrak S$ into finitely many families $\mathfrak S_i$, so that each $\mathfrak S_i$ is
pairwise-$\pitchfork$ and $G$ acts on $\{\mathfrak S_i\}_i$ by permutations. We say that
$(\calX,\mathfrak S)$ is \emph{colorable} if it is $\mathrm{Aut}(\mathfrak S)$-colorable.  We call the $\mathfrak S_i$ \emph{BBF families}.

\end{definition}

The following is \cite[Theorem 2.9]{DMS_bary}, the proof of which was mainly an application of \cite[Proposition 5.8]{BBFS}:

\begin{theorem}\label{thm:stable proj}
Let $(\calX, \mathfrak S)$ be a $G$-colorable HHS for $G < \mathrm{Aut}(\mathfrak S)$ with
standard projections $\hpi_-, \hrho^-_-$.  There exists $\theta>0$ and refined projections
$\pi_-, \rho^-_-$ with the same domains and ranges, respectively, and such that:
\begin{enumerate}
 \item If $X,Y$ lie in different $\mathfrak S_j$, and $\hrho^X_Y$ is defined, then $\rho^X_Y=\hrho^X_Y$.
 \item If $X,Y\in\mathfrak S_j$ are distinct, then the Hausdorff distance between $\rho^X_Y$ and $\hrho^X_Y$ is at most $\theta$.
 \item If $x\in\calX$ and $Y\in\mathfrak S$, then the Hausdorff distance between $\pi_Y(x)$ and $\hpi_Y(x)$ is at most $\theta$. \label{item:same proj}
 \item If $X,Y,Z \in \mathfrak S_j$ for some $j$ are pairwise distinct and $d_Y(\rho^X_Y, \rho^Z_Y) > \theta$, then $\rho^X_Z=\rho^Y_Z$.
 \item Let $x\in\calX$, and let $Y,Z\in\mathfrak S_j$ for some $j$ be pairwise distinct. If $d_Y(\pi_Y(x), \rho^Z_Y)>\theta$ then $\pi_Z(x)=\rho^Y_Z$.\label{item:strict_Behrstock_X}
\end{enumerate}

Moreover, $(\calX, \mathfrak S)$ equipped with $\pi_-, \rho^-_-$ is an HHS,  $G <
\mathrm{Aut}(\mathfrak S)$, and it is $G$-colorable.
\end{theorem}

\begin{rem}\label{rem:rho_can_be_points}
 This remark on HHS structures allows us to simplify the setup that we have to deal with in Sections \ref{sec:stable trees}, and aligns us with the setup in \cite[Remark 1]{DMS_bary}. The remark is that, given an HHS, we can $\mathrm{Aut}(\cuco{X},\mathfrak S)$--equivariantly change the structure in a way that all $\pi_V(x)$ and $\rho^U_V$ for $U\propnest V$ are points, rather than bounded sets, and that moreover the new structure has stable projections if the old one did. This can be achieved by replacing each $\calC(V)$ by the nerve of the covering given by subsets of sufficiently large diameter (which is quasi-isometric to $\calC(V)$). In particular, the vertices of the new $\calC(V)$ are labeled by bounded sets, and we can redefine $\pi_V(x)$ to be the vertex labeled by $\pi_V(x)$, and similarly for $\rho^U_V$.
\end{rem}

We shall make a standing assumption that any colorable HHS is equipped with \emph{stable projections} in the sense of Theorem \ref{thm:stable proj} with projections single points, as in Remark \ref{rem:rho_can_be_points}.

The following useful consequence of Theorem \ref{thm:stable proj} gets used below in the proof of Proposition \ref{prop:sporadic domains}:

\begin{lemma}\label{lem:partial order}
    For any $x,y \in \calX$, the following hold:
    \begin{itemize}
        \item If $V_1, V_2, V_3$ are pairwise transverse, then $\pi_{V_1}(y) = \pi_{V_1}(x') = \rho^{V_2}_{V_1}$
        \item There exists $K = K(\mathfrak S)>0$ and $B_0= B_0(\mathfrak S)>0$ so that if $\calV \subset \Rel_K(x,y)$ contains no pairwise transverse triple of domains, then $\#\calV < B_0$.
    \end{itemize}
\end{lemma}

\begin{proof}
    The proof of (1) is contained in the proof of \cite[Lemma 2.11]{DMS_bary}.  For (2), observe that if $\calV$ does not contain any pairwise transverse triples of domains, then there is a subcollection $\calV_0 \subset \calV$ of size $\#\calV_0$ proportional (in $\mathfrak S$) to $\#\calV$, so that $\calV_0$ is pairwise non-transverse.  But then $\#\calV_0$ is bounded by \cite[Lemma 2.2]{HHS_II}, completing the proof.
\end{proof}

\subsection{Strong passing-up}\label{subsec:SPU}

Besides the stable projections provided by Theorem \ref{thm:stable proj}, the main tool in this subsection is a powerful generalization of the basic ``passing-up'' lemma from \cite[Lemma 2.5]{HHS_II}.  This version, the Strong Passing-up Proposition \ref{prop:SPU} below, is \cite[Proposition 4.3]{Dur_infcube}.

The basic passing-up property says that given a collection of relevant domains $\calV$ for some pair of points $a,b \in \calX$, as long as we arrange for $\#\calV$ to be very large, then we can find a relevant domain $W$ with $d_W(a,b)$ as large as we like and some domain $V \in\calV$ so that $V \nest W.$

Roughly speaking, Strong Passing-up says that by making $\#\calV$ very large, we can find a subcollection $\calV_0 \subset \calV$ so that $V \nest W$ for all $V \in \calV_0$ and so that the $\rho$-sets for domains in $\calV_0$ spread out along the geodesic in $\calC(W)$.  Moreover, by again increasing $\#\calV$ if necessary, we can force the coarse density of these $\rho$-sets along this geodesic to increase.  Making this precise requires the notion of a $\sigma$-subdivision.

\begin{definition}[$\sigma$-subdivision] \label{defn:sigma subdivision}
Let $\gamma:I \to \calC(W)$ be a geodesic in $\calC(W)$ between $\pi_W(a),\pi_W(b)$ where $a,b\in \calX$.  For $\sigma>0$, we say that a subdivision $\{x_i\}$ of $I$ is a $\sigma$-\emph{subdivision} of $\gamma$ if the $x_i$ decompose $I$ into a collection of subintervals $[x_i, x_{i+1}]$ so that for all but at most one $i$ we have $|x_{i+1}-x_i| = \sigma$, with the (possibly nonexistent) extra subinterval for which $|x_{i+1} - x_i| < \sigma$.
\end{definition}

\begin{notation}
\label{not:ES}
    We denote $\ES$ a large constant that depends only on the constants in the definition of an HHS.  In order for the results in this section to hold, it suffices to fix it once and for all.  See \cite[Section 5]{Dur_infcube} for specifics.
\end{notation}

Suppose now that $x,y \in \calX$, $K_1 \geq 50E_{\mathfrak S}$, $\calV' \subset \Rel_{K_1}(a,b)$ and $W \in \Rel_{K_1}(a,b)$ so that $V \nest W$ for all $V \in \calV'$.  Let $\gamma:I \to \calC(W)$ be a geodesic between $a,b$ in $\calC(W)$.

Given a $\sigma$-subdivision of $\gamma$, let $\calW_i$ denote the set of domains $V \in \calV'$ so that $p_{\gamma}(\rho^V_W) \cap [x_i, x_{i+1}] \neq \emptyset$, where $p_{\gamma}: \calC(W) \to \gamma$ is a closest point projection.  Note that any given $V \in \calV'$ belongs to at most two $\calW_i$ when $\sigma \geq 10E_{\mathfrak S}$.

The following proposition is \cite[Proposition 4.3]{Dur_infcube}:

\begin{proposition}[Strong Passing-up]\label{prop:SPU}
For any $K_2 \geq K_1 \geq 50E_{\mathfrak S}$, there exists $P_1 = P_1(K_1, K_2)>0$ so that for any $x, y\in \calX$, if $\calV \subset \Rel_{K_1}(a,b)$ with $\#\calV > P_1$, then there exists $W \in \Rel_{K_2}(a,b)$ and $\calV' \subset \calV$ so that $V \nest W$ for all $V \in \calV'$ and 
$$\diam_W\left(\bigcup_{V \in \calV'} \rho^V_W\right) > K_2.$$

Moreover, for any $\sigma\geq 10E_{\mathfrak S}$ and $n \in \mathbb{N}$, there exists $P_2(K_1,K_2,\sigma, n)>0$ so that if $\#\calV > P_2$, then we can arrange the following to hold:
\begin{itemize}
    \item If $\gamma:I \to \calC(W)$ is geodesic in $\calC(W)$ between $a,b$ and $\{x_i\}$ is a $\sigma$-subdivision of $\gamma$ determining sets $\calW_i$ as above, then 
    $$\#\{1 \leq i \leq k| \calW_i \neq \emptyset\} \geq n.$$
\end{itemize}
\end{proposition}

\subsection{Distinguished domains}

We are now ready to prove our domain control statements.  The first regards domains $U \in \calU'$ where the projections of $F,F'$ do not behave in the expected way.

\begin{definition}[Distinguished domains]\label{defn:distinguished}
A domain $U \in \calU'$ is \emph{distinguished} if either
\begin{enumerate}
    \item $U \in \calU$ and $\pi_U(x) \neq \pi_U(x')$ for all $x \in F$ and some $x' \in F'-F$.
    \item $U \in \calU' - \calU$ and there exist $x,y \in F$ so that $\pi_U(x) \neq \pi_U(y)$.
\end{enumerate}
\begin{itemize}
    \item We let $\calD(F,F')$ denote the set of distinguished domains.
\end{itemize}
\end{definition}

\begin{remark}\label{rem:non-distinguished}
    We remark that if $U \in \calU$ is not distinguished, then for all $x' \in F' - F$, there exists $x \in F$ so that $\pi_U(x) = \pi_U(x')$, i.e. $\pi_U(F') = \pi_U(F)$.  In particular, in case (1), adding the projections of $F'- F$ to $F$ does not create any new data.  Similarly, if $U \in \calU' - \calU$ is not distinguished, then $\pi_U(F)$ is a single point.
\end{remark}

The next proposition bounds the number of distinguished domains.

\begin{proposition} \label{prop:distinguished domains}
    There exists $D_1 = D_1(\mathfrak S, \#F')>0$ so that $\#\calD(F,F') < D_1$.
\end{proposition}

\begin{proof}
In this proof, many of the constants and bounds will depend on our largeness constant $K$.  In these arguments, we are free to choose $K= K(\mathfrak S)$ to be as large as necessary, though still bounded in terms of the ambient HHS structure $\mathfrak S$.

    Let $\calV$ be a collection of distinguished domains.  Unless their number is bounded in terms of $\mathfrak S, \#F'$, we can pass to a subcollection so that the domains in $\calV$ are all either
    \begin{enumerate}
        \item distinguished of type (1), so that all domains in $\calV$ are
        \begin{itemize}
            \item distinguished by some fixed $z \in F'$, i.e. so that $\pi_V(z) \neq \pi_V(f)$ for all $f \in F$;
            \item contained in $\Rel_K(x,y)$ for fixed $x,y \in F$; and
            \item lie in a single BBF family $\mathfrak S_j$.
        \end{itemize}
        \item distinguished of type (2), so that all domains in $\calV$ are
        \begin{itemize}
            \item distinguished by a fixed pair $x,y \in F$, i.e., so that $\pi_V(x) \neq \pi_V(y)$ for all $V \in \calV$;
            \item contained in $\Rel_K(x',y')$ for fixed $x',y' \in F'-F$; and
            \item lie in a single BBF family $\mathfrak S_j$.
        \end{itemize}
        
    \end{enumerate}

    Assume we are in case (1).  Let $\sigma= K/100$.  Assuming that $\#\calV$ is sufficiently large (in terms of $\mathfrak S, K$), Strong Passing-up \ref{prop:SPU} provides a domain $W \in \Rel_K(x,y)$ with $V \nest W$ for all $V\in \calV$, so that if $\gamma$ is a geodesic between $\pi_W(x), \pi_W(y)$, and $\{w_1,\dots, w_n\}$ is a $\sigma$-subdivision of $\gamma$, then there exist $V,V' \in \calV$ so that
    
    \begin{itemize}
        \item $\rho^V_W$ and $\rho^{V'}_W$ are at least $K/8$ far from $x,y$ and
        \item $\rho^V_W$ and $\rho^{V'}_W$ are at least $K/100$ apart.
    \end{itemize}

Moreover, assuming without loss of generality that $d_W(x,p_{\gamma}(z))>K/2$, by increasing $\#\calV$ only a bounded amount (in terms of $\mathfrak S,K$) if necessary, we can arrange that

\begin{itemize}
    \item $p_{\gamma}(\rho^V_W), p_{\gamma}(\rho^{V'}_W)$ separate $x$ from $p_{\gamma}(z)$ along $\gamma$.
\end{itemize}

The point here is that by increasing the coarse density constant $n$ in Proposition \ref{prop:SPU} a bounded amount (in $\mathfrak S, K$), we can arrange for $V,V' \in \calV$ to be proximate to two non-adjacent $\sigma$-subdivision subintervals between $x, p_{\gamma}(z)$.

\begin{figure}
    \centering
    \includegraphics[width=\textwidth]{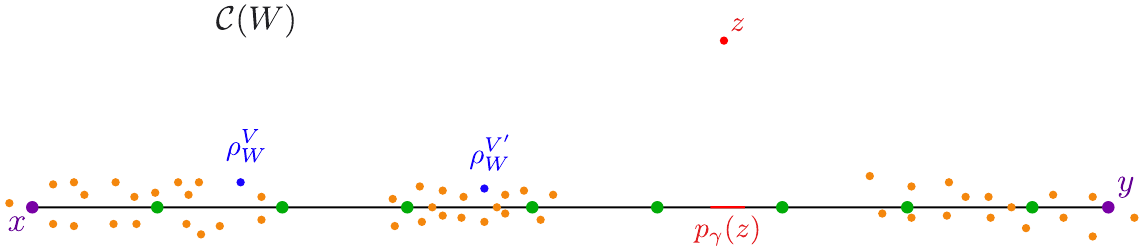}
    \caption{Case (1) of Proposition \ref{prop:distinguished domains}: $\pi_{\gamma}(z)$ must be far from at least one of $x,y$, so Proposition \ref{prop:SPU} provides domains $V, V' \in \calU$ distinguished by $z$ with $V,V' \nest W$ so that $x,y, \rho^V_W, \rho^{V'}_W$, and $\pi_{\gamma}(z)$ are separated from each other in useful ways.  This allows us to conclude via Theorem \ref{thm:stable proj} that $\pi_V(z)$ is one of $\pi_V(x)$ or $\pi_V(y)$, which is a contradiction.}
    \label{fig:distinguished}
\end{figure}

Assuming then that $p_{\gamma}(\rho^{V'}_W)$ appears after $p_{\gamma}(\rho^{V}_W)$ along $\gamma$ on the way from $x$ to $p_{\gamma}(z)$, a basic consequence of the Bounded Geodesic Image axiom implies that $d_{V'}(y, \rho^{V'}_V)>\theta$ and $d_{V'}(z, \rho^{V'}_V)>\theta$ for $\theta = \theta(\mathfrak S)$ as in Theorem \ref{thm:stable proj}.  This requires making $K = K(\mathfrak S)$ sufficiently large so that $K/100$ is large enough to invoke the Bounded Geodesic Image axiom.  Hence by item \eqref{item:strict_Behrstock_X} of Theorem \ref{thm:stable proj}, we have $\pi_V(y) = \rho^{V'}_V = \pi_V(z)$, which contradicts the assumption that $V$ is distinguished by $z$, i.e. $\pi_V(z) \neq \pi_V(f)$ for all $f \in F$.  This deals with case (1).

For case (2), we make variations on the above argument.  Again let $\sigma = K/100$.  Assuming that $\#\calV$ is sufficiently large (in terms of $\mathfrak S, K$), Proposition \ref{prop:SPU} again provides a domain $W \in \Rel_{100K}(x',y')$ with $V \nest W$ for all $V \in \calV$, so that if $\gamma$ is a geodesic between $\pi_W(x'), \pi_W(y')$, and $\{w_1,\dots, w_n\}$ is a $\sigma$-subdivision of $\gamma$, then there exist $V_1,V_2,V_3 \in \calV$ so that
    
    \begin{itemize}
        \item Each of the $\rho^{V_i}_W$ is at least $K/8$ from $x',y'$ and
        \item The $\rho^{V_i}_W$ are pairwise at least $K/100$ apart.
    \end{itemize}

There are two subcases.  In the first, both of $p_{\gamma}(x)$ and $p_{\gamma}(y)$ are at least $K/20$ away from one of $x'$ or $y'$.  Then we can argue again as in case (1) to arrange  $\pi_V(x)=\pi_V(y)$, a contradiction.  Otherwise, we are in the second subcase, where $p_{\gamma}(x)$ is close to $x'$ and $p_{\gamma}(y)$ is close to $y'$.  Then we can arrange for each of the $p_{\gamma}(\rho^{V_i}_W)$ to separate $p_{\gamma}(x)$ and $p_{\gamma}(y)$ along $\gamma$, with $p_{\gamma}(\rho^{V_2}_W)$ in between $p_{\gamma}(\rho^{V_1}_W)$ and $p_{\gamma}(\rho^{V_3}_W)$.  But then an application of item \eqref{item:strict_Behrstock_X} of Theorem \ref{thm:stable proj} implies that $\pi_{V_2}(x) = \rho^{V_1}_{V_2} = \pi_{V_2}(x')$ and $\pi_{V_2}(y) = \rho^{V_3}_{V_2} = \pi_{V_2}(y')$, so that $V_2 \in \Rel_K(x,y)$, which is a contradiction.  This completes the proof of the case and the proposition.
\end{proof}

\subsection{Involved domains}

Our next goal is to control the number of domains in $\calU' - \calU$ which nest into domains in $\calU$.  This will give us control over the $\rho$-sets that are involved in the stable tree construction (Subsection \ref{subsec:stable trees defined}).

\begin{definition}\label{defn:involved}
    A domain $U \in \calU$ is \emph{involved} if $\{V \in \calU| V \nest U\} \neq \{V \in \calU'|V \nest U\}$.
    \begin{itemize}
        \item We let $\calI(F,F')$ denote the set of involved domains.
    \end{itemize}
\end{definition}

\begin{proposition}\label{prop:involved domains}
    There exists $D_2 = D_2(\mathfrak S, \#F')>0$ so that $\#\calI(F,F') < D_2$.
\end{proposition}

\begin{proof}
For each $U \in \calI(F,F')$, choose some $V \in \calU' - \calU$ so that $V \nest U$, and let $\calV$ denote the collection of these domains $V$.  By passing to a subcollection of $\calI(F,F')$ of size uniformly (in terms of $\mathfrak S, |F'|$) proportional to $|\calI(F,F')|$ if necessary, we may assume that 
\begin{enumerate}
    \item There are fixed $x,y \in F$ so that if $U \in \calI(F,F')$, then $U \in \Rel_K(x,y)$;
    \item We have $\calV \subset \mathfrak S_i$ for a fixed BBF family;
    \item There are fixed $x',y' \in F'$ so that if $V \in \calV$, then $V \in \Rel_K(x',y')$;
    \item We have $\calW \subset \mathfrak S_j$ for some other fixed BBF family.
\end{enumerate}

As in the proof of Proposition \ref{prop:distinguished domains}, we want to apply Strong Passing-up \ref{prop:SPU} and Theorem \ref{thm:stable proj} to obtain a contradiction by forcing $V \in \calV$ to be in $\Rel_K(F)$.  Also, as in that proof, we can choose $K = K(\mathfrak S,|F'|)$ to be as large as we like.

By further assuming that $\#\calV$ is sufficiently large (in terms of $\mathfrak S, |F'|$), we get some domain $W \in \Rel_{100K}(x,y)$ with $V \nest W$ for all $V \in \calV$, so that if $\gamma$ is a geodesic between $x,y$ in $\calC(W)$, then there exist $V_1, V_2,V_3 \in \calV$ so that
\begin{itemize}
    \item Each $p_{\gamma}(\rho^{V_i}_W)$ is at least $2K$ away from $x,y$,
    \item The $p_{\gamma}(\rho^{V_i}_W)$ are at least pairwise $2K$ apart, with the $p_{\gamma}(\rho^{V_i}_W)$ appearing in order from $x,y$ along $\gamma$.
\end{itemize}

By the Bounded Geodesic Image axiom, any geodesic in $\calC(W)$ between $x',y'$ must pass uniformly close (depending only on $\mathfrak S$ and not on $K$) to each $\rho^{V_i}_W$.  Once again, applying item \eqref{item:strict_Behrstock_X} of Theorem \ref{thm:stable proj}, it follows that $\pi_{V_2}(x) = \rho^{V_1}_{V_2} = \pi_{V_2}(x')$ and $\pi_{V_2}(y) = \rho^{V_3}_{V_2} = \pi_{V_2}(y')$, making $V_2 \in \Rel_K(x,y)$ while also $V_2 \in \calU' - \calU$ by assumption, which is a contradiction.  This completes the proof.

\end{proof}

\subsection{Sporadic domains}

In this subsection, we turn toward proving a key bound for our Stabler Tree Theorem \ref{thm:stabler tree}.  Roughly, when $U \in \calU$, there may be unboundedly-many domains $V \nest U$ for which $V \in \calU' - \calU$.  For our tree modeling purposes in Section \ref{sec:stable trees}, we only need to control the number and location of these additional domains that appear near $\hull_U(F)$.  In terms of Gromov modeling trees for the hulls, our main goal here is to say that all but boundedly-many of them appear along the ``new'' branches associated to the points in $F' - F$.  These troublesome domains are captured in the following definition:

\begin{definition}[Sporadic domains]\label{defn:sporadic domains}

Given $D>0$, a domain $V \in \calU' - \calU$  is $D$-\textbf{sporadic} if for some $U \in \calU$ with $V \nest U$, we have 
\begin{equation} \label{eq:sporadic}
    \rho^V_U \notin \bigcup_{x' \in F' - F} \left(\bigcap_{x \in F} \calN_D(\hull_U(x,x'))\right)
\end{equation}

\begin{itemize}
    \item We let $\calV_D$ denote the set of $D$-sporadic domains.
\end{itemize}
\end{definition}

\begin{figure}
    \centering
    \includegraphics[width=.75\textwidth]{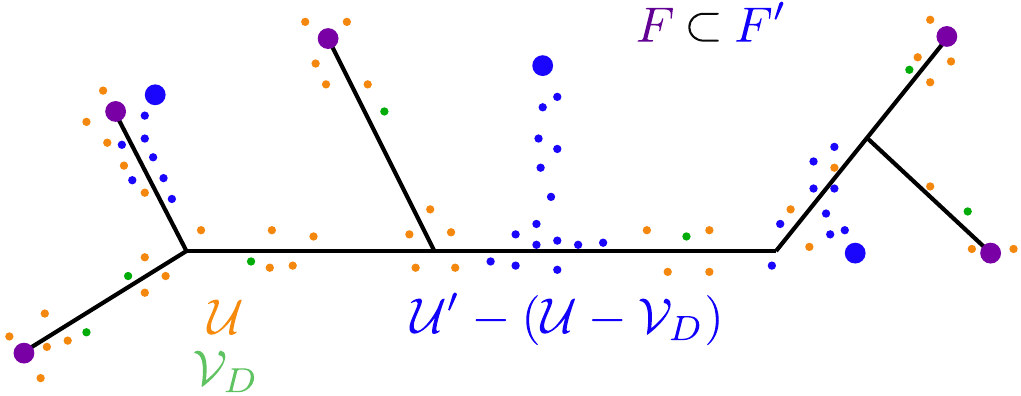}
    \caption{The motivating picture for sporadic domains.  For any large domain $V \in \calU$ with $V \nest U$, $\rho^V_U$ lies close to the hull of $F$ in $\calC(U)$, and hence near the Gromov modeling tree for $\hull_U(F)$.  Sporadic domains are precisely those whose $\rho$-points $\calV_D$ lie far away from the ``branch points'' in the tree corresponding to the points in $F' - F$.  Proposition \ref{prop:sporadic domains} controls the number of these points by Theorem \ref{thm:stable proj} and some basic Bounded Geodesic Image axiom arguments.}
    \label{fig:sporadic domains}
\end{figure}

In words, sporadic domains are those which do not cluster near that closest point projections of $x' \in F' - F$ on $\hull_U(F)$, which we think of as the ``new'' branch points associated to the points in $F' - F$.

The following proposition says we can control the number and size of sporadic domains.

\begin{proposition}\label{prop:sporadic domains}
There exists $D_0 = D_0(\mathfrak S)>0$ so that for any $D > D_0$, there exists $N = N(D, \#F',\mathfrak S)>0$ so that $|\calV_D| < N$.  Moreover, for each $V \in \calV_D$ with $V \nest U$ for $U \in \calU$, then $V \in \Rel_{K - 2E_{\mathfrak S}}(F)$.

\end{proposition}

\begin{proof}

As before, we will derive a contradiction by assuming that $\#\calV_D$ is very large.  Suppose that $V \in \calV_D$.  We begin by producing $x,y \in F$ so that $V$ is a $(K-2E_{\mathfrak S})$-large domain for $x,y$, which will prove the ``moreover'' part of the statement.

In the first case, there exist $x \in F$ and $x' \in F' - F$ so that $V \in \Rel_K(x,x')$.  By definition of $\calV_D$, there exists some $y \in F$ so that $\hull_U(x',y) \cap \calN_D(\rho^V_U) = \emptyset$.  Taking $D>\ES$ (Notation \ref{not:ES}), the Bounded Geodesic Image axiom implies that $V \in \Rel_{K - \ES}(x,y)$.

In the second case, there exist $x', y' \in F' - F$ so that $V \in \Rel_K(x',y')$.  Again by definition of $\calV_D$, there exist $x, y \in F$ so that $\hull_U(x,x') \cap \calN_D(\rho^V_U) = \emptyset$ and $\hull_U(y,y') \cap \calN_D(\rho^V_U) = \emptyset$.  Then the Bounded Geodesic Image axiom implies that $V \in \Rel_{K-2E}(x,y)$.  Hence the ``moreover'' part of the statement holds.

Since there are only boundedly-many BBF families (Definition \ref{defn:colorable}), we may pass to a subcollection $\calV_0 \subset \calV_D \subset \Rel^i_{K-2E}(x,y)$ for a fixed $i$ and fixed $x,y \in F$, where each of the domains in $\calV_0$ arises as either in the first or second cases above.  In particular, this subcollection $\calV_0$ has size proportional (in terms of $\mathfrak S, \#F',D$) to $\calV_D$.

In the first case, taking $K$ large enough depending only on $\mathfrak S$, then the first part of Lemma \ref{lem:partial order} implies that if $V_1,V_2,V_3$ are pairwise transverse, then $\pi_{V_1}(y) = \pi_{V_1}(x') = \rho^{V_2}_{V_1}$, and so in fact $V \in \Rel_K(x,y)$, which is a contradiction.  On the other hand, the second part of Lemma \ref{lem:partial order} then implies that $\#\calV_0$ is bounded in terms of $\mathfrak S$.  A similar contradiction arises if the $\calV_0$ are in the second case.  This completes the proof.
\end{proof}

\section{Stable trees} \label{sec:stable trees}

In this section, we prove a refinement of our original Stable Tree Theorem \cite[Theorem 3.2]{DMS_bary}.  This refinement is crucial for our stable tree comparison result (Theorem \ref{thm:stabler tree}) in Section \ref{sec:stabler trees}, which in turn is the key result from this paper for proving our Stabler Cubulations Theorem \ref{thm:stabler cubulations} via \cite{Dur_infcube}.

\subsection{Motivating the stable tree constructions}

When building our cubical models for the hull of a finite subset $F \subset \calX$ of an HHS later in Section \ref{sec:stabler cubulations}, the first step is to project $F$ to all of the $K$-relevant hyperbolic spaces $\calU = \Rel_K(F)$ for $K = K(\mathfrak S)$.  For each such domain $U \in \calU$, we then want to construct a tree $T_U$ which encodes the projection of $F$ to $\calC(U)$, plus all of the relative projection data from domains $V \in \calU$ with $V \nest U$, namely the $\rho$-sets $\rho^V_U \subset \calC(U)$.  These so-called \emph{stable trees} (Definition \ref{defn:stable tree}) then get modified in certain hierarchically-informed ways to become the input into the cubulation machine.

The purpose of any ``stable tree''-type theorem is to build trees $T_U$ as above which transform in a controlled way under reasonable modifications of the input set $F\rightsquigarrow F'$.  In \cite{DMS_bary}, we were interested in the case where $F, F' \subset \calX$ are at bounded Hausdorff distance in $\calX$.  We will describe this case in a bit of detail because we utilize the constructions from \cite{DMS_bary} in a fundamental way.

There we performed an analogous analysis to Section \ref{sec:controlling domains} to show that the set of $K$-relevant domains $\calU, \calU'$ for $F,F'$, respectively, had bounded symmetric difference.  On those boundedly-many domains $U \in \calU \hspace{.025in}\triangle \hspace{.05in} \calU'$, we were then faced with the fact that $F,F'$ had similar but not exactly the same projections to $\calC(U)$, and that possibly had boundedly-many different relevant $V \nest U$.  In other words, the input data into the tree construction---the projections and relative projections---were boundedly different.  The goal of the ``stable tree''-type theorem in that context was to build a tree construction so that the resulting trees were only boundedly different.  This involved a careful decomposition of the two trees $T_U, T'_U$ into a collection of ``stable'' pieces which were mostly identical, with boundedly-many nearly identical pieces, with the complementary ``unstable'' pieces of $T_U,T'_U$ bounded in number and diameter.  Such a decomposition then allows us to see $T_U,T'_U$ are isometric up to collapsing the unstable pieces to points.  See the Definition \ref{defn:stable decomp} of a \emph{stable decomposition} and the Stable Tree Theorem \ref{thm:stable tree} for precise details.

Our current goal is philosophically aligned, but the details are quite different.  In our current case, we are adding points to our set $F \subset \calX$, to obtain a new set $F'\subset \calX$.  Importantly, the points in $F' - F$ can be arbitrarily far from $F$.   So even though in Section \ref{sec:controlling domains} we obtained a bound on the number of domains $U \in \calU$ where the projections of $F,F'$ are different, i.e. $\pi_U(F) \neq \pi_U(F')$ (Proposition \ref{prop:distinguished domains}), and where the relative projections are different for $F,F'$ (Proposition \ref{prop:involved domains}), we are still in a very different situation in the (boundedly-many) remaining domains $U \in \calU$ where the projection data is possibly arbitrarily different.  In the end, however, we are not looking for the resulting trees $T_U,T'_U$ to be identical up to collapsing subtrees.  Rather, we need $T_U$ to admit a convex embedding into $T'_U$ up to collapsing the unstable pieces.  Hence we need stable decompositions of $T_U,T'_U$ so that stable pieces of $T_U$ are mostly identical stable pieces of $T'_U$, up to a few which are nearly identical, with the unstable pieces of $T_U$ being bounded in diameter and number.

The following is an informal statement which combines the Stabler Tree Theorem \ref{thm:stabler tree} with Proposition \ref{cor:collapsed isometry} which describes the control we gain after collapsing the unstable pieces.  The precise statements are in Theorem \ref{thm:stabler tree} and Definition \ref{defn:stable decomp}.

To set a bit of notation, we let $\calZ$ be a $\delta$-hyperbolic space.  Given any finite $F \subset \calZ$, we let $\lambda(F)$ denote any Gromov modeling tree in $\calZ$ for the points $F$ (see Subsection \ref{subsec:basic tree setup} for a discussion of this network function $\lambda$).   If $F\subset F'$ is finite, then $\lambda(F)$ is contained in a uniform neighborhood of $\lambda(F')$, coarsely coinciding with $\hull_{\lambda(F')}(F)\subset \lambda(F')$.  By the branches corresponding to $F'-F$, we mean the coarse complement of this hull in $\lambda(F')$.

\begin{theorem}[Stabler trees, informal]\label{thm:stabler tree informal}
    Let $F \subset F' \subset \calZ$ be finite subsets of a $\delta$-hyperbolic space $\calZ$, and let $\calY \subset \calY'$ be finite sets of points uniformly close to $\hull_{\calZ}(F)$ and $\hull_{\calZ}(F')$, respectively.  If the number of points in $\calY'- \calY$ that avoid the branches corresponding to $F'-F$ is uniformly bounded, then
    \begin{enumerate}
        \item There exist trees $T, T'$ which model $\hull_{\calZ}(F)$ and $\hull_{\calZ}(F')$, respectively;
        \item There are partitions $T = T_s \cup T_u$ and $T' = T'_s \cup T'_u$ into subtrees and a bijection $\alpha:\pi_0(T_s) \to \pi_0(T'_s)$, so that the following holds:
        \begin{itemize}
            \item If $\Delta:T \to \hT$ and $\Delta':T' \to \hT'$ collapse the components of $T_u, T'_u$ to points, then there is a convex embedding $\Phi:\hT' \to \hT'$ which identifies components of $T_s,T'_s$ identified by $\alpha$.
        \end{itemize}
    \end{enumerate}
\end{theorem}

\begin{remark}
    In Theorem \ref{thm:stabler tree informal}, the sets $\calY$ and $\calY'$ represent the relative projections that we will eventually need to consider in the hierarchical setting.  In the full version, Theorem \ref{thm:stabler tree} and its consequence Corollary \ref{cor:collapsed isometry}, the convex embedding $\Phi:\hT \to \hT'$ between collapsed trees carries additional information about the points in $F$ and the sets $\calY, \calY'$.  While too technical to state informally, this extra information is crucial in Section \ref{sec:stabler cubulations}.
\end{remark}

\subsection{Structure of Sections \ref{sec:stable trees} and \ref{sec:stabler trees}}

The formulation and proof of the Stabler Tree Theorem \ref{thm:stabler tree} will happen over the course of the next two sections.  Roughly, in Section \ref{sec:stable trees} we prove a refined version of our stable tree results from \cite[Section 3]{DMS_bary}, and in Section \ref{sec:stabler trees}, we use this refinement as a tool for proving our more powerful ``add a point'' version in Theorem \ref{thm:stabler tree}.

The rest of this section proceeds by introducing and recalling the key facts about our construction of stable trees from \cite[Section 3]{DMS_bary}.  With these in hand, we will then turn to proving the refined version.  This involves a careful analysis of how changes of the input data result in various changes in the output trees.  See Subsection \ref{subsec:stable tree outline} for a detailed outline, which will be possible once we have established basic terminology.

\subsection{Basic setup for the stable tree construction} \label{subsec:basic tree setup}

For the rest of this section, fix a $\delta$-hyperbolic geodesic space $\calZ$.  For a finite subset $F\subset \calZ$ let $\hull_{\calZ}(F)\subset\calZ$
be the set of geodesics connecting points of $F$. Hyperbolicity tells us that $\hull_{\calZ}(F)$
can be approximated by a finite tree with accuracy depending only on $\delta$ and the
cardinality $\# F$. To systematize this for the purposes of this section, we make the
following definitions.

First, we fix a function $\ltree$ which, to any finite subset $F$ of
$\calZ$, assigns a minimal network spanning $F$. That is, $\ltree(F)$ is a 1-complex embedded in $\calZ$ with the property
that $\ltree(F)\cup F$ is connected, and has minimal length among all such
1-complexes (where the length of a 1-complex embedded in $\calZ$ is the sum of the lengths of all edges).  Minimality implies $\ltree(F)$ is a tree.

The following lemma summarizes the properties of the minimal networks $\lambda(F)$, we leave its proof to the reader.

\begin{lemma}\label{lem:basic tree lemma} 
Let $\calZ$ be a geodesic $\delta$-hyperbolic space and $\lambda$ the minimal network function
as above. For any choice of $k>0$, there exists $\epsilon_0=\epsilon_0(k,\delta)$ so that for all $\epsilon\geq \epsilon_0$ there exists $\ep' = \ep'(k, \ep)>0$ such that, if $F\subset \calZ$  with $|F|<k$ then
\begin{enumerate}
  \item There is 
a $(1,\ep/2)$-quasi-isometry $\ltree(F)\to \hull_{\calZ}(F)$ which is $\ep/2$-far from the
inclusion of $\ltree(F)$ in $\calZ$.
\item
For any two points $x,y\in
\calN_\ep(\ltree(F))$, any geodesic joining them is in $\calN_{\ep'}(\ltree(F))$.
\end{enumerate}
\end{lemma}

\begin{remark}\label{rem:ep notation}
    In what follows, we will need to work with different values of $\ep_0$ for different values of $k$.  Toward that end, we will adopt the notation $\ep_0(k)$ to denote this constant for a given value of $k>0$.
\end{remark}

We similarly define an additional network function 
$\lambda'$ which assigns, to any finite collection $A_1,\ldots,A_k$ of finite subsets of $\calZ$,
a minimal network that spans them. 
That is, $\lambda'(A_1,\ldots,A_k)$ is a 1-complex in $\calZ$ of minimal length with the
property that the quotient of $\lambda'(A_1,\dots,A_k)$ obtained by collapsing
each $A_i$ to a point is connected. Minimality again implies that this collapsed graph is a
tree, and hence $\lambda'(A_1, \dots, A_k)$ is a forest.  For convenience we write $\ltree(\{x_1,\ldots,x_k\}) =
\lambda'(\{x_1\},\ldots,\{x_k\})$.

We make an additional requirement, following a definition.  We say that subsets $A_1, A_3$ are \emph{$\epsilon$-separated} by $A_2$ if there exists a minimal length $\calZ$-geodesic $\sigma$ which connects $A_1,A_3$ and passes within $2\epsilon$ of $A_2$.

\begin{lemma}\label{lem:inductive forest}
Suppose that $A_1, \dots, A_n, A'_{n+1}, \dots, A'_m \subset \calZ$ are collections of pairwise disjoint finite subsets in $\calZ$, with $d_{\calZ}(A_n, A'_{n+1})<\epsilon$ for $\epsilon>\epsilon_0$ as in Lemma \ref{lem:basic tree lemma}.  Suppose that for any $1\leq i \leq n-1$ and $n+1 \leq j \leq m$, we have that $A_i, A'_j$ are $\epsilon$-separated by $A_n$.

Then any component of $\lambda'(A_1, \dots, A_n)$ with an endpoint on $A_n$ is a component of $\lambda'(A_1, \dots, A_n \cup A'_{n+1}, A'_{n+2}, \dots, A'_m)$, and the latter forest contains no components connecting $A_1, \dots, A_{n-1}$ to $A'_{n+2}, \dots, A'_m$.
\end{lemma}

\begin{proof}
    The proof is a straight-forward application of basic $\delta$-hyperbolic geometry.  The main idea is that, because $A_k$ separates the $A_1, \dots, A_{n-1}$ from $A'_{n+1}, \dots, A'_m$, adding the points in $A'_{n+1}$ to $A_n$ will not create the ability to reduce the length of any component of $\lambda'(A_1, \dots, A_n)$ by connecting it to some point in $A'_{n+1}, \dots, A'_m$.  Thus, by inducting both on the number of subsets $n,m$ and their cardinalities (which are finite by assumption), we can arrange that components of the minimal networks connecting the various $A_i$ to $A_n$ and $A_n\cup A'_{n+1}$ in $\lambda'(A_1, \dots, A_n)$ and $\lambda'(A_1, \dots, A_n\cup A'_{n+1}, \dots, A'_m)$, respectively, are identical.  We leave the details to the reader.
\end{proof}

\subsection{$\epsilon$-setups}\label{subsec:ep-setup}

The next definition is used throughout the rest of this section and Section \ref{sec:stabler cubulations}.

\begin{definition}\label{defn:ep setup}
    For $\epsilon>\epsilon_0$, an $\epsilon$-\textbf{setup} is a pair $(F, \calY)$ where $\calY \subset \calZ$ is a finite (but possibly arbitrarily large) set of points of $\calZ$ with the property that $d_{\calZ}(\lambda(F), y) < \epsilon/2$ for all $y \in \calY$.  We call such an arrangement $(F;\calY)$ an $\epsilon$-\emph{setup} in $\calZ$.
\end{definition}

\begin{remark}[Relation to HHSs]
    In the hierarchical setting, $\calZ$ will be one of the hyperbolic spaces $\calC(U)$ for $U \in \calU$, and the points $\calY$ will be the points in $\rho^V_U$ for $V \nest U \in \calU$ with a large projection.  The proximity condition of the points in $\calY$ to $\lambda(F)$ is encoding the Bounded Geodesic Image axiom.
\end{remark}

Finally, observe that as a consequence of Lemma \ref{lem:basic tree lemma}, if $C, C' \subset \calY$ are finite subsets, then $\lambda'(C,C') \subset \calN_{\epsilon'}(\lambda(F))$.

\subsection{Cluster and shadows} \label{subsec:cluster basics}

The purpose of this subsection and the next is to describe our stable tree construction.  The first step of this process is defining the cluster graph.

\begin{definition}\label{defn:clusters, separation}
Let $\ep>0$ and $(F; \calY)$ be an $\epsilon$-setup in $\calZ$ as in Definition \ref{defn:ep setup}.  Given $\ep', E \gg \epsilon$, let $\calC_E(F \cup \calY)$ be the graph whose edges connect points in $\calY \cup F$ which are at most $E$ apart.  We call the connected components of $\calC_E(F \cup \calY)$ $E$-\emph{clusters}.

As above, given three $E$-clusters $C_1,C_2,C_3$, we say that $C_2$ $\ep'$-\emph{separates} $C_1$ from $C_3$ if there exists a minimal length $\calZ$-geodesic segment $\sigma$ with endpoints on $C_1,C_3$ which passes through $N_{2\epsilon'}(C_2)$ in $\calZ$.
\end{definition}

\begin{definition}[Cluster separation graph]\label{defn:cluster sep graph}
Given constants $\ep, \ep', E>0$, the $(\ep, \ep',E)$-\emph{cluster separation graph} for the $\epsilon$-setup $(F;\calY)$ is the graph $\calG_E(F \cup \calY)$ given by the following data:
\begin{itemize}
\item The vertices of $\calG_E(F \cup \calY)$ are $E$-clusters.
\item Two $E$-clusters $C_1,C_2$ are connected by an edge in $\calG_E(F \cup \calY)$ whenever $C_1$ and $C_2$ are not $\ep'$-separated by another $E$-cluster.
\end{itemize}
\end{definition}

In \cite[Subsection 3.1]{DMS_bary}, we analyzed the structure of this graph by using the minimal network $\lambda(F)$ as a reference object.

\begin{definition}[Shadows]\label{defn:shadows}
    Given any subset $A \subset \calZ$, the \emph{shadow} $s(A)$ of $A$ on the tree $\ltree(F)$ is the convex hull (in $\ltree(F)$) of all points in $\calN_{\ep}(A) \cap \lambda(F)$. 
\end{definition}

Roughly speaking, the shadow $s(C)$ of an $E$-cluster encodes the location information of $C$ onto the tree $\lambda(C)$, allowing use hyperbolic geometry arguments to understand how the clusters are arranged.

Toward that end, the tree $\ltree(F)$ has valence bounded in terms of $\#F$, so if its diameter is large enough, most $E$-clusters $C$ will determine a bivalent vertex of $\calG_E(F \cup \calY)$.  Some such clusters will contain points of $F$.

\begin{definition}[Bivalent cluster]\label{defn:bivalent}
A cluster $C$ is \emph{bivalent} if it determines a bivalent vertex of $\calG_E(F \cup \calY)$ and does not contain a point of $F$.  We let $\calE^0$ denote the set of bivalent clusters.
\end{definition}

The following lemma gives basic properties of the cluster separation graph and shadows:

\begin{lemma}\label{lem:cluster lemma}
If $2\ep' > \ep + \ep'$ and $E \geq 8\epsilon'$, then the following hold:
\begin{enumerate}
    \item $\calG_E(F \cup \calY)$ is connected.
    \item For distinct $E$-clusters $C,C'$:
    \begin{enumerate}
        \item $s(C) \cap s(C')$ contains no leaf of $s(C)$ or $s(C')$,
        \item The diameter of $s(C) \cap s(C')$ is bounded in terms of $\#F,E$,
        \item If at least one of $s(C)$ or $s(C')$ is an interval along an edge $\lambda(F)$, then $s(C) \cap s(C') = \emptyset$.
    \end{enumerate}
     \item The number of non-bivalent clusters $\#(\calG^0 - \calE^0)$ is bounded in terms of $\#F$.
    \item Every bivalent cluster $C$ has shadow $s(C)$ lying along an edge of $\lambda(F)$. 
    \item If $C_1, \dots, C_n$ are bivalent clusters whose shadows lie along an edge of $\lambda(F)$ in that order, then their shadows $s(C_i)$ on $\lambda(F)$ are disjoint, lie along the edge in the given order $s(C_1), \dots, s(C_n)$, and $d_{\calZ}(C_i, C_j) \geq (j-i)M_1$, for $M_1 = M_1(k, \delta)$.
    \item For any $D>0$, there exists $n = n(D, k, \delta)>0$ so that if $C_1, C_2 \in \calG^0$ with $d_{\calG}(C_1,C_2)\geq n$, then $d_{\calZ}(C_1, C_2) \geq D$.
\end{enumerate}
\end{lemma}

\begin{proof}
    Item (1) is \cite[Lemma 3.4]{DMS_bary}, parts (a)--(c) of item (2) is \cite[Lemma 3.6]{DMS_bary}, item (3) is \cite[Lemma 3.11]{DMS_bary}, and item (4) is \cite[Lemma 3.8]{DMS_bary}, while item (5) combines \cite[Lemma 3.9 and Claim 1 of Lemma 3.6]{DMS_bary}.

    Finally, item (6) is not explicitly contained in \cite{DMS_bary}, but is an easy consequence of our work there, as follows.  Suppose that $C_1 = D_0, \dots, D_n = C_2$ is a geodesic in $\calG$ between $C_1,C_2$.  By item (3) of this lemma, any such geodesic contains boundedly many (in $k,\delta$) non-bivalent clusters.  Hence if $n$ is sufficiently large (in $k,\delta$), then we can find some subcollection $D_j, \dots, D_k$ of consecutive bivalent $D_i$.  Since the $D_j, \dots, D_k$ form a geodesic in $\calG$, they are contained in a single component of the subgraph $\calE$ of $\calG$ induced by the vertices of $\calE^0$.

    By \cite[Lemma 3.12]{DMS_bary}, there exists some edge $e$ of $\lambda(F)$ so that the shadows $s(D_j), \dots, s(D_k)$ lie along $e$.  Now by item (5) of this lemma, we have that $d_{\calZ}(D_j, D_k) > M_1 (k-j)$, for $M_1 = M_1(k,\delta)$.  On the other hand, by definition of $\calG$, the $D_j, \dots, D_k$ $\ep$-separate $C_1$ from $C_2$, and hence any minimal length geodesic from $C_1$  to $C_2$ must pass within $\ep/2$ of each of the $D_j, \dots, D_k$.

    Thus by forcing $n$ to be sufficiently large in terms of $k,\delta$ and our given $\calZ$-distance bound $D$, we can force $k-j$ to be large enough so that $d_{\calZ}(C_1, C_2)>D$.  This completes the proof of (6) and the lemma.
\end{proof}

For the rest of this section, we will freely cite other specific references from \cite{DMS_bary} directly, not necessarily stating them independently.

\subsection{Fixing notation}\label{subsec:fixing constants, stable tree}

For the rest of this section, we fix the following information:

\begin{itemize}
    \item A natural number $k$, which globally controls the size of our finite subsets.
    \item A positive number $\ep>\ep_0(k, \delta)>0$ as in  Lemma \ref{lem:basic tree lemma};
    \item An $\ep$-setup $(F;\calY)$ in $\calZ$ with $|F|\leq k$.
    \item A positive number $\ep' = \ep'(k, \delta, \ep)>0$ so that $2\ep' > \ep + \ep'$ as in Lemma \ref{lem:basic tree lemma}.
    
    \item A cluster separation constant $E = E(k, \delta, \ep)>0$ sufficiently large so that $E > 8\ep'$ as in Lemma \ref{lem:cluster lemma}.
\end{itemize}

These constants are the parameters which control the cluster separation graph $\calG_E(F)$ associated to any $\ep$-setup $(F; \calY)$, which is the key input for our stable tree construction, which we give in the next subsection.

\subsection{Stable trees defined}\label{subsec:stable trees defined}

We are now ready to define the stable tree for any $\epsilon$-setup $(F; \calY)$.  The idea is that we want to the minimal network function $\lambda'$ to connect adjacent clusters in $\calG_E(F \cup \calY)$ and the other network function $\lambda$ to internally connect a cluster via its various connection points to its neighbors.

More precisely, we define two forests $T_c(F, \calY)$ and $T_e(F, \calY)$ as follows.  Let $\calV$ denote the set of closures of components $\pi_0(\calG - \calE^0)$.  For each $V \in \calV$, let $V^0$ denote the clusters which form its vertex set.  We note that some elements of $\calV$ are single edges $[C,C']$ connecting clusters $C,C' \in \calE^0$, while others are subgraphs of $\calG_E(F \cup \calY)$ containing vertices of $\calG^0 - \calE^0$ (whose cardinality is bounded by Lemma \ref{lem:cluster lemma}).  Note that these forests implicitly depend on the constants $\ep,\ep', E$ as chosen in Subsection \ref{subsec:fixing constants, stable tree}.

For each $V \in \calV$, let $\lambda(V^0)$ denote the minimal network connecting the clusters in $V^0$.  We define our \emph{edge forest} as
$$T_e = T_e(F, \calY) = \bigsqcup_{V \in \calV} \lambda'(V^0).$$

To define our cluster forest, we want to internally connect each cluster $C \in \calG^0$ as follows.  Let $r(c) = C \cap (T_e \cup F)$.  Define $\mu(C)$ to be the tree $\lambda(r(C))$, and define the \emph{cluster forest} to be
$$T_c = T_c(F, \calY) = \bigsqcup_{C \in \calG^0} \mu(C).$$

We can now define our stable trees as the abstract union of these two forests:

\begin{definition}[Stable tree]\label{defn:stable tree}
    The $(\ep, \ep', E)$-\emph{stable tree} for an $\epsilon$-setup $(F; \calY)$ in $\calZ$,  is
    $$T(F,\calY) = T_e(F, \calY) \cup T_c(F, \calY).$$
\end{definition}

Observe that if we collapse all components of $T_c$ to points, then $T$ becomes a connected network $N$, which is just a union of trees connected at vertices, each of which corresponds to a cluster in $\calE^0$.  But each cluster in $\calE^0$ disconnects $\calG_E(F \cup \calY)$, and so $N$ is a tree.  This observation will be important going forward.

Also observe that the minimal network maps combine to give a global map 
$$\phi:T \to \calZ,$$ meaning that stable trees wear two hats, one being as abstract unions of the edge and cluster forests, and the other being as a concrete realization of their components in $\calZ$.  Notably, the images in $\calZ$ under $\phi$ of the components of $T$ can overlap, but only a bounded amount.  See \cite[Figures 10 and 11]{DMS_bary} for a discussion.

The following lemma gives the basic properties of our stable trees:

\begin{lemma}\label{lem:stable tree basics}
    For a choice of $k>0$ and constants $\ep=\ep(k,\delta)>0, \ep'(\ep, k, \delta)>0$, and $E=E(\ep', k, \delta)>0$ as chosen in Subsection \ref{subsec:fixing constants, stable tree} and any $\ep$-setup $(F; \calY)$ in $\calZ$, the following hold for the $(\ep, \ep', E$)-stable tree $T = T_e \cup T_c$:
    \begin{enumerate}
        \item The natural map $\phi:T \to \calZ$ is a $(K_1,K_1)$-quasi-isometric embedding with $d^{Haus}_{\calZ}(\ltree(F), \phi(T))<K_1$ for $K_1=K_1(k, \delta)$.
        \item The total branching $b=b(T)$ is bounded in terms of $k, \delta$, and the leaves of $T$ are contained in $F \cup \calY$.
        \item There exists $D_i = D_i(k, \delta)>0$ for $i=1,2$ so that for each cluster $C \in \calG^0$, we have $\mu(C) \subset \calN_{D_1}(C)$, so $T_c \subset \calN_{D_2}(F \cup \calY)$.
        \item There exists $D_3 = D_3(k, \delta)>0$ so that for all $p \in T_e$, we have $d_{\calZ}(p, F \cup \calY) \geq \frac{1}{b}d_T(p, \partial T_e) - D_3$.
        
    \end{enumerate}
\end{lemma}

\begin{proof}
    Item (1) is \cite[Proposition 3.14]{DMS_bary}, while items (2)--(4) are from \cite[Lemma 3.13]{DMS_bary}.
\end{proof}

\begin{remark}
    In the rest of this section, we will usually ignore the constants $\ep,\ep', E$ when talking about the stable tree for a given $\ep$-setup $(F;\calY)$.  The dependence of our arguments on these constants is only particularly relevant in Section \ref{sec:stabler trees}.
\end{remark}

\subsection{Admissible setups and stable decompositions}\label{subsec:stable decomp}

Having defined stable trees above (Definition \ref{defn:stable tree}), we are almost ready to describe stable decompositions of stable trees.  The motivation for the definitions in this section comes from our eventual desire to plug things into our cubical model machinery.  The main theorem of the next section, Theorem \ref{thm:stabler tree}, proves that two admissible (as in the next definition) $\epsilon$-setups $(F; \calY)$ and $(F';\calY')$ admit stable trees $T, T$ with compatible stable decompositions.  The main upshot is contained in Proposition \ref{cor:collapsed isometry}, which says that once we collapse the ``unstable'' pieces of these trees, then the (collapsed version of) $T$ admits a convex embedding into the (collapsed version of) $T$, which also preserves the various information encoded in the stable decomposition.

First, we need to set some notation for the rest of the section:

\begin{definition}[Admissible setups]\label{defn:admissible setup}
Given two $\epsilon$-setups $(F;\calY)$ and $(F';\calY')$, we say that $(F;\calY)$ is $\epsilon$-\textbf{admissible} with respect to $(F';\calY')$ if $\calY \cap  \calY' \subset \calN_{\epsilon/2}(\ltree(F)) \cap \calN_{\epsilon/2}(\ltree(F'))$.  Moreover, given $N>0$, we say that it is $(N,\epsilon)$-\textbf{admissible} if $|\calY-\calY'| \leq N$.

\end{definition}

\begin{remark}
    We will deal with two cases over Sections \ref{sec:stable trees} and \ref{sec:stabler trees}.  The first, which is relevant for our (refined) Stable Tree Theorem \ref{thm:stable tree}, is where $F=F'$ and both of $(F; \calY)$ and $(F; \calY')$ are $(N,\epsilon)$-admissible to each other, so that in particular, $|\calY \triangle \calY'|<N$.  In the second case, which is relevant for the Stabler Tree Theorem \ref{thm:stabler tree} in Section \ref{sec:stabler trees}, deals with the case where $F \subset F'$ and $\calY \subset \calY'$, i.e., $(F;\calY)$ is $(0,\epsilon)$-admissible with respect to $(F';\calY')$.
\end{remark}

As the name indicates, a stable decomposition involves decomposing a pair of stable trees for a pair of $(N,\epsilon)$-admissible $\epsilon$-setups into subtrees, with some maps which identify and organize the various pieces.  Each stable decomposition of the stable tree $T$ decomposes $T = T_s \cup T_u$ into two collections of subtrees, the \emph{stable components}, which together form $T_s$ and are all intervals, and the \emph{unstable components}, which are the complementary subtrees.  We call this general kind of decomposition an \emph{edge decomposition}.

The purpose of the next definition is to give shorthand for requiring that the various maps involved respect how the various stable components on each stable tree are oriented towards the endpoints of the ambient trees.

To set some notation, suppose $T$ is a tree with a distinguished finite subset $F \subset T$.  Let $E \subset T$ be an interval in an edge of $T$ with $E \cap F = \emptyset$.  Then every point $f \in F$ has a closest point in $E$, denoted $E(f)$, which is necessarily an endpoint of $E$.  Finally, given a stable tree $T = T_e \cup T_c$ for an $\epsilon$-setup $(\calF;\calY)$, if $y \in \calY \cup \calF$, we let $C_y$ denote the cluster containing $y$, and $\mu(C_y)$ the corresponding component of $T_c$.

The following definition contains the stability properties we want:

\begin{definition}[Stable decomposition] \label{defn:stable decomp}
Let $\calZ$ be $\delta$-hyperbolic and geodesic, and $N,\epsilon>0$.  Let $(F;\calY)$ and $(F';\calY')$ with $F \subseteq F'$ be an $(N,\epsilon)$-admissible pair of $\epsilon$-setups.  Let $T = T_e \cup T_c$ and $T' = T'_e \cup T'_c$ denote their stable trees with their associated maps $\phi:T \to \calZ$ and $\phi':T \to \calZ'$ as provided by Lemma \ref{lem:stable tree basics}.

Given $L_1,L_2>0$, $\calY_0 \subset \calY \cap \calY'$, and two edge decompositions $T_s \subset T_e$ and $T'_s \subset \hull_{T'}(F) \cap T'_e$, we say that $T_s$ is $\calY_0$-\textbf{stably} $(L_1,L_2)$-\textbf{compatible} with $T'_s$ if the following hold:

\begin{enumerate}

\item There is a bijection $\alpha:\pi_0(T_s) \to \pi_0(T'_s)$ between the sets of stable components. \label{item:stable bijection}
\item For each \textbf{stable pair} $(E, E')$ identified by $\alpha$, there exists an isometry $i_{E,E'}:E \to E'$. \label{item:stable pairs}

\item For all but at most $L_1$ pairs of stable components $(E,E')$ identified by $\alpha$, we have $\phi(E) = \phi'(E') \subset \calZ$ and $\phi(x) = \phi'(i_{E,E'}(x))$ for all $x \in E$. \label{item:identical pairs}
\item For the (at most) $L_1$-many remaining stable pairs $(E,\alpha(E))$, we have $$d_{\calZ}(\phi(x), \phi'(i_{E,\alpha(E)}(x)))< L_2$$ for all $x \in E$. \label{item:close pairs}

\item The complements $T_{e} - T_s$ and $(\hull_{T'}(F) \cap T'_{e}) - T'_s$ consist of at most $L_1$ \textbf{unstable components} of diameter at most $L_2$.\label{item:unstable components}
\item There exist \textbf{unstable forests} $T_{\diff} \subset T$ and $T'_{\diff} \subset \hull_{T'}(F)$ each the union of at most $L_1$ components of $T_e,T_c$ and $T'_e,T'_c$, respectively, so that the components of $T - T_{\diff}$ and $\hull_{T'}(F) - T'_{\diff}$ are identical.  Moreover, $T_e - T_s \subset T - T_{\diff}$ and $T'_e - T'_s \subset T' - T'_{\diff}$. \label{item:unstable forests}
\item There exists a bijection $\beta:\pi_0(T-T_s) \to \pi_0(\hull_{T'}(F) - T'_s)$ which satisfies: \label{item:adjacency}
\begin{enumerate}
\item (Identifying clusters) For any $y \in \calY_0 \cup F$, let $D_y, D'_y$ denote the components of $T-T_s$ and $T' - T'_s$ containing $\mu(C_y), \mu(C'_y)$, respectively.  Then $\beta(D_y) = D'_y$. \label{item:cluster identify}
\item (Adjacency-preserving) If stable components $E_1,E_2 \in \pi_0(T_s)$ are adjacent to a component $D \in \pi_0(T - T_s)$ at points $x \in E_1$ and $y \in E_2$, then $\alpha(E_1)$ and $\alpha(E_2)$ are adjacent to $\beta(D)$ at $i_{E_1, \alpha(E_1)}(x)$ and $i_{E_2,\alpha(E_2)}(y)$.\label{item:Adjacency-preserving}
\end{enumerate}
\end{enumerate}

\end{definition}

\begin{remark}\label{rem:stable tree compare}
    The original Stable Tree Theorem \cite[Theorem 3.2]{DMS_bary} provides a number of the above properties of a stable decomposition in the case of two $(N, \epsilon)$-admissible $\ep$-setups $(F; \calY)$ and $(F; \calY')$, where $|\calY \symdiff \calY'|<N$.  In particular, when $F'=F$, we have $\hull_{T'}(F) = T'$, simplifying the notation.  The original theorem provides a bijection $\alpha:\pi_0(T_s) \to \pi_0(T'_s)$ satisfying items \eqref{item:stable bijection}--\eqref{item:unstable components} of Definition \ref{defn:stable decomp}.  Item \eqref{item:unstable forests} will be an easy consequence of Claims 1 and 2 in the proof of \cite[Theorem 3.2]{DMS_bary} (see the beginning of the proof of Theorem \ref{thm:stable tree, one point} below).  Thus for the refined version of that theorem, namely Theorem \ref{thm:stable tree} below, our main task is to show that $\alpha$ can be used to define a bijection $\beta:\pi_0(T- T_s) \to \pi_0(T' - T'_s)$  which satisfies the extra properties of item \eqref{item:adjacency} of Definition \ref{defn:stable decomp}. 
\end{remark}

\begin{remark}
It is crucial that the stable components of the stable decomposition $T_s \subset T$ are all intervals.  Their two-sidedness plays an important role in certain parts of the argument.    
\end{remark}

\begin{remark}[Simplicialization]
    In order to plug into \cite{Dur_infcube}, we will need to arrange that the components of any stable decomposition can be taken to be simplicial trees, i.e. trees where all edge lengths are integers and branch points are at integer points.  We explain how to do this in Subsection \ref{subsec:simplicialization} in Section \ref{sec:stabler trees}.
\end{remark}

\begin{remark}[Motivating Definition \ref{defn:stable decomp}]
    Definition \ref{defn:stable decomp} is intricate but carefully crafted towards out cubulation ends in Section \ref{sec:stabler cubulations}.  See Subsection \ref{subsec:collapsing motivation} at the end of Section \ref{sec:stabler trees} for the main application.
\end{remark}

\subsection{The (refined) Stable Tree Theorem: the (refined) statement}

We are now ready to state a refined version of the original Stable Trees Theorem, namely \cite[Theorem 3.2]{DMS_bary}.  As in the rest of this section, we are working with our fixed base $\ep$-setup $(F;\calY)$ and its associated constants $\ep,\ep',E$ all controlled by $k, \delta$.  Now, however, we are adding some number $N>0$ of cluster points and the output stable trees from our construction.  In particular, we will consider another $\ep$-setup $(F;\calY')$ for $F$ where $|\calY' - \calY|<N$.

The theorem says that such an $(N,\epsilon)$-admissible pair of $\ep$-setups admits a stable decomposition in the sense of Definition \ref{defn:stable decomp}.

\begin{theorem}\label{thm:stable tree}
Let $\calZ$ be $\delta$-hyperbolic and geodesic, and $N>0$.  Suppose that $(F;\calY)$ and $(F;\calY')$ are an $(N,\epsilon)$-admissible pair of $\epsilon$-setups.  Let $T = T_e \cup T_c$ and $T' = T'_e \cup T'_c$ denote their $(\ep, \ep', E)$-stable trees.

There exist $L_1 = L_1(N,k,\delta)>0$, $L_2 = L_2(N,k,\delta)>0$, and two edge decompositions $T_s \subset T_e$ and $T'_s \subset T'_e$ such that $T_s$ is $\calY$-\emph{stably} $(L_1,L_2)$-\emph{compatible} with $T'_s$, with the maps $\alpha$ and $\beta$ as in Definition \ref{defn:stable decomp} being bijections. 

\end{theorem}

\subsection{Outline of the proof of Theorem \ref{thm:stable tree}} \label{subsec:stable tree outline}

The proof of Theorem \ref{thm:stable tree} takes the proof of the original Stable Tree Theorem \cite[Theorem 3.1]{DMS_bary} as its starting point, as discussed in Remark \ref{rem:stable tree compare}.  The proof is in two parts.

First, we will consider the base case where $N=1$, that is where $\calY' - \calY =\{w\}$ is a single cluster point.  The proof of this base case (Theorem \ref{thm:stable tree, one point}) requires a careful analysis of how clusters and the cluster separation graph change when adding this point.  In Lemma \ref{lem:bounded affected}, we show that there is a controlled number of \emph{affected clusters} (Definition \ref{defn:affected cluster}), namely those clusters whose composition or adjacency properties change.  This allows us to define \emph{unstable cores} of the corresponding stable trees for $(F;\calY)$ and $(F;\calY \cup \{w\})$ and prove in Proposition \ref{prop:unstable core} that the stable trees are identical outside of these unstable cores.  With these structural statements in hand, the proof of the base case (Theorem \ref{thm:stable tree, one point}) then proceeds by establishing the various endpoint data preservation properties of Definition \ref{defn:stable decomp} that \cite[Theorem 3.1]{DMS_bary} did not previously provide.

The general case, where $\calY' - \calY =\{w_1, \dots, w_n\}$ is some finite number of points, requires an iterative setup.  In particular, we prove in Proposition \ref{prop:stable iteration} that given a chain (Definition \ref{defn:chain of setups}) of pairwise admissible $\ep$-setups whose stable trees admit compatible stable decompositions, one can iteratively combine the corresponding stable decompositions to produce stable decompositions for the end links of the chain.  The main statement here is Proposition \ref{prop:stable iteration, pair}, which shows that one can do this for a chain of link 3.  With these established, Theorem \ref{thm:stable tree} in our current setting follows from iterated applications of Theorem \ref{thm:stable tree, one point} and a single application of Proposition \ref{prop:stable iteration}.  Notably, the main iteration Proposition \ref{prop:stable iteration} is fairly general and we use it again at the end of Section \ref{sec:stabler trees}.

\subsection{Affected clusters}

For this subsection, fix an $\epsilon$-setup $(F, \calY)$ in $\calZ$ as Subsection \ref{subsec:fixing constants, stable tree} and suppose $w \in \calZ$ is such that $(F, \calY \cup \{w\})$ is also an $\epsilon$-setup.  That is, we want to restrict our attention to adding one cluster point to a given setup.  We will deal with adding multiple points via an iterative argument in Subsection \ref{subsec:stable iteration}.

The goal of this subsection is to prove some structural results about how the stable trees $T$ and $T'$ for the two setups $(F, \calY)$ and $(F, \calY \cup \{w\})$ are related.

The next definition begins our analysis of the structure of the cluster separation graph $\calG = \calG(F, \calY)$ with respect to the new cluster point $w$.

\begin{definition}[Absorbed clusters]\label{defn:absorbed cluster}
    We say that a cluster $A \in \calG^0$ is \textbf{absorbed} if the new cluster point $w$ satisfies $d_{\calZ}(A,w)<E$.  Let $\calA_0$ denote the set of absorbed clusters.
    \end{definition}

    In other words, if $A$ is absorbed, then $A \cup \{w\}$ is contained in some cluster $C_w$ for the setup $(F, \calY \cup \{w\})$, and in fact $C_w$ is the union of $w$ and the absorbed clusters.  

\begin{definition}[Affected clusters]\label{defn:affected cluster}
    We say that a cluster $A \in \calG^0$ is \textbf{affected} if one of the following holds:
\begin{enumerate}
    \item $A$ is absorbed,
    \item $A$ is adjacent in $\calG$ to an absorbed cluster, or
    \item There is a non-absorbed cluster $B \in \calG^0$ such that $A,B$ are adjacent in $\calG$ but not adjacent in $\calG' = \calG(F, \calY \cup \{w\})$.
\end{enumerate}
\end{definition}

We remark that adding $w$ to $\calY$ can only remove edges from $\calG$ when building $\calG'$, hence the statement of (3).

We let $\calA$ denote the set of affected clusters.  We note that $\calA_0 \subset \calA$ and that $\calA_0$ can be empty, while $\calA$ is always nonempty.

\begin{figure}
    \centering
    \includegraphics[width=.80\textwidth]{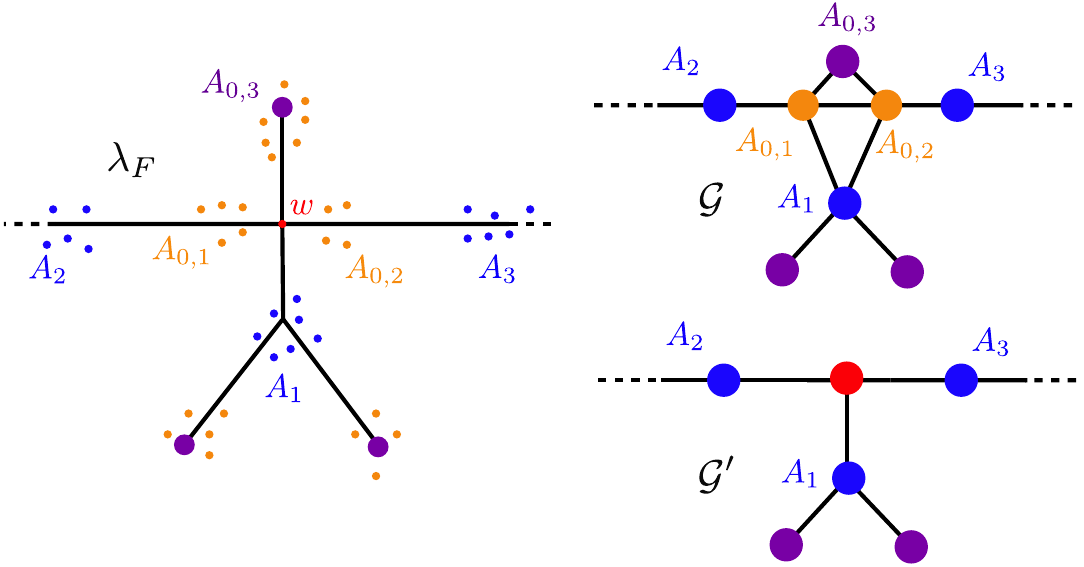}
    \caption{Affected clusters of types (1) and (2): Adding the cluster point $w$ can only affect the composition of boundedly-many clusters, namely the \emph{absorbed} $A_{0,i}$.  It can also affect which points are closest, such as in the type (2) affected cluster $A_1$.  See Figure \ref{fig:affected2} for an example of a type (3) affected cluster.}
    \label{fig:affected_cluster}
\end{figure}

\begin{lemma}\label{lem:bounded affected}
There exists $A_0 = A_0(k,\delta)>0$ and $A_1 = A_1(k,\delta)>0$ so that the following hold:
\begin{enumerate}
    \item $\#\calA < A_0$, and
    \item If $C_1,C_2 \in \calA$, then $d_{\calG}(C_1,C_2) < A_1$.
\end{enumerate}
\end{lemma}

\begin{proof}

To prove item (1), we first show that the number of absorbed clusters, that is $\#\calA_0$, is bounded in terms of $k,\epsilon, \ep', E, \delta$ and hence only in terms of $k, \delta$.  For this, we use \cite[Lemma 3.11]{DMS_bary}, which says that all but boundedly-many (in terms of $\delta$ and $\#F < k$) clusters in $\calG^0$ are \emph{bivalent}, that is have valence $2$ in $\calG$ and does not contain a point of $F$.  By \cite[Lemma 3.10]{DMS_bary}, any such bivalent cluster $C$ has shadow $s(C)$ which is contained in an edge of $\lambda(F)$.

Note that not only is the number of non-bivalent clusters bounded, but so is the size of any set of clusters $\calC$ with the following property: There is no edge of $\lambda(F)$ containing the shadows of two bivalent clusters from $\calC$. This is because the number of edges of $\lambda(F)$ is bounded in terms of $\#F, \delta$ and hence $k, \delta$.

So to bound $\#\calA_0$ it suffices to consider a collection $A_1, \dots, A_n \in \calA_0$ of absorbed clusters which are bivalent and whose shadow is contained in a single edge.  By part (c) of \cite[Lemma 3.6]{DMS_bary}, we must have that $s(A_i) \cap s(A_j) = \emptyset$ for $i \neq j$, while \cite[Claim 1 of Lemma 3.6]{DMS_bary} forces a lower bound on $d_{\lambda(F)}(s(A_i), s(A_j))$ in terms of $E, \epsilon$ for $i \neq j$.  However $d_{\calZ}(A_i,w)< E$ for all $i$, so the same claim forces the clusters to be pairwise close as a function of $E, \epsilon, \delta$, which thus bounds $n$, and bounds $\#\calA_0$ (in $k,\delta$) in turn as required.

    \begin{figure}
    \centering
    \includegraphics[width=.75\textwidth]{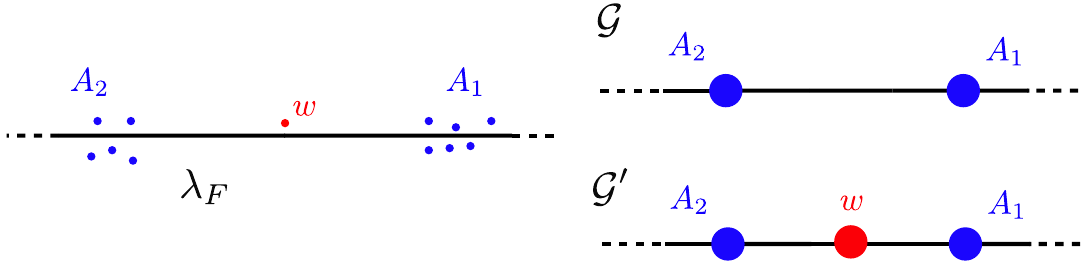}
    \caption{A type (3) affected cluster: The new cluster point $w$ can get between a pair of clusters that are adjacent in $\calG$.  When bounding the number of such clusters in the proof of Lemma \ref{lem:bounded affected}, we can reduce to the case where they are bivalent, where the change between $\calG$ and $\calG'$ is isolated to the picture in the figure.}
    \label{fig:affected2}
\end{figure}

    For the bound on $\#\calA$, observe that the bound on $\# \calA_0$ bounds the number of affected clusters of types (1) and (2), where the latter uses the bound on the valence of $\calG$ (Lemma \ref{lem:stable tree basics}).  Since the number of non-bivalent clusters (of any kind) is bounded, it thus suffices to bound the number of bivalent clusters which are affected of type (3).
    
    Suppose $A, A' \in \calA$ are bivalent, non-absorbed, and adjacent in $\calG$ but not in $\calG'$.  In this (degenerate) case, \cite[Lemma 3.10]{DMS_bary} implies that $s(A)$ and $s(A')$ are intervals inside an edge of $\lambda(F)$, and it follows that $s(w)$ must lie on that same edge in $\lambda(F)$ between $\lambda(A)$ and $\lambda(A')$.  Since $A,A'$ are not absorbed, we must have that $s(w)$ is disjoint from $s(A)$ and $s(A')$, and that the only difference between $\calG$ and $\calG'$ is that an extra vertex labeled by $w$ has been added, and that this vertex forms the connection between $A$ and $A'$ in $\calG'$, replacing the edge in $\calG$.  This completes the proof of item (1) of the lemma.

    For item (2), observe that if $C_1, C_2 \in \calA$, then each is distance at most $2$ from some absorbed clusters $B_1,B_2$, whose distance in $\calZ$ is bounded by $E = E(k,\delta)>0$.  On the other hand, item (6) of Lemma \ref{lem:cluster lemma} now says that the distance in $\calG$ between $B_1,B_2$ is bounded as a function of $k,\delta$ and $E$, and hence in $k,\delta$.  This thus bounds the distance betwen $C_1,C_2$ in terms of $k,\delta$, as required, completing the proof of the lemma.
\end{proof}

\begin{remark}
    The proof of the above lemma says more: Either one is in the degenerate case at the end of the proof, or no two bivalent affected clusters have shadows lying in a single edge of $\lambda(F)$. 
\end{remark}

\subsection{The unstable core} \label{subsec:unstable core}

The goal of this subsection is to prove the following proposition, which provides \textbf{unstable cores} $T_{\calA} \subset T$ and $T'_{\calA} \subset T'$ for the stable trees, outside of which the stable trees are exactly the same.  It is one of the key technical steps in our proof of Theorem \ref{thm:stable tree, one point}.

\begin{remark}[Identical subtrees, etc.]\label{rem:identical}
    In what follows, we will often refer a bit informally to two subtrees $R \subset T$ and $R' \subset T'$ being ``identical''.  The formal meaning in these situations is that there is an isometry $i_{R,R'}:R \to R'$ so that for each $x \in R$, we have $\phi(x) = \phi'(i_{R,R'}(x))$, where $\phi:T \to \calZ$ and $\phi':T' \to \calZ$ are the maps provided by Lemma \ref{lem:stable tree basics}. 
\end{remark}

\begin{proposition}\label{prop:unstable core}
    There exist subtrees $T_{\calA} \subset T$ and $T'_{\calA} \subset T'$ which are the union of boundedly-many edge and cluster components from the decompositions $T = T_e \cup T_c$ and $T' = T'_e \cup T'_c$, with the bound controlled by $k,\delta$, satisfying the following properties:  \begin{itemize}
        \item There is a bijection $\gamma:\pi_0(T-T_{\calA}) \to \pi_0(T' - T'_{\calA})$ where identified components are identical subtrees of $T$ and $T'$.
    \end{itemize}
\end{proposition}

The proof of Proposition \ref{prop:unstable core}, which we complete in Subsection \ref{subsec:proof of unstable core prop}, will require some supporting lemmas and notation.  The rough strategy is to first isolate the parts of $\calG$ which are affected by the addition of $w$ (this uses Lemma \ref{lem:bounded affected}), and then to isolate these affected parts by a buffer of bivalent clusters which insulate the rest of the graph from this smaller part.

\begin{definition}[Raw core]\label{defn:raw core}
    The \textbf{raw core} of $\calG$ with respect to $\calA$ is
    $$R_{\calA} = \calN^{\calG}_{A_1+ 2}(\calA),$$
    namely the $(A_1+2)$-neighborhood in $\calG$ of $\calA$, where $A_1 = A_1(k,\delta)>0$ is the constant from item (2) of Lemma \ref{lem:bounded affected}.
\end{definition}

\begin{lemma}\label{lem:raw basic}
    $R_{\calA}$ is a connected subgraph of $\calG$ consisting of boundedly-many vertices, with the bound controlled by $k,\delta$.  Any cluster in a component of $\calG-R_{\calA}$ is at least distance $2$ in $\calG$ from any absorbed cluster.
\end{lemma}

\begin{proof}
    By Lemma \ref{lem:bounded affected}, $\calN^{\calG}_{A_1}(\calA)$ is connected and consists of boundedly many vertices, with the bound controlled by $k, \delta$.  Hence $R_{\calA}$ is connected and also consists of boundedly-many vertices, since the valence of $\calG$ is bounded by item (3) of Lemma \ref{lem:cluster lemma}.  Finally, any vertex in $\calG - R_{\calA}$ is at least distance $2$ in $\calG$ from any absorbed cluster since all absorbed clusters are contained in $\calA$.  This completes the proof.
\end{proof}

Our next goal is to add a layer of insulation to $R_{\calA}$ to build an unstable core $T_{\calA} \subset T$ in such a way that allows us to build a mirror subtree $T'_{\calA} \subset T'$ satisfying Proposition \ref{prop:unstable core}.  We do this by isolating $R_{\calA}$ in a complementary component in $\calG$ of boundedly-many nearby bivalent clusters.

Recall that a cluster $C$ is bivalent (Definition \ref{defn:bivalent}) if it forms a bivalent vertex of $\calG$ and does not contain a point of $F$.  We let $\calE^0$ denote the set of bivalent clusters.

\begin{lemma}\label{lem:bivalent insulation}
    There exists a collection $\calE_{\calA} \subset \calE^0$ of bivalent clusters so that $R_{\calA}$ is contained in a single component of $\calG - \bigcup_{E \in \calE_{\calA}} E$.  Denoting the closure of this component by $S_{\calA}$,  we have that $\calE_{\calA}$ and $S_{\calA}$ both involve boundedly-many clusters, with the bound controlled by $k,\delta$.
\end{lemma}

\begin{proof}

Let $\calE'_0$ denote the set of bivalent clusters which do not lie in $R_{\calA}$.  Since $R_{\calA}$ is connected  (Lemma \ref{lem:raw basic}), there exists a unique component of $\calG - \calE'_0$ containing $R_{\calA}$.  Call the closure of this component $S_{\calA}$.  We let $\calE_{\calA}$ denote the bivalent clusters in the boundary of $S_{\calA}$; some of the boundary clusters may contain points of $F$ and we exclude those.  Excluding the clusters in $\calE_{\calA}$, the bivalent clusters in $S_{\calA}$ are precisely those in $R_{\calA}$, and so it has boundedly-many vertices (in $k,\delta$) by Lemma \ref{lem:cluster lemma}.  In particular, $S_{\calA}$ and $\calE_{\calA}$ involve boundedly-many clusters (in $k,\delta$), as required.  This completes the proof.
\end{proof}

\begin{figure}
    \centering
    \includegraphics[width=.6\textwidth]{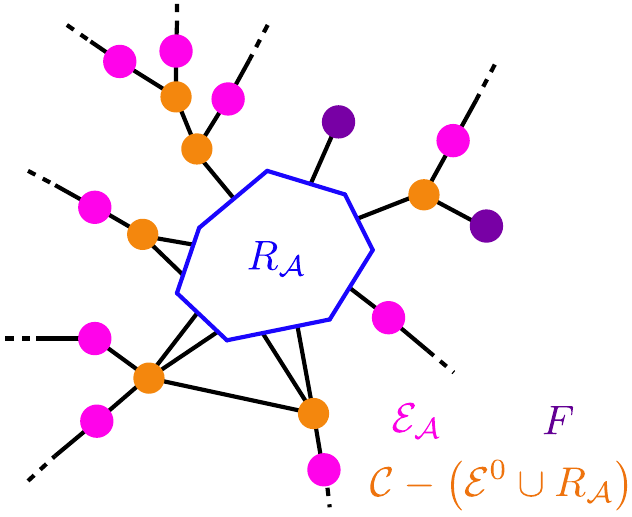}
    \caption{A schematic of how the (minimal) set of bivalent clusters $\calE_{\calA}$ separates the raw core $R_{\calA}$ from the rest of $\calG$.  Each (pink) cluster in $\calE_{\calA}$ cuts $\calG$ into two graphs, one containing $R_{\calA}$ because it is connected.  Some of the ends of $\calG$ outside of $R_{\calA}$ end in a (purple) cluster containing a point of $F$.}
    \label{fig:unstable core}
\end{figure}

Our next step is to analyze the structure of the complementary components of $\calG - S_{\calA}$ and their mirror images in $\calG'$, the cluster separation graph for our other setup $(F; \calY \cup \{w\})$.

\begin{lemma}\label{lem:E' from E}
    Every cluster $C \in \calE_{\calA}$ forms an identical bivalent cluster for the setup $(F;\calY \cup \{w\})$.
\end{lemma}

\begin{proof}
    By definition of $\calA$, the composition of any cluster not in $\calA$ is unchanged with respect to the setup $(F;\calY \cup \{w\})$.  Moreover, non-membership in $\calA$ also guarantees that its adjacency relations in $\calG$ are the same as in $\calG'$.  Since any cluster in $\calE_{\calA}$ is at least distance two from any absorbed cluster by Lemma \ref{lem:raw basic}, this proves the lemma.
\end{proof}

We let $E'_{\calA}$ denote the set of clusters in the setup $(F;\calY \cup \{w\})$.  We let $\calU$ denote the closure of components of $\calG - \calE_{\calA}$, and similarly let $\calU'$ denote the closure of components of $\calG' - E'_{\calA}$.

\begin{lemma}\label{lem:buffer identical}
There is a bijection $\zeta:\calU \to \calU'$.  Setting $S'_{\calA} = \zeta(S_{\calA})$, then the following hold:
    \begin{enumerate}
      \item If $D \in \calU - \{S'_{\calA}\}$, then $\zeta(D)$ and $D$ are identical, in the sense that there is a graph isomorphism $D \to D'$ so that the clusters identified by this isomorphism consist of exactly the same cluster points in $\calY$.
      \item The number of vertices in $S'_{\calA}$ is bounded in terms of $k,\delta$.
    \end{enumerate}
\end{lemma}

\begin{proof}
The argument for item (1) is essentially the same as in Lemma \ref{lem:E' from E}, since the only clusters whose membership or adjacency change are those in $\calA$, which the closures of components in $\calU$ avoid.  Each such component has a bivalent cluster $C \in \calE_{\calA}$ in its boundary, and so there is an identical such component in $\calG - E'_{\calA}$ with $C$ as its boundary.

Item (2) follows from the simple observation that the number of clusters in for the setup $(F;\calY)$ is at most one less than the number of clusters for $(F;\calY \cup \{w\})$, because either $w$ gets absorbed into an existing cluster or it forms it own cluster.  Hence since the components in $\calU - \{S_{\calA}\}$ and $\calU' - \{S'_{\calA}\}$ are identical and $S_{\calA}$ has boundedly-many vertices, so must $S'_{\calA}$.  This completes the proof.
\end{proof}

The last step before defining the unstable cores involves an observation about the construction of the stable trees $T$ for $(F;\calY)$ and $T'$ for $(F;\calY \cup \{w\})$.

\medskip

We recall the notation involved in defining $T = T_e \cup T_c$ (Subsection \ref{subsec:stable trees defined}).  Let $\calE^0$ denote the set of bivalent clusters (Definition \ref{defn:bivalent}).  Let $\calV$ be the set of closures of connected components of $\calG - \calE^0$.  For each $V \in \calV$, let $\calV^0$ denote its vertex set.  Then the edge components of $T$ are defined as
    $$T_e = \bigsqcup_{V \in \calV} \lambda'(V^0)$$
    while the cluster components are defined by
    $$T_c = \bigsqcup_{C \in \calG^0} \mu(C)$$
    where $\mu(C) = \lambda(r(C))$ and $r(C) = C \cap (T_e \cup F)$.  In the above, $\lambda$ and $\lambda'$ are the network functions defined in Subsection \ref{subsec:basic tree setup}.

The edge and cluster components of $T' = T'_e \cup T'_c$ are defined analogously, with associated notation $(\calE')^0, \calV'$, etc.

\begin{lemma}\label{lem:buffer has same components}
    The components of $\calU$ and $\calU'$ satisfy the following:
    \begin{enumerate}
        \item Every component of $\calU$ (resp. $\calU'$) is a union of components of $\calV$ (resp. $\calV'$).
        \item $S_{\calA}$ and $S'_{\calA}$ are unions of boundedly-many components of $\calV$ and $\calV'$ respectively, with the bound controlled by $k,\delta$.
    \end{enumerate}

\end{lemma}

\begin{proof}
    Item (1) follows immediately from the fact that $\calU$ and $\calU'$ are defined as the closures of the complementary components of the bivalent clusters in $\calE_{\calA} \subset \calE^0$.

    On the other hand, item (2) follows immediately from item (1) and the bound on the number of clusters contained in $S_{\calA}$ and $S'_{\calA}$ from Lemmas \ref{lem:bivalent insulation} and \ref{lem:buffer identical}.  This completes the proof of the lemma.
\end{proof}

\medskip

Let $\calV_{\calA} \subset \calV$ denote the components in $\calV$ which are contained in $S_{\calA}$, and $\calC_{\calA}$ the set of clusters contained in $S_{\calA}$.  Define $\calV'_{\calA}$ and $\calC'_{\calA}$ analogously.

\begin{definition}[Unstable cores]\label{defn:unstable cores}
    The \textbf{unstable core} $T_{\calA} \subset T$  of $T$ is the union of the cluster and edge components involved in $S_{\calA}$, namely:
    $$T_{\calA} = \left(\bigcup_{V \in \calV_{\calA}} \lambda'(V^0)\right) \cup \left(\bigcup_{C \in \calC_{\calA}} \mu(C)\right).$$
    Similarly, \textbf{unstable core} $T'_{\calA} \subset T'$  of $T'$ is
     $$T'_{\calA} = \left(\bigcup_{V \in \calV'_{\calA}} \lambda'(V^0)\right) \cup \left(\bigcup_{C \in \calC'_{\calA}} \mu(C)\right).$$
\end{definition}

    \subsection{Proof of Proposition \ref{prop:unstable core}}\label{subsec:proof of unstable core prop}

First, observe that $T_{\calA} \subset T$ is path connected (and hence a subtree) because any point of $T_{\calA}$ belongs to some edge or cluster component defined by the connected subgraph $S_{\calA} \subset \calG$.  Hence given two points in $T_{\calA}$, one can pass to adjacent clusters subtrees (or do nothing, if the points are in cluster subtrees) which are vertices of $S_{\calA}$.  Any path in $S_{\calA}$ (as a path in $\calG$) between these two adjacent clusters determines a path in $T_{\calA}$ between the points, by following along corresponding chain of minimal networks.  The same argument shows that $T'_{\calA}$ is also a subtree.

Both $T_{\calA}$ and $T'_{\calA}$ are unions of edge and cluster components, and moreover a bounded number of these by their definition and item (2) of Lemma \ref{lem:buffer has same components}.  This proves the first part of the statement.

The second part of the statement follows from combining Lemmas \ref{lem:buffer identical} and Lemma \ref{lem:buffer has same components} with the definitions of $T,T'$.  In particular, the former provides  an isomorphism $\zeta:\calU \to \calU'$ which, by item (1) of that lemma, also provides graph isomorphisms for components other than $S'_{\calA} = \zeta(S_{\calA})$, with identified vertices corresponding to identical clusters.  The latter lemma then says that all edge and cluster components of $T,T'$ not contained in $T_{\calA}, T'_{\calA}$, respectively, are contained in these complementary components.  Since the combinatorial data of these components of $T, T'$ are defined using identical cluster separation graph and cluster membership data, they define identical collections of minimal networks by our fixed choices of $\lambda,\lambda'$ (Subsection \ref{subsec:basic tree setup}.  This completes the proof of the proposition.

\subsection{The (refined) Stable Tree Theorem: the (refined) statement} \label{subsec:refined stable tree}

We are now ready to prove a refined version of the original Stable Trees Theorem, namely \cite[Theorem 3.2]{DMS_bary}.  We consider $(1,\epsilon)$-admissible setups $(F, \calY)$ and $(F, \calY \cup \{w\})$, that is, we add a cluster point to an $\epsilon$-setup $(F, \calY)$, where we are using our fixed setup as in Subsection \ref{subsec:fixing constants, stable tree}.

The theorem says that such an $(1,\epsilon)$-admissible pair admits a stable decomposition in the sense of Definition \ref{defn:stable decomp}.  In Proposition \ref{prop:stable iteration} below, we will see how to iterate this procedure to allow for adding a bounded number of cluster points.

\begin{theorem}\label{thm:stable tree, one point}
Let $\calZ$ be $\delta$-hyperbolic and geodesic. Let $k>0$ and $\ep=\ep(k, \delta)>0$, $\ep'=\ep'(k,\ep)>0$, and $E = E(k, \ep')>0$ as in Subsection \ref{subsec:fixing constants, stable tree}.  Suppose $(F;\calY)$ and $(F;\calY \cup \{w\})$ are an $(1,\epsilon)$-admissible pair of $\epsilon$-setups.  Let $T = T_e \cup T_c$ and $T' = T'_e \cup T'_c$ denote their $(\ep, \ep', E)$-stable trees.

There exist $L_1 = L_1(k,\delta)>0$, $L_2 = L_2(k,\delta)>0$, and two edge decompositions $T_s \subset T_e$ and $T'_s \subset T'_e$ such that $T_s$ is $\calY$-\emph{stably} $(L_1,L_2)$-\emph{compatible} with $T'_s$, with the maps $\alpha$ and $\beta$ as in Definition \ref{defn:stable decomp} being bijections. 

\end{theorem}

\begin{proof}
The statement of \cite[Theorem 3.2]{DMS_bary} provides forests $T_s \subset T_e \subset T$ and $T'_s \subset T'_e \subset S$ consisting of intervals and a bijection $\alpha:\pi_0(T_s) \to \pi_0(T'_s)$ which satisfy items \eqref{item:stable bijection}--\eqref{item:unstable components} of Definition \ref{defn:stable decomp}.  Item \eqref{item:unstable forests} follows quickly from Claims 1 and 2 of \cite[Theorem 3.2]{DMS_bary}, which show that the symmetric difference of the edge and vertex sets of $\calG, \calG'$ are bounded in terms of $k, \delta$.  Combining this with the bound on branching and valence of $\calG,\calG'$ (Lemma \ref{lem:cluster lemma}) implies that there are boundedly-many (in $k,\delta$) different components of $T_e, T'_e$ and $T_c,T'_c$, meaning that they can be grouped into boundedly-many (in $k,\delta$) subtrees outside of which $T,T'$ are identical.

Thus the main task of the proof is showing that the bijection $\alpha:\pi_0(T_s) \to \pi_0(T'_s)$ can be used to define a bijection $\beta:\pi_0(T- T_s) \to \pi_0(T' - T'_s)$  which satisfies the extra properties of item \eqref{item:adjacency} of Definition \ref{defn:stable decomp}. 

By items \eqref{item:identical pairs} and \eqref{item:close pairs}, the components of $T_s$ and $T'_s$ break into two collections, namely pairs of intervals which are exactly the same in $\calZ$, and at most $L_1$-many pairs of intervals which have the same length (by item \eqref{item:stable pairs}) and are $L_2$-close in $\calZ$, for $L_i = L_i(k,\delta)>0$.  Let us call the first kind \emph{identical pairs} and the second kind \emph{approximate pairs}.  Importantly, since there are only boundedly-many (in $k,\delta$) intervals in approximate pairs, we may assume that each such interval is as long as we would like, say $M = M(k,\delta)$, by adding intervals shorter than this to the collection of unstable components.  Note that this maintains the original proximity bound $L_2$ by increasing the size of $L_1$ (while keeping it bounded in $k,\delta$).  In particular, by assuming that $M > 4L_2$, we may arrange that if $(C, \alpha(C))$ are an approximate pair with length at least $M$, then $i_{C, \alpha(C)}:C \to \alpha(C)$ sends each endpoint of $C$ to the endpoint of $\alpha(C)$ within $L_2$ of it (item \eqref{item:close pairs} of Definition \ref{defn:stable decomp}).

With this arranged, we next associate to each component $C \in \pi_0(T_s)$ a collection of labels which we think of as \emph{gluing data} as follows:  For any other stable component $D \in \pi_0(T_s)$, let $g^D_C$ be the endpoint of $C$ adjacent to the component of $T - C$ containing $D$.

Observe that the bijection $\alpha:\pi_0(T_s) \to \pi_0(T'_s)$ naturally associates the endpoints of $C$ to the endpoints of $\alpha(C)$, via the isometry $i_{C, \alpha(C)}:C \to \alpha(C)$.  For each $C$, let $\alpha_0$ denote this induced map on endpoints.

The following key claim says that the bijection $\alpha$ preserves this gluing data.

\begin{claim} \label{claim:gluing data}
 For any $C,D \in \pi_0(T_s)$, we have $\alpha_0(g^D_C) = g^{\alpha(D)}_{\alpha(C)}$.   
\end{claim}

Note that the gluing data is combinatorially defined in terms of the structure of the given tree.  To prove Claim \ref{claim:gluing data}, we need to connect this combinatorial data to the metric data of the hyperbolic space $\calZ$.  The idea is that we want the gluing data $g^D_C$ for a pair of stable components $C,D \in \pi_0(T_s)$ to be coarsely realized by the closest point projection of $D$ to $C$, since $\alpha(D)$ and $\alpha(C)$ are close to $D,C$, respectively.  However, this will only work nicely when both $D$ and $C$ are sufficiently long, or are separated in $T$ by some subtree which is mirrored in $T'$.

The issue here is that the collections $T_s, T'_s$ can contain short identical components, whose separation properties are hard to pin down.  This is where our work in the previous subsection comes in, as it will allow us to refine the collections $T_s, T'_s$ so that all stable components are either long or are separated from each other by their respective unstable cores $T_{\calA}$ and $T'_{\calA}$, that is, the subtrees $T_{\calA}\subset T$ and $T'_{\calA} \subset T'$ given by Proposition \ref{prop:unstable core}. Recall that these, in particular, are unions of boundedly-many components of the decompositions $T= T_e \cup T_c$ and $T' = T'_e \cup T'_c$, with the bound depending on $k,\delta$.  Note that the subtree $T_{\calA}$ separates every pair of components $T - T_{\calA}$, and similarly for $T'_{\calA}$ and the components of $T'-T'_{\calA}$.  Finally, there is a bijection $\gamma: \pi_0(T - T_{\calA}) \to \pi_0(T' - T'_{\calA})$, where identified components are identical subtrees of $T \cap T'$.

As discussed above, every approximate pair of stable components can be made as long as necessary.  On the other hand, every pair of identical components $D \in \pi_0(T_s)$ and $D' \in \pi_0(T'_s)$ coincides with a component of $T_e \cap T'_e$, as identical components are contained in their respective edge subtrees $T_e,T'_e$.  By Proposition \ref{prop:unstable core}, any such pair has the property that either $D = D' \subset T_{\calA} \cap T'_{\calA}$, or $D = D' \subset S \in \pi_0(T - T_{\calA}) = \pi_0(T' - T'_{\calA})$.  Note that there are only boundedly-many pairs of the former type by Proposition \ref{prop:unstable core}, so we can remove the ones not satisfying a lower diameter bound (in terms of $k,\delta$) from $T_s, T'_s$ without creating an issue with respect to the other items in Definition \ref{defn:stable decomp}.

Abusing notation, we refer to $T_s,T'_s$ as these slightly refined collections, and now observe that all pairs of stable components $D \in \pi_0(T_s)$ and $D' \in \pi_0(T'_s)$ satisfy the following: 
\begin{itemize}
    \item Either the lengths of $D,D'$ are bounded below by some $M=M(k,\delta)>0$ to be determined below, or
    \item $D=D'$ is contained in some component of $T - T_{\calA} = T' - T'_{\calA}$.
\end{itemize}

We are finally ready to prove our key claim about the gluing data for these refined stable decompositions $T_s \subset T$, $T'_s \subset T'$:

\begin{proof}[Proof of Claim \ref{claim:gluing data}]
First, observe that every pair of approximate stable components is contained in $T_{\calA}$ and $T'_{\calA}$, and if $D$ is an identical stable component contained in $T_{\calA}$, then $\alpha(D)$ is contained in $T'_{\calA}$ by Proposition \ref{prop:unstable core}.  Moreover, we can arrange that such a stable component $D$ is as long as we need, say at least $M=M(k, \delta)$-long.  To confirm the gluing data property in the claim, there are three cases.

First, suppose that $D_1,D_2$ are stable components contained in $T_{\calA}$, so that $\alpha(D_1), \alpha(D_2)$ are also contained in $T'_{\calA}$.  By choosing the lower bound $M = M(k,\delta)>0$ for the length of such stable components to be sufficiently large and using the fact that $T,T'$ are uniformly (in $k,\delta$) quasi-isometrically embedded (and hence uniformly quasiconvex) in $\calZ$, we can arrange for the endpoint associated to $g^{D_1}_{D_2}$ to be coarsely (in $k,\delta$) the closest point projection of $D_1$ to $D_2$ in $\calZ$.  A similar statement holds for $\alpha(D_1)$ and $\alpha(D_2)$.  By again choosing the length parameter $M$, we must have that $\alpha_0(g^{D_1}_{D_2})$ is coarsely (in $k,\delta$) the closest point projection of $\alpha(D_1)$ to $\alpha(D_2)$, and thus $\alpha(D_1)$ must be in the corresponding component of $T' - \alpha(D_2)$.  Thus the claim holds in this case.

Now suppose that $D_1,D_2 \in \pi_0(T_s)$ are stable components in identical pairs outside of $T_{\calA}$.  If they are both in the same component of $T - T_{\calA}$, then $\alpha(D_1), \alpha(D_2)$ are in the corresponding component of $T'  - T'_{\calA}$, which is identical, so the gluing data is preserved.  If $D_1,D_2$ are in different components of $T-T_{\calA}$, then $g^{D_1}_{D_2}$ is the endpoint of $D_2$ corresponding to the component of $T - D_2$ containing both $T_{\calA}$ and $D_1$.  Note that this uses that $T_{\calA}$ is connected by Proposition \ref{prop:unstable core}.  On the other hand, using that $T'_{\calA}$ is connected by the same proposition, we must have that $g^{\alpha(D_1)}_{\alpha(D_2)}$ coincides with the corresponding end of $\alpha(D_2)$, namely the end of $\alpha(D_2)$ which is adjacent to the component of $T' - \alpha(D_2)$ containing $T'_{\calA}$ and $\alpha(D_1)$.  On the other hand, this endpoint is exactly $\alpha_0\left(g^{D_1}_{D_2}\right)$ by definition of $\alpha_0$.  Thus the claim holds in this case.

Finally, for the mixed case, suppose $D_1$ is an identical stable component contained in $S \in \pi_0(T-T_{\calA})$ and $D_2$ is a stable component in $T_{\calA}$.  Since the components of $T - T_{\calA}$ are identical to the components of $T' - T'_{\calA}$ by Proposition \ref{prop:unstable core}, we can arrange that any such component $S$ have a diameter lower bound (controlled by $k,\delta$) so that $g^{D_1}_{D_2}$ coarsely coincides with the closest point projection of $S$ to $D_2$.  We can arrange this by adding such pairs of components to $T_{\calA}$ and $T'_{\calA}$, while still preserving the key properties of Proposition \ref{prop:unstable core}, namely that they are connected, have identical complements in $T,T'$ respectively, and are the unions of boundedly-many components of $T_e,T_c$.  This last item uses the fact that every component of $T_e$ has a lower-diameter bound (in $k,\delta$) and every pair of components of $T_c$ are separated by at least one component of $T_e$.

Now having already arranged for $D_2$ to have a large diameter (controlled by $k,\delta$), this endpoint $g^{D_1}_{D_2}$ is the endpoint of $D_2$ closest to $S$.  Our lower bounds on the lengths of $S$ and $D_2$ provide that $\alpha_0(g^{D_1}_{D_2}) = g^{\alpha(D_1)}_{\alpha(D_2)}$.  A similar argument show that $\alpha_0(g^{D_2}_{D_1}) = g^{\alpha(D_2)}_{\alpha(D_1)}$, completing the proof of the claim.
\end{proof}

With Claim \ref{claim:gluing data} in hand, we can define our desired bijection $\beta:\pi_0(T - T_s) \to \pi_0(T'-T'_s)$ and confirm that the properties in item \eqref{item:adjacency} in Definition \ref{defn:stable decomp} hold.
\medskip

\textbf{\underline{The bijection $\beta:\pi_0(T- T_s) \to \pi_0(T' - T'_s)$}}: Let $C \in \pi_0(T - T_s)$.  Let $E_1, \dots, E_n$ denote the stable components adjacent to $C$, where $E_i$ is adjacent to $C$ at its endpoint $e_i$.  Observe that by definition, we have $e_i = g^{E_j}_{E_i}$ for all $j \neq i$.

We claim that there is a unique component $C' \in \pi_0(T' - T'_s)$ such that $\alpha(E_1), \dots, \alpha(E_n)$ are the stable components adjacent to $C'$, with $\alpha(E_i)$ adjacent to $C'$ at $\alpha_0(e_i)$.

For this, suppose first that $\alpha(E_i), \alpha(E_j)$ are adjacent to some unstable component $C'$.  Then $\alpha(E_i)$ is adjacent to $C'$ at $g^{\alpha(E_i)}_{\alpha(E_j)} = \alpha_0(g^{E_i}_{E_j}) = \alpha_0(e_i)$ by Claim \ref{claim:gluing data}.  Thus if the images of the $E_i$ are adjacent to some unstable component, then they are adjacent at the correct endpoints required, i.e. the corresponding endpoints provided by $\alpha_0$.

Now suppose for a contradiction that $\alpha(E_i)$ and $\alpha(E_j)$ are not adjacent to an unstable component.  This implies that they are separated in $T'$ by some other stable component $E'$.  Thus $g^{\alpha(E_i)}_{E'} \neq g^{\alpha(E_j)}_{E'}$.  On the other hand, $\alpha^{-1}(E')$ does not separate $E_i$ from $E_j$, and since they are all intervals in a tree, this means that $g^{E_i}_{\alpha^{-1}(E')} = g^{E_j}_{\alpha^{-1}(E')}$, and this contradicts Claim \ref{claim:gluing data}.  Hence all of the $\alpha(E_i)$ are adjacent to a common unstable component $C'$ at the correct endpoints.

Finally, a similar argument shows that no other stable component $E''$ can be adjacent to $C'$ in $T'$, because then there would be some $E_i$ which separates $\alpha^{-1}(E'')$ from $E_j$ for all $j \neq i$.  This is because the union of $C'$ with the $E_i$ is a connected subtree of $T$, so each complementary component of that union---one of which contains the supposed $\alpha^{-1}(E'')$---is separated by some $E_i$ from all of the other $E_j$.  Hence the existence of such an $E''$ would result in a similar contradiction via Claim \ref{claim:gluing data}, namely that $g^{E''}_{\alpha(E_i)} = \alpha_0(e_i),$ while $g^{\alpha^{-1}(E'')}_{E_i}$ is the opposite endpoint of $E_i$.

Thus we can define our bijection $\beta:\pi_0(T - T_s) \to \pi_0(T' - T'_s)$ by $\beta(C) = C'$ as defined above.
 
\smallskip

\textbf{\underline{Verifying item \eqref{item:adjacency} of Definition \ref{defn:stable decomp}}}:  Observe that $\beta$ satisfies item \eqref{item:Adjacency-preserving} by construction.  For item \eqref{item:cluster identify}, let $y \in \calY \cup F$ and let $C_y, C'_y$ denote the clusters containing $y$ for the setups $(F, \calY), (F, \calY \cup \{w\})$, respectively.  Let $D_y, D'_y$ be the components of $T-T_s, T'-T'_s$ containing $\mu(C_y), \mu(C'_y)$, respectively, and let $\beta(D_y) = D'$.  We want to show that $D'_y = D'$.

Now either $C_y$ is contained in the unstable core $T_{\calA}$ or not.  If it is, then $C'_y$ is contained in the unstable core $T'_{\calA}$ and it follows that both $D'_y$ and $D'$ intersect $T'_{\calA}$.  If $D'_y \neq D'$, then there must be some long stable component $E' \in \pi_0(T'_s) \cap T'_{\calA}$ separating them, as all stable components in $T'_{\calA}$ can be made as long as desired.  But now this says that $y$ is on opposite sides of $E'$ and $\alpha^{-1}(E')$, which is impossible.

On the other hand, if $C_y$ is not contained in the unstable core, then $\mu(C_y)$ is  contained in an identical component $S \subset T - T_{\calA} = T' - T'_{\calA}$, which says that $C_y = C'_y$.  While it is possible that $D_y, D'_y$ are not entirely contained in $S$, both overlap it.  If all of $S$ is an unstable component, i.e. a component of $T-T_s$ and $T'-T'_s$, then $S = D'_y = D'$.  Otherwise, $D_y, D'_y$ are both adjacent to some collection of identical stable components on the same side of $S$, and hence $\beta(D_y) = D_y$ by definition of the bijection.  This completes the proof of the theorem.
\end{proof}

\subsection{Iteratively refining stable decompositions}\label{subsec:stable iteration}

In this subsection, we prove our iteration statement, Proposition \ref{prop:stable iteration}.  It allows us to iteratively define stable decompositions between a pair admissible setups when their is a chain of admissible setups that interpolate between them.  The key definition of the refined decomposition follows closely the analogous discussion in \cite[Subsection 8.10]{Dur_stableinterval}.  The main work is the base case of combining two pairs of stable decompositions across a common setup.

\begin{definition}[Links, chains]\label{defn:chain of setups}
    Given three admissible $\epsilon$-setups $(F_i;\calY_i)$ for $i=1,2,3$ and a subset $\calY_0 \subset \calY_1 \cap \calY_2 \cap \calY_3$, we say that $(F_2; \calY_2)$ is a $(\calY_0, L_1, L_2)$-\textbf{link} between $(F_1;\calY_1)$ to $(F_3;\calY_3)$ if the stable tree of $(F_2; \calY_2)$ admits $\calY_0$-stable $(L_1,L_2)$-compatible decompositions with the stable trees for $(F_1;\calY_1)$ to $(F_3;\calY_3)$.
More generally, we say that an $n$-tuple $(F_1; \calY_1), \dots, (F_n; \calY_n)$ is a $(\calY_0, L_1, L_2)$-\textbf{chain} if each $(F_i; \calY_i)$ is a $(\calY_0, L_1, L_2)$-link between $(F_{i-1}; \calY_{i-1})$ and $(F_{i+1}; \calY_{i+1})$ for each $2\leq i \leq n-1$.
\end{definition}

The following is the main result of this subsection.  Roughly, it says that stable compatibility is transitive:

\begin{proposition}\label{prop:stable iteration}
For every $n\geq 2$, there exists $M_n = M_n(L_1,L_2,n)>0$ so that if $(F;\calY_1), \dots, (F;\calY_n)$ is a $(\calY_0,L_1,L_2)$-chain of pairwise admissible $\epsilon$-setups, then $(F_1;\calY_1)$ and $(F_n;\calY_n)$ admit $\calY_0$-stable $M_n$-compatible stable decompositions. 
\end{proposition}

The proof of Proposition \ref{prop:stable iteration} is a straight-forward iterative application of the corresponding statement for the base case where $n=3$.  We deal with this case next in Proposition \ref{prop:stable iteration, pair}, the proof of which completes the proof of Proposition \ref{prop:stable iteration}.

\begin{proposition}\label{prop:stable iteration, pair}
Suppose that $(F_i;\calY_i)$ for $i=1,2,3$ is a $(\calY_0, L_1, L_2)$-chain of admissible $\epsilon$-setups, with stable trees $T_1,T_2,T_3$.  Then exist $\calY_0$-stable $(4L_1^2, 4L_2^2)$-compatible decompositions $T^{2,3}_{s,1} \subset T_{e,1}$ and $T^{2,1}_{s,3} \subset T_{e,3}$.
\end{proposition}

\begin{proof}
The first step involves refining the given stable decompositions on $T_1,T_2$ and $T_2,T_3$ using the stable decompositions from $T_2,T_3$ and $T_1,T_2$, respectively.

Following our notation from before (Definition \ref{defn:stable tree}), let $\phi_i:T_i \to \calZ$ denote the maps into $\calZ$ of the respective stable trees.  Moreover, assume that we have
\begin{itemize}
\item a $\calY_0$-stable decomposition $T^1_{s,2} \subset T_{e,2}$ which is $(L_1,L_2)$-compatible with $T^2_{s,1} \subset T_{e,1} $ for the setups $(a_1,b_1;\calY_1)$ and $(a_2,b_2;\calY_2)$, and
\item a $\calY_0$-stable decomposition $T^3_{s,2} \subset T_{e,2}$ which is $(L_1,L_2)$-compatible with $T^2_{s,3} \subset T_{e,3}$ for the setups $(a_2,b_2;\calY_2)$ and $(a_3,b_3;\calY_3)$.
\end{itemize}

Note that the stable decompositions $T^1_{s,2}$ and $T^3_{s,2}$ both live on $T_2$.  Set $T^{1,3}_{s,2} = T^1_{s,2} \cap T^3_{s,2}$.  Observe that $T_{e,2} - T^{1,3}_{s,2} = (T_{e,2} - T^1_{s,2}) \cup (T_{e,2} - T^3_{s,2})$ has at most $2L_1$-many components each of diameter at most $4L_2^2$; note that the larger diameter bound accounts for unstable components to combine, but there are at most $2L_1$-many of them.

\begin{figure}
    \centering
    \includegraphics[width=1\textwidth]{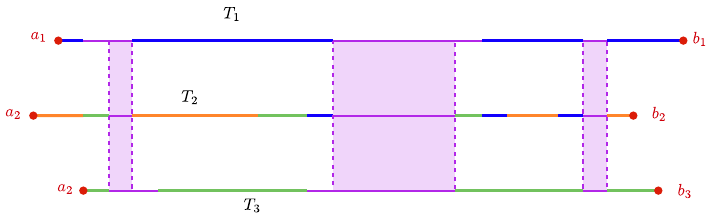}
    \caption{Proposition \ref{prop:stable iteration, pair}, iterating stable decompositions: The idea for the iteration argument is to intersect the stable decompositions $T^1_{s,2} \cap T^3_{s,2} \subset T_2$ for $T_1,T_3$, respectively, and then use the isometries which identify the stable pairs for $T_1, T_2$ to give a refined stable decomposition for $(T_1,T_2)$, which can be combined with a refinement for $(T_2,T_3)$ via simple composition.  In the schematic (which is \emph{not} happening in $\calZ$), we have an example where all three trees are intervals.  These intervals $T_1,T_2,T_3$ are lined up to indicate how to refine the pairwise stable decompositions for $(T_1,T_2)$ and $(T_2,T_3)$ for one for $(T_1,T_3)$.  The orange segments correspond to cluster components of the intervals, while the blue and green segments correspond to the unstable components for the pairs $(T_1,T_2)$ and $(T_2,T_3)$, respectively.  The complement of the orange, blue, and green segments on $T_2$ form the refined stable decomposition on $T_2$, and the purple squares indicate how these segments induce refined stable decompositions on $T_1$ and $T_3$.}
    \label{fig:iterated stable decomp}
\end{figure}

We first produce refined stable decompositions of $T^{2,3}_{s,1} \subset T^2_{s,1}$ and $T^{1,3}_{s,2} \subset T^1_{s,2}$, and the same argument produces the refined stable decompositions $T^{1,3}_{s,2}  \subset T^3_{s,2}$ and $T^{2,1}_{s,3} \subset T^2_{s,3}$.

We can induce a decomposition $T^{2,3}_{s,1}$ on $T_1$ by taking the collection of all 
$$i^{-1}_{D,D'}\left(D' \cap T^{1,3}_{s,2}\right)$$
over all stable pairs $(D,D')$ of $T_1,T_2$ as identified by the bijection $\alpha_{1,2}:\pi_0(T_{s,1}) \to \pi_0(T_{s,2})$.  We first observe that items \eqref{item:stable bijection}--\eqref{item:unstable forests} of Definition \ref{defn:stable decomp} follow fairly quickly, and then prove the properties in item \eqref{item:adjacency}.

First, observe that any component $V \subset T^{2,3}_{s,1}$ is contained in some component $V \subset D \subset T_{s,1} \subset T_{e,1}$, and the map $i_{D,D'}|_V:V \to i_{D,D'}(V)$ is an isometry, where $D' = \alpha_{1,2}(D)$, and vice versa.  Thus the components of $T^{2,3}_{s,1}$ are in bijective correspondence with the components of $T^{1,3}_{s,2}$, with identified components being isometric via these restrictions.  Moreover, these restrictions induce a bijection $\alpha^3_{1,2}:\pi_0(T^{2,3}_{s,1}) \to \pi_0(T^{1,3}_{s,2})$ (as in item \eqref{item:stable bijection} of Definition \ref{defn:stable decomp}) which coincides with the bijective correspondence induced from $\alpha_{1,2}$ and the various isometries $i_{D,D'}$.  Next observe that the images $\phi_1(V)$ and $\phi_2(i_{D,D'}(V))$ are either exactly the same or are $L_2$-Hausdorff close in $\calZ$ as items \eqref{item:identical pairs} and \eqref{item:close pairs} of Definition \ref{defn:stable decomp}.  Item \eqref{item:unstable components} of Definition \ref{defn:stable decomp} follows from the above bounds of $2L_1$ and $4L_2$ on the number and diameter, respectively, of complementary components of $T_{1,e} - T^{2,3}_{s,1}$ and $T_{2,e} - T^{1,3}_{s,2}$.

Finally, for item \eqref{item:unstable forests} of Definition \ref{defn:stable decomp}, there are unstable forests $T^2_{1,\diff} \subset T_1$ and $T^1_{2,\diff} \subset T_2$ whose complements $T_1 - T^2_{1,\diff}$ and $T_2 - T^1_{2,\diff}$ consist of identical components.  Similarly, there are unstable forests $T^3_{2,\diff} \subset T_2$ and $T^2_{3,\diff} \subset T_3$ whose complements $T_2 - T^3_{2,\diff}$ and $T_3 - T^2_{3,\diff}$ are identical.  In the same way we have defined the components of the induced stable decompositions $T^{2,3}_{1,s}$ and $T^{1,3}_{s,2}$, we can take the intersection $T^{1,3}_{s,\diff} := T^1_{2,\diff} \cap T^3_{2,\diff}$---which is a union of components of $T_{2,e}$ and $T_{2,c}$ by construction---and push it to $T_1$ to obtain a subforest $T^{2,3}_{1,\diff} \subset T^2_{1,\diff}$ whose components are identical to the components of $T^{1,3}_{2,\diff}$.  Moreover, by construction, the components of $T^{2,3}_{1,\diff}$ and $T^{1,3}_{2, \diff}$ are unions of components of $T_{1,e}, T_{1,c}$ and $T_{2,e},T_{2,c}$, respectively, and also their complements $T_1 - T^{2,3}_{1,\diff}$ and $T_2 - T^{1,3}_{2, \diff}$ both consist of boundedly-many (in $k, \delta$) components of $T_{1,e}, T_{1,c}$ and $T_{2,e},T_{2,c}$, respectively.  Thus item \eqref{item:unstable forests} holds.

Thus the decompositions $T^{2,3}_{s,1} \subset T_1$ and $T^{1,3}_{s,2} \subset T_2$ satisfy all the properties of Definition \ref{defn:stable decomp} except possibly the gluing data condition in item \eqref{item:stable bijection} and the adjacency conditions in item \eqref{item:adjacency}.

To prove \eqref{item:stable bijection}, we will want to prove the following analogue of the gluing data Claim \ref{claim:gluing data} from the proof of Theorem \ref{thm:stable tree}.  As before, given components $D,D' \in \pi_0(T^{2,3}_{s,1})$, we let $g^{D'}_D$ be the endpoint of $D$ adjacent to the component of $T - D$ containing $D'$, and similarly for components of $\pi_0(T^{1,3}_{s,2})$.  We also let $\wa^3_{1,2}$ denote the endpoint map on paired components $D$ and $\alpha^3_{1,2}(D)$ induced from the corresponding isometry.

\begin{claim}\label{claim:gluing data, iteration}
For any $D, D' \in \pi_0(T^{2,3}_{s,1})$, we have $\wa^3_{1,2}(g^{D'}_D) = g^{\alpha^3_{1,2}(D')}_{\alpha^3_{1,2}(D)}$.
\end{claim}

\begin{proof}[Proof of Claim \ref{claim:gluing data, iteration}]
The idea of the proof is to use the fact that the bijections $\alpha_{1,2}:\pi_0(T^{2}_{s,1}) \to \pi_0(T^1_{s,2})$  preserve the gluing data, and the fact that each component of $T^{2,3}_{s,1}$ is contained in a component of $T^2_{s,1}$.

Given a component $A \in \pi_0(T^2_{s,1})$, let $D_1, \dots, D_n$ be the components of $T^{2,3}_{s,1}$ contained in $A$, which we can take to be occurring in order going from one endpoint $A_1$ of $A$ to the other, $A_2$.  Let $D'_1, \dots, D'_n \in \pi_0(T^{1,3}_{s,2})$ be the corresponding components of $T^{1,3}_{s,2}$ contained in $\alpha_{1,2}(A) \in \pi_0(T^1_{s,2})$, where $D'_i = \alpha^3_{1,2}(D_i)$ by definition of $\alpha^3_{1,2}$.  Because of the ordering, it is easy to see that the gluing data condition between the various $D_i$ is satisfied, that is $\wa^3_{1,2}(g^{D_i}_{D_j}) = g^{D'_i}_{D'_j}$.

For the gluing data for other pairs of stable components outside of $A$, let $B \in \pi_0(T^2_{s,1}) - \{A\}$ be any other component.  Then for any of the $D_i$, the endpoint of $D_i$ adjacent to the component of $T- D_i$ containing $B$ is endpoint closest to the endpoint of $A$ adjacent to the component of $T - A$ containing $B$.  The same is true for any potential component $C \in \pi_0(T^{2,3}_{s,1})$ contained in $B$.  Hence the gluing data condition holds for these components, namely $\wa^3_{1,2}(g^C_{D_i}) = g^{\alpha^3_{1,2}(C)}_{\alpha^3_{1,2}(D_i)}$.
\end{proof}

Thus item \eqref{item:stable bijection} of Definition \ref{defn:stable decomp} is satisfied.  The proof of item \eqref{item:adjacency} now follows from an essentially identical argument to the corresponding part of the proof of Theorem \ref{thm:stable tree}.  

At this point, we have produced $\calY_0$-stable $(2L_1,2L_2)$-compatible decompositions $T^{2,3}_{s,1} \subset T_{s,1}$ and $T^{1,3}_{s,2} \subset T_{s,2}$, and a similar argument produces decompositions $T^{2,1}_{s,3} \subset T_{s,3}$ and $T^{3,1}_{s,2} = T^{1,3}_{s,2}$.  

 We now want to see that $T^{2,3}_{s,1}$ and $T^{2,1}_{s,3}$ are $\calY_0$-stable $(4L_1^2,4L_2^2)$-compatible decompositions for $(F_1;\calY_1)$ and $(F_3;\calY_3)$.  For this, define $\alpha_{1,3}:\pi_0(T^{2,3}_{s,1}) \to \pi_0(T^{2,1}_{s,3})$ to be $\alpha_{1,3} = \alpha^{1}_{2,3} \circ \alpha^3_{1,2}$.  Observe now that since $\alpha^{1}_{2,3}$ and $\alpha^3_{1,2}$ both preserve the gluing data meaning their composition does, too: if $D, D' \in \pi_0(T^{2,3}_{s,1})$ are distinct components, then 
 $$\wa^{3,1}_2\left(\wa^{2,3}_1(g^{D'}_{D})\right) = \wa^{3,1}_2\left(g^{\alpha^{2,3}_1(D')}_{\alpha^{2,3}_1(D)}\right) = g^{\alpha^{3,1}_2(\alpha^{2,3}_1(D'))}_{\alpha^{3,1}_2(\alpha^{2,3}_1(D))} = g^{\alpha_{1,3}(D')}_{\alpha_{1,3}(D)},$$
 as required.  As before, the fact that the gluing data is preserved provides the desired adjacency-preserving bijection $\beta_{1,3}:\pi_0(T_1 - T^{2,3}_{s,1}) \to \pi_0(T_3 - T^{2,1}_{s,3})$.

This completes the proof of the proposition.\end{proof}

\section{The Stabler Trees Theorem}\label{sec:stabler trees}

In this section, we state and prove our main stable tree comparison result, Theorem \ref{thm:stabler tree}, where we produce a sort of combinatorial almost-embedding between the stable trees associated to appropriately compatible finite subsets $F \subset F' \subset \calZ$.  This is the main technical result from this paper for proving our Stabler Cubulations Theorem \ref{thm:stabler cubulations}.  The statement requires a bit more setup.

\subsection{Basic lemmas and fixing notation}\label{subsec:basic tree setup, redux}

In this subsection, we fix some basic notation for the rest of the section.

The first lemma is an expansion of Lemma \ref{lem:basic tree lemma}.  In particular, item (3) allows for comparing the trees produced for subsets $F \subset F'$ by the minimal network function $\lambda$ from Subsection \ref{subsec:basic tree setup}; we leave the proof to the reader. Recall that $p_A$ denotes a closest point projection to a subset $A$ of a given metric space.

\begin{lemma}\label{lem:basic tree lemma, redux} 
Let $\calZ$ be a geodesic $\delta$-hyperbolic space and $\lambda$ the minimal network function
as in Subsection \ref{subsec:basic tree setup}. For any choice of $k>0$, there exists $\epsilon_{0,k}=\epsilon_{0,k}(k,\delta)\geq 0$ with $\ep_{0,k} \to \infty$ as $k \to \infty$ so that for all $\epsilon\geq \epsilon_{0,k}$ there exists $\ep'_k = \ep'(k, \ep)>0$ such that, if $F\subset \calZ$  with $|F|\leq k$ then
\begin{enumerate}
  \item There is 
a $(1,\ep/2)$-quasi-isometry $\ltree(F)\to \hull(F)$ which is $\ep/2$-close to the inclusion of $\ltree(F)$ into $\calZ$.
\item For any two points $x,y\in
\calN_\ep(\ltree(F))$, any geodesic joining them is in $\calN_{\ep'_k}(\ltree(F))$.
\item If $F \subset F'$ with $|F'|<k$, then
\begin{enumerate}
    \item $d^{Haus}_{\calZ}(\lambda(F), \hull_{\lambda(F')}(F))<\epsilon'_k.$
    \item For any $x \in F' - F$, we have $d_{\calZ}(p_{\lambda(F)}(x), p_{\hull_{\lambda(F')}(F)}(x))<\ep'_k.$
\end{enumerate} 
\end{enumerate}
\end{lemma}

\begin{remark}
    Item (3b) says that the distance in $\calZ$ between the closest point projection to $\lambda(F)$ of a point $x \in F' - F$ is within $\ep'_k$ of its projection to the hull in $\lambda(F')$ of $F$.  One can interpret this as saying that the expected branch point in $\lambda(F)$ for adding $x$ to $F$ is close to the actual branch point corresponding to $x$ in $\lambda(F')$.  This plays a role in the proofs of Lemmas \ref{lem:fake basic} and \ref{lem:fake separate} below.
\end{remark}

\begin{notation}\label{not:fixed based setup, redux}
    For the rest of this section, we fix the following collection of sets and constants:
\begin{enumerate}
    \item A natural number $k$, which globally controls the size of our finite subsets.
    \item Minimal network functions $\lambda, \lambda'$ controlled by $k$ as in Lemma \ref{lem:basic tree lemma, redux}.
    \item A positive number $\ep_k>\ep_{0,k}$ as in Lemma \ref{lem:basic tree lemma, redux}.
     \item Finite subsets $F \subset F' \subset \calZ$ with $|F'| \leq  k$;
     \begin{itemize}
         \item In particular, the embedding maps $\lambda(F), \lambda(F')\to \calZ$ are $(1,\ep_k/2)$-quasi-isometric embeddings.
     \end{itemize}
     \item A natural number $k_0 = k_0(k, \delta)>0$ large enough so that 
     \begin{enumerate}
        \item $\lambda(F) \subset \calN_{\frac{\ep_{0,k_0}}{2}}(\lambda(F'))$;
         \item $\ep_{0,k_0}/2> 10\ep'_k$.  This controls the constants we use in the stable tree construction;
         \item If $p, q\in \calN_{\ep_k}(\lambda(F))$, $p',q' \in \lambda(F')$ are closest points, $\gamma$ is a geodesic between $p,q$, and $\gamma' \subset \lambda(F')$ is a geodesic in $\lambda(F')$ between $p',q'$, then $$d^{Haus}_{\calZ}(\gamma, \gamma') < \ep_{0,k_0}.$$
     \end{enumerate}
   \item Positive numbers $\ep_{k_0}>\ep_{0,k_0}$ as in Lemma \ref{lem:basic tree lemma, redux}.
   
       \item Cluster graph constants: We fix 
       \begin{enumerate}
           \item A cluster proximity constant $\ep'_{k_0}$ so that $2\ep'_{k_0} > \ep'_{k_0} + \ep_{k_0}$.
           \item A cluster separation constant $E_0 = E_0(k_0, \delta, \ep_{k_0})>0$ so that $E_0 > 8\ep'_{k_0}.$
       \end{enumerate}
\end{enumerate}

\end{notation}

\begin{remark}
    The bounds on $k_0$ in item (5) from Notation \ref{not:fixed based setup, redux} deserve some comment.  Each of them uses the fact that $\ep_{0,k} \to \infty $ as $k\to \infty$.  For (5a), both $\lambda(F), \lambda(F')$ are $(1,\ep_k/2)$-quasi-isometrically embedded by assumption (4), so $\hull_{\lambda(F')}(F)$ and $\lambda(F)$ are uniformly close depending only on $\ep_k, \delta$.  Item (5b) is immediate.  And item (5c) is an application of the Morse property.
\end{remark}

\begin{remark}[$k$ and $k_0$]\label{rem:fixed size}
    The relationship between the constants $k$ and $k_0$ and the versions of the constants $\ep,\ep', E$ that they determine is worth a comment.  The constructions related to stable trees will use the $k_0$-versions of these constants, namely $\ep_{0,k_0}, \ep'_{k_0}, E_0$.  On the other hand, all of the constants associated to the minimal network functions $\lambda$ for $F,F'$ and the proximity constraint for the sets $\calY, \calY'$ to $\lambda(F), \lambda(F')$ will depend on the $k$-versions, in particular $\ep_{0,k}$ and $\ep'_{k}$.  
    
    In applications, such as building metrics on colorable HHSs in Section \ref{sec:asymp CAT(0)}, we will only need arguments which involve modeling the hulls of a controlled number of points.  As we shall see below, setting $k_0$ to be much larger than such a controlled number allows us to have a unified set of constants for building and comparing stable trees for finite sets of points of cardinality less than $k$.

    In the end, however, all of these constants fundamentally depend only on $k$ and $\delta$ (with the latter determined by $k$ in the HHS setting).  Thus throughout the section, we will only indicate the dependency of our constants on $k, \delta$.
\end{remark}

\subsection{Sporadic cluster points and iterated admissible setups}

Given finite subsets $F \subset F' \subset \calZ$ as in Notation \ref{not:fixed based setup, redux}, let $F' - F = \{x_1, \dots, x_n\}$.  The main part of our Stabler Tree Theorem \ref{thm:stabler tree} involves explaining how to build stable decompositions for adding one $x_i$ at a time to $F$.  The process of iterating this procedure to handle all of the $x_i$ simultaneously will essentially be an application of the iteration ideas from Proposition \ref{prop:stable iteration}.  Nonetheless, the setup requires setting some notation.

Fix $\ep_k$-setups (Definitions \ref{defn:ep setup}) $(F;\calY)$ and $(F';\calY')$ with $F \subset F'$, so that $(F;\calY)$ is $(0,\ep_k)$-admissible with respect to $(F';\calY')$ (Definition \ref{defn:admissible setup}), meaning in particular that $\calY \subset \calY'$.  The following definition, which was the motivation for Definition \ref{defn:sporadic domains}, puts an extra constraint on admissibility in the case where $F' - F = \{x\}$.  The set in the definition should be thought of as a union of neighborhoods of the additional branches of $\ltree(F')$ compared to $\ltree(F)$.

\begin{definition}\label{defn:sporadic cluster points}
    For $S\geq 0$, a cluster point $p \in \calY' - \calY$ is $S$-\textbf{sporadic} if
    $$p \notin \bigcap_{f \in F}\mathcal N_{S}\left(\hull(x,f)\right).$$
    \begin{itemize}
        \item We let $\calY_{\spor}(S)$ denote the set of $S$-sporadic cluster points.
    \end{itemize}
\end{definition}

\begin{figure}
    \centering
    \includegraphics[width=.75\textwidth]{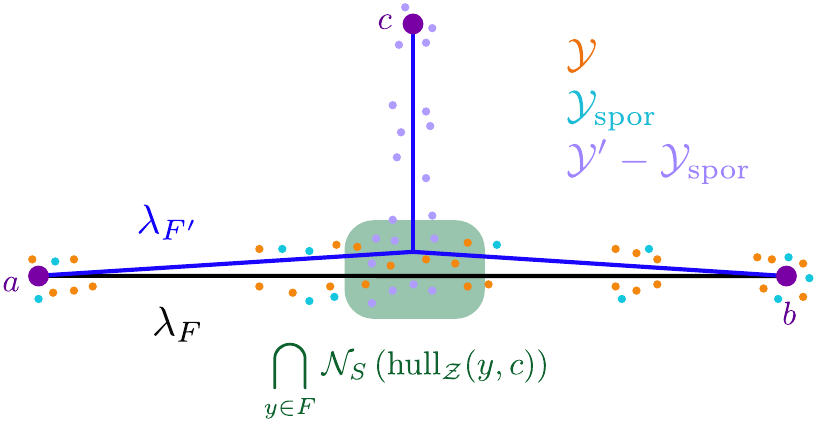}
    \caption{A relatively simple pair of setups with $F = \{a,b\}$ and $F' = \{a,b,c\}$.  The sporadic cluster points $\calY_{\spor}$ are those in $\calY' - \calY$ which lie outside of a wide neighborhood of the ``branch point'' for $c$ in $\lambda(F)$.  In the hierarchical setting and that of the Stabler Tree Theorem \ref{thm:stabler tree}, we will be able to control the number of these sporadic cluster points.}
    \label{fig:sporadic cluster}
\end{figure}

We now want to compare setups $(F;\calY)$ and $(F';\calY')$ as above where now $F' - F = \{x_1, \dots, x_n\}$.  Toward that end, for $1 \leq i \leq n$, set $F_i = F \cup \{x_1, \dots, x_i\}$.

\begin{definition}[Well-layered setups] \label{defn:well-layered}
Given $S,N>0$, we say that two $\ep_k$-setups $(F;\calY)$ and $(F';\calY')$ with $F \subset F'$ and $F' - F = \{x_1, \dots, x_n\}$ are $(S,N,\ep_k)$-\emph{well-layered} if there exists $\calY = \calY_0 \subset \calY_1 \subset \cdots \subset \calY_n = \calY'$ so that
\begin{enumerate}
    \item Each $(F_i;\calY_i)$ is $(0, \ep_k)$-admissible with respect to $(F_{i+1};\calY_{i+1})$, and
    \item If $\calY_{\spor}(S;i)$ denotes the set of $S$-sporadic domains for the setups $(F_i;\calY_i)$ and $(F_{i+1};\calY_{i+1})$, then $|\calY_{\spor}(S;i)|<N$.
    \item For each $1 \leq i \leq n$, we have $\calY_{\spor}(S;i) \subset \calN_{\ep_k/2}(\lambda(F_i))$.
\end{enumerate}
\end{definition}

\begin{remark}
    In the HHS setting, each tuple of points $F_i$ will come with some collection of domains $\calU_i$ to which it has a large diameter projection.  The hyperbolic space $\calZ$ will correspond to some $U \in \calU_i$, and the cluster points $\calY_i$ will correspond to $\rho$-points for domains $V \in \calU_i$ with $V \nest U$.

    The well-layered property corresponds to the fact that large projections for $F_i$ are large projections for $F_{i+1}$.  The proximity condition in item (3) of Definition \ref{defn:well-layered} will follow from the Bounded Geodesic Image axiom and our choice of $k$.  See Subsection \ref{subsec:basic tree setup, redux} below for a detailed discussion.
\end{remark}

\subsection{Statement of the Stabler Trees Theorem} \label{subsec:stabler trees}

We are now ready to state the theorem, which we note requires a bound on the number of sporadic domains and, crucially, that $\calY_{\spor}$ is $\ep_k/2$-close to $\lambda(F)$, so that $(F; \calY \cup \calY_{\spor})$ is an $\ep_k$-setup.

 \begin{theorem}\label{thm:stabler tree}
  There exists $S_0 = S_0(k, \delta)>0$ so that if $S > S_0$, then there exists an $L = L(S, N, k, \delta)>0$ so that the following holds.  Suppose that $(F; \calY)$ and $(F';\calY')$ are $(S,N,\ep_k)$-well-layered $\ep_k$-setups where $F \subset F'$ and $|F'|\leq k$.
  
Then the $(\ep_{k_0}, \ep'_{k_0}, E_0)$-stable trees $T, T'$ with respect to $(F, \calY)$ and $(F', \calY')$ admit stable decompositions $T_s \subset T$ and $T'_s \subset T'$ so that $T_s$ is $\calY$-stably $(L,L)$-compatible with $T'_s$.
\end{theorem}

\begin{remark}
    To be clear, the stable decomposition constant $L$ in Theorem \ref{thm:stabler tree} depends on sporadicity constant $S$ and the corresponding bound $N$, as well as the larger ambient constants $k_0$ and its related stable tree constants, but thus also on the various constants associated to the other bound $k$.  See Remark \ref{rem:fixed size} for further discussion.
\end{remark} 

\begin{remark}[Simplicialization] \label{rem:simplicialization}
    In Subsection \ref{subsec:simplicialization}, we explain how to prove a simplicialized version of Theorems \ref{thm:stable tree} and \ref{thm:stabler tree}.  This allows us to assume that the edge and stable components of stable trees and their the stable decompositions, respectively, are simplicial trees.  This modification is necessary for our work in Section \ref{sec:stabler cubulations}.  See Proposition \ref{prop:simplicialization} for a precise statement.
\end{remark}

\begin{remark}[Motivating Theorem \ref{thm:stabler tree}]
    As mentioned in Section \ref{sec:stable trees}, the main motivation for Theorem \ref{thm:stabler tree} is Corollary \ref{cor:collapsed isometry} in Subsection \ref{subsec:collapsing motivation} below.  It provides the raw materials we use in Section \ref{sec:stabler cubulations} to plug into the cubulation machinery from \cite{Dur_infcube}.
\end{remark}

\subsection{Outline of proof}\label{subsec:stabler tree outline}

We give an outline before proceeding with the proof, which will involve proving some supporting lemmas.  As we will see below, the main part of the proof is the case of adding a single point, that is when $F' = F \cup \{x\}$.  The proof of this base case involves a strategic application of Theorem \ref{thm:stable tree} as follows:

Starting with the observation above that $(F; \calY \cup \calY_{\spor})$ is an $\ep_k$-setup by assumption, we first add to each of the setups $(F, \calY \cup \calY_{\spor})$ and $(F', \calY')$ a bounded number of \emph{fake cluster points}, denoted by $\calY_{\fake}$ and $\calY'_{\fake}$, near the projection of $x$ to $\lambda(F)$ in $\calZ$ and the branch point corresponding to $x$ in $\lambda(F')$. Very roughly, adding these fake cluster points will allow us to just ``add a branch'' to the stable tree for $F$ to obtain the one for $F'$.  Fake cluster points are defined in Definitions \ref{defn:fake F} and \ref{defn:fake F'}, and their basic properties are recorded in Lemma \ref{lem:fake basic}.  The goal is to prove that there is a sequence of admissible setups:
\[(F; \calY) \rightsquigarrow (F; \calY \cup \calY_{\spor}) \rightsquigarrow (F; \calY \cup \calY_{\spor} \cup \calY_{\fake}) \rightsquigarrow (F'; \calY' \cup \calY'_{\fake}) \rightsquigarrow (F'; \calY').\]

In terms of building stable decompositions using Theorem \ref{thm:stable tree}, we will see that the transitions between the first, second, and fourth pairs are straightforward from the definitions (Subsection \ref{subsec:fake clusters}).  Building a stable decomposition for the third pair requires a detailed analysis and is the main work in this section.

As a first step, the fake cluster points act as a buffer by separating cluster points in $\calY \cup \calY_{\spor}$ from those in $\calY' - (\calY \cup \calY_{\spor})$ in a controlled way; see Lemma \ref{lem:fake separate}.  This in turn will help us control the combinatorial setup of the respective cluster separation graphs for these expanded ``fake'' setups; see Lemma \ref{lem:fake graph}.  In particular, by adding these fake cluster points so that their shadows go deep enough into their respective trees, we will be able to arrange (Proposition \ref{prop:edge inject}) that every edge component of the stable tree $T_{\fake} = T_{e,\fake} \cup T_{c, \fake}$ for $(F, \calY \cup \calY_{\spor} \cup \calY_{\fake})$ is an edge component of $T'_{\fake} = T'_{e,\fake} \cup T'_{c,\fake}$, the stable tree for the $\epsilon$-setup $(F \cup \{x\}, \calY' \cup \calY'_{\fake})$.

At this point, we can then apply Proposition \ref{prop:stable iteration} to build compatible stable decompositions bridging
between $(F, \calY)$ and $(F, \calY \cup \calY_{\spor} \cup \calY_{\fake})$, and $(F', \calY')$ and $(F', \calY' \cup \calY'_{\fake})$, respectively.  Since the stable pieces of the stable decompositions are contained in the (edges of) the edge components, an analogous argument to the proof of Proposition \ref{prop:stable iteration, pair} will allow us to build compatible stable decompositions on the common components of the stable trees corresponding to $(F, \calY \cup \calY_{\spor} \cup \calY_{\fake})$ and $(F', \calY' \cup \calY'_{\fake})$, since the edge components of the stable tree of the former are edge components of the stable tree for the latter by Proposition \ref{prop:edge inject}.

Upgrading from the base case to the general case---where we are adding more than one point---will then be another variation on this iteration argument.

\subsection{Fake cluster points} \label{subsec:fake clusters}

For the rest of the section, we will utilize the setup and notation from Subsection \ref{subsec:basic tree setup, redux} and the statement of Theorem \ref{thm:stabler tree}.  The core of the argument relates to the base case where we are adding a single point to $F$, i.e. where $F' - F = \{x\}$.  Iteratively adding points requires essentially no new ideas.

Toward that end, until our iterative argument in Subsection \ref{subsec:stabler iteration}, we assume that $F' - F = \{x\}$.

The first step of the proof is to define and establish properties of the fake cluster points, as discussed in the outline above (Subsection \ref{subsec:stabler tree outline}).  The sets of fake cluster points will be defined as \emph{nets} on the trees $\lambda(F)$, so we set this notation in advance:

\begin{definition}[Net]\label{defn:net}
Given constants $a,A\geq 0$, a $(a,A)$-\textbf{net} on a subspace $Y \subset X$ of a metric space is a collection $Z \subset Y$ of points so that:
\begin{itemize}
\item (density) $Z$ is $a$-coarsely dense in $Y$, and
\item (proximity) For any $z,z' \in Z$, we have $d_X(z,z') \geq A$.
\end{itemize}
\end{definition}

We first define the sets of fake cluster points $\calY_{\fake}$ and $\calY'_{\fake}$ for our setups $(F; \calY \cup \calY_{\spor})$ and $(F'; \calY')$, and then establish their basic properties.

\begin{definition}[Fake cluster points for $F$]\label{defn:fake F}
    Given constants $a, A, B\geq 0$, a set of $(a,A,B)$-\textbf{fake cluster points} for $(F; \calY \cup \calY_{\spor})$ is a choice of a $(a,A)$-net of the $B$-neighborhood in $\lambda(F)$ of the closest point projection $p_{\lambda(F)}(x)$ of $x$ to $\lambda(F)$.  We denote such a choice by $(a,A,B)-\calY_{\fake}$ or simply $\calY_{\fake}$ when the setup is fixed.
    
\end{definition}

\begin{remark}\label{rem:fake constants}
    In the definition of fake cluster points, the third parameter $B$ controls the diameter of the set around the ``attaching'' point of $x$ to $\lambda(F)$.  The first two parameters $a,A$ control the density and spacing within that set.  Hence, when controlling these parameters, we will need to make $a,A$ small and $B$ large.
\end{remark}

\begin{remark}[Sporadic and fake clusters]\label{rem:sporadic fake}
    The sporadicity constant $S_0$ from Theorem \ref{thm:stabler tree} is related to and partially determines the fake cluster diameter constant $B$, in that we will always need to take $B$ larger than (likely some controlled multiple of) $S_0$.  The idea is that sporadic cluster points avoid a neighborhood of the ``branch point'' in $\lambda(F)$ of the new point $x \in F' - F$.  On the other hand, the fake cluster points $\calY_{\fake}$ are precisely defined to cover such a neighborhood.  See Lemma \ref{lem:fake basic}.
\end{remark}

For the rest of this subsection, fix $a,A,B \geq 0$ and a set of $(a,A,B)$-fake cluster points $\calY_{\fake}$  for $(F; \calY \cup \calY_{\spor})$, where we will adjust $a,A,B$ as necessary as in Remark \ref{rem:fake constants}.

\begin{definition}[Fake cluster points for $F'$]\label{defn:fake F'}
    A set of $(a,A,B)$-\textbf{fake cluster points} for $(F'; \calY')$ is the set $\calY_{\fake}$ along with a choice of $(a,A)$-net of the $B$ neighborhood in $\lambda(F') - \hull_{\lambda(F')}(F)$ of the closest point projection $p_{\hull_{\lambda(F')}(F)}(x)$ of $x$ to the hull in $\lambda(F')$ of $F$.  We denote such a choice by $(a,A,B)-\calY'_{\fake}$, or simply $\calY'_{\fake}$ when the setup is fixed.
\end{definition}

\begin{figure}
    \centering
    \includegraphics[width=.75\textwidth]{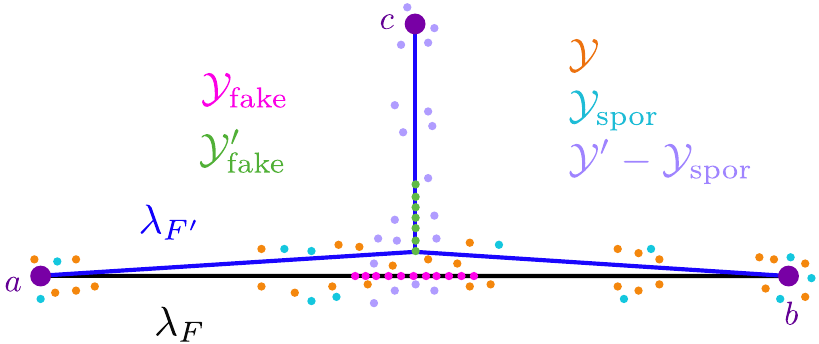}
    \caption{The example from Figure \ref{fig:sporadic cluster} revisited.  The fake cluster points $\calY_{\fake}$ lie along $\lambda(F)$ near this ``branch point'', while the points in $\calY'_{\fake}$ extend $\calY_{\fake}$ up the genuine branch for $c$ in $\lambda(F')$.  In Lemma \ref{lem:fake basic}, we show that there are two clusters $C_{\fake}, C'_{\fake}$ containing $\calY_{\fake}, \calY'_{\fake}$ respectively.  Then in Lemma \ref{lem:fake separate}, we show that these fake clusters separate cluster points in $\calY \cup \calY_{\spor}$ from those in $\calY'- \calY_{\spor}$.  }
    \label{fig:fake cluster}
\end{figure}

\begin{remark}
By definition, $\calY'_{\fake}$ is an extension of $\calY_{\fake}$, i.e. $\calY_{\fake} \subset \calY'_{\fake}$.  The extension is a net which moves into the ``new branch'' of $\lambda(F')$ corresponding to the added point $x \in F' - F$.  By construction, the set $\calY_{\fake}$ already ``covers'' the neighborhood in $\hull_{\lambda(F')}(F)$ of branch point corresponding to this ``new branch''.  See Figure \ref{fig:fake cluster} for a schematic.
\end{remark}

As we will see, once its parameters $a,A,B$ are carefully chosen, the set $\calY_{\fake}$ will acts as a buffer which insulates the rest of the setup $(F; \calY \cup \calY_{\spor})$ from the changes which occur when adding the additional cluster points in $\calY'- (\calY \cup \calY_{\spor})$.  Similarly, the set $\calY'_{\fake}$ of fake cluster points for $(F'; \calY')$ also acts as a buffer to the possibly unbounded amount of combinatorial data in $(F'; \calY')$ associated to the new point $x$.

We first establish some basic properties for both of $\calY_{\fake}$ and $\calY'_{\fake}$ in the following lemma.  In what follows, it may help the reader to recall the various definitions from Subsection \ref{subsec:cluster basics}.  In particular, we recall that we have fixed a cluster proximity constant $E_0$ and a cluster separation constant $\ep'_{k_0}$ as in Notation \ref{not:fixed based setup, redux}.

\begin{lemma}\label{lem:fake basic}
    There exist constants $a_0 = a_0(k, \delta)>0$, $A_0 = A_0(k, \delta)>0$, and $B_0 = B_0(k, \delta,S)>0$ so that for any $0 \leq a \leq a_0$, $0 \leq A \leq A_0$, and $B\geq B_0$ and any choices $\calY_{\fake}:=(a,A,B)-\calY_{\fake}$ and $\calY'_{\fake}:=(a,A,B)-\calY'_{\fake}$ of $(a,A,B)$-fake cluster points, the following hold:
    \begin{enumerate}
        \item The pairs $(F; \calY \cup \calY_{\spor}(B) \cup \calY_{\fake})$ and $(F'; \calY' \cup \calY'_{\fake})$ are $\ep_k$-setups.
        \item The setup $(F; \calY \cup \calY_{\spor}(B) \cup \calY_{\fake})$ is $\frac{\ep_{k_0}}{2}$-admissible with respect to $(F; \calY \cup \calY_{\spor}(B) \cup \calY_{\fake})$.
        \item There exists a unique cluster $C_{\fake}$ for the setup $(F; \calY \cup \calY_{\spor}(B) \cup \calY_{\fake})$ so that $\calY_{\fake} \subset C_{\fake}$.
        \item There exists a unique cluster $C'_{\fake}$ for the setup $(F'; \calY'\cup \calY'_{\fake})$ so that $\calY'_{\fake} \subset C'_{\fake}$.
        \item As sets, we have $C_{\fake} \subset C'_{\fake}$.
        \item If $p \in \calY' - \calY_{\spor} = \calY_{\spor}(S)$ satisfies $p \in \calN_{\ep_k}(\lambda(F))$ then $p \in C'_{\fake}$.
        \item There exists a constant $D_{\fake} = D_{\fake}(a,A,B)>0$ so that $\#\calY_{\fake} < D_{\fake}$ and $\#\calY'_{\fake} < D_{\fake}$.
        
    \end{enumerate}
\end{lemma}

\begin{proof}
    Item (1) is clear because we have chosen $\calY_{\fake} \subset \lambda(F)$ and $\calY'_{\fake} \subset \lambda(F')$.  Item (2) follows from item (5) of Notation \ref{not:fixed based setup, redux}.  Item (7) is automatic from the construction.
    
    Item (3) can be arranged by choosing the proximity constant $a = a(k, \delta)>0$ sufficiently small, in particular less than $E_0/2$ (for $E_0 = E_0\geq 8\ep'_{k_0}$ as in Notation \ref{not:fixed based setup, redux}), since this guarantees a chain of cluster points at pairwise distance $E_0/2$ which spans the whole neighborhood, and hence meaning they are all in one cluster.  Finally, making these constants smaller does not change this fact.

    Item (4) follows from a similar argument, after a couple of observations.  First, observe that, as in item (3), all cluster points in $\calY'_{\fake} - \calY_{\fake}$ form a single cluster for sufficiently small $a,A$.  To show that there is a cluster point in $\calY_{\fake}$ close to a point in $\calY'_{\fake} - \calY_{\fake}$, recall that in Subsection \ref{subsec:basic tree setup, redux}, we fixed $E_0 \geq 8\ep'_{k_0}$.  In particular, item (3b) of Lemma \ref{lem:basic tree lemma, redux} says that endpoint of the neighborhood in $\lambda(F') - \hull_{\lambda(F')}(F)$ covered by the net $\calY'_{\fake} - \calY_{\fake}$ is within $\ep'_k< \ep'_{k_0}$ of the projection of $x$ to $\lambda(F)$.  Hence taking $a,A$ sufficiently small says that there are fake cluster points in $\calY_{\fake}$ and $\calY'_{\fake} - \calY_{\fake}$ within $E_0$ of each other and thus $\calY'_{\fake}$ forms a single cluster.

    Finally, item (5) is basically by definition and the fact that $\calY_{\fake} \subset \calY'_{\fake}$ by construction (Definition \ref{defn:fake F'}).  In particular, if $p \in \calY \cup \calY_{\spor}$ is connected by some chain of $E_0$-close cluster points in $\calY \cup \calY_{\spor} \subset \calY'$ to a fake cluster point $q \in \calY_{\fake}$, this fact does not change if we instead consider $q$ as a fake cluster point in $\calY'_{\fake}$.  This completes the proof of the lemma.
\end{proof}

\begin{remark}
    Lemma \ref{lem:fake basic} is the first place using our assumption from Subsection \ref{subsec:basic tree setup, redux} that there is a global constant $k_0$ which controls the rest of the various constants $\ep_0,\ep,\ep',E$.  This assumption also plays a crucial role in Lemma \ref{lem:fake graph}.
\end{remark}

\begin{remark}
    We note that $C_{\fake} - \calY_{\fake}$ and $C'_{\fake} - \calY'_{\fake}$ can be nonempty.
\end{remark}

\begin{remark}
    The bounds in item (7) of Lemma \ref{lem:fake basic} are crucial for our argument, as they allow us to use Theorem \ref{thm:stable tree} to build stable decompositions with controlled constants for $(F; \calY \cup \calY_{\spor})$ and $(F; \calY \cup \calY_{\spor} \cup \calY_{\fake})$, and similarly for $(F'; \calY')$ and $(F';\calY' \cup \calY'_{\fake})$.
\end{remark}

\subsection{Structure of (fake) cluster graphs}\label{subsec:structure of fake clusters}

In this subsection, we analyze the structure of the cluster graphs for the setups $(F; \calY \cup \calY_{\spor} \cup \calY_{\fake})$ and $(F'; \calY' \cup \calY'_{\fake})$.  We point the reader to Subsections \ref{subsec:cluster basics} and \ref{subsec:stable trees defined} for the basic definitions of clusters, cluster graphs, and stable trees.

Let $\calG_{\fake}$ and $\calG'_{\fake}$ denote the cluster separation graph for $(F, \calY \cup \calY_{\spor} \cup \calY_{\fake})$ and $(F', \calY' \cup \calY'_{\fake})$, and $\calG^0_{\fake}$ and $\calG'^0_{\fake}$ their vertices.  By Lemma \ref{lem:fake basic}, the sets of fake cluster points are contained in single clusters,  $C_{\fake}$ and $C'_{\fake}$ similarly.

Recall that our choices of a fixed cluster proximity constant $E_0$ and cluster separation constant $\ep'_{k_0}$ above in Notation \ref{not:fixed based setup, redux} guarantees that the parameters of cluster formation for the setup $(F;\calY \cup \calY_{\spor} \cup \calY_{\fake})$ is the same for $(F';\calY' \cup \calY'_{\fake})$.

The first lemma organizes the clusters for $(F';\calY' \cup \calY'_{\fake})$:

\begin{lemma} \label{lem:organize F'}
    There exist cluster constants $a_1= a_1(k,\delta)>0$, $A_1 = A_1(k, \delta)>0$ and $B_1 = B_1(k, \delta)>0$ so that for any $0\leq a\leq \min\{a_0,a_1\}$, $0\leq A \leq \min\{A_0,A_1\}$ and $B>\max\{B_0,B_1\}$ for $a_0,A_0,B_0$ as in Lemma \ref{lem:fake basic}, any cluster point in $\calY' \cup \calY'_{\fake}$ belongs to a cluster of one of the following types:
    \begin{enumerate}
        \item[(C1')] A cluster consisting only of cluster points in $\calY \cup \calY_{\spor}$,
        \item[(C2')] The cluster $C'_{\fake},$ or
        \item[(C3')] A cluster consisting only of cluster points in $\calY' - (\calY \cup \calY_{\spor})$.
    \end{enumerate}
\end{lemma}

\begin{proof}
    Suppose that $y \in \calY \cup \calY_{\spor}$ and $y' \in \calY' - (\calY \cup \calY_{\spor})$ lie in the same cluster $C$.  We claim that $C = C'_{\fake}$.   
    By construction of the clusters, there exists a sequence $y = y_1, \dots, y_n = y'$ of cluster points in $\calY' \cup \calY'_{\fake}$, with consecutive pairs being $E$-close.  Since each $y_i \in N_{2\ep_k}(\lambda(F))$, there must be some $y_i$ within $2(E+\ep_k)$ of $p_{\hull_{\lambda(F')}(F)}(x)$.  Hence by choosing the density and proximity constants $a,A$ sufficiently small and the diameter constant $B$ sufficiently large---all controlled by $k,\delta$---at least one of the $y_i$ must be within $E_0$ of some fake cluster point.  This proves the claim and thus the lemma.
\end{proof}

\begin{figure}
    \centering
    \includegraphics[width=1\textwidth]{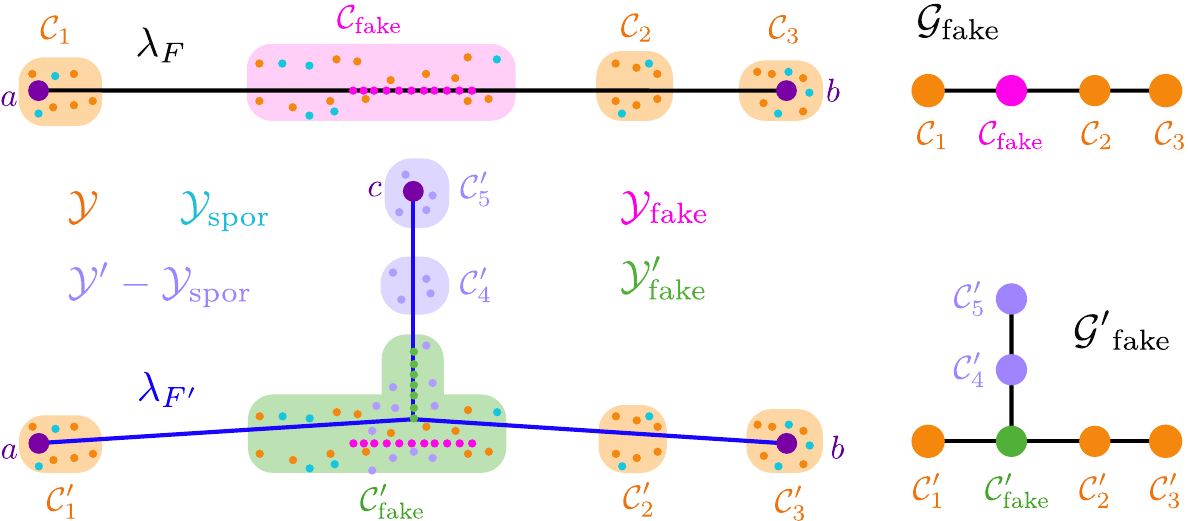}
    \caption{The example from Figures \ref{fig:sporadic cluster} and \ref{fig:fake cluster} revisited.  In Lemmas \ref{lem:organize F'} and \ref{lem:organize F}, we show that clusters in $(F;\calY \cup \calY_{\spor} \cup \calY_{\fake})$ and $(F';\calY' \cup \calY'_{\fake})$ come in 2 or 3 types, respectively.  In the diagrams of $\calG_{\fake}$ and $\calG'_{\fake}$, the orange clusters are of types (1) and (1') respectively, the lavender clusters are of type (3'), and the pink and green clusters are of type (2) and (2') respectively.  Later in Lemma \ref{lem:fake graph}, we prove that the cluster graph $\calG_{\fake}$ for $(F;\calY \cup \calY_{\spor} \cup \calY_{\fake})$ admits a type-preserving embedding in the cluster graph $\calG'_{\fake}$ for $(F';\calY' \cup \calY'_{\fake})$, sending $C_{\fake}$ to $C'_{\fake}$.}
    \label{fig:fake organize}
\end{figure}

Essentially the same proof gives a similar organization for the clusters for $(F;\calY \cup \calY_{\spor} \cup \calY_{\fake})$:

\begin{lemma}\label{lem:organize F}
    There exist cluster constants $a_1= a_1(k,\delta)>0$, $A_1 = A_1(k, \delta)>0$ and $B_1 = B_1(k, \delta, S)>0$ so that for any $0\leq a\leq \min\{a_0,a_1\}$, $0\leq A \leq \min\{a_0,a_1\}$ and $B>\max\{B_0,B_1\}$, any cluster point in $\calY \cup \calY_{\spor} \cup \calY_{\fake}$ belongs to a cluster of one of the following types:
    \begin{enumerate}
        \item[(C1)] A cluster consisting only of cluster points in $\calY \cup \calY_{\spor}$, or
        \item[(C2)] The cluster $C_{\fake}$.
    \end{enumerate}
\end{lemma}

The following lemma says that the fake clusters act as a buffer between the cluster points associated strictly to $F'$ and those associated to $F$.

\begin{lemma}\label{lem:fake separate}
    There exists cluster constants $a_2= a_2(k, \delta)>0$, $A_2 = A_2(k, \delta)>0$ and $B_2 = B_2(k, \delta, S)>0$ so that if $0 \leq a \leq \min\{a_0,a_1,a_2\}$, $0 \leq A \leq \min\{A_0,A_1,A_2\}$, and $B \geq \max\{B_0,B_1,B_2\}$, the following hold:
    
    If $p \in \calY \cup \calY_{\spor}$ and $q \in \calY' \cup \calY'_{\fake} - (\calY \cup \calY_{\spor} \cup \calY_{\fake})$, respectively, then any geodesic between $p,q$ must pass $2\ep'_{k_0}$-close to a point in $\calY_{\fake}$.  In particular, $C_{\fake}$ $\ep'_{k_0}$-separates any clusters of type (C1') and (C3').
\end{lemma}

\begin{proof}
    By our setup from Theorem \ref{thm:stabler tree}, both of $p,q$ lie in $\calN_{\ep_k/2}(\lambda(F'))$, hence any geodesic $\gamma$ between them is contained in $\calN_{\ep'_k}(\lambda(F'))$ by item (2) of Lemma \ref{lem:basic tree lemma, redux}.  Moreover, by item (3b) of that lemma and choosing the fake cluster diameter constant $B = B(k, \delta, S)>0$ sufficiently large, we can guarantee that $p$ lies outside of the $(10E_0+S)$-neighborhood of both $p_{\lambda(F)}(x)$ and $p_{\hull_{\lambda(F')}(F)}(x)$.

      If $p', q' \in \lambda(F')$ are closest points in $\lambda(F')$ to $p,q$, respectively, then the geodesic $\gamma'$ in $\lambda(F')$ between $p', q'$ is a uniform quasi-geodesic in the ambient space $\calZ$ whose quality is controlled by $\ep_k, \delta$ (recall that $\lambda(F') \to \hull(F')$ is a $(1,\ep_k/2)$-quasi-isometry).  This quasi-geodesic $\gamma'$ necessarily passes through some point $z \in \calY_{\fake}$.  Now item (5c) of Notation \ref{not:fixed based setup, redux} implies that $d^{Haus}_{\calZ}(\gamma, \gamma')<\ep_{0,k_0} < 2\ep'_{k_0}$, and hence $\gamma$ passes within $2\ep'_{k_0}$ of $\calY_{\fake}$, as required.  This completes the proof.
\end{proof}

The next lemma says that having chosen our cluster parameters $\ep'_{k_0}, E_0$ as in Notation \ref{not:fixed based setup, redux}, then the $E_0$-cluster graph $\calG_{\fake}$ admits an injection into $\calG'_{\fake}$ with fake clusters being identified.

\begin{lemma}\label{lem:fake graph}
    There exists $B_3 = B_3(k, \delta, S)>0$ so that for any $B>\max\{B_0, B_1, B_2, B_3\}$, the following holds:
    \begin{enumerate}
        \item There is an injection $I:\calG^0_{\fake} \to (\calG'_{\fake})^0$ sending $C_{\fake}$ to $C'_{\fake}$. \label{item:inject}
        \item If $I$ identifies $C \in \calG^0_{\fake} - \{C_{\fake}\}$ with $C' \in \left(\calG'_{\fake}\right)^0$, then $C = C'$ as sets. \label{item:identify}
        \item This injection extends to an embedding of graphs $\calG_{\fake} \to \calG'_{\fake}$ whose image $I(\calG_{\fake})$ is the induced subgraph of $\calG'_{\fake}$ on the vertices in $I(\calG^0_{\fake})$. \label{item:induced subgraph}
        \item The closure of each component of $\calG_{\fake} - C_{\fake}$ is the closure of some component of $\calG'_{\fake} - C'_{\fake}$. \label{item:components}
    \end{enumerate}   
\end{lemma}

\begin{proof}

We begin with the following claim, in which we show that clusters of type (C1') (from Lemma \ref{lem:organize F'}) correspond exactly to clusters of type (C1) (from Lemma \ref{lem:organize F}):

\begin{claim}\label{claim:fake vertex}
    If $c \in \calY \cup \calY_{\spor}$ is a cluster point of type (C1') contained in a cluster $C'$ for $(F';\calY' \cup \calY'_{\fake})$, there exists a cluster $C$ for $(F;\calY \cup \calY_{\spor} \cup \calY_{\fake})$ so that $C = C'$ as sets.
\end{claim}

\begin{proof}[Proof of Claim \ref{claim:fake vertex}]
    Suppose first that $c \in C'$ is of type (C1') and that $C'$ is a singleton.  Then $d_{\calZ}(c,c')>E_0$ for any $c' \in \calY \cup \calY_{\spor} - \{c\}$, otherwise $c,c'$ would be in the same cluster.  Hence Lemma \ref{lem:organize F'} implies that $\{c\}$ forms a singleton cluster in $(F;\calY \cup \calY_{\spor} \cup \calY_{\fake})$, as well.

    Now suppose $c,c' \in C'$ are distinct cluster points of type (C1') in the same cluster $C'$ for $(F'; \calY' \cup \calY'_{\fake})$.  Then each belongs in clusters of type (C1) by Lemma \ref{lem:organize F}.  If $c = c_1, \dots, c_n = c'$ is a chain of $E_0$-close cluster points in $C'$, then by Lemma \ref{lem:organize F'}, each of the $c_i$ are of type (C1).  Moreover, since consecutive pairs are $E_0$-close, we must have that the whole chain is in one cluster of type (C1) for $(F;\calY \cup \calY_{\spor} \cup \calY_{\fake})$.  Hence if $c,c' \in C$ is that cluster, then $C' \subset C.$

    To see that $C' =C$, suppose that $z \in C-C'$.  Then Lemma \ref{lem:organize F} again implies that there exists a chain $z = z_1, \dots, z_n = c$ of type (C1) cluster points so that consecutive pairs are $E_0$-close.  But then $z \in C'$, completing the proof.
\end{proof}

With Claim \ref{claim:fake vertex} in hand, we can now define the map $I: \calG^0_{\fake} \to (\calG'_{\fake})^0$ as follows.  First, set $I(C_{\fake}) = C'_{\fake}$.  Then for any cluster $C \in \calG^0_{\fake} - \{C_{\fake}\}$, Claim \ref{claim:fake vertex} uniquely identifies $C$ set-wise with a cluster $C' \in (\calG'_{\fake})^0$.  This proves item \eqref{item:inject} from the lemma, and item \eqref{item:identify} of the lemma is explicitly part of Claim \ref{claim:fake vertex}.

Item \eqref{item:induced subgraph} of the lemma is the following claim:

\begin{claim}\label{claim:fake edge}
    $I:\calG^0_{\fake} \to \left(\calG'_{\fake}\right)^0$ extends to an embedding $I:\calG_{\fake} \to \calG'_{\fake}$.  The image $I(\calG_{\fake})$ is an induced subgraph of $\calG'_{\fake}$ on $I(\calG^0_{\fake})$.
\end{claim}

\begin{proof}[Proof of Claim \ref{claim:fake edge}]
We first show that $I$ send pairs of adjacent clusters in $\calG_{\fake}$ to adjacent clusters in $\calG'_{\fake}$.  There are two cases, depending on whether or not $C_{\fake}$ is involved.

Suppose first that $C \in \calG^0_{\fake}$ is adjacent to $C_{\fake}$ in $\calG_{\fake}$.  We claim that $C' = I(C)$ is adjacent to $C'_{\fake}$.

Suppose for a contradiction that there exists another cluster $C'_2 \in \calG'_{\fake}$ which separates $C'$ from $C'_{\fake}$.  This means that there exists some minimal length geodesic $\gamma$ connecting points in $C', C'_{\fake}$ which passes within $2\ep'_{k_0}$ of some point $w \in C'_2$.  By choosing the fake cluster diameter constant $B = B(k, \delta)>0$ to be sufficiently large, we can force that the endpoint of $\gamma$ on $C'_{\fake}$ to actually be in $C_{\fake}$.  This forces $C'_2$ to be of type (C1) and hence corresponds to some cluster $C_2$ for the setup $(F; \calY \cup \calY_{\spor} \cup \calY_{\fake})$ by Claim \ref{claim:fake vertex}.  But this would contradict the assumption that $C$ is adjacent to $C_{\fake}$.

Next suppose that $C_1 \in \calG^0_{\fake}$ is adjacent to some type (1) cluster $C_2$ in $\calG_{\fake}$.  By a similar argument and choosing the fake cluster diameter constant $B = B(k, \delta, S)>0$ sufficiently large, any cluster that might $\ep'_{k_0}$-separate $I(C_1)$ from $I(C_2)$ in $\calG'_{\fake}$ would be of type (C1'), since the points in $\calY_{\fake}$ would separate them from all cluster points in $\calY'_{\fake} - \calY_{\fake}$ and $\calY' - (\calY \cup \calY_{\spor})$.  Hence such a cluster would correspond set-wise to a cluster of type (C1) by Claim \ref{claim:fake vertex}, and such a corresponding cluster would have separated $C_1, C_2$ in the construction of $\calG_{\fake}$, a contradiction.

To complete the proof of the claim, we need to show that $\calG'_{\fake}$ does not contain any unexpected edges between vertices in $I(\calG^0_{\fake})$.  Suppose then that $C_1,C_2\in \calG^0_{\fake}$ are not connected by an edge.  Then there exists a cluster $C_3 \in \calG^0_{\fake}$ $\ep'_{k_0}$-separating them, i.e. there exists a minimal length geodesic $\gamma$ connected points on $x_1 \in C_1$ and $x_2 \in C_2$ which intersects the $2\ep'_{k_0}$-neighborhood of a cluster point $p \in C_3$.  Regardless of whether or not $C_i = C_{\fake}$ for $i=1,2$, we have $x_i \in C_i \subset I(C_i)$ for $i=1,2,3$ by item \eqref{item:identify} of this lemma (proven above).  In particular, $p \in I(C_3)$ will $\ep'_{k_0}$-separate $x_1$ from $x_2$ and hence $I(C_1)$ from $I(C_2)$.  Thus it is impossible for $I(C_1)$ to be adjacent to $I(C_2)$ if $C_1,C_2$ were not.  This completes the proof of the claim.
\end{proof}

Finally, to finish the proof of the lemma, we prove item \eqref{item:components} in the following claim:

\begin{claim}\label{claim:components}
    The closure of each component of $\calG_{\fake} - C_{\fake}$ is the closure of some component of $\calG'_{\fake} - C'_{\fake}$.
\end{claim}

\begin{proof}[Proof of Claim \ref{claim:components}]
    It suffices to prove that any cluster in $(\calG'_{\fake})^0 - I(\calG^0_{\fake})$ is not connected to a cluster in $I(\calG^0_{\fake})-C'_{\fake}$ by an edge of $\calG'_{\fake}$.  Equivalently, we must show that $C'_{\fake}$ $\ep'_{k_0}$-separates any such pair of clusters.  By Lemma \ref{lem:organize F'}, any cluster $C' \in (\calG'_{\fake})^0 - I(\calG^0_{\fake})$ is necessarily of type (C3'), i.e. $C'$ consists entirely of cluster points in $\calY' - (\calY \cup \calY_{\spor})$.  On the other hand, by Lemma \ref{lem:organize F} and item \ref{item:identify} of this lemma, any cluster $C \in I(\calG^0_{\fake})-C'_{\fake}$ is necessarily of type (C1), i.e. it consists entirely of cluster points in $\calY \cup \calY_{\spor}$, and hence corresponds to a cluster of type (C1') by Claim \ref{claim:fake vertex}.

    In order for $I(C),C'$ to be connected by an edge in $\calG'_{\fake}$, any minimal length geodesic $\gamma$ between $I(C) = C$ and $C'$ must avoid the $2\ep'_{k_0}$ neighborhood of every other cluster.  However this would contradict Lemma \ref{lem:fake separate}.  This completes the proof of the claim and hence the lemma.
\end{proof}

\end{proof}

\subsection{Controlling edge components of $T_{\fake}$ and $T'_{\fake}$} With Lemma \ref{lem:fake graph} in hand, our last step in the buildup to the proof of Theorem \ref{thm:stabler tree} is the following culminating proposition.  In it, we show that the edge components of the $(\ep_{k_0},\ep'_{k_0}, E_0)$-stable tree $T_{\fake}$ for $(F; \calY \cup \calY_{\spor} \cup \calY_{\fake})$ are edge components of $T'_{\fake}$, the $(\ep_{k_0},\ep'_{k_0}, E_0)$-stable tree  for $(F'; \calY \cup \calY'_{\fake})$.  For the proof, the reader may benefit from reviewing the construction of the edge and cluster components of stable trees in Subsection \ref{subsec:stable trees defined}.

\begin{proposition} \label{prop:edge inject}
There exists $B_4 = B_4(k, \delta, S)>0$ so that if $B > \max\{B_0, B_1, B_2,B_3,B_4\}$, then every edge component of $T_{\fake} = T_{e,\fake} \cup T_{c,\fake}$ is an edge component of $T'_{\fake} = T'_{e,\fake} \cup T'_{c,\fake}$.
\end{proposition}

\begin{proof}
    Let $\calE^0_{\fake}$ and $(\calE'_{\fake})^0$ denote the bivalent clusters for the graphs $\calG_{\fake}$ and $\calG'_{\fake}$, and recall that the edge components of the stable trees $T_{e,\fake} \subset T_{\fake}$ and $T'_{e,\fake} \subset T'_{\fake}$ are defined as closures of the components of $\calG_{\fake} - \calE^0_{\fake}$ and $\calG'_{\fake} - \calE'^0_{\fake}$, respectively.

    Observe first that $C_{\fake}$ may be bivalent, while $C'_{\fake}$ is not bivalent by construction since $C'_{\fake}$ either disconnect $\calG'_{\fake}$ into more than two components or contains a point of $F'$, namely $x$.  It will be useful to first consider edge components defined without the involvement of $C_{\fake}$ and $C'_{\fake}$.
    
    Toward that end, observe that items \eqref{item:induced subgraph} and \eqref{item:components} of Lemma \ref{lem:fake graph} imply that if $C \in \calE^0_{\fake} -\{C_{\fake}\}$ then $I(C) \in (\calE'_{\fake})^0 - \{C'_{\fake}\}$, where $I:\calG^0_{\fake} \to \left(\calG'_{\fake}\right)^0$ is the injection provided by that lemma.  Hence if we consider the set of closures of components $\calV_{\fake}$ of $\calG_{\fake} - \calE^0_{\fake}$ and $\calV'_{\fake}$ of $\calG'_{\fake} - (\calE'_{\fake})^0$ as in the Definition \ref{defn:stable tree} of the stable trees $T_{\fake} = T_{e,\fake} \cup T_{c,\fake}$ and $T'_{\fake} = T'_{e,\fake} \cup T'_{c,\fake}$, then every component $V \in \calV_{\fake}$ not containing $C_{\fake}$ appears as a component in $\calV'_{\fake}$, again by item \eqref{item:components} of Lemma \ref{lem:fake graph}.

It remains to consider the components involving $C_{\fake}$ and $C'_{\fake}$.  For the latter, observe that there exists a unique component $V' \in \calV'_{\fake}$ containing $C'_{\fake}$, since $C'_{\fake}$ is not bivalent.  On the other hand, there are at most two such exceptional components in $\calG_{\fake} - \calE^0_{\fake}$, when $C_{\fake}$ is bivalent (the two components of which it is a boundary), and only one component when $C_{\fake}$ is not bivalent.

\smallskip
\textbf{\underline{$C_{\fake}$ is not bivalent}}: Suppose first that $C_{\fake}$ is not bivalent and let $V \in \calV_{\fake}$ be the closure of the (unique) component of $\calG_{\fake} - \calE^0_{\fake}$ containing $C_{\fake}$.  Let $B_1, \dots, B_n \in \calE^0_{\fake}$ be the bivalent boundary clusters for $V$.  By Lemma \ref{lem:fake graph}, they are all bivalent in $\calG'_{\fake}$ and also are boundary clusters of $V'$.  While $V'$ may have other bivalent clusters $B'_1, \dots, B'_k$ in its boundary, each of these additional bivalent clusters is of type (C3') (as in Lemma \ref{lem:organize F'}), i.e. each $B'_i$ consists only of points in $\calY' - (\calY \cup \calY_{\spor})$.

\begin{figure}
    \centering
    \includegraphics[width=1\textwidth]{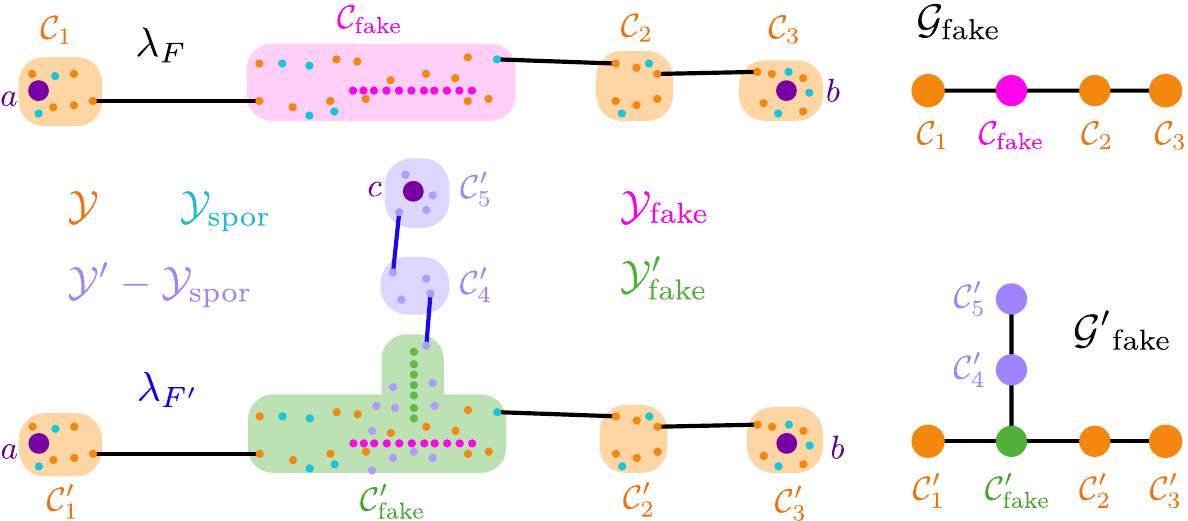}
    \caption{The example from Figures \ref{fig:sporadic cluster},  \ref{fig:fake cluster}, and \ref{fig:fake organize} revisited.  In Proposition \ref{prop:edge inject}, we show that every edge component of the stable tree $T_{\fake}$ for $(F; \calY \cup \calY_{\spor} \cup \calY_{\fake})$ is an edge component of the stable tree $T'_{\fake}$ for $(F';\calY' \cup \calY'_{\fake})$.}
    \label{fig:edge inject}
\end{figure}
Now consider the minimal networks $\lambda'(V)$ and $\lambda'(V')$.  These networks are minimal length forests which connect the clusters in $V$, $V'$ respectively.  Observe that any cluster besides $C_{\fake}$ involved in $V$ is of type $(C1)$.  By taking the fake cluster diameter constant $B = B(k, \delta, S)>0$ for  $C_{\fake}$ on $\lambda(F)$ to be sufficiently large, Lemma \ref{lem:fake separate} provides that the closest cluster point in $V^0 - \{C_{\fake}\}$ to any cluster point in a type (C3') cluster in $(V')^0$ is a cluster point in $C_{\fake}$.   Hence if clusters $C_1, \dots, C_k \in V^0 - \{C_{\fake}\}$ are connected by a component of $\lambda'(V)$, then this component also appears in $\lambda'(V')$, as one cannot find a shorter minimal network by replacing $\lambda'(C_1 \cup \cdots \cup C_k)$ by a smaller tree connecting different points on $C_1 \cup \cdots \cup C_n$, otherwise $\lambda'(V)$ would have used that tree instead by Lemma \ref{lem:inductive forest} (see Subsection \ref{subsec:basic tree setup} for the inductive definition of $\lambda'$).  By the same logic and again using Lemma \ref{lem:fake separate} and the inductive definition of $\lambda'$, if $D_1, \cdots, D_m, C_{\fake}$ are connected by a component of $\lambda'(V)$, then that component appears as a component of $\lambda'(V')$, as any such component would have to connect points contained in the corresponding clusters in $V^0$.

\smallskip
\textbf{\underline{$C_{\fake}$ is bivalent}}: Finally, suppose that $C_{\fake}$ is bivalent, while again $C'_{\fake}$ necessarily is not bivalent.  Observe that $C_{\fake}$ is part of two components $V_1,V_2 \in \pi_0\left(\calG_{\fake} - \calE^0_{\fake}\right)$ which determine forests $\lambda'(V_1^0)$ and $\lambda'(V_2^0)$.  Bivalency implies that the components of these forests connected to $C_{\fake}$ are segments connecting $C_{\fake}$ to the clusters $C_1,C_2 \in \calG^0_{\fake}$ to which it is adjacent in $\calG_{\fake}$.  By Lemma \ref{lem:fake graph}, $C_1,C_2$ determine clusters in $(F', \calY' \cup \calY'_{\fake})$ which (as vertices) are adjacent to $C'_{\fake}$ in $\calG'_{\fake}$.

Let $V' \in \pi_0(\calG'_{\fake} - \calE'^0_{\fake})$ be the component containing $C'_{\fake}$, and observe that $C_1,C_2$ are necessarily contained in its closure by Lemma \ref{lem:fake graph}.  Again by construction and the inductive definition of $\lambda'$, the segments in $\lambda'(V_1)$ and $\lambda'(V_2)$ connecting $C_1, C_{\fake}$ and $C_2, C_{\fake}$ must appear as components of $\lambda'(V')$ because $C_{\fake}$ separates every cluster point in $\calY \cup \calY_{\spor}$ from every cluster point in $\calY' \cup \calY'_{\fake} - (\calY \cup \calY_{\spor} \cup \calY_{\fake})$ by Lemma \ref{lem:fake separate}.  This completes the proof of the proposition.
\end{proof}

\subsection{Proof of Theorem \ref{thm:stabler tree} in the case where $F'- F = \{x\}$}

With Proposition \ref{prop:edge inject} in hand and our discussion of the fake cluster setups complete, there are two remaining steps in the proof of Theorem \ref{thm:stabler tree}.  The first is to complete the proof of the base case where we are adding a single point to $F$, namely where $F' - F = \{x\}$.  The second and last step is to explain how to iteratively add points.

\smallskip

Recall that we have a collection of setups as follows:
\[(F; \calY) \rightsquigarrow (F; \calY \cup \calY_{\spor}) \rightsquigarrow (F; \calY \cup \calY_{\spor} \cup \calY_{\fake}) \rightsquigarrow (F'; \calY' \cup \calY'_{\fake}) \rightsquigarrow (F'; \calY').\]

By our assumptions in the statement of Theorem \ref{thm:stabler tree}, the pair $(F; \calY) \rightsquigarrow (F; \calY \cup \calY_{\spor})$ is $(N, \ep_k)$-admissible.  The pairs $(F; \calY \cup \calY_{\spor}) \rightsquigarrow (F; \calY \cup \calY_{\spor} \cup \calY_{\fake})$ and $(F'; \calY' \cup \calY'_{\fake}) \rightsquigarrow (F'; \calY')$ are both $D_{\fake}$-admissible by Lemma \ref{lem:fake basic}.  Here, $D_{\fake}$ depends on our chosen tuple of fake cluster constants $(a,A,B)$ where $0 < a < \min\{a_0,a_1,a_2\}$, $0 < A < \min\{A_0,A_1,A_2\}$ as in Lemmas \ref{lem:fake basic}, \ref{lem:organize F'}, \ref{lem:organize F}, \ref{lem:fake separate}, and $B > \max\{B_0, B_1,B_2,B_3,B_4\}$ as in those lemmas plus Lemma \ref{lem:fake graph} and Proposition \ref{prop:edge inject}.

Hence we can apply Theorem \ref{thm:stable tree} and Proposition \ref{prop:stable iteration} to provides $\calY$-stable $M$-compatible decompositions for $(F; \calY)$ and $(F; \calY \cup \calY_{\spor} \cup \calY_{\fake})$ and similarly for $(F'; \calY')$ and $(F'; \calY' \cup \calY'_{\fake})$, where $M = M(S,N,k, \delta)>0$.

Denote these stable decompositions by $T_s \subset T_e$ and $T_{s,\fake} \subset T_{e,\fake}$, and similarly $T'_{s} \subset T'_{e}$ and $T'_{s,\fake} \subset T'_{e,\fake}$.  By Definition \ref{defn:stable decomp}, every component of these stable decompositions lies in some edge component of the corresponding stable tree.  On the other hand, Proposition \ref{prop:edge inject} says that every edge component of $T_{e,\fake}$ is a component of $T'_{e,\fake}$.  We are now done by essentially the same argument as in the proof of Proposition \ref{prop:stable iteration, pair}.

In that proof, we had three stable setups $T_1,T_2,T_3$ with $T_1,T_2$ and $T_2,T_3$ admitting compatible stable decompositions.  We used $T_2$ as a ``bridge'' on which to intersect the components of the stable decompositions coming from $T_1$ and $T_3$, respectively, and define the various maps required for Definition \ref{defn:stable decomp} via restrictions and compositions of the maps provided for the pairs $T_1,T_2$ and $T_2,T_3$.

In our current situation, this bridge is provided by Proposition \ref{prop:edge inject}, namely the common edges of $T_{e,\fake} \subset T'_{e,\fake}$.  A nearly identical argument involving appropriate intersections and restrictions then allows us to induce $\calY$-stably $4M$-compatible stable decompositions on $T_e\subset T$ and $T'_e \subset T'$.  We leave details to the reader.

In particular, this allows us to induce a uniformly compatible $\calY$-stable decomposition for $(F; \calY \cup \calY_{\spor} \cup \calY_{\fake})$ and $(F';\calY' \cup \calY'_{\fake})$.  We are now done with the base case $F' - F = \{x\}$ by applying Proposition \ref{prop:stable iteration}.

\subsection{Completing the proof when $F' - F = \{x_1, \dots, x_n\}$} \label{subsec:stabler iteration}

Assuming now that $F' - F = \{x_1, \dots, x_n\}$ with $(F;\calY)$ and $(F';\calY')$ being $(S,N,\ep_k)$-well-layered $\ep_k$-setups as in the statement of the theorem and hence Definition \ref{defn:well-layered}.

Let $(F_0; \calY_0) = (F;\calY)$, and for each $1 \leq i \leq n$, let $F_i = F \cup \{x_1, \dots, x_i\}$.  The base case of the theorem (which we established above) provides uniformly stable decompositions for each pair $(F_i;\calY_i)$ and $(F_{i+1};\calY_{i+1})$.  The proof is now complete by Proposition \ref{prop:stable iteration}, which says that the $(\ep_k, \ep'_k, E_0)$-stable trees for $(F;\calY)$ and $(F';\calY')$ admits $\calY$-stably $L$-compatible stable decompositions (Definition \ref{defn:stable decomp}) for $L = L(n, S, N, k, \delta)>0$, as required.  This completes the proof of Theorem \ref{thm:stabler tree}.

\subsection{Simplicialization} \label{subsec:simplicialization}

In this subsection, we make some observations that will allow us to make a minor but important (for what follows) modifications to our stable trees and stable decompositions.  First, we explain the motivation for the modification.

In the next section, we will want to plug two versions of our stable trees into the cubulation machine from \cite{Dur_infcube}.  The first will be obtained by the thickening and collapsing procedure described in Subsection \ref{subsec:HFT} directly to the edge components of the stable trees.  The second will be obtained by further collapsing each complementary component of a stable decomposition to a point (Definition \ref{defn:stable decomp}), leaving a tree built from the stable components.  In both cases, we will need to know that the resulting object is a simplicial tree (i.e., where all edges have integer lengths) where the collapsed points are vertices (to fit into the Definition \ref{defn:HFT} of a hierarchical family of trees).  This requires that 
\begin{enumerate}
    \item the edge components of our stable trees are simplicial trees, and
    \item that the length of each component of the stable decomposition has integer length.
\end{enumerate}

Realizing both of these properties requires a minor modification to our setup and arguments.  Since it is possible to do so, we have decided to comment on them here instead of earlier to allow the reader to focus on those already involved constructions without having an additional thing of which to keep track.

To arrange (1), we can modify the outputs of our minimal network functions $\lambda, \lambda'$ from Subsection \ref{subsec:basic tree setup} by simply collapsing segments of length less than 1 from the ends of the edges of the components of the networks.  For a finite set of points $F \subset \calZ$ with $|F|\leq k$ or a finite collection of finite subsets $A_1, \dots, A_n$, the resulting objects $\lambda(F)$ and $\lambda'(A_1, \dots, A_n)$ are still a tree and a forest, respectively, which are uniformly $(1,C)$-quasi-isometric to their originals, where $C$ depends on $\delta,k,$ and $n$, respectively.  Notably, this means that their images in $\calZ$ need no longer be connected, but connectivity of the image in $\calZ$ never plays a role in any of the arguments in this paper or in \cite{DMS_bary}.  In fact, the images of the original (unsimplicialized) network functions need not be embedded, see \cite[Remark 3 and Figure 10]{DMS_bary} for further similar pathologies.

As an upshot, we may assume that given any $\ep_k$-setup $(F, \calY)$ in $\calZ$, its corresponding stable tree $T = T_e \cup T_c$ has the property that each component of $T_e$ is a simplicial tree.

To arrange (2), we need a related set of observations.  Suppose that $(F, \calY)$ and $(F', \calY')$ with $F \subset F'$ are $(N,\ep_k)$-admissible $\ep_k$-setups satisfying the assumptions of Theorem \ref{thm:stabler tree}, so that they admit $\calY$-stably $L$-compatible stable decompositions $T_s \subset T$ and $T'_s \subset T'$, for $L = L(k, N, \delta, S)>0$ (Definition \ref{defn:stable decomp}).  In particular, there is a bijection $\alpha:\pi_0(T_s) \to \pi_0(T'_s)$, where components $B = \alpha(B)$ identified by $\alpha$ are identified by an isometry $i_{B,\alpha(B)}:B \to \alpha(B)$.  All but $L$-many of these components are identical components of $T_e \cap T'_e$, and hence are already simplicial.  The remaining (at most) $L$-many components (which are all segments) can thus be trimmed by removing segments of length less than $1$ using the isometries to identify segments to be collapsed.  Moreover, we can arrange that the endpoints of the components of $T_{U,s}$ lie at the vertices of components of $T_{U,e}$ which contain them.  Since there are at most $L$-many of these and each collapsed segment is short, the resulting collection of components still results in a stable decomposition, where we possibly have to increase $L$ by a bounded amount.

Combining these two observations, we get:

\begin{proposition}\label{prop:simplicialization}
    We may assume that all edge components of stable trees and all components of stable decompositions are simplicial trees.  Moreover, the leaves of any component of a stable decomposition lies at vertices of the edge component that contains it.
\end{proposition}

\subsection{Collapsed trees and stable decompositions}\label{subsec:collapsing motivation}

The following corollary motivates Definition \ref{defn:stable decomp} and the Stabler Trees Theorem \ref{thm:stabler tree}.  Roughly, it says that the combinatorial data used to construct the stable trees and stable decompositions are correctly encoded when we collapse the complementary components of the stable decompositions.

\begin{corollary}\label{cor:collapsed isometry}
Let $(F;\calY)$ and $(F'; \calY')$ with $\calY \subset \calY'$ be $(N,\epsilon)$-admissible $\epsilon$-setups with $\calY$-stable $L$-compatible decompositions $T_s \subset T_e$ and $T'_s \subset T'_e$.   Let $\Delta:T \to \hT$ and $\Delta':T' \to \hT'$ denote the quotients obtained by collapsing each component of $T-T_s$ and $\hull_{T'}(F)-T'_s$ to a point.  The following hold:
\begin{enumerate}
\item $\hT$ and $\hT'$ are simplicial trees where each collapsed component of $T-T_s$ and $\hull_{T'}(F)-T'_s$ is a vertex in the simplicial structure, and
\item There exists an isometric embedding $\Phi:\hT \to \hT'$ which restricts to the isometry of pairs of stable components $i_{E,\alpha(E)}:E \to \alpha(E)$ for each $E \in \pi_0(T_s)$.   Moreover, we have
\begin{enumerate}
\item $\Phi(\Delta(f)) = \Delta'(f)$ for all $f \in F$,
\item If $y \in \calY$ and $D_y \in \pi_0(T - T_s)$ and $D'_y \in \pi_0(T - T_s)$ contain $\mu(C_y)$ and $\mu(C'_y)$, respectively, then $\Phi(\Delta(D_y)) \subset \Delta'(D'_y)$. 
\end{enumerate}
\end{enumerate}
\end{corollary}

\begin{proof}
First, observe that $\hT$ and $\hT'$ are trees because $T,T'$ are and each component of $T-T_s$ and $T'-T'_s$ is a subtree.  They are simplicial trees with the given description by Proposition \ref{prop:simplicialization}, and so the distance between the points corresponding to pairs of collapsed components of $T-T_s$ and $T'-T'_s$ is a positive integer.

For the second conclusion, we can define $\Phi:\hT\to \hT'$ as follows: For each vertex $v \in \hT$ corresponding to a component $C \in \pi_0(T-T_s)$, we define $\Phi(v) = v'$, where $v'$ is the vertex of $\hT'$ corresponding to $\beta(C) \in \pi_0(\hull_{T'}(F)-T'_s)$.  If now $x \in \hT$ is not such a vertex, then $x \in \Delta(E)$ for some component $E \in \pi_0(T_s)$.  So we define $\Phi(x) = \Delta'(i_{E, \alpha(E)}(x))$.  Note that $\Phi:\hT\to \hT'$ is an embedding because $\alpha$ and $\beta$ are bijections.  Moreover, $\Phi$ is an isometric embedding by item \eqref{item:Adjacency-preserving}, since any geodesic between $x, y \in \hT$ consists of a sequence of edges in $\Delta(T_s)$ connected at vertices corresponding to components of $T_u$, and \eqref{item:Adjacency-preserving} implies that adjacent components of $T_s$ are (isometrically) identified with adjacent components of $T'_s$ at the corersponding endpoints.  Hence $\Phi$ sends geodesics to geodesics.

Finally, both of the subitems of (2) follow directly from item \eqref{item:cluster identify} of Definition \ref{defn:stable decomp}.  This completes the proof.
\end{proof}

\section{Stabler cubulations} \label{sec:stabler cubulations}

In this section, we prove our main stability statement about cubical approximations in colorable HHSs, stated as Theorem \ref{thm:stabler cubulations} below.

The rough idea is that given a pair of finite subsets $F \subset F' \subset \calX$ of a colorable HHS $\calX$, we would like the cubical approximation $\calQ_F$ for $F$ to admit a convex embedding into the cubical approximation $\calQ_{F'}$ for $F'$.  While this is not true on the nose, it is true up to deleting boundedly-many hyperplanes from $\calQ_F \to \calR_F$ and $\calQ_{F'} \to \calR_{F'}$ to obtain refined cubical models.  This, along with an equivariance property, is the content of Theorem \ref{thm:stabler cubulations}, which we restate below:

\stabler*

\begin{remark}\label{rem:equivariance}
The special case of Theorem \ref{thm:stabler cubulations} where $F'=gF$ is already covered by \cite[Theorem 4.1]{DMS_bary}. Hence, we can and will only prove Theorem \ref{thm:stabler cubulations} for $g$ equal to the identity, and applying this case together with \cite[Theorem 4.1]{DMS_bary} we can put the relevant diagrams together and obtain the full statement of Theorem \ref{thm:stabler cubulations}.
\end{remark}

The proof of Theorem \ref{thm:stabler cubulations} proceeds in a few steps.  The first involves explaining how to convert the hierarchical data associated to finite subsets $F \subset F' \subset \calX$ into an appropriate input for our Stabler Tree Theorem \ref{thm:stabler tree}.  This involves using our work from Section \ref{sec:controlling domains} in a crucial way.  Then with the output of Theorem \ref{thm:stabler tree} in hand, we need to do a bit more work to convert it to an appropriate input for the cubulation machinery from \cite{Dur_infcube}.  We give a detailed outline of the argument below in Subsection \ref{subsec:stabler cubes, outline}, after laying out the basic setup.

\subsection{Fixing a setup and some notation} \label{subsec:fixed base setup, cubes}

For the rest of this section, fix a colorable HHS $(\calX, \mathfrak S)$ (Definition \ref{defn:colorable}, endowed with the stable projections from Theorem \ref{thm:stable proj}.  We further fix a constant $k>0$ and finite subsets $F \subset F' \subset \calX$ with $|F'|<k$.  Fixing a projection threshold $K = K(\mathfrak S)>0$ sufficiently large as in Theorem \ref{thm:stable proj}, we let $\calU = \Rel_K(F)$ and $\calU' = \Rel_K(F')$ denote the $K$-relevant domains for $F,F'$ respectively (see Notation \ref{not:rel}).

For each $U \in \calU$, we denote the projection of $F$ to $\calC(U)$ by $F_U$.  We also denote the set of $K$-relevant domains nesting for $F$ into $U$ by $\calU_U = \{V \in \calU| V \nest U\}$.  We then denote the set of their projections to the $\delta$-hyperbolic space $\calC(U)$ by $\calY_U = \{\rho^V_U|V \in \calU_U\}$.  We define $F'_U, \calU'_U,$ and $\calY'_U$ analogously for the set $F'$.  We note that $\delta = \delta(\mathfrak S)>0$ depends only on the ambient HHS structure.  Furthermore, recall that we have arranged via Remark \ref{rem:rho_can_be_points} that the images of all of these projections $\pi_U$ and $\rho^V_U$ are points in $\calC(U)$.  

For each $U \in \calU \cup \calU'$, let $\lambda_U, \lambda'_U$ be the minimal network functions for $\calC(U)$ as defined in Subsection \ref{subsec:basic tree setup}.  The following is an easy consequence of the Bounded Geodesic Image axiom and (the uniform) hyperbolicity of $\calC(U)$: 

\begin{lemma}\label{lem:HHS to ep setup}
    There exist $\ep_{\mathfrak S} = \ep_{\mathfrak S}(\mathfrak S, k)>0$ so that $\rho^V_U \subset \calN^{\calC(U)}_{\ep_{\mathfrak S}/2}(\lambda_U(F))$ for all $V \in \calU_U$.  In particular, $(F_U;\calY_U)$ is an $\ep_{\mathfrak S}$-setup for each $U \in \calU$.  The same statement holds for the setups $(F'_U; \calY'_U)$ for each $U \in \calU'$.
\end{lemma}

We now set the following notation, along the lines of Notation \ref{not:fixed based setup, redux} from Subsection \ref{subsec:basic tree setup, redux}.  For completeness, this involves reiteration of some of the above notation in this subsection as well as the items of Notation \ref{not:fixed based setup, redux}.

\begin{notation}\label{not:fixed setup, cubes}
    For the rest of this section, we fix the following collection of sets and constant:
    \begin{enumerate}
        \item \textbf{Global notation from the ambient colorable HHS $(\calX, \mathfrak S)$}:
        \begin{enumerate}
        \item A natural number $k$, which globally controls the size of our finite subsets.
        \item Finite subsets $F \subset F' \subset \calX$ with $|F'|\leq k$.
        \item A projection threshold $K = K(k,\mathfrak S)>0$ at least as large as the one in Theorem \ref{thm:stable proj}. 
        \item Sets of $K$-relevant domains $\calU = \Rel_K(F)$ and $\calU' = \Rel_K(F')$ for $F,F'$, respectively.
        \item For each $U \in \mathfrak S$, we denote the projections of $F,F'$ to $\calC(U)$ by $F_U,F'_U$.
        \item For each $U \in \calU \cup \calU'$, we set $\calU_U = \{V \in \calU| V \nest U\}$ and $\calU'_U = \{V \in \calU'| V \nest U\}$.
        \item For each $U \in \calU \cup \calU'$, we set $\calY_U = \{\rho^V_U|V \in \calU_U\}$ and $\calY'_U = \{\rho^V_U| V \in \calU'_U\}$.  Note that $\calY_U = \calY_U \cap \calY'_U$ by definition.
      \end{enumerate}
      \item \textbf{Notation for stabler tree construction}:
      \begin{enumerate}
        \item A positive number $\ep> \max\{\ep_{\mathfrak S}, \ep_{0,k}\}$, where $\ep_{0,k} =\ep_{0,k}(k, \mathfrak S)>0$ is the constant provided by Lemma \ref{lem:basic tree lemma, redux}.  In particular:
        \begin{enumerate}
            \item The embedding maps $\lambda_U(F_U), \lambda_U(F'_U) \to \calC(U)$ are $(1, \ep/2)$-quasi-isometric embeddings.
            \item Both $(F_U;\calY_U)$ and $(F'_U;\calY'_U)$ are $\ep$-setups in $\calC(U)$ by Lemma \ref{lem:HHS to ep setup}.
        \end{enumerate}
        \item A positive number $\ep' = \ep'(k, \mathfrak S)>0$ so that $2\ep' > \ep + \ep'$.
        \item A natural number $k_0 = k_0(k, \mathfrak S)>0$ large enough so that
       \begin{enumerate}
        \item $\lambda_U(F_U) \subset \calN_{\frac{\ep_{0,k_0}}{2}}(\lambda(F'))$;
         \item $\ep_{0,k_0}/2> 10\ep'$.
         \item If $p, q\in \calN_{\ep}(\lambda_U(F_U))$, $p',q' \in \lambda_U(F'_U)$ are closest points, $\gamma$ is a geodesic between $p,q$, and $\gamma' \subset \lambda_U(F'_U)$ is a geodesic in $\lambda_U(F'_U)$ between $p',q'$, then $$d^{Haus}_{\calZ}(\gamma, \gamma') < \ep_{0,k_0}.$$
     \end{enumerate}
   \item Positive numbers $\ep_{k_0}>\max\{\ep_{0,k_0}, \ep\}$ as in Lemma \ref{lem:basic tree lemma, redux}.
   
       \item Cluster graph constants: We fix 
       \begin{enumerate}
           \item A cluster proximity constant $\ep'_{k_0}$ so that $2\ep'_{k_0} > \ep'_{k_0} + \ep_{k_0}$.
           \item A cluster separation constant $E_0 = E_0(k_0, \delta, \ep_{k_0})>0$ so that $E_0 > 8\ep'_{k_0}.$
       \end{enumerate}
       \item For each $U \in \calU$, the $(\ep_{k_0}, E_{k_0}, \delta)$-stable tree $T_U = T_{U,e} \cup T_{U,c}$ for $(F_U;\calY_U)$.  Define $T'_U = T'_{U,e} \cup T'_{U,c}$ similarly for $(F'_U;\calY'_U)$ when $U \in \calU'$.
\end{enumerate}
    \item \textbf{Sporadicity constants and notation}:
    \begin{enumerate}
        \item We fix $D>\max\{D_0,S_0\}$ where $D_0 = D_0(\mathfrak S)>0$ is as in Proposition \ref{prop:sporadic domains} and $S_0 = S_0(k, \mathfrak S)>0$ is as in Theorem \ref{thm:stabler tree}.
    \item  For each $U \in \calU \cap \calU'$, we set $\calV_{U,D}$ to be the set of $D$-sporadic domains in $U$ as in Definition \ref{defn:sporadic domains}.
    \item For each $U \in \calU \cap \calU'$, we let $\calY_{U, spor}(D) = \{\rho^V_U|V \in \calV_{U,D}\}$ denote the set of $D$-sporadic cluster points for the domains in $\calV_{U,D}$.
     \end{enumerate}
    \end{enumerate}
\end{notation}

\begin{remark}\label{rem:K}
    The largeness threshold $K = K(k, \mathfrak S)>0$ is the main constant that we will need to be able to periodically adjust during our arguments.  Since this happens many times explicitly and implicitly (e.g., through the arguments from \cite{Dur_infcube}), we will merely comment on it during proofs, making note that any increases will only depend on our choice of $k$ and the ambient setup $(\calX, \mathfrak S)$. 
\end{remark}

We now want to record an important but straightforward consequence of the above choices of constants and notation.  Note that the lemma only applies to domains in $\calU \cap \calU'$, since there are the only domains where we will need to apply our stabler tree techniques.

\begin{lemma}\label{lem:HHS basic setups are compatible}
    For each $U \in \calU \cap \calU'$, the $\ep$-setups $(F_U;\calY_U)$ and $(F'_U;\calY'_U)$ are $(D,N,\ep)$-well-layered as in Definition \ref{defn:well-layered}, where $N = N(D,k, \mathfrak S)>0$ is as in Proposition \ref{prop:sporadic domains}.
    \begin{itemize}
        \item In particular, Theorem \ref{thm:stabler tree} provides a constant $L = L(k, \mathfrak S)>0$ and $L$-stably $\calY_U$-compatible stable decompositions $T_{s,U} \subset T_{U,e}$ and $T'_{s,U} \subset T'_{U,e}$.
    \end{itemize}
\end{lemma}

\begin{proof}
    First, $(F_U;\calY_U)$ and $(F'_U;\calY'_U)$ are $\ep$-setups by Lemma \ref{lem:HHS to ep setup} and our choice of $\ep$ in Notation \ref{not:fixed setup, cubes}.  We now check the conditions well-layered conditions in Definition \ref{defn:well-layered}.

    Let $F' - F = \{x_1, \dots, x_n\}$ and set $F_0 = F, F_n = F'$ and $F_i = F \cup \{x_1, \dots, x_i\}$ for all $i >0$.  Set $\calU_i = \Rel_K(F_i)$.  For each $U \in \calU \cap \calU'$ and $0 \leq i \leq n$, set $\calY_{U,i} = \{\rho^V_U| V \in \calU_i, V \nest U\}$.
    
    Now observe that since $\calY_{U,i} \subset \calY_{U,i+1}$ for all $0 \leq i <n$, we have that each $(F_i; \calY_{U,i})$ is $(0,\ep)$-admissible with respect to $(F_{i+1};\calY_{U, i+1})$ for all $0 \leq i <n$ by our choice of $\ep$ from Notation \ref{not:fixed setup, cubes}.  This proves item (1) of Definition \ref{defn:well-layered}.
    
    To see items (2) and (3), for each $0 \leq i \leq n$, let $\calY_{U,i,\spor} \subset \calY_{u,i+1}$ denote the set of $\rho$-points $p$ where 
    $$p \notin \bigcap_{f \in F_i} \calN_D(\hull_{\calC(U)}(x_{i+1}, f)).$$

    Item (3) follows immediately because of the Bounded Geodesic Image axiom plus the ``moreover'' statement of Proposition \ref{prop:sporadic domains}, which says that all domains in $\calV_{U, \spor}$ are $(K-2\ES)$-relevant for $F$, where $\ES=\ES(\mathfrak S)>0$ is ten times larger than the constants in the HHS definition (as in Notation \ref{not:ES}).  In particular, we can arrange this by making $K=K(\mathfrak S)>0$ large enough.
    
    Finally, for item (2), by our choice of sporadic constant $D = D(k, \mathfrak S)>0$ in Notation \ref{not:fixed setup, cubes}, Proposition \ref{prop:sporadic domains} provides an $N = N(k, \mathfrak S)>0$ so that $|\calY_{U,i,\spor}|<N$ for each $0\leq i < n$.  This completes the proof.
\end{proof}

\subsection{Outline of the rest of the section}\label{subsec:stabler cubes, outline}

With the output of Lemma \ref{lem:HHS basic setups are compatible} in hand---namely, the uniformly compatible stable decompositions for the stable tree $T_U, T'_U$ for each $U \in \calU \cap \calU'$---we can proceed to modify these trees so that they are appropriately configured for input into the cubulation machine from \cite{Dur_infcube}.  We outline this now.

The first main step is to obtain an HHS-like system of projections---called a \emph{hierarchical family of trees (HFT)} (Definition \ref{defn:HFT})---on collapsed versions of the collections of stable trees $\{T_U\}_{U \in \calU}$ and $\{T'_U\}_{U \in \calU'}$, where all of the coarseness from the HHS setup has been removed.  This is accomplished by first observing that the hierarchical data provided by the stable trees, subsurface projections, and relative projections induce a family of projections on the trees themselves, called a \emph{reduced tree system} (Subsection \ref{subsec:cubical models background}).  In our setting, this information is encoded in the cluster components of the $T_U, T'_U$.  Since the projections in a reduced tree system are still coarse, one then \emph{thickens} along these components (Subsection \ref{subsec:thickenings}), and then collapses down the thickened components to obtain a two families of simplicial trees $\{\hT_U\}_{U \in \calU}$ and $\{\hT'_U\}_{U \in \calU'}$ which are HFTs (Subsection \ref{subsec:HFT}).

These HFTs have naturally associated \emph{$0$-consistent subsets} $\calQ, \calQ'$ which are uniformly quasi-isometric to the hierarchical hulls, $\hull_{\calX}(F)$ and $\hull_{\calX}(F')$, respectively (Subsection \ref{subsec:Q defined}).  On the other hand, they are CAT(0) cube subcomplexes of the products of the $\hT_U$ and $\hT'_U$, respectively (Subsection \ref{subsec:HFT dual}), thus providing us the cubical models for $\hull_{\calX}(F)$ and $\hull_{\calX}(F')$.

The next main step is then employing the Stabler Tree Theorem \ref{thm:stabler tree} in this context.  In Subsection \ref{subsec:stable decomp thick}, we show that the thickenings of the trees $T_U, T'_U$ from above admit stable decompositions (in the sense of Definition \ref{defn:stable decomp}), which we then collapse as above in Subsection \ref{subsec:collapsing unstable parts} to obtain new HFTs with $0$-consistent sets $\calQ_0$ and $\calQ'_0$, respectively.  We use the results from Section \ref{sec:controlling domains} to control the size and number of collapsed pieces (see Lemma \ref{lem:controlling unstable parts} in particular).  Finally, in Subsection \ref{subsec:collapse hyperplane deletion}, we show that the corresponding maps $\calQ \to \calQ_0$ and $\calQ' \to \calQ'_0$ are actually hyperplane deletion maps deleting a controlled number of hyperplanes.  With this, we will have defined all of the maps that appear in Theorem \ref{thm:stabler cubulations}.

The final main step is then to prove that the various pieces of the diagram in Theorem \ref{thm:stabler cubulations} coarsely commute.  Lemma \ref{lem:commute collapse} shows that the upper and lower triangles coarsely commute.  Proposition \ref{prop:cubical convex embedding} uses Proposition \ref{cor:collapsed isometry} and results from \cite{Dur_infcube} to show that the map $\calQ_0 \to \calQ'_0$ is a convex embedding.  Finally, Proposition \ref{prop:cc comm diagram} proves that the middle triangle coarsely commutes. 

\subsection{Reduced tree systems from stable trees}\label{subsec:cubical models background}

The notion of a reduced tree system axiomatizes the basic hierarchical properties satisfied by the Gromov modeling trees for the projections of a finite subset $F \subset \calX$ to each of the hyperbolic spaces of the relevant set for $F$. They are the input into the cubical model construction in \cite{Dur_infcube} which we are using in this paper.  The definition \cite[Definition 6.15]{Dur_infcube} is somewhat involved and its content is not relevant to us, so we will provide a rough idea.  Given a finite subset $F \subset \calX$ with $\calU = \Rel_K(F)$ for $K = K(\calX)>0$ sufficiently large, an $R$-\emph{reduced tree system} is a collection of trees $\{T_U\}_{U \in \calU}$ so that each $T_U$ is $C$-median $(C,C)$-quasi-isometric to $\hull_U(F)$, along with a family of projection maps $\delta^V_U:T_V \to T_U$ for $V,U$ not orthogonal which satisfy analogous properties (e.g., a version of the Bounded Geodesic Image axiom) to the projections in an HHS up to some coarseness constant $R$.

Every family of uniform Gromov modeling trees for a finite set $F \subset \calX$ admits many $R$-reduced tree systems for uniform $R$ \cite[Corollary 6.16]{Dur_infcube}, but for our purposes, it will be necessary to use the edge/cluster decomposition associated to its family of stable trees to define the projections.

The following is \cite[Proposition 7.4]{Dur_infcube}:

\begin{proposition}\label{prop:reduced tree system}
    Let $F \subset \calX$ be finite with $\calU = \Rel_K(F)$ for $K=K(\calX)>0$ sufficiently large.  Suppose that for each $U \in \calU$ there is a tree $T_U$ and a $C$-quasi-median $(C,C)$-quasi-isometric embedding $\phi_U:T_U \to \calC(U)$ so that $d^{Haus}_U(\phi_U(T_U)), \hull_U(F))<C$, and so that each $T_U$ admits a decomposition $T_U = T_{e,U} \cup T_{c,U}$ with the following properties:
\begin{enumerate}
\item For each component $D \subset T^c_U$, there is an associated collection of domains $\calU_U(D) \subset \calU$ so that
\begin{enumerate}
    \item For each $V \in \calU$ with $V \nest U$, there exists a unique component $D_V \subset T_{c,U}$ so that $V \in \calU_U(D_V)$,
    \item $V \nest U$ for each $V \in \calU_U(D)$, and
    \item $d^{Haus}_U(\phi_U(D), \bigcup_{V \in \calU_U(D)} \rho^V_U)<C$ with the endpoints of $\phi_U(D)$ being contained in 
    $$\phi^{-1}_U\left(p_{T_U}\left(\{\rho^V_U|V \in \calU\} \cup \pi_U(F)\right)\right).$$
\end{enumerate}
\item For each $f \in F$, there exists a \textbf{marked point} $f_U \in T_U$ and a component $C_f \subset T_{c,U}$ so that $f_U \in C_U$ and $d_U(\phi_U(f_U), f)<C$.
\end{enumerate}

Then there exists $R = R(\calX, \#F, K, C)>0$ so that $\{T_U\}_{U \in \calU}$ is an $R$-reduced tree system with respect to the following projections $\delta^V_U:T_V \to T_U$ for $V,U$ not orthogonal:
\begin{itemize}
    \item \underline{$V \nest U$}: $\delta^V_U = \phi^{-1}_U(p_{D_V}(\rho^V_U))$ where $p_{D_V}$ is closest point projection to $D_V$ in $\calC(V)$;
    \item \underline{$V \pitchfork U$}: $\delta^V_U = \phi^{-1}_U(p_{T_U}(\rho^V_U))$, where $p_{T_U}$ is closest point projection to $\phi_U(T_U)$ in $\calC(U)$;
    \item \underline{$U \nest V$}: $\delta^V_U= \phi^{-1}_U \circ p_{T_U} \circ \rho^V_U \circ \phi_V:T_V \to T_U$.
\end{itemize}
\end{proposition}

The following lemma is the main takeaway from this discussion:

\begin{lemma}\label{lem:stable tree system}
    Given our standing setup from Notation \ref{not:fixed setup, cubes}, there exists $R_0 = R_0(\calX, k)>0$ so that the collection of stable trees $\{T_U\}_{U \in \calU}$ associated to $F \subset \calX$ is a $R_0$-reduced tree systems with respect to the projections $\delta^V_U$ from Proposition \ref{prop:reduced tree system}.
\end{lemma}

\begin{proof}
This is basically an application of Lemmas \ref{lem:cluster lemma} and \ref{lem:stable tree basics} to verify conditions (1) and (2) from the above proposition, so we mostly sketch how this goes.  Items (1a) and (1b) are immediate, and item (1c) follows from item (3) of Lemma \ref{lem:stable tree basics}.  Finally, item (2) of the proposition is immediate from the Definition \ref{defn:stable tree}---that is, every point of $F$ is contained in some cluster $C$---and again item (3) of Lemma \ref{lem:stable tree basics}.  This completes the proof.
\end{proof}

\subsection{Thickenings} \label{subsec:thickenings}

As is evident from Proposition \ref{prop:reduced tree system}, the structure of an $R$-reduced tree system is coarse.  The next step is to remove this coarseness via a process of thickening and collapsing, converting a reduced tree system into a hierarchical family of trees (Definition \ref{defn:HFT} below).  We begin with a discussion of thickenings from \cite{Dur_infcube} and relate it to our current context.

Let $T$ be an tree with a decomposition $T = A \cup B$ into collections of subtrees.  Given $r_1,r_2\geq 0$, we can define a sequence of thickened decompositions $T = \calA_n \cup \calB_n$ as follows:
\begin{itemize}
\item First, take the $r_1$-neighborhoods in $T$ of the components of $B$.  Call the collection of these $\calB_1$ and their complement $T - \calB_1 = \calA_1$.
\item Second, connect any two subtrees in $B_1$ which are within distance $r_2$ of each other in $T$ by the geodesic between them.  Call the resulting collection of subtrees $\calB_2$ and set $\calA_2 = T- \calB_2$.
\item Iterate inductively: Given $\calB_n$, define $\calB_{n+1}$ to be the collection of subtrees in $\calB_n$ along with the geodesic between any pair of them which are $r_2$-close.  And set $\calA_{n+1} = T - \calB_{n+1}$.
\end{itemize}

The following is \cite[Lemma 7.1]{Dur_infcube}:

\begin{lemma}\label{lem:thickenings exist}
Let $T = A \cup B$ be a decomposition of a tree into a pair of forests and suppose that $T$ has branching bounded by $m>0$.  Using the above notation, for any $r_1, r_2\geq 0$ there exists $N = N(r_1,r_2,m)>0$ and $D = D(r_1,r_2,m)>0$  with $D\geq r_1$ so that if $n\geq N$, then 
\begin{enumerate}
\item For any $l\geq n$, we have $\calB_n = \calB_l$ and $\calA_n = \calA_l$;
\item We have $\calB_n \subset \calN_D(B)$ and every pair of distinct components of $\calB_n$ are at least $r_2$-far away;
\item We have $\calN_{r_1}(B) \subset \calB_n$.
\end{enumerate}
\end{lemma}

 \begin{definition}[Tree thickenings]\label{defn:general thickening}
Given a tree $T = A \cup B$ with a decomposition into two forests, a bound $m>0$ on its branching, and $r_1,r_2>0$, the $(r_1,r_2)$-\emph{thickening} of $T$ \emph{along $B$} defines a decomposition $T = \bT_e \cup \bT_c$ where $\bT_e = \calA_N$ and $\bT_c = \calB_N$, for  $N = N(r_1,r_2,m)>0$ as in Lemma \ref{lem:thickenings exist}.
\end{definition}

In our current context, we have two possible decompositions with respect to which we can thicken each $T_U$, which we note all have branching bounded in terms of $k, \mathfrak S$.  For the first decomposition, we could take $B_U$ to be the collection of $\delta^V_U \subset T_U$ subtrees for $V \in \calU_U$ along with the marked points $f_U$ associated to points $f \in F$ on $T_U$, both as defined in Proposition \ref{prop:reduced tree system}.  Alternatively, we could take $B_U$ to be the cluster components $T_{U,c}$ of the stable trees $T_U = T_{U,e} \cup T_{U,c}$ themselves.  The following is \cite[Proposition 7.9]{Dur_infcube}:

\begin{lemma}\label{lem:same thickenings}
    There exists $R_0 = R_0(k, \mathfrak S)>0$ so that if $r_1,r_2> R_0$, then the two $(r_1,r_2)$-thickenings of each $T_U$ with respect to the two above decompositions are exactly the same.
\end{lemma}

Hence we can define:

 \begin{definition}[Thickening the $T_U$]\label{defn:thickening}
For every $U \in \calU$ and $r_1,r_2> R_0 = R_0(k, \mathfrak S)>0$ as in Lemma \ref{lem:same thickenings}, the $(r_1,r_2)$-\textbf{thickening} of $T_U = T_{U,e} \cup T_{U,c}$ \emph{along $T_{U,c}$} defines a decomposition $T_U = \bT_{U,e} \cup \bT_{U,c}$ where $\bT_{U,e} = \calA_N$ and $\bT_{U,c} = \calB_N$, for  $N = N(r_1,r_2,k, \mathfrak S)>0$ as in Lemma \ref{lem:thickenings exist}.
\end{definition}

\begin{figure}
    \centering
    \includegraphics[width=.6\textwidth]{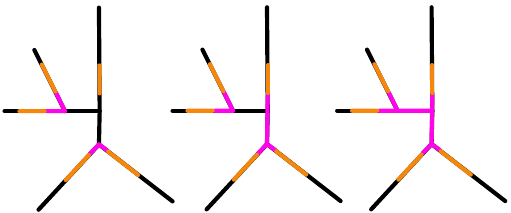}
    \caption{The number of iterations of the $(r_1,r_2)$-thickening process is controlled by the branching of the given tree.  In the schematic, $r_1$-neighborhoods of components are in gold, and only the bottom two components are within $r_2$ of each other.  Each absorption of a new gold component requires a new branch point.  Figure borrowed from \cite{Dur_infcube}.}
    \label{fig:cluster collapse}
\end{figure}

Hence by Lemma \ref{lem:thickenings exist} and Lemma \ref{lem:stable tree system}, we get:

\begin{lemma}\label{lem:stable thickenings}
    For any $U \in \calU$ and $r_1,r_2> R_0= R_0(k, \mathfrak S)>0$ as in Lemma \ref{lem:same thickenings}, the $(r_1,r_2)$-thickening of $T_U = T_{U,e} \cup T_{U,c}$ along $T_{U,c}$ defines a decomposition $T_U = \bT_{U,e} \cup \bT_{U,c}$ satisfying:
    \begin{enumerate}
        \item For any distinct components $C_1, C_2 \subset \bT_{U,c}$, we have $d_{T_U}(C_1,C_2)> r_2$.
        \item Every cluster point in $\calY_U$ and marked point $f_U$ of $T_U$ is contained in some component of $\bT_{U,c}$, and moreover any such cluster or marked points is at least $r_1$-far in $T_U$ from any endpoint of the component containing it.
    \end{enumerate}
\end{lemma}

We remark that we could perform the same set of constructions for the setup associated to  $F'$.  We will proceed using the analogous notation, namely $T'_U = \bT'_{U,c} \cup \bT'_{U,e}$, etc.

\subsection{Collapsed trees and hierarchical families of trees} \label{subsec:HFT}

The next step is to, for each $U \in \calU$, obtain a new tree from $T_U = \bT_{U,e} \cup \bT_{U,c}$ by collapsing the thickened cluster components $\bT_{U,c}$.  We then want to induce new HHS-like projections on these collapsed trees, and then observe (via \cite{Dur_infcube}) that in fact this family of collapsed trees with these projections satisfy HHS-like axioms.

For each $U \in \calU$, define 
$$q_U:T_U \to \hT_U$$
to be the quotient map which collapses each component of $\bT_{c,U}$ to a point.  Thus the resulting object $\hT_U$ consists of the components of $\bT_{U,e}$, with two attached at a point if and only if they were both adjacent to a given thickened cluster component of $\bT_{U,c}$. 

\begin{figure}
    \centering
    \includegraphics[width=.75\textwidth]{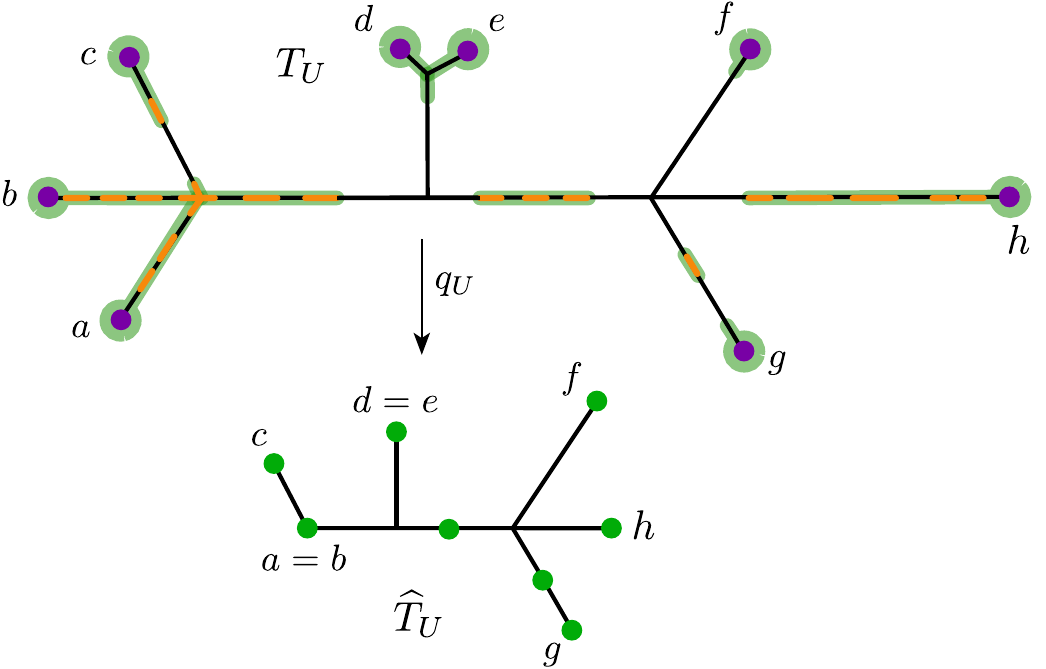}
    \caption{Collapsing the components of $\bT_{U,c}$ (in green) to points can collapse whole branches of $T_U$ and identify to points of $F$, i.e., the marked points of $T_U$.  In the collapsed simplicial tree $\hT_U$, these the cluster components (and hence marked points) become vertices.  This figure is a modified version of \cite[Figure 8]{Dur_infcube}.}
    \label{fig:cluster collapse 2}
\end{figure}

\begin{remark}
    It is possible that the cluster components for a given tree $T_U$ will constitute the whole tree, meaning that $\hT_U$ is a point while $T_U$ had arbitrarily large diameter.  Proving that the collection of collapsed trees still coarsely encodes the geometry of $\hull_{\calX}(F)$ is nontrivial and relates to the ``passing-up'' axiom for HHSs (see \cite[Section 4]{Dur_infcube}).
\end{remark}

Using these quotient maps, we can also induce a family of \emph{collapsed projections} between the $\hT_U$:

\begin{itemize}
    \item \underline{$V \nest U$}: $\hd^V_U = q_U \circ \delta^V_U \in \hT_U$ which is a vertex of $\hT_U$ corresponding to the (collapsed) cluster component of $T_U$ corresponding to $\rho^V_U$;
    \item \underline{$V \pitchfork U$}: $\hd^V_U = q_U(\delta^V_U)$ which is the vertex $\hf_V \in \hT_U$ corresponding to a point $f \in F$ with $d_U(\rho^V_U, f) < E_{\mathfrak S}$.
    \item \underline{$U \nest V$}: $\hd^V_U= q_U \circ \delta^V_U \circ q^{-1}_V: \hT_V \to \hT_U$.
\end{itemize}

The idea now is that these collapsed projections make the collection of collapsed trees $\{\hT_U\}$ into something like an exact HHS, which is the content of the following definition:

\begin{definition}[Hierarchical family of trees] \label{defn:HFT}
A \emph{hierarchical family of trees} is the following collection of objects and properties:

\begin{enumerate}
\item A finite index set $\calU$ of \emph{domains} with \emph{relations} $\nest, \pitchfork, \perp$ so that
\begin{itemize}
    \item $\nest$ is anti-reflexive and anti-symmetric, and gives a partial order;
    \item $\pitchfork$ is symmetric and anti-reflexive;
    \item There is a unique $\nest$-maximal element $S$;
    \item Any pair of domains $U,V \in \calU$ with $U \neq V$ satisfies exactly one of the above relations. 
\end{itemize}
\item To each $U \in \calU$ there is an associated finite simplicial tree $\hT_U$. 
\item A finite collection of labels $F$ and, on each tree, a finite set of \emph{marked points}, which are labeled by $F$, with each element $f \in F$ labeling exactly one marked point $\hf_U$ of $\hT_U$.  Moreover, for each $U \in \calU$, each leaf of $\hT_U$ is a marked point.
\item A family of \emph{relative projections} determining:
\begin{itemize}
    \item For each $U, V \in \calU$ with $V \nest U$ or $V \pitchfork U$, there is a vertex $\hd^V_U \in \hT^{(0)}_U$.  Moreover, each component of $\hT_U - \hd^V_U$ contains a marked point.
    \item If $U,V,W \in \calU$ with $U \perp V$ and $V \nest W$ and either $U \nest W$ or $U \pitchfork W$, then $\hd^U_W = \hd^V_W$.
\end{itemize}
    \item(BGI) There exist projection maps $\hd^U_V: \hT_U \to \hT_V$ when $V \nest U$, satisfying the following bounded geodesic image property:\label{item:BGI HFT}
    \begin{itemize}
        \item If $C \subset \hT_U - \hd^V_U$ is a component, then $\hd^U_V(C)$ coincides with $\hf_V$ for any $f \in F$ with $\hf_U \in C$. 
    \end{itemize}
\end{enumerate}
\end{definition}

The following proposition is a consequence of \cite[Lemma 12.2]{Dur_infcube}:

\begin{proposition}\label{prop:stable tree thickenings}
    There exist $R_1=R_1(\calX, \#F)>0$ and $R_2=R_2(\calX, \#F)>0$ so that for any integers $r_1>\max\{R_0,R_1\}$ and $r_2>\max\{R_0,R_2\}$, the family of collapsed trees $\{\hT_U\}_{U \in \calU}$ equipped with their associated collapsed projections $\hd^V_U$ arising from an $(r_1,r_2)$-thickening of the stable trees $\{T_U\}_{U \in \calU}$ along their cluster components satisfies Definition \ref{defn:HFT}.
\end{proposition}

\begin{remark}\label{rem:size of r}
    In what follows, we are free to choose the constants $r_1,r_2$ to be as large as we need depending only on $\calX$ and $k
    $.  Notably, in order for the machinery from \cite{Dur_infcube} to work, we need to keep $r_1,r_2$ independent of the largeness constant $K$, as it is a global constant. 
\end{remark}

\subsection{Consistent subsets and the hull}\label{subsec:Q defined}

The hierarchical family of trees $\{\hT_U\}_{U \in\calX}$ associated to our base stable tree setup (Subsection \ref{subsec:fixed base setup, cubes}) is an exact model for the hull $\hull_{\calX}(F)$ of our finite subset $F \subset \calX$.  In this subsection, we explain how a certain hierarchically-defined subset $\calQ$ of the product of the collapsed trees is quasi-median quasi-isometric to $\hull_{\calX}(F)$.

Given an HFT $\{\hT_U\}_{U \in \calU}$ with collapsed projections $\hd^V_U$, let
$$\calW = \prod_{U \in \calU} \hT_U.$$
Any point $\hx \in \calW$ is a tuple $\hx = (\hx_U) \in \prod_{U \in \calU} \hT_U$.  As in an HHS, we can ask that its coordinates satisfy exact version of the HHS consistency conditions (see \cite[Proposition 1.8]{HHS_II}), i.e., where there is no coarseness.

\begin{definition}[$0$-consistency and $\calQ$]\label{defn:Q consistent}
A tuple $(\hx_U) \in \calW$ is $0$-\emph{consistent} if we have
\begin{enumerate}
    \item If $U \pitchfork V \in \calU$, then either $\hx_U = \hd^V_U \in \hT_U$ or $\hx_V = \hd^U_V \in \hT_V$.
    \item If $U \nest V \in \calU$, then either $\hx_V = \hd^U_V \in \hT_V$ or $\hx_U \in \hd^V_U(\hx_V) \subset \hT_U$.
\end{enumerate}
\begin{itemize}
\item We set $\calQ$ to be the set of $0$-consistent tuples in $\calW$.
\end{itemize}
\end{definition}

\begin{remark}
    The above definition is essentially \cite[Definition 7.8]{Dur_infcube} with the crucial difference that we have not included any canonicality conditions.  This is because we are working only with finite sets of interior points unlike in the more general setting of \cite{Dur_infcube}, which also handles finite sets of hierarchy rays.
\end{remark}

Since $\calQ$ is a subset of a product of metric spaces, we can give it any number of $\ell^p$-metrics, which are all quasi-isometric.  In proofs, one usually uses the $\ell^1$-metric, in analogy with the HHS distance formula.

Given the HFT $\{\hT_U\}_{U \in \calU}$ arising via Proposition \ref{prop:stable tree thickenings} from our stable tree setup (Subsection \ref{subsec:fixed base setup, cubes}), we next observe that there is a canonical map 
$$\hPsi:\hull_{\calX}(F) \to \calQ$$
which is a quasi-median quasi-isometry by Theorem \ref{thm:HFT model} below.

The map $\hPsi$ is defined coordinate-wise via maps $\hpsi_U:\hull_{\calX}(F) \to \hT_U$ for each $U \in \calU$.  Given $x \in \hull_{\calX}(F)$, we define $\hpsi_U(x) = q_U \circ \phi_{U}^{-1}\circ p_{T_U} \circ \pi_U(x)$, where
\begin{itemize}
    \item $\pi_U:\calX \to \calC(U)$ is the usual subsurface projection;
    \item $p_{T_U}:\calC(U) \to \phi_U(T_U)$ is closest point
    projection in $\calC(U)$ to the image of the stable tree $T_U$ under the model map $\phi_U:T_U \to \calC(U)$ provided by Lemma \ref{lem:stable tree basics};
    \item $q_U:T_U \to \hT_U$ is the quotient map described in Subsection \ref{subsec:HFT} above.
\end{itemize}

Checking that the image of $\hPsi$ lies in $\calQ$ is a definition chase, see \cite[Proposition 7.11]{Dur_infcube}.  The next theorem, which is far more involved, follows from \cite[Theorem 15.23]{Dur_infcube}.

\begin{theorem}\label{thm:HFT model}
    Given our basic setup as in Subsection \ref{subsec:fixed base setup, cubes}, the map $\hPsi:\hull_{\calX}(F) \to \calQ$ is a $C$-quasi-median $(C,C)$-quasi-isometry for $C= C(\calX, \#F)>0$.
\end{theorem}

As it will be useful going forward, we set some notation.  We let $\hO:\calQ \to \hull_{\calX}(F)$ denote the coarse inverse to $\hPsi$, which is $C$-quasi-median $(C,C)$-quasi-isometry.

\begin{remark}
    We note that the $0$-consistent subspace $\calQ$ associated to a general HFT may be empty.  Part of the content of Theorem \ref{thm:HFT model} is that, in our setting, $\calQ$ is non-empty.
\end{remark}

\begin{remark}
    We note that any quasi-median quasi-isometry necessarily has a coarse inverse which is itself a quasi-median quasi-isometry of a comparable quality, the map $\hO$ admits an explicit definition which plays a key role in the results from \cite{Dur_infcube} supporting the discuss and proofs in this section.  However, it will suffice for our current purposes to simply have an abstract coarse inverse.
\end{remark}

\subsection{$\calQ$ is a CAT(0) cube complex}\label{subsec:HFT dual}

Every hierarchical family of trees is simultaneously an HHS-like object and, as a product of simplicial trees, a CAT(0) cube complex.  Being hierarchical, an HFT determines a $0$-consistent subset $\calQ$ (Definition \ref{defn:Q consistent}).  This subset is evidently a cubical subcomplex.  However, we can say more.  We recall the specifics in what follows, see \cite[Sections 13, 14, and 15]{Dur_infcube} for details.

Any product $\prod_i X_i$ of CAT(0) cube complexes admits a cubical structure where the hyperplanes are of the form $\hh_i \times \prod_{i\neq j} X_j$, where $\hh_i$ is a hyperplane in $X_i$.  Given a subspace $W \subset \prod_i X_i$, we can induce a wallspace on $W$ by intersecting these \emph{product hyperplanes} with $W$ and obtain the dual cube complex by applying Sageev's construction \cite{Sageev:cubulation}.

In our context, the components of $\prod_{U \in \calU} \hT_U$ are simplicial trees, so their hyperplanes are midpoints of their edges.  Each such tree hyperplane $h_U \in \hT_U$ determines a $\calQ$-hyperplane $\calQ_{h_u} = \{\hx|\hx_U  = h_U\}$, and a pair of half-spaces corresponding to the fibers of the two half-spaces $T_U - h_U$ under the canonical projection map $\hpi_U:\prod_{U \in \calU} \hT_U \to \hT_U$. Let $\calD(\calQ)$ denote the dual cube complex to the system of walls $\calW(\calQ)$ on $\calQ$ obtained by intersecting the product hyperplanes from $\prod_{U \in \calU} \hT_U$.

This dual complex $\calD(\calQ)$ is a CAT(0) cubical subcomplex of $\prod_{U \in \calU} \hT_U$.  This is because the Sageev consistency conditions for product hyperplanes in $\prod_{U \in \calU} \hT_U$ are determined component-wise, in that it is only possible for two walls in $\prod_{U \in \calU} \hT_U$ to have empty intersection if they are defined by two non-intersecting half--trees of some $\hT_U$.  On the other hand, the Sageev consistency conditions for walls in $\calW(\calQ)$ satisfy both this component-wise consistency and additional hierarchical consistency conditions following from Definition \ref{defn:Q consistent}; see \cite[Section 12]{Dur_infcube}.

Setting $\calQ^0\subset \calQ$ to be the set of tuples of vertices of the $\hT_U$, there is a natural dualization map $\calD:\calQ^0 \to \calD(\calQ)$ defined as follows:
\begin{itemize}
    \item Given a tuple $\hx = (\hx_U) \in \calQ^0$, we set $\calD(\hx)$ to be the set of orientations on walls in $\calW(\calQ)$ where, for all $U\in \calU$ and component hyperplanes $h_U \in \hT_U$, the corresponding $\calQ$-hyperplane $\hh_U$ chooses the $\calQ$-half-space $\calQ_U$ corresponding to the half-space of $\hT_U - \{h_U\}$ containing $\hx_U$.
\end{itemize}

The following is a consequence of \cite[Theorem 13.6 and Proposition 15.19]{Dur_infcube}:

\begin{theorem}\label{thm:Q dual}
    Given our basic setup as in Subsection \ref{subsec:fixed base setup, cubes}, the map $\calD:\calQ \to \calD(\calQ)$ is a $0$-median isometry.  In particular, $\calD$ is surjective.
\end{theorem}

Thus $\calQ$ is actually the full cubical subcomplex $\calD(\calQ)$ and hence CAT(0) (see \cite[Theorem 15.19]{Dur_infcube}):

\begin{theorem}\label{thm:Q CCC}
    $\calQ = \calD(\calQ)$ is a CAT(0) cubical subcomplex of $\prod_{U \in \calU} \hT_U$.
\end{theorem}

In particular, the median structure on $\calQ$ is inherited from the median structure on $\prod_{U \in \calU} \hT_U$; see \cite[Section 15]{Dur_infcube}.

\subsection{Stable decompositions for thickened trees} \label{subsec:stable decomp thick}

Having defined the cubical models $\calQ$ and $\calQ'$ for $\hull_{\calX}(F)$ and $\hull_{\calX}(F')$ respectively above, we now turn towards proving our Stabler Cubulations Theorem \ref{thm:stabler cubulations}.

In the construction of our cubical models via HFTs above, we first thicken our trees $T_U = T_{U,e} \cup T_{U,c}$ along their edge decompositions.  However, we used this edge decomposition in Section \ref{sec:stable trees} to define the stable decompositions that we would now like to deploy.  In order to do so, we observe that the thickening process preserves stable compatibility in the sense of Definition \ref{defn:stable decomp}.

As in our fixed base setup from Subsection \ref{subsec:fixed base setup, cubes}, for each $U \in \calU$, let $T_U = T_{U,e} \cup T_{U,c}$ and $T'_U = T'_{U,e} \cup T'_{U,c}$ denote the edge decompositions for the stable trees.  Let $T_{U,s} \subset T_{U,e}$ and $T'_{U,s} \subset \hull_{T'}(F) \cap T'_{U,e}$ denote the corresponding stable decompositions.

Let $\bT_U = \bT_{U,e} \cup \bT_{U,c}$ and $\bT'_U = \bT'_{U,e} \cup \bT'_{U,c}$ denote the $(r_1,r_2)$-thickening of the stable trees along their cluster components for $r_1,r_2$ positive integers controlled by $\mathfrak S, k$ as in Proposition \ref{prop:stable tree thickenings}.  Recall that $\bT_{U,e} \subset T_{U,e}$ and $\bT'_{U,e} \subset T'_{U,e}$ for each $U \in \calU$.

For each $U \in \calU$, define $\bT_{U,s} = \bT_{U,e} \cap T_{U,s}$ and $\bT'_{U,s} = \bT'_{U,e} \cap T'_{U,s}$, that is, the stable decomposition restricted to the slightly contracted edge components in the thickening.  These need not be stable decompositions, in part because of the boundedly-many non-identical components (as in item \eqref{item:close pairs} of Definition \ref{defn:stable decomp}).  However, the following lemma says that we can perform boundedly-many (in $k,\mathfrak S$) small surgeries to arrange that Definition \ref{defn:stable decomp} be satisfied.

\begin{lemma}\label{lem:thickened stable decomps}
    For any positive integers $r_1>\max\{R_1,R_0\}$ and $r_2>\max\{R_2,R_0\}$ as in Proposition \ref{prop:stable tree thickenings} and each $U \in \calU$, the $(r_1,r_2)$-thickenings $\bT_U = \bT_{U,e} \cup \bT_{U,c}$ and $\bT'_{U} = \bT'_{U,e} \cup \bT'_{U,c}$ admit $\calY_U$-compatible $L'$-stable decompositions for $L'= L'(\mathfrak S, k)$.
\end{lemma}

\begin{proof}

    The thickening process likely affects almost all components of $T_{U,s}$ and $T'_{U,s}$.  We first show that for all but boundedly-many (in $k,\mathfrak S$) such components, the thickenings are exactly the same.  We then explain how to deal with the remaining components.
    
    By item \eqref{item:unstable forests} of Definition \ref{defn:stable decomp}, there are unstable forests $T_{U, \diff} \subset T_U$ and $T'_{U, \diff} \subset T'_U$ which are the union of boundedly-many (in $k, \mathfrak S$) components of $T_{U,e},T_{U,c}$ and $T'_{U,e}, T'_{U,c}$, respectively, so that the complements $T_U - T_{U,\diff}$ and $\hull_{T'_U}(F) - T'_{U,\diff}$ consist of identical components.  For each such identical component $C$, there are boundedly-many components of $T_{U,c}$ and $T'_{U,c}$ which are contained in $C$ and are adjacent to $T_{U,\diff}$ and $T'_{U,\diff}$.  Hence there are boundedly-many edge components of the complements $T_U - T_{U,\diff}$ and $\hull_{T'_U}(F) - T'_{U,\diff}$ where their thickenings may differ.  In particular,  all but boundedly-many (in $k, \mathfrak S$) components of $T_{U,s}$ and $T'_{U,s}$ are identical after an $(r_1,r_2)$-thickening of $T_U$ and $T'_U$ along $T_{U,c}$ and $T'_{U,c}$, respectively.

    For the remaining boundedly-many components where the thickenings may differ, some of them come from identical components of $T_{U,s}$ and $T'_{U,s}$ (item \eqref{item:identical pairs} of Definition \ref{defn:stable decomp}), and other from proximate components (item \eqref{item:close pairs} of Definition \ref{defn:stable decomp}).
    
    Let $E \in \pi_0(T_{U,s})$ be such a component and recall that $E$ is an interval in a component of $T_{U,e}$, and in particular avoids branch points.  The $(r_1,r_2)$-thickening of $T_U$ along $\bT_{U,c}$ might overlap any such component $E \in \pi_0(T_{U,s})$ by at most $D=D(r_1,r_2, k, \mathfrak S)>0$ (as in Lemma \ref{lem:thickenings exist}) into either of its two ends, possibly in a non-integer amount.  A similar statement holds for its paired component $E' = \alpha_U(E) \in \pi_0(T'_{U,s})$.
    
    First, restrict $E$ to $E \cap \bT_{U,e} \cap i_{E,E'}(E' \cap \bT'_{U,e})$ and similarly restrict $E'$ to $E' \cap \bT'_{U,e} \cap i^{-1}_{E,E'}(E \cap \bT_{U,e})$, and observe that these are restrictions are isometric via $i_{E,E'}$ and $L$-close (or possibly identical) in $\calC(U)$ by construction.  Next observe that $E$ and $E'$ are both segments in edges of components of $T_{U,e}$ and $T'_{U,e}$, respectively, and hence these restrictions are also segments.  Finally, further trim these restricted segments at their ends in an identical fashion by segments of length less than $1$ to arrange that they are simplicial intervals.  It is straightforward to check that the resulting collection of components and isometries results in a uniform stable decomposition.  This completes the proof.
\end{proof}

\subsection{Unstable parts of collapsed trees} \label{subsec:unstable parts of collapsed trees}

With the observation in Lemma \ref{lem:thickened stable decomps} about stable decompositions for thickened trees, we can now proceed to use the Stabler Tree Theorem \ref{thm:stable tree} to further refine the collections of collapsed stable trees $\{\hT_U\}_{U \in \calU}$ and $\{\hT_U\}_{U \in \calU'}$ using the stable decompositions provided by the Stabler Tree Theorem \ref{thm:stabler tree}.

Using our fixed base setup from Subsection \ref{subsec:fixed base setup, cubes}, for each $U \in \calU$, we let $\bT_{U,s} \subset \bT_{U,e} \subset T_{U,e}$ and $
\bT'_{U,s} \subset \hull_{T'}(F) \cap \bT'_{U,e} \subset \hull_{T'}(F)  \cap T'_{U,e}$ denote the stable decompositions for the thickened trees $\bT_U,\bT'_U$, respectively, as provided by Lemma \ref{lem:thickened stable decomps}.

The main output of the next definition are the unstable parts of the various trees $\hT_U, \hT'_U$.  These come in two flavors, depending on whether $U \in \calU$ or $U \in \calU'- \calU$.  In the definition, $\hull_{\hT'_U}(F)$ is the hull in $\hT'_U$ of the vertices corresponding to the marked points labeled by $F$.

\begin{definition}\label{defn:stable parts}
    For each $U \in \calU$,
    \begin{itemize}
        \item The \emph{stable parts} of $\hT_U$ and $\hT'_U$ are $\hT_{U,s} = q_U(\bT_{U,s})$ and $\hT'_{U,s} = q'_U(\bT'_{U,s})$, respectively.
        \item The \emph{unstable parts} of $\hT_U$ and $\hT'_U$ are $\hT_{U,u} = \hT_U - \hT'_{U,s}$ and $\hT'_{U,u} = \hull_{\hT'_U}(F) - \hT'_{U,s}$, respectively.
    \end{itemize}
    If $U \in \calU' - \calU$,
    \begin{itemize}
        \item The \emph{unstable part} of $\hT'_U$ is $\hT'_{U,u} = \hull_{\hT'_U}(F)$.
    \end{itemize}
\end{definition}

\begin{remark}\label{rem:F in F'}
    Recall that if $U \in \calU' - \calU$ is not distinguished (Definition \ref{defn:distinguished}), then $\pi_U(F)$ is a point, and hence the hull of the marked points corresponding to $F$ in $\hT'_U$ is just a vertex.  Hence by Proposition \ref{prop:distinguished domains}, for all but boundedly-many $U \in \calU' - \calU$, this hull is just a vertex.
\end{remark}

Our ultimate goal is to collapse the complements of the stable parts of the $\hT_U$ to again obtain a family of simplicial trees.  For this, the following observation is useful:

\begin{lemma}\label{lem:simplicial stable}
    For each $U \in \calU$, each component of $\hT_{U,s}$ is a simplicial interval whose leaves are vertices of the component of $\hT_{U,e}$ containing it, and similarly for $\hT'_{V,s}$ with $V \in \calU'$.
    \begin{itemize}
        \item Moreover, if $U \in \calU' - \calU$, then the hull in $\hT'_{U,0}$ of the marked points corresponding to $F$ is a vertex.
    \end{itemize}
\end{lemma}

\begin{proof}
    This follows essentially directly from Definition \ref{defn:stable decomp} and Lemma \ref{lem:thickened stable decomps}: For $U \in \calU$ the stable decompositions $\bT_{U,s} \subset \bT_{U,e}$ and $\bT'_{U,s} \subset \bT'_{U,e}$ are preserved identically under the collapsing maps $q_U:T_U \to \hT_U$ and $q'_U:T'_U \to \hT'_U$.  This collapsing map collapses simplicial subtrees, whose complementary components are simplicial subtrees.  Hence the quotient $\hT_{U,s}$ satisfies the description in the lemma.  For $U \in \calU' - \calU$, the convex hull in a simplicial tree of a set of vertices is a simplicial subtree, and hence collapsing it results in a simplicial tree.  This completes the proof.
\end{proof}

\begin{remark} \label{rem:collapsed trees}
In a collapsed tree, $\hT_U$, a uniform neighborhood of the each component of $T^c_U$ has been collapsed to point (Lemma \ref{lem:thickenings exist}), with $r_2$-close components combined and collapsed.  Thus $\hT_U$ is a union of the thickened edge components $\bT_{U,e}$ glued at \emph{cluster vertices}, i.e., the vertices corresponding to the collapsed (uniform) neighborhoods of cluster components, with two edge components glued at a cluster vertex exactly when they are both adjacent in $T_U$ to the corresponding component of $T_{U,c}$.  In this way, the stable and unstable parts are subtrees of the edges components $\hT^e_U$.  It is worth noting that if $\hT_{U,u} = \emptyset$ for some $U \in \calU'$, then $\hT_U = \hT_{U,s}$, because every component of $\hT^e_U$ is contained in $\hT_{U,s}$, and thus so are the cluster points.  As we will see next, most domains satisfy this property.
\end{remark}

The following lemma allows us to control the number and diameter of the unstable components of the various trees in our collections.  It is a consequence of the Stable Tree Theorem \ref{thm:stable tree} and our work in Section \ref{sec:controlling domains}:

\begin{lemma}\label{lem:controlling unstable parts}
    There exists $M_0 = M_0(\calX, k)>0$ so that the following hold:
    \begin{enumerate}
        \item We have $\#\{U \in \calU| \hT_{U,u} \neq \emptyset \hspace{.05in} \text{or} \hspace{.05in} \hT'_{U,u} \neq \emptyset\}<M_0$.
        \item For all $U \in \calU$, the number of components of $\hT_{U,u}$ and $\hT'_{U,u}$ are bounded by $M_0$.
        \item For all $U \in \calU$, each component of $\hT_{U,u}$ and $\hT'_{U,u}$ has diameter bounded by $M_0$.
    \end{enumerate}
    \begin{itemize}
        \item In particular, if $U \in \calU$ is neither distinguished nor involved, then $T_U = T_{U,e} \cup T_{U,c}$ and $T'_U = T'_{U,e} \cup T'_{U,c}$ are identical as trees and have identical edge/cluster decompositions.
    \end{itemize}
\end{lemma}

\begin{proof}
    For (1), we observe that by construction of the stable trees, it suffices to show that for all but boundedly-many domains $U \in \calU$, we have that $T_U = T_{U,e} \cup T_{U,c}$ and $T'_U = T'_{U,e} \cup T'_{U,c}$ are identical both as trees and in terms of having identical edge/cluster decompositions.  This happens precisely when 
    \begin{itemize}
        \item[(a)] $\pi_U(F') = \pi_U(F)$, and
        \item[(b)] $\calU_U = \{V \nest U| V \in \calU\}$ and $\calU'_U = \{V \nest U|V \in \calU'\}$ are identical.
    \end{itemize}
    A domain $U \in \calU$ which fails to have property (a) is ``distinguished'' in the sense of Definition \ref{defn:distinguished}, and hence Proposition \ref{prop:distinguished domains} bounds the number of these domains in terms of $\mathfrak S, k$.  On the other hand, a domain which fails to have property (b) is ``involved'' in the sense of Definition \ref{defn:involved}, and hence Proposition \ref{prop:involved domains} bounds the number of such domains in terms of $\mathfrak S, k$.  Hence the number of domains failing (1) is bounded in terms of $\mathfrak S, k$, as required.

    For (2) and (3), if $U \in \calU$, then we are in the boundedly-many cases excluded by (1).  In these cases, the trees $T_U, T'_U$ admit uniformly compatible stable decompositions via Theorem \ref{thm:stabler tree}, which bounds both the number and diameter of the complementary components of $T_{U,e} - T_{U,s}$ and $\hull_{T'_U}(F) \cap T'_{U,e} - T'_{U,s}$ in terms of $\mathfrak S, k$ as well as the number of ``sporadic'' domains nested into $U$ in the sense of Definition \ref{defn:sporadic domains}.  In particular, the bound (in terms of $\mathfrak S, \#F'$) in Proposition \ref{prop:sporadic domains} on sporadic domains is precisely giving us the bound in the inset equation in the statement of Theorem \ref{thm:stabler tree}.  Combining these bounds gives (2) and (3) when $U \in \calU$.

    When $U \in \calU' - \calU$, we want to invoke Proposition \ref{prop:distinguished domains}.  Recall that a domain $U \in \calU' - \calU$ is not distinguished (Definition \ref{defn:distinguished}) if $\pi_U(F)$ is a single point, and that Proposition \ref{prop:distinguished domains} bounds the number of distinguished domains in terms of $\mathfrak S, k$.  Note that by definition of $U \in \calU' - \calU$, that $\diam_U(F) < K$.  Hence combining these, if $U \in \calU' - \calU$ is distinguished, the hull of (the points corresponding to) $F$ in the (collapsed) tree $\hT'_U$ has diameter coarsely bounded by $K$.  Since there is at most one such bounded-diameter component for each of the boundedly-many distinguished domains, this proves (2) and (3) for $U \in \calU' - \calU$, and hence completes the proof of the lemma.
\end{proof}

\subsection{Stabilizing the trees by collapsing the unstable parts} \label{subsec:collapsing unstable parts}

We now explain how to process of collapsing the unstable parts from all of the trees in the families $\{\hT_U\}_{U \in \calU}$ and $\{\hT'_U\}_{U\in \calU'}$ combine at the level of HFTs (Subsection \ref{subsec:HFT}) to give a hyperplane collapsing map in the corresponding CAT(0) cube complexes.

For each $U \in \calU$, let $\Delta_U:\hT_U \to \hT_{U,0}$ and $\Delta'_U:\hT'_U \to \hT'_{U,0}$ collapse each component of the respective unstable parts of $\hT_{U,u} \subset \hT_U$ and $\hT'_{U,u} \subset \hull_{\hT'_U}(F)$ to a point.  These component-wise maps combine to give maps 
$$\Delta_0:\prod_{U \in \calU} \hT_U \to \prod_{U \in \calU} \hT_{U,0} \hspace{.25in} \text{and} \hspace{.25in} \Delta'_0:\prod_{U \in \calU'} \hT'_U \to \prod_{U \in \calU'} \hT'_{U,0}.$$

We need to add one last piece of information to the trees $\hT_{U,0}$.  By carrying the information over via the map $\Delta_U$, each such tree $\hT_{U,0}$ has marked points which are labeled by points of $F$ and various $\hd^V_U$ for $V \pitchfork U$ where $V \in \calU$.  Similarly, most have cluster points which are labeled by domains $V \nest U$ where $V \in \calU$.  Finally, we get new collapsed projections between the collapsed trees by appropriately composing with these collapsing maps:
\begin{itemize}
    \item $U \nest V$ or $U \pitchfork V$, then $(\hd^U_V)_0 = \Delta_V(\hd^V_U)$;
    \item $U \nest V$ then $(\hd^V_U)_0$ is the map $\Delta_U \circ \hd^V_U \circ \Delta^{-1}_V$.
\end{itemize}

The following is now essentially by construction:

\begin{lemma}\label{lem:simplicial sct}
    For each $U \in \calU$, $\hT_{U,0}$ is a simplicial tree with marked and cluster points being vertices, and similarly for $\hT'_{V,0}$ for any $V \in \calU'$.  Moreover $\{\hT_{U,0}\}_{U \in \calU}$ and $\{\hT'_{U,0}\}_{U \in \calU'}$ are HFTs.
\end{lemma}

\begin{proof}
    By construction, any such tree $\hT_{U,0}$ is a union of the uncollapsed stable components of $\hT_U$, attached along common leaves, which are either marked points, cluster points, or points arising from collapsing an unstable part, the latter of which we have labeled now as cluster points.  Each of these stable components is a simplicial tree by Proposition \ref{prop:simplicialization}, and hence $\hT_{U,0}$ satisfies the statement.  The situation is the same for $\hT'_{V,0}$ with $V \in \calU'$, where we simply observe for distinguished domains $V \in \calU' - \calU$ that collapsing subtrees only makes the HFT axioms easier to confirm.

    The last conclusion, that $\{\hT_{U,0}\}_{U \in \calU'}$ and $\{\hT'_{U,0}\}_{U \in \calU'}$ are HFTs, is an exercise in the definitions which we leave to the reader.  This completes the proof.
\end{proof}

\begin{remark}[Unstable components and cluster points] \label{rem:unstable and clusters}
    An unstable component can be adjacent to a cluster point in $\hT_U$, in which case collapsing the unstable component in $\hT_U$ identifies it with that cluster point, though unstable components need not be adjacent to cluster points.
\end{remark}

\subsection{Collapsing maps as hyperplane deletions} \label{subsec:collapse hyperplane deletion}

In this subsection, we observe that these component-wise collapsings give rise to hyperplane deletions at the level of cube complexes.  For this discussion, we are using the notion of a hyperplane deletion map, which is just a ``restriction quotient'' from \cite{CapraceSageev}.

\begin{definition}[Hyperplane deletion]\label{defn:hyperplane deletion}
Let $\calQ$ be a CAT(0) cube complex and let $\calH$ denote the set of hyperplanes in $\calQ$.  Given a subset $\calG \subset \calH$ of hyperplanes in $\calQ$, the \textbf{hyperplane deletion map for $\calG$}, denoted $\Delta_{\calG}:\calQ \to \calD(\calW(\calH - \calG))$ sends each tuple of orientations on the half-spaces associated to the hyperplanes $\calH$ to the unique tuple of identical orientations on the restricted set of half-spaces for the hyperplanes in $\calH'$.
\end{definition}

We record the following straightforward lemma:

\begin{lemma}
For any $L>0$, if $\calQ$ is a CAT(0) cube complex and $\calG \subset \calH$ is a subset of its hyperplanes with $|\calG|<L$, then $\Delta_{\calG}:\calQ \to \calD(\calH -\calG)$ is a $0$-median $(1,L)$-quasi-isometry.
\end{lemma}

We can now state our collapsing proposition:

\begin{proposition}\label{prop:cube collapse}
There exists a constant $C_0 = C_0(\mathfrak S, k)>0$ so that the following hold:
\begin{enumerate}
    \item The maps $\Delta, \Delta'$ restrict to hyperplane deletion maps $\Delta:\calQ \to \calQ_0$ and $\Delta':\calQ' \to \calQ'_0$ which collapse at most $C_0$-many hyperplanes.
    \begin{itemize}
        \item In particular, $\Delta,\Delta'$ are $0$-median $(1,C_0)$-quasi-isometries.
    \end{itemize}
    \item There exist injective $(1,C_0)$-quasi-isometries $\Xi:\calQ_0 \to \calQ$ and $\Xi':\calQ'_0 \to \calQ'$ so that $\Delta_0 \circ \Xi = id_{\calQ_0}$ and $\Delta'_0 \circ \Xi' = id_{\calQ'_0}$.
\end{enumerate}
\end{proposition}

\begin{proof}
Observe that Lemma \ref{lem:controlling unstable parts} implies that the map $\Delta$ collapses boundedly-many bounded-diameter subtrees across all of the $\hT_U$.

  The rest is now essentially a consequence of the Tree Trimming techniques from \cite[Theorems 10.3 and 15.27]{Dur_infcube}.  The first theorem implies that this map is an almost isometry of uniform quality (in $k, \mathfrak S$) with a well-defined inverse, while the second theorem implies that this map is a hyperplane deletion map.  In particular, the constant $C_0$ in this part can be taken to depend only on the constant $L = L(\mathfrak S, k)>0$ from Theorem \ref{thm:stable tree} and the constant $M_0 = M_0(\mathfrak S, k)>0$ from Lemma \ref{lem:controlling unstable parts}.

  The discussion for $\calQ'$ and $\calQ'_0$ is essentially identical. This completes the proof.
\end{proof}

\subsection{The convex embedding $\Phi: \calQ_0 \to \calQ'_0$}\label{subsec:convex embedding}

We are now ready to prove that $\calQ_0$ admits a convex embedding into $\calQ'_0$.

Recall from Proposition \ref{cor:collapsed isometry} that for each $U \in \calU$, we have a convex embedding $\Phi_U:\hT_{U,0} \to \hT'_{U,0}$ which coherently encodes the marked point and cluster data.  This will define our global map on those coordinates.

For domains  $U \in \calU' - \calU$, we let $\hT_{U,0}$ be a point.  In order to define a map $\Phi_U:\hT_{U,0} \to \hT'_{U,0}$ for $U \in \calU' - \calU$, we thus need to be able associate a single point in $\hT'_U$.  This was arranged in Subsection \ref{subsec:unstable parts of collapsed trees}, where our trees $\hT'_{U,0}$ were constructed so that the marked points corresponding to $F$ in $\hT'_{U,0}$ are a single point when $U \in \calU' - \calU$, as described in Lemma \ref{lem:simplicial stable}.  This gives us a well-defined target in the factors corresponding to $U \in \calU' - \calU$.

 Hence by using the coordinate-wise maps $\Phi_U$ for $U \in \calU$ and setting $\Phi_V \equiv (\hf_V)'$ for any $f \in F$ and $V \in \calU' - \calU$, we get a well-defined global map:

\begin{proposition}\label{prop:cubical convex embedding}
    There exist thickening constants $r_1 = r_1(\mathfrak S, k)>0$ and $r_2 = r_2(\mathfrak S,k)$ map $\Phi:\prod_{U \in \calU} \hT_{U,0} \to \prod_{U \in \calU'} \hT'_{U,0}$ is an isometric embedding which restricts to an $\ell^1$-convex embedding $\Phi:\calQ_0 \to \calQ'_0$.
\end{proposition}

\begin{proof}
We have arranged so that the component-wise maps $\Phi_U:\hT_{U,0} \to \hT'_{U,0}$ are all convex embeddings, so the global map $\Phi:\prod_{U \in \calU'} \hT_{U,0} \to \prod_{U \in \calU'} \hT'_{U,0}$ is an isometric embedding.  This uses the fact that $\Phi_U \equiv \hf'_U$ for any $f \in F$ and any $U \in \calU' - \calU$, so in particular the distance in the coordinate $\hT'_{U,0}$ for any pair of points in the image of $\Phi$ is $0$.

The main work of the proof is showing that $\Phi(\calQ_0) \subset \calQ'_0$.  So we first assume this and prove that $\Phi:\calQ_0 \to \calQ'_0$ is a $\ell^1$-convex embedding.

Let $\gamma$ be an $\ell^1$ geodesic in $\calQ'_0$ between points in the image of $\Phi$, namely $\Phi(\calQ_0)$.  By \cite[Lemma 14.7]{Dur_infcube}, the projection $\hpi'_U(\gamma)$ of $\gamma$ to $\hT'_{U,0}$ is an unparameterized geodesic segment in each $\hT'_{U,0}$ for each $U \in \calU'$.  Since $\Phi_U:\hT_{U,0} \to \hT'_{U,0}$ is a convex embedding for each $U \in \calU$ and otherwise $\hpi'_V(\Phi(\calQ_0)) = (\hf_V)'$ for any $f \in F$, it follows that $\hpi'_U(\gamma) \subset \Phi(\calQ_0)$ for each $U \in \calU'$.  Hence $\Phi(\calQ_0)$ is a convex subcomplex.
\medskip

It remains to prove that $\Phi(\calQ_0) \subset \calQ'_0$.  To see this, let $\hx = (\hx_U) \in \calQ_0$ and set $\Phi(\hx) = \hy = (\hy_V)_{V \in \calU'}$, where by definition $\hy_U = \Phi_U(\hx_U)$ for $U \in \calU$ and $\hy_W = (\hf_W)'$ for any (and every) $f \in F$.  We check $0$-consistency of $\hy$ (Definition \ref{defn:Q consistent}) case-wise.

For the notation of the proof, we denote $\hd$-coordinates in $\calQ, \calQ_0, \calQ',$ and $\calQ'_0$ by $\hd^V_U, \left(\hd^V_U\right)_0, \left(\hd^V_U\right)', $ and $\left(\hd^V_U\right)'_0$ respectively.

\medskip
\textbf{\underline{$U \pitchfork V$ for $U, V \in \calU'$}}:  If $U, V \in \calU$, then we are done by item (2b) of Proposition \ref{cor:collapsed isometry}, which implies that $\Phi_U(\Delta_U(\hd^V_U)) = \Delta'_U(\left(\hd^V_U\right)') = \left(\hd^V_U\right)'_0$ and $\Phi_V(\Delta_V((\hd^U_V))) = \Delta'_V(\left(\hd^U_V\right)') = \left(\hd^U_V\right)'_0$.  In particular, since $0$-consistency of $\hx$ implies one of $\hx_U = \left(\hd^V_U\right)_0$ or $\hx_V = \left(\hd^U_V\right)_0$, we must have either $\hy_U = \left(\hd^V_U\right)'_0$ or $\hy_V = \left(\hd^U_V\right)'_0$, respectively, which is what we wanted.

Next suppose both $U,V \in \calU' - \calU$.  
Note that by definition of $\Phi$, we have $\hy_U = \left(\hf_U\right)'_0$ and $\hy_V = \left(\hf_V\right)'_0$ for every $f \in F$.  So it suffices to show that one of $\left(\hf_U\right)'_0 = \left(\hd^V_U\right)'_0$ or $\left(\hf_V\right)'_0 = \left(\hd^U_V\right)'_0$.  By uniform (in $k,\mathfrak S$) consistency of $f$ in the HHS sense, we have that one of $\pi_U(f)$ or $\pi_V(f)$ is uniformly close in $\calC(U)$ to $\rho^V_U$ or $\rho^U_V$, respectively.  Assuming the former without loss of generality, then $\rho^V_U$ is uniformly close (in terms of $\mathfrak S$) to $\hull_U(F)$ in $\calC(U)$, and hence $\left(\delta^V_U\right)' \subset T'_U$ is uniformly close (in terms of $k,\mathfrak S$) to $\hull_{T'_U}(F)$, where $T'_U$ is the stable tree for $F'$ in $\calC(U)$.  On the other hand, since $\diam_U(F)<K$, it follows that $\hull_{T'_U}(F)$ is contained in a uniform neighborhood of $\left(\delta^V_U\right)'$.  Hence choosing the thickening constant $r_1= r_1(\mathfrak S, k)$ sufficiently large guarantees the claim by Lemma \ref{lem:thickenings exist}.

Finally, suppose now that $U \in \calU$ and $V \in \calU' - \calU$, while $\hy'_U \neq (\hd^V_U)'_0$.  Note that it is not possible that all $f\in F$ have $\pi_V(f)$ uniformly far from $\rho^U_V$, for otherwise (HHS) consistency of $f$ would imply that all $\pi_U(f)$ are uniformly close to $\rho^V_U$, contradicting $U\in\mathcal U$. Therefore, there exists $f\in F$ with $\rho^U_V$ uniformly close to $\pi_V(g)$, and by a similar argument in the last paragraph, we can choose sufficiently (but uniformly) large thickening constants to arrange that $\Phi_U(\hx_V) = (\hg_V)'_0 = \left(\hd^U_V\right)'_0$, as required.

\medskip

\textbf{\underline{$U \nest V$ for $U, V \in \calU'$}}:  As above, we are done if $U,V \in \calU$ by Proposition \ref{cor:collapsed isometry},  If $V \in \calU' - \calU$, we are also done by essentially the same argument as for the $U \pitchfork V$ case.

Suppose then that $U \in \calU' - \calU$ and $V \in \calU$.  We claim that $(\hd^U_V)'_0 = (\hf_V)'_0$ for every $f \in F$.  Similar to the above, is suffices to show that $\rho^U_V$ is uniformly close in $\calC(V)$ to $\hull_V(F)$.  But this follows immediately from the Bounded Geodesic Image axiom because $U \in \calU$.

This completes the proof of the proposition.
\end{proof}

\subsection{Maps to $\calX$ and commutativity}

In this subsection, we first explain how to build a map from the refined cubical models $\calQ_0$ and $\calQ'_0$ to the HHS $\calX$, and then why these maps coarsely commute in Proposition \ref{prop:cc comm diagram}.  This will complete the proof of the Stabler Cubulations Theorem \ref{thm:stabler cubulations}.

We have the following collection of maps:
\begin{itemize}
    \item From Subsection \ref{subsec:Q defined}, we have maps $\hO:\calQ \to \calX$ and $\hO': \calQ' \to \calX$ which are median quasi-isometries by Theorem \ref{thm:HFT model} whose coarseness constants depend on $\mathfrak S$ and $k$.
    \item From Subsection \ref{subsec:collapsing unstable parts}, we have maps $\Delta_0:\calQ \to \calQ_0$ and $\Delta'_0: \calQ' \to \calQ'_0$ which collapsed boundedly-many hyperplanes, and in particular are $0$-median almost-isometries by Proposition \ref{prop:cube collapse}, with error controlled in terms of $\mathfrak S$ and $k$.
    \item Finally, Proposition \ref{prop:cube collapse} provides sections $\Xi:\calQ_0 \to \calQ$ and $\Xi':\calQ'_0 \to \calQ'$ so that $\Delta_0 \circ \Xi = id_{\calQ_0}$ and $\Delta'_0 \circ \Xi' = id_{\calQ'_0}$.
\end{itemize}

With these in hand, we obtain maps 
$$\hO_0 = \hO \circ \Xi:\calQ_0 \to \calX \hspace{.5in} \text{and} \hspace{.5in} \hO'_0= \hO' \circ \Xi':\calQ'_0 \to \calX.$$ By Theorem \ref{thm:HFT model} and Proposition \ref{prop:cube collapse}, these are median quasi-isometric embeddings of quality bounded in terms of $\mathfrak S$ and $k$.

Moreover, our particular setup provides the following commutativity statement:

\begin{lemma}\label{lem:commute collapse}
    There exists $D_1 = D_1(\calX, |F'|)$ so that:
    \begin{enumerate}
        \item For any $\hx \in \calQ$, we have $d_{\calX}(\hO(\hx), \hO_0(\Delta(\hx)))<D_1$.
        \item For any $\hx' \in \calQ'$, we have $d_{\calX}(\hO'(\hx'), \hO'_0(\Delta'(\hx')))<D_1$.
    \end{enumerate}
\end{lemma}

\begin{proof}

We only prove the version for $\calQ, \calQ_0$, since the statement for $\calQ', \calQ'_0$ has essentially the same proof.  

If $\hy \in \calQ$, then $\hy \in \Delta_0^{-1}(\Delta_0(\hy))$ and  $\Xi(\Delta(\hy)) \in \Delta_0^{-1}(\Delta_0(\hy))$, while $\Delta_0(\hy) = \Delta_0(\Xi(\Delta_0(\hy)) \in \calQ_0$; here $\Delta_0^{-1}$ denotes the preimage of $\Delta_0$, as opposed to the section $\Xi$.  Since $\Delta_0$ is a uniform almost-isometry, we must have that $\diam_{\calQ}(\Delta_0^{-1}(\hy))$ is bounded, and hence $d_{\calQ}(\hy, \Xi(\Delta_0(\hy))$ is bounded.  But now $\hO:\calQ \to \calX$ is a quasi-isometric embedding with bounded constants, and hence the distance in $\calX$ between $\hO(\hy)$ and $\hO(\Xi(\Delta_0(\hy))) = \hO_0(\Delta_0(\hy))$ is bounded in terms of $k,\mathfrak S$.  This completes the proof.
\end{proof}

With this lemma in hand, we can now prove our main commutativity statement and thus complete the proof of Theorem \ref{thm:stabler cubulations}:

\begin{proposition}\label{prop:cc comm diagram}
There exists a constant $B = B(\calX, |F'|)>0$ so that the diagram
     \begin{equation}\label{Phi diagram}
  \begin{tikzcd}
   \calQ \arrow[ddrr,"\hO", bend left=40] \arrow[dr,"\Delta \hspace{.075in}" left] &  \\
    &\calQ_0 \arrow[dr,"\hO_0 \hspace{.2in}" below]\arrow[dd, "\Phi"] \\
    & & \calX\\
    & \calQ'_0 \arrow[ur,"\hO'_0"] \\
    \calQ'\arrow[uurr,"\hspace{.2in} \vspace{.1in} \hO'" below, bend right=40] \arrow[ur,"\Delta'"] & \\
  \end{tikzcd}
  \end{equation}
commutes up to error $B$. Moreover, $\Phi$ is an isomorphism if $d_{Haus}(F,F')\leq 1$.
\end{proposition}

\begin{notation}\label{not:close}
    In the proof, we will make a number of coarse comparisons with the coarseness only depending on $k,\mathfrak S$ (Subsection \ref{subsec:fixed base setup, cubes}).  To improve readability, any use of the term \emph{uniform} will indicate errors depending only on $k, \mathfrak S$.  Moreover, when two points $a,b \in X$ of a metric space are uniformly close, we will write $a \sim_X b$, to indicate their uniform proximity in $X$.  We will never be stringing together more than boundedly-many (in $k,\mathfrak S$) instances of this $\sim$ notation or the word ``uniform'', so the errors will not build up.
\end{notation}

\begin{proof}
    Coarse commutativity of the top and bottom triangles is by Lemma \ref{lem:commute collapse}.  It thus suffices to prove coarse commutativity of the middle triangle.

    Let $H = \hull_{\calX}(F) \subset \calX$ and $H' = \hull_{\calX}(F') \subset \calX$.  Next observe that Theorem \ref{thm:HFT model} and Proposition \ref{prop:cube collapse} imply that $\Delta \circ \hPsi:H \to \calQ_0$ and $\hO_0:\calQ_0 \to H$ are uniform quasi-inverses, and similarly for $\Delta' \circ \hPsi':H' \to \calQ'_0$ and $\hO'_0:\calQ'_0 \to H'$.

Thus to establish commutativity of the middle triangle, we must prove that for any $\hx \in \calQ_0$,  we have $\hO_0(\hx) \sim_{\calX} \hO'_0 \circ \Phi(\hx)$ (recall Notation \ref{not:close}).

\medskip
We claim that it suffices to prove that if $x \in H$, then
    \begin{equation}\label{eq:commute}
        \Delta' \circ \hPsi'(x) \sim_{\calQ'_0} \Phi \circ \Delta \circ \hPsi(x).
    \end{equation}

We first prove the proposition assuming \eqref{eq:commute}, then establish \eqref{eq:commute}.

\medskip

First, observe that $\hO_0(\hx) \in H$ by definition.  Coarse commutativity of the upper triangle implies that 
$$\Delta \circ \Psi \circ \hO_0(\hx) \sim_{\calQ_0} \hx.$$
Hence applying $\Phi$ to both sides and using the fact that it is an isometric embedding, we see that 
$$\Phi(\hx) \sim_{\calQ'_0} \Phi \circ \Delta \circ \Psi \circ \hO_0 (\hx).$$  But by \eqref{eq:commute} and the fact that $\hO_0(\hx) \in H$, we have that 
$$\Phi \circ \Delta \circ \Psi \circ \hO_0 (\hx) \sim_{\calQ'_0} \Delta' \circ \Psi' \circ \hO_0(\hx).$$
Now applying $\hO'_0$ to both sides and using the fact that it is a uniform quasi-isometry, we see that
$$\hO'_0(\Phi(\hx)) \sim_{\calX} \hO'_0 \circ \Delta' \circ \Psi' \circ \hO_0(\hx).$$
However, the composition $\hO'_0 \circ \Delta'$ is uniformly close to the identity on $H'$, and hence 
$$\hO'_0 \circ \Delta' \circ \Psi' \circ \hO_0(\hx) \sim_{\calX} \hO_0(\hx).$$

Combining these last two observations, we see that $\hO_0(\hx) \sim_{\calX} \hO'_0 \circ \Phi(\hx)$, which is what we needed to prove.

\medskip

We now set about establishing \eqref{eq:commute}.  Set $\Delta' \circ \hPsi'(x) = \hx'_0 = (\hx'_{U,0})$ and $\Phi \circ \Delta \circ \hPsi(x) = \hy'_0= (\hy'_{U,0})$, where we note that both $\hx'_0,\hy'_0 \in \calQ'_0$.    We argue component-wise that $\hx'_0 \sim_{\calQ'_0} \hy'_0$ with respect to the $\ell^1$-metric on $\calQ'_0$.

Recall that the map $\hPsi:H \to \calQ$ (Subsection \ref{subsec:Q defined}) is defined component-wise (over the set $\calU$) by $\hpsi_U = q_U \circ \phi_U^{-1} \circ p_U \circ \pi_U$, where 
\begin{itemize}
    \item $\pi_U:\calX \to \calC(U)$ is the HHS projection,
    \item $p_U:\calC(U) \to \phi_U(T_U)$ is closest point projection,
    \item $\phi_U:T_U \to \calC(U)$ is the $(L_0,L_0)$-quasi-isometric embedding of the stable tree $T_U$ provided by Lemma \ref{lem:stable tree basics} (here $L_0=L_0(\mathfrak S, k)$), and
    \item $q_U:T_U \to \hT_U$ is the map which collapses the $(r_1,r_2)$-thickening of $T_U$ along the components of $T_{U,c}$, with constants $r_1,r_2$ depending only on $k, \mathfrak S$ (Definition \ref{defn:thickening}).
\end{itemize}

The map $\hPsi':H' \to \calQ'$ is defined analogously.

\medskip

\textbf{\underline{Reducing to distinguished and involved components}}: We first show how to exclude domains in $\calU' - \calU$ and non-distinguished and non-involved domains in $\calU$, leaving us with the uniformly boundedly-many  distinguished domains in $\calU$ (Proposition \ref{prop:distinguished domains}).

First, suppose that $U \in \calU' - \calU$.  Then $\pi_U(x)$ is uniformly close to $\hull_U(F)$, implying that its image in $T_U$ and $T'_U$ via the compositions $(\phi)^{-1}_U \circ p_U \circ \pi_U(x)$ and $(\phi')^{-1}_U \circ p'_U \circ \pi_U(x)$ are uniformly close to $\hull_{T_U}(F) = T_U$ and $\hull_{T'_U}(F)$, respectively, where $T_U$ has uniformly bounded diameter.  Hence a uniform $(r_1,r_2)$-thickening of $T_U$ and $T'_U$ will result in $\hpsi_(x) = \hpsi_U(F) = \hT_U$ and $\hpsi'_U(x) = \hpsi'_U(F) \in \hT'_U$, respectively, where $\hpsi'_U(F)$ is just a point by Lemma \ref{lem:simplicial stable}.  Since $\Phi_U(\hpsi_(x)) = (f_U)'_0$ by definition (Subsection \ref{subsec:convex embedding}, we have $\hx'_{U,0} = \hy'_{U,0} = \hf'_{U,0}$ for any $f \in F$.  Hence we may ignore domains in $\calU' - \calU$.

Next observe that by Lemma \ref{lem:controlling unstable parts}, if $U \in \calU$ is neither distinguished (Definition \ref{defn:distinguished}) nor involved (Definition \ref{defn:involved}), then the stable trees $T_U, T'_U$ for the two setups are exactly the same, and hence have identical thickenings, and thus the stable decompositions for the pair are the (identical) collection of edge components for each, and their corresponding collapsed trees $\hT_U = \hT_{U,0}$ and $\hT'_U = \hT'_{U,0}$ are identical.  Hence, for such $U$, we have $\hx'_{U,0} = \hy'_{U,0}$.  Finally, by Propositions \ref{prop:distinguished domains} and \ref{prop:involved domains}, all but uniformly boundedly-many domains in $\calU$ are neither distinguished nor involved.  Thus it suffices to bound the diameter $d_{\hT'_{U,0}}(\hx'_{U,0}, \hy'_{U,0})$ for the remaining boundedly-many domains $U$ in terms of $\mathfrak S, k$.

\medskip

\textbf{\underline{Bounding distances in distinguished and involved components}}:  Having reduced to considering components corresponding to the boundedly-many distinguished and involved domains $U \in \calU$, the rough idea of the rest of the proof is that if $\hx'_{U,0}$ and $\hy'_{U,0}$ are very far apart, then we can detect this by common stable components of the trees $T_U, T'_U$.  But the definitions of $\hx'_{U,0}$ and $\hy'_{U,0}$ from our initial point $x$ both factor through closest point projection to the trees $T_U$ and $T'_U$, respectively.  This observation, plus a case analysis, will allow us to produce a contradiction.  We note that the definitions of distinguished and involved are not used in this part of the proof.

\smallskip
Let $U \in \calU$ and suppose that $\hx'_{U,0} \neq \hy'_{U,0}$.  Note that $\phi_U(T_U)$ is within a uniform neighborhood of $\phi'_U(T'_U)$, and hence the closest point projections $p_U(\pi_U(x))$ and $p'_U(\pi_U(x))$ of $\pi_U(x)$ to $\phi_U(T_U)$ and $\phi'_U(T'_U)$ are uniformly close as well.  We now want to use the assumption that $d_{\hT'_{U,0}}(\hx'_{U,0}, \hy'_{U,0})$ is large to derive a contradiction of this fact.

The distance $d_{\hT'_{U,0}}(\hx'_{U,0}, \hy'_{U,0})$ is measured in the edge components of $\hT'_{U,0}$ which separate them, and these edge components are stable components of $T'_U$.  This means that if the distance $d_{\hT'_{U,0}}(\hx'_{U,0}, \hy'_{U,0})$ is very large, then there must be either 
\begin{enumerate}
\item some long stable component $E' \subset T'_{U,s}$ so that some long segment of $\Delta'_U(q'_U(E'))$ separates $\hx'_{U,0}$ from $\hy'_{U,0}$, or
\item if no such long component exists, there must be some large number of shorter stable components separating them.
\end{enumerate}

In both cases, we will argue that large separation in $\hT'_{U,0}$ by either a long segment or a long sequence of short segments forces $p_U(\pi_U(x))$ and $p'_U(\pi_U(x))$ to be far apart.

In case (1), let $E \subset T_{U,s}$ be the stable component paired with $E' \subset T'_{U,s}$.  Note that $\phi_U(E) \cup \phi'_U(E')$ is uniformly quasi-convex in $\calC(U)$, and hence $\pi_U(x)$ has a uniformly bounded projection to it.  However the fact that a long segment of $\Delta'_U(q_U(E')) \subset \hT'_{U,0}$ separates $\hx'_{U,0}$ from $\hy'_{U,0}$ means that some large diameter portion of $\phi_U(E) \cup \phi'_U(E')$ uniformly coarsely separates $p_U(\pi_U(x))$ from $p'_U(\pi_U(x))$, which is impossible.

In case (2), we can assume that no such long stable component exists, so that the (large) distance $d_{\hT'_{U,0}}(\hx'_{U,0}, \hy'_{U,0})$ corresponds to some sequence of segments $E'_1, \dots, E'_n \subset T'_{U,s}$ in stable components which separate $\hx'_{U,0}$ from $\hy'_{U,0}$ in $\hT'_{U,0}$, appearing in that order along the geodesic between $\hx'_{U,0}$ and $\hy'_{U,0}$ in $\hT'_{U,0}$.  

Let $E_i \subset T_{U,s}$ denote the segment identified with $E'_i$ in the stable component identified with the component containing $E'_i$.   Since each of the $E'_i$ are uniformly close (in terms of $\mathfrak S,k$) to $E_i$ and the bijection $\alpha:\pi_0(T_{U,s}) \to \pi_0(T'_{U,s})$ is adjacency preserving (Definition \ref{defn:stable decomp} item \eqref{item:Adjacency-preserving}), it follows that the hulls $\hull_{\phi_U(T_U)}(E_1, E_n)$ and $\hull_{\phi'_U(T'_U)}(E'_1,E'_n)$ in $\phi_U(T_U)$ and $\phi'_U(T'_U)$ respectively are within uniform Hausdorff distance in $\calC(U)$ of each other, and hence their union is uniformly quasi-convex.  Thus, if $\Delta'_U(q'_U(E'_1)), \dots, \Delta'_U(q'_U(E'_n))$ separate $\hx'_{U,0}$ from $\hy'_{U,0}$ in $\hT'_{U,0}$, then the union of these hulls coarsely separates $p_U(\pi_U(x))$ from $p'_U(\pi_U(x))$, which by the above are bounded distance.

Now, since we are not in case (1), we may assume that each of the segments $E'_i$ (and hence $E_i$) have length bounded above in terms of $\mathfrak S, k$.  Each adjacent pair $E'_i, E'_{i+1}$ is separated by an unstable component which is either one of the uniformly boundedly-many uniformly bounded diameter unstable components which are not associated to some cluster component for domains in $\calU$ (item \eqref{item:unstable components} of Definition \ref{defn:stable decomp}), or by an unstable component containing some cluster component for domains in $\calU$ which are common to both $T_U$ and $T'_U$ (i.e., to the setups for both $F,F'$ in $\calC(U)$).  Moreover, there are only uniformly boundedly-many unstable components that do not coincide with cluster components because of the bound on the number of unstable components in item \eqref{item:unstable components} of Definition \ref{defn:stable decomp}.  Thus any such long sequence $E'_i$ (i.e., with $n$ very large) contains a long subsequence of cluster components corresponding to both $F,F'$.

Since for each $U \in \calU$, all but boundedly-many clusters are bivalent by item (3) of Lemma \ref{lem:cluster lemma}, item (6) of that same lemma implies that the distance in $\calC(U)$ between $\phi'_U(E'_1)$ and $\phi'_U(E'_n)$ is bounded below by a function of $n, \mathfrak S, k$.  But that distance uniformly coarsely bounds from below the distance between $p_U(\pi_U(x))$ from $p'_U(\pi_U(x))$, which is uniformly bounded.  Hence $n$ is uniformly bounded, as required for case (2).  This completes the proof of the proposition.

\end{proof}

\bibliographystyle{alpha}
\bibliography{asymp}

\newcommand{\etalchar}[1]{$^{#1}$}
\begin{thebibliography}{CCG{\etalchar{+}}25b}

\bibitem[ABD17]{ABD}
Carolyn Abbott, Jason Behrstock, and Matthew~Gentry Durham.
\newblock Largest acylindrical actions and stability in hierarchically hyperbolic groups.
\newblock {\em arXiv preprint arXiv:1705.06219}, 2017.

\bibitem[ABM{\etalchar{+}}25]{Random_quot}
C.~Abbott, D.~Berlyne, G.~Mangioni, T.~Ng, and A.~Rasmussen.
\newblock {S}pinning and random quotients preserve hierarchical hyperbolicity.
\newblock {\em In preparation}, 2025.

\bibitem[Bar18]{Bartels:ICM}
Arthur Bartels.
\newblock {$K$}-theory and actions on {E}uclidean retracts.
\newblock In {\em Proceedings of the {I}nternational {C}ongress of {M}athematicians---{R}io de {J}aneiro 2018. {V}ol. {II}. {I}nvited lectures}, pages 1041--1062. World Sci. Publ., Hackensack, NJ, 2018.

\bibitem[BB97]{bestvina1997morse}
Mladen Bestvina and Noel Brady.
\newblock Morse theory and finiteness properties of groups.
\newblock {\em Inventiones mathematicae}, 129:445--470, 1997.

\bibitem[BB19]{BB:FJ_MCG}
Arthur Bartels and Mladen Bestvina.
\newblock The {F}arrell-{J}ones conjecture for mapping class groups.
\newblock {\em Invent. Math.}, 215(2):651--712, 2019.

\bibitem[BBF15]{BBF}
Mladen Bestvina, Ken Bromberg, and Koji Fujiwara.
\newblock Constructing group actions on quasi-trees and applications to mapping class groups.
\newblock {\em Publ. Math. Inst. Hautes \'{E}tudes Sci.}, 122:1--64, 2015.

\bibitem[BBFS20]{BBFS}
Mladen Bestvina, Ken Bromberg, Koji Fujiwara, and Alessandro Sisto.
\newblock Acylindrical actions on projection complexes.
\newblock {\em Enseign. Math.}, 65(1-2):1--32, 2020.

\bibitem[BD80]{BorDyd}
Karol Borsuk and Jerzy Dydak.
\newblock What is the theory of shape?
\newblock {\em Bulletin of the Australian Mathematical Society}, 22(2):161--198, 1980.

\bibitem[BDS11]{behrstock2011median}
Jason Behrstock, Cornelia Dru{\c{t}}u, and Mark Sapir.
\newblock Median structures on asymptotic cones and homomorphisms into mapping class groups.
\newblock {\em Proceedings of the London Mathematical Society}, 102(3):503--554, 2011.

\bibitem[Bes96]{bestvina1996local}
Mladen Bestvina.
\newblock Local homology properties of boundaries of groups.
\newblock {\em Michigan Mathematical Journal}, 43(1):123--139, 1996.

\bibitem[BF08]{BellFuj_asdim}
Gregory~C Bell and Koji Fujiwara.
\newblock The asymptotic dimension of a curve graph is finite.
\newblock {\em Journal of the London Mathematical Society}, 77(1):33--50, 2008.

\bibitem[BFL14]{BFL:lattices}
A.~Bartels, F.~T. Farrell, and W.~L\"{u}ck.
\newblock The {F}arrell-{J}ones conjecture for cocompact lattices in virtually connected {L}ie groups.
\newblock {\em J. Amer. Math. Soc.}, 27(2):339--388, 2014.

\bibitem[BH99]{bridson-haefliger}
M.~R. Bridson and A.~Haefliger.
\newblock {\em Metric spaces of non-positive curvature}, volume 319 of {\em Grundlehren der Mathematischen Wissenschaften [Fundamental Principles of Mathematical Sciences]}.
\newblock Springer-Verlag, Berlin, 1999.

\bibitem[BHMS24]{BHMS:kill_twists}
Jason Behrstock, Mark Hagen, Alexandre Martin, and Alessandro Sisto.
\newblock A combinatorial take on hierarchical hyperbolicity and applications to quotients of mapping class groups.
\newblock {\em J. Topol.}, 17(3):Paper No. e12351, 94, 2024.

\bibitem[BHS17]{HHS:asdim}
Jason Behrstock, Mark~F. Hagen, and Alessandro Sisto.
\newblock Asymptotic dimension and small-cancellation for hierarchically hyperbolic spaces and groups.
\newblock {\em Proc. Lond. Math. Soc. (3)}, 114(5):890--926, 2017.

\bibitem[BHS19]{HHS_II}
Jason Behrstock, Mark Hagen, and Alessandro Sisto.
\newblock {Hierarchically hyperbolic spaces II: Combination theorems and the distance formula}.
\newblock {\em Pacific Journal of Mathematics}, 299(2):257--338, 2019.

\bibitem[BHS20]{BHS:quasi}
Jason Behrstock, Mark~F Hagen, and Alessandro Sisto.
\newblock Quasiflats in hierarchically hyperbolic spaces.
\newblock {\em Duke Mathematical Journal}, 2020.
\newblock to appear.

\bibitem[BKW13]{bux2013higher}
Kai-Uwe Bux, Ralf K{\"o}hl, and Stefan Witzel.
\newblock Higher finiteness properties of reductive arithmetic groups in positive characteristic: the rank theorem.
\newblock {\em Annals of Mathematics}, pages 311--366, 2013.

\bibitem[BL12]{BL:hyp_CAT0}
Arthur Bartels and Wolfgang L\"{u}ck.
\newblock The {B}orel conjecture for hyperbolic and {${\rm CAT}(0)$}-groups.
\newblock {\em Ann. of Math. (2)}, 175(2):631--689, 2012.

\bibitem[BLR08]{BLR:hyperbolic}
Arthur Bartels, Wolfgang L\"{u}ck, and Holger Reich.
\newblock The {$K$}-theoretic {F}arrell-{J}ones conjecture for hyperbolic groups.
\newblock {\em Invent. Math.}, 172(1):29--70, 2008.

\bibitem[BM91]{BestMess}
Mladen Bestvina and Geoffrey Mess.
\newblock The boundary of negatively curved groups.
\newblock {\em Journal of the American Mathematical Society}, 4(3):469--481, 1991.

\bibitem[BM00]{burger2000lattices}
Marc Burger and Shahar Mozes.
\newblock Lattices in product of trees.
\newblock {\em Publications Math{\'e}matiques de l'IH{\'E}S}, 92:151--194, 2000.

\bibitem[Bow16]{Bow:CAT0_cone}
Brian~H. Bowditch.
\newblock Some properties of median metric spaces.
\newblock {\em Groups Geom. Dyn.}, 10(1):279--317, 2016.

\bibitem[Bow18]{Bow:cubulation}
Brian Bowditch.
\newblock Convex hulls in coarse median spaces.
\newblock {\em Preprint}, 2018.

\bibitem[BR20]{BR:combination}
Federico Berlai and Bruno Robbio.
\newblock A refined combination theorem for hierarchically hyperbolic groups.
\newblock {\em Groups Geom. Dyn.}, 14(4):1127--1203, 2020.

\bibitem[BR22]{BR:graph_prods}
Daniel Berlyne and Jacob Russell.
\newblock Hierarchical hyperbolicity of graph products.
\newblock {\em Groups Geom. Dyn.}, 16(2):523--580, 2022.

\bibitem[BS07]{buyalo2007elements}
Sergei Buyalo and Viktor Schroeder.
\newblock {\em Elements of asymptotic geometry}, volume~3.
\newblock European Mathematical Society, 2007.

\bibitem[CCG{\etalchar{+}}25a]{CCGHO}
J{\'e}r{\'e}mie Chalopin, Victor Chepoi, Anthony Genevois, Hiroshi Hirai, and Damian Osajda.
\newblock Helly groups.
\newblock {\em Geometry \& Topology}, 29(1):1--70, 2025.

\bibitem[CCG{\etalchar{+}}25b]{Helly}
J\'er\'emie Chalopin, Victor Chepoi, Anthony Genevois, Hiroshi Hirai, and Damian Osajda.
\newblock Helly groups.
\newblock {\em Geom. Topol.}, 29(1):1--70, 2025.

\bibitem[CK00]{CrokeKleiner}
Christopher~B Croke and Bruce Kleiner.
\newblock Spaces with nonpositive curvature and their ideal boundaries.
\newblock {\em Topology}, 39(3):549--556, 2000.

\bibitem[Cor17]{Cordes:boundary}
Matthew Cordes.
\newblock Morse boundaries of proper geodesic metric spaces.
\newblock {\em Groups Geom. Dyn.}, 11(4):1281--1306, 2017.

\bibitem[CRHK21]{CRHK}
Montserrat Casals-Ruiz, Mark Hagen, and Ilya Kazachkov.
\newblock Real cubings and asymptotic cones of hierarchically hyperbolic groups.
\newblock {\em preprint}, 2021.

\bibitem[CS11]{CapraceSageev}
Pierre-Emmanuel Caprace and Michah Sageev.
\newblock Rank rigidity for cat (0) cube complexes.
\newblock {\em Geometric and functional analysis}, 21(4):851--891, 2011.

\bibitem[Dah03]{dahmani2003classifying}
Fran{\c{c}}ois Dahmani.
\newblock Classifying spaces and boundaries for relatively hyperbolic groups.
\newblock {\em Proceedings of the London Mathematical Society}, 86(3):666--684, 2003.

\bibitem[Dan25]{danielski2025boundaries}
Daniel Danielski.
\newblock On boundaries of bicombable spaces.
\newblock {\em arXiv preprint arXiv:2503.06673}, 2025.

\bibitem[DDLS24]{DDLS}
Spencer Dowdall, Matthew~G. Durham, Christopher~J. Leininger, and Alessandro Sisto.
\newblock Extensions of {V}eech groups {II}: {H}ierarchical hyperbolicity and quasi-isometric rigidity.
\newblock {\em Comment. Math. Helv.}, 99(1):149--228, 2024.

\bibitem[DL16]{DL:flat_torus}
Dominic Descombes and Urs Lang.
\newblock Flats in spaces with convex geodesic bicombings.
\newblock {\em Anal. Geom. Metr. Spaces}, 4(1):68--84, 2016.

\bibitem[DMS20]{DMS_bary}
Matthew~G Durham, Yair~N Minsky, and Alessandro Sisto.
\newblock Stable cubulations, bicombings, and barycenters.
\newblock {\em arXiv preprint arXiv:2009.13647}, 2020.

\bibitem[Dra06]{Dra:BM_formula}
A.~N. Dranishnikov.
\newblock On {B}estvina-{M}ess formula.
\newblock In {\em Topological and asymptotic aspects of group theory}, volume 394 of {\em Contemp. Math.}, pages 77--85. Amer. Math. Soc., Providence, RI, 2006.

\bibitem[Dur23]{Dur_infcube}
Matthew~Gentry Durham.
\newblock Cubulating infinity in hierarchically hyperbolic spaces.
\newblock {\em arXiv preprint arXiv:2308.13689}, 2023.

\bibitem[Dur24]{Dur_stableinterval}
Matthew~Gentry Durham.
\newblock Bicombing the mapping class group via stable cubical intervals.
\newblock {\em \url{https://www.dropbox.com/scl/fi/ciqy4pj0csuo6xna3301s/Durham-Stable-intervals.pdf?rlkey=j0kc92exxo01bic3uolndlic4&dl=0}}, 2024.

\bibitem[Eng78]{dimension_theory}
Ryszard Engelking.
\newblock {\em Dimension theory}, volume~19 of {\em North-Holland Mathematical Library}.
\newblock North-Holland Publishing Co., Amsterdam-Oxford-New York; PWN---Polish Scientific Publishers, Warsaw, 1978.
\newblock Translated from the Polish and revised by the author.

\bibitem[FJ93]{farrell1993isomorphism}
F~Thomas Farrell and Lowell~E Jones.
\newblock Isomorphism conjectures in algebraic k-theory.
\newblock {\em Journal of the American Mathematical Society}, 6(2):249--297, 1993.

\bibitem[FL05]{FL:EZ}
F.~T. Farrell and J.-F. Lafont.
\newblock E{Z}-structures and topological applications.
\newblock {\em Comment. Math. Helv.}, 80(1):103--121, 2005.

\bibitem[GM19]{GuilMor}
Craig~R Guilbault and Molly~A Moran.
\newblock Proper homotopy types and z-boundaries of spaces admitting geometric group actions.
\newblock {\em Expositiones Mathematicae}, 37(3):292--313, 2019.

\bibitem[GMT19]{guilbault2019boundaries}
Craig Guilbault, Molly Moran, and Carrie Tirel.
\newblock Boundaries of baumslag--solitar groups.
\newblock {\em Algebraic \& Geometric Topology}, 19(4):2077--2097, 2019.

\bibitem[Gro87]{Gromov87}
Mikhael Gromov.
\newblock Hyperbolic groups.
\newblock In {\em Essays in group theory}, pages 75--263. Springer, 1987.

\bibitem[Hag23]{Hagen:non-col}
Mark Hagen.
\newblock Non-colorable hierarchically hyperbolic groups.
\newblock {\em Internat. J. Algebra Comput.}, 33(2):337--350, 2023.

\bibitem[Han51]{polyhedron_anr}
Olof Hanner.
\newblock Some theorems on absolute neighborhood retracts.
\newblock {\em Ark. Mat.}, 1:389--408, 1951.

\bibitem[HHP23]{HHP}
Thomas Haettel, Nima Hoda, and Harry Petyt.
\newblock Coarse injectivity, hierarchical hyperbolicity and semihyperbolicity.
\newblock {\em Geom. Topol.}, 27(4):1587--1633, 2023.

\bibitem[HMS24]{MMS:Artin}
Mark Hagen, Alexandre Martin, and Alessandro Sisto.
\newblock Extra-large type {A}rtin groups are hierarchically hyperbolic.
\newblock {\em Math. Ann.}, 388(1):867--938, 2024.

\bibitem[Hod24]{hoda2024strongly}
Nima Hoda.
\newblock Strongly shortcut spaces.
\newblock {\em Algebraic \& Geometric Topology}, 24(6):3291--3325, 2024.

\bibitem[HP23]{HagenPetyt}
Mark~F Hagen and Harry Petyt.
\newblock Projection complexes and quasimedian maps.
\newblock {\em Algebraic \& Geometric Topology}, 22(7):3277--3304, 2023.

\bibitem[HS20]{HS:cubical}
Mark~F. Hagen and Tim Susse.
\newblock On hierarchical hyperbolicity of cubical groups.
\newblock {\em Israel J. Math.}, 236(1):45--89, 2020.

\bibitem[Hu65]{Hu:retracts}
Sze-tsen Hu.
\newblock {\em Theory of retracts}.
\newblock Wayne State University Press, Detroit, MI, 1965.

\bibitem[Hum17]{Hume_andim}
David Hume.
\newblock Embedding mapping class groups into a finite product of trees.
\newblock {\em Groups, Geometry, and Dynamics}, 11(2):613--647, 2017.

\bibitem[Kar08]{Kar_thesis}
Aditi Kar.
\newblock {\em Discrete groups and CAT (0) asymptotic cones}.
\newblock PhD thesis, The Ohio State University, 2008.

\bibitem[Kar11]{Kar_asymp}
Aditi Kar.
\newblock Asymptotically cat (0) groups.
\newblock {\em Publicacions Matem{\`a}tiques}, pages 67--91, 2011.

\bibitem[KL96]{KL:actions}
Michael Kapovich and Bernhard Leeb.
\newblock Actions of discrete groups on nonpositively curved spaces.
\newblock {\em Math. Ann.}, 306(2):341--352, 1996.

\bibitem[LS05]{LangSchlich_Nagata}
Urs Lang and Thilo Schlichenmaier.
\newblock Nagata dimension, quasisymmetric embeddings, and lipschitz extensions.
\newblock {\em International Mathematics Research Notices}, 2005(58):3625--3655, 2005.

\bibitem[L{\"u}c10]{Lueck:ICM}
Wolfgang L{\"u}ck.
\newblock {$K$}- and {$L$}-theory of group rings.
\newblock In {\em Proceedings of the {I}nternational {C}ongress of {M}athematicians. {V}olume {II}}, pages 1071--1098. Hindustan Book Agency, New Delhi, 2010.

\bibitem[Mar14]{martin2014non}
Alexandre Martin.
\newblock Non-positively curved complexes of groups and boundaries.
\newblock {\em Geometry \& Topology}, 18(1):31--102, 2014.

\bibitem[Mor16]{Moran_metric}
Molly~A Moran.
\newblock Metrics on visual boundaries of cat (0) spaces.
\newblock {\em Geometriae Dedicata}, 183:123--142, 2016.

\bibitem[OP09]{osajda2009boundaries}
Damian Osajda and Piotr Przytycki.
\newblock Boundaries of systolic groups.
\newblock {\em Geometry \& Topology}, 13(5):2807--2880, 2009.

\bibitem[Pet21]{petyt2021mapping}
Harry Petyt.
\newblock Mapping class groups are quasicubical.
\newblock {\em to appear in Amer. J. Math.}, 2021.

\bibitem[PZS24a]{petyt2024constructing}
Harry Petyt, Abdul Zalloum, and Davide Spriano.
\newblock Constructing metric spaces from systems of walls.
\newblock {\em arXiv preprint arXiv:2404.12057}, 2024.

\bibitem[PZS24b]{PZ_walls}
Harry Petyt, Abdul Zalloum, and Davide Spriano.
\newblock Constructing metric spaces from systems of walls.
\newblock {\em arXiv preprint arXiv:2404.12057}, 2024.

\bibitem[QR22]{QR_CAT0}
Yulan Qing and Kasra Rafi.
\newblock Sublinearly morse boundary i: Cat (0) spaces.
\newblock {\em Advances in Mathematics}, 404:108442, 2022.

\bibitem[RST23]{russell2023convexity}
Jacob Russell, Davide Spriano, and Hung~Cong Tran.
\newblock Convexity in hierarchically hyperbolic spaces.
\newblock {\em Algebraic \& Geometric Topology}, 23(3):1167--1248, 2023.

\bibitem[Sag95]{Sageev:cubulation}
Michah Sageev.
\newblock Ends of group pairs and non-positively curved cube complexes.
\newblock {\em Proc. London Math. Soc. (3)}, 71(3):585--617, 1995.

\bibitem[Sch]{schwer2324cat}
Petra Schwer.
\newblock Cat (0) cube complexes.
\newblock {\em Lecture Notes in Mathematics}, 2324.

\bibitem[Sis19]{HHS_survey}
Alessandro Sisto.
\newblock What is a hierarchically hyperbolic space?
\newblock In {\em Beyond hyperbolicity}, volume 454 of {\em London Math. Soc. Lecture Note Ser.}, pages 117--148. Cambridge Univ. Press, Cambridge, 2019.

\bibitem[Vie27]{vietoris1927hoheren}
Leopold Vietoris.
\newblock {\"U}ber den h{\"o}heren zusammenhang kompakter r{\"a}ume und eine klasse von zusammenhangstreuen abbildungen.
\newblock {\em Mathematische Annalen}, 97(1):454--472, 1927.

\bibitem[Vir24]{virk2024contractibility}
{\v{Z}}iga Virk.
\newblock Contractibility of the rips complexes of integer lattices via local domination.
\newblock {\em To appear in \emph{Transactions of the American Mathematical Society}; arXiv:2405.09134}, 2024.

\bibitem[Weg12]{Wegner:FJ}
Christian Wegner.
\newblock The {$K$}-theoretic {F}arrell-{J}ones conjecture for {CAT}(0)-groups.
\newblock {\em Proc. Amer. Math. Soc.}, 140(3):779--793, 2012.

\bibitem[Zar22]{Zaremsky_contract}
Matthew~CB Zaremsky.
\newblock Bestvina--brady discrete morse theory and vietoris--rips complexes.
\newblock {\em American Journal of Mathematics}, 144(5):1177--1200, 2022.

\bibitem[Zar24]{zaremsky2024contractible}
Matthew~CB Zaremsky.
\newblock Contractible vietoris-rips complexes of $\mathbb{Z}^n$.
\newblock {\em arXiv preprint arXiv:2410.11993}, 2024.

\end{thebibliography}

\end{document}